\documentclass[reqno,a4paper,draft]{amsart}
\usepackage{enumitem}

\setenumerate{label=\textnormal{(\arabic*)}}

\usepackage{amsmath,amssymb,dsfont,verbatim,mathtools,bm,geometry,fge,endnotes}

\usepackage{booktabs}
\newcommand{\ra}[1]{\renewcommand{\arraystretch}{#1}}

\usepackage[latin1]{inputenc}
\usepackage[raggedright]{titlesec}
\usepackage{tikz}
\usepackage{float,subfig}
\usepackage{amsrefs}

\titleformat{\chapter}[display]
{\normalfont\huge\bfseries}{\chaptertitlename\\thechapter}{20pt}{\Huge}
\titleformat{\section}
{\normalfont\Large\bfseries\center}{\thesection}{1em}{}
\titleformat{\subsection}
{\normalfont\large\bfseries}{\thesubsection}{1em}{}
\titleformat{\subsubsection}[runin]
{\normalfont\normalsize\bfseries}{\thesubsubsection}{1em}{}
\titleformat{\paragraph}[runin]
{\normalfont\normalsize\bfseries}{\theparagraph}{1em}{}
\titleformat{\subparagraph}[runin]
{\normalfont\normalsize\bfseries}{\thesubparagraph}{1em}{}
\titlespacing*{\chapter} {0pt}{50pt}{40pt}
\titlespacing*{\section} {0pt}{3.5ex plus 1ex minus .2ex}{2.3ex plus .2ex}
\titlespacing*{\subsection} {0pt}{3.25ex plus 1ex minus .2ex}{1.5ex plus .2ex}
\titlespacing*{\subsubsection}{0pt}{3.25ex plus 1ex minus .2ex}{1.5ex plus .2ex}
\titlespacing*{\paragraph} {0pt}{3.25ex plus 1ex minus .2ex}{1em}
\titlespacing*{\subparagraph} {\parindent}{3.25ex plus 1ex minus .2ex}{1em}


\newcommand{\hs}{\hspace{-0.5pt}}

\def\xcirc{\hs\circ\hs}

\newtheorem{theorem}{Theorem}[section]
\newtheorem{lemma}[theorem]{Lemma}
\newtheorem{proposition}[theorem]{Proposition}
\newtheorem{corollary}[theorem]{Corollary}

\theoremstyle{definition}
\newtheorem{definition}[theorem]{Definition}

\newtheorem{notation}[theorem]{Notation}

\theoremstyle{remark}
\newtheorem{remark}[theorem]{Remark}

\DeclareMathOperator{\Aut}{Aut}

\DeclareMathOperator{\CH}{CH}

\DeclareMathOperator{\J}{J}

\DeclareMathOperator{\ide}{id}

\DeclareMathOperator{\Supp}{Supp}

\DeclareMathOperator{\en}{en}

\DeclareMathOperator{\factors}{factors}

\DeclareMathOperator{\st}{st}

\DeclareMathOperator{\dir}{dir}

\DeclareMathOperator{\Succ}{Succ}
\DeclareMathOperator{\Pred}{Pred}
\DeclareMathOperator{\Dir}{Dir}
\DeclareMathOperator{\HH}{H}
\DeclareMathOperator{\linspan}{linspan}
\DeclareMathOperator{\lcm}{lcm}

\newcommand{\ov}{\overline}

\newcommand{\wt}{\widetilde}

\begin{document}

\title{On the shape of possible counterexamples to the Jacobian Conjecture}

\author{Jorge A. Guccione}
\address{Departamento de Matem\'atica\\ Facultad de Ciencias Exactas y Naturales-UBA,
Pabell\'on~1-Ciudad Universitaria\\ Intendente Guiraldes 2160 (C1428EGA) Buenos Aires, Argentina.}
\address{Instituto de Investigaciones Matem\'aticas ``Luis A. Santal\'o"\\ Facultad de
Ciencias Exactas y Natu\-rales-UBA, Pabell\'on~1-Ciudad Universitaria\\
Intendente Guiraldes 2160 (C1428EGA) Buenos Aires, Argentina.}
\email{vander@dm.uba.ar}

\author{Juan J. Guccione}
\address{Departamento de Matem\'atica\\ Facultad de Ciencias Exactas y Naturales-UBA\\
Pabell\'on~1-Ciudad Universitaria\\ Intendente Guiraldes 2160 (C1428EGA) Buenos Aires, Argentina.}
\address{Instituto Argentino de Matem\'atica-CONICET\\ Saavedra 15 3er piso\\
(C1083ACA) Buenos Aires, Ar\-gentina.}
\email{jjgucci@dm.uba.ar}

\thanks{Jorge A. Guccione and Juan J. Guccione were supported by PIP 112-200801-00900 (CONICET)}

\author{Christian Valqui}
\address{Pontificia Universidad Cat\'olica del Per\'u, Secci\'on Matem\'aticas, PUCP,
Av. Universitaria 1801, San Miguel, Lima 32, Per\'u.}

\address{Instituto de Matem\'atica y Ciencias Afines (IMCA) Calle Los Bi\'ologos 245. Urb San C\'esar.
La Molina, Lima 12, Per\'u.}
\email{cvalqui@pucp.edu.pe}

\thanks{Christian Valqui was supported by PUCP-DGI-2011-0206, PUCP-DGI-2012-0011 and Lucet 90-DAI-L005.}

\begin{abstract} We improve the algebraic methods of Abhyankar for the Jacobian Conjecture in dimension two and describe the shape of possible
counterexamples. We give an elementary proof of the result of Heitmann in~\cite{H}, which states that $\gcd(\deg(P),\deg(Q))\ge 16$ for any
counterexample $(P,Q)$. We also prove that $\gcd(\deg(P),\deg(Q))\ne 2p$ for any prime $p$ and analyze thoroughly the case $16$, adapting a
reduction of degree technique introduced by Moh.
\end{abstract}

\maketitle

\section*{Introduction} Let $K$ be a field of characteristic zero.
The Jacobian Conjecture (JC) in dimension two, stated by Keller in~\cite{K},
says that any pair of polynomials $P,Q\in L:= K[x,y]$ with $[P,Q]:=\partial_x P \partial_y Q -
\partial_x Q \partial_y P\in K^{\times}$ defines an automorphism of $K[x,y]$.

In this paper we improve the algebraic methods of Abhyankar describing the shape of the support
of possible counterexamples. We use elementary algebraic methods combined with basic discrete analytic
geometry on the plane, i.e. on the points $\mathds{N}_0\times\mathds{N}_0$ in the case of
$L=K[x,y]$ and in $\frac 1l\mathds{Z}\times\mathds{N}_0$ in the case of $L^{(l)}:=K[x^{\pm \frac{1}{l}},y]$.

The first innovation is a definition of the directions and an order relation on them, based
on the crossed product of vectors, which simplifies substantively the treatment of consecutive
directions associated with the Newton polygon of Jacobian pairs. It is related to~\cite{H}*{Lemma~1.15}
and enables us to simplify substantially the treatment of the Newton polygon and its edges (compare with~\cite{A2}*{7.4.14}).

The second innovation lies in the use of the polynomial $F$ with
$[F,\ell_{\rho,\sigma}(P)]=\ell_{\rho,\sigma}(P)$, obtained in Theorem~\ref{central} for a given
Jacobian pair $(P,Q)$. This element can be traced back to 1975 in~\cite{J}. There also appears
the element $G_0\in K[P,Q]$, which becomes important in the proof of our
Proposition~\ref{proporcionalidad de direcciones mayores}. The polynomial
$F$ mentioned above is well known and used by many authors, see for
example~\cite{J}, \cite{ML} and \cite{vdE}*{10.2.8} (together with \cite{vdE}*{10.2.17 i)}).
In Theorem~\ref{central}, we
add some geometric statements on the shape of the supports, especially
about the endpoints (called $\st$ and $\en$) associated to an edge of
the Newton Polygon. In~\cite{H}*{Proposition~1.3} some of these statements,
presented in an algebraic form, can be found.

We will apply different endomorphisms in order to deform the support of a Jacobian pair. Opposed
to most of the authors working in this area (\cite{H}, \cite{Z}, \cite{M}), we remain all the
time in $L$ (or $L^{(l)}$). In order to do this we use the following very simple expression of
the change of the Jacobian under an endomorphism $\varphi\colon L\to L$ (or $ L\to L^{(l)}$,
or $L^{(l)}\to L^{(l)}$):
$$
[\varphi(P),\varphi(Q)]=\varphi([P,Q])[\varphi(x),\varphi(y)].
$$

Another key ingredient is the concept of regular corners and its classification, which we
present in Section~5. The geometric fact
that certain edges can be cut above the diagonal, Proposition~\ref{case IIb}, was already known
to Joseph and used in~\cite{J}*{Theorem~4.2}, in order to prove the polarization theorem.

In Section~6 we give an elementary proof of a result of~\cite{H}:
If
$$
B := \begin{cases}\infty & \text{if the jacobian conjecture is true,}\\
\min\bigl(\gcd(v_{1,1}(P),v_{1,1}(Q))\bigr)&\text{if it is false, where $(P,Q)$ runs on the counterexamples,}
\end{cases}
$$
then $B\ge 16$. In spite of Heitmann's
assertion ``Nothing like this appears in the literature but results of this type are known
by Abhyankar and Moh and are easily inferred from their published work'', referring to his
result, we do not know how to do this, and we did not find anything like this in the literature
till now. For example, in the survey papers~\cite{A2} and~\cite{A3}, this result is not mentioned,
although in~\cite{A3}*{Corollary~8.9} it is proven that $B \ge 9$.

In Section~7 we present our main new results:
Propositions~\ref{proporcionalidad de direcciones mayores} and~\ref{proporcionalidad de direcciones menores}
and Corollaries~\ref{fracciones de F} and~\ref{fracciones de F1}. At first sight they look
rather technical, but are related to the fact that for a Jacobian pair $(P,Q)$ in $K[x,y]$
we know that $P$ and $Q$ are star symmetric.
Propositions~\ref{proporcionalidad de direcciones mayores} and~\ref{proporcionalidad de direcciones menores}
yield partial star symmetries between elements in $K[P,Q]$ and $P$, whereas
Corollaries~\ref{fracciones de F} and~\ref{fracciones de F1} guarantee that
the leading forms of $P$ associated with certain directions can be written as powers of certain
polynomials. This
allows us to establish a very strong divisibility criterion for the possible
regular corners, Theorem~\ref{divisibilidad}, which enables us to
prove that $B=16$ or $B>20$, a result that is consistent with~\cite{H}*{Theorem~2.24}, where the two smallest
possible corners are $(D,E)=(12,4)$ and $(D,E)=(15,6)$, and $D+E$ gives the $\gcd$ of the total
degrees of $P$ and $Q$.

We also prove that $B\ne 2p$ for all prime $p$. This result is announced
to be proven by Abhyankar in a remark after~\cite{H}*{Theorem~1.16}, and it is said that it can be
proven similarly to~\cite{H}*{Proposition~2.21}. However, we could not translate the proof
of~\cite{H}*{Proposition~2.21} to our setting nor modify it to give $B\ne 2p$. Once again in the survey
articles~\cite{A2} and~\cite{A3}, this result is not mentioned, although in~\cite{A2}*{Lemma~6.1} it is
proven that $\gcd(\deg(P),\deg(Q))\ne p$. We also found~\cite{Z}*{Theorem~4.12} from which $B\ne 2p$
follows. But the proof relies on~\cite{Z}*{Lemma~4.10}, which has a gap, since it claims without proof
that $I_2 \subset \frac 1m \Gamma(f_2)$, an assertion which cannot be proven to be true. The same
article claims to have proven that $B>16$, and the author claims to have verified that $B>33$, but
it relies on the same flawed argument, so $B\ge 16$ remains up to the moment the best lower limit
for $B$.

One part of our strategy is described by~\cite{H}: ``The underlying strategy is the
minimal counterexample approach. We assume the Jacobian conjecture is false and derive
properties which a minimal counterexample must satisfy. The ultimate goal is either a
contradiction (proving the conjecture) or an actual counterexample.'' Actually this is
the strategy followed by Moh in~\cite{M}, who succeeded in proving that for a
counterexample $(P,Q)$, $\max(\deg(P),\deg(Q))>100$. The trouble of this strategy
is that the number of equations and variables one has to solve in order to discard
the possible counterexamples, grows rapidly, and the brute force approach with computers
gives no conceptual progress, although it allows us to increase the lower bound for $\max(\deg(P),\deg(Q))$.

The approach followed in~\cite{H} is more promising, since every possible $B$ ruled out
actually eliminates a whole infinite family of possible counterexamples and cannot be
achieved by computer power. The possible counterexample at $B=16$ is still within reach,
and in Section 8 we give a very detailed description of its shape, after reducing the
degrees, following essentially the same strategy of Moh in~\cite{M}, who does it in the
particular case $m=3$, $n=4$. This shows the advantage of the present method compared
to~\cite{H}, where the author says ``we have no nice way to handle these cases", referring
to the exceptional cases found by Moh.

Using the classification of regular corners we can produce the algebraic data corresponding to a
resolution at infinity, and these data are strongly related to the shape of a possible
counterexample. It would be interesting to describe thoroughly the relation between the
algebraic and topological methods used in the different approaches mentioned above.

The results in the first six sections of this paper are analogous to those established for
the one dimensional Dixmier conjecture in~\cite{G-G-V1}. The first section is just a
reminder of definitions and properties from~\cite{G-G-V1}. In Section~2 we give an
improved version of the analogous results in that paper, the main difference being
the proof of the existence of $G_0$ in Theorem~\ref{central} and Proposition~\ref{pavadass}(5).
In section~3 we recall some of the results of~\cite{G-G-V1} about the order of directions.
At the beginning of Section~4 we introduce the concept of a minimal pair and
prove that a minimal pair can be assumed to have a trapezoidal shape.

The results corresponding to Proposition~5.3 of~\cite{G-G-V1} now are distributed along
various propositions that classify regular corners in section~5.
In section~6 we obtain the fact that $B\ge 16$, in the same way as the corresponding
result in~\cite{G-G-V1}. The rest of the results
in this paper are new.

We point out that the proofs that $B=16$ or $B>20$ and that $B\ne 2p$ for any prime number~$p$
 can be adapted easily to the case of the Dixmier conjecture.

\section{Preliminaries}\label{preliminares} We recall some notations and properties from~\cite{G-G-V1}.
For each $l\in \mathds{N}$, we consider the commutative $K$-algebra $L^{(l)}$, generated by variables
$x^{\frac{1}{l}}$, $x^{\frac{-1}{l}}$ and $y$, subject to the relation $x^{\frac{1}{l}} x^{\frac{-1}{l}} = 1$.
In other words $L^{(l)} := K[x^{\frac{1}{l}},x^{\frac{-1}{l}},y]$. Obviously, there is a canonical
inclusion $L^{(l)}\subseteq L^{(h)}$, for each $l,h\in \mathds{N}$ such that $l|h$. We define the set of
directions by
\begin{equation*}
\mathfrak{V} := \{(\rho,\sigma)\in \mathds{Z}^2: \text{$\gcd(\rho,\sigma) = 1$} \}.
\end{equation*}
We also define
\begin{align*}
&\mathfrak{V}_{\ge 0} := \{(\rho,\sigma)\in \mathfrak{V}:\rho+\sigma\ge 0\},\\
&\mathfrak{V}_{>0} := \{(\rho,\sigma)\in \mathfrak{V}:\rho+\sigma>0\}
\shortintertext{and}
&\mathfrak{V}^0 := \{(\rho,\sigma)\in \mathfrak{V}:\rho+\sigma>0\text{ and }\rho>0\}.
\end{align*}
Note that $\mathfrak{V}_{\ge 0} = \mathfrak{V}_{>0}\cup \{(1,-1),(-1,1)\}$.

\begin{definition} For all $(\rho,\sigma)\in \mathfrak{V}$ and $(i/l,j)\in \frac{1}{l}\mathds{Z}\times
\mathds{Z}$ we write $v_{\rho,\sigma}(i/l,j):= \rho i/l+\sigma j$.
\end{definition}

\begin{definition}\label{not valuaciones para polinomios} Let $(\rho,\sigma)\in \mathfrak{V}$. For
$P = \sum a_{\frac{i}{l},j} x^{\frac{i}{l}} y^j\in L^{(l)}\setminus\{0\}$, we define:

\begin{itemize}

\smallskip

\item[-] The {\em support} of $P$ as $\Supp(P) := \left\{\left(i/l,j\right): a_{\frac{i}{l},j}\ne
    0\right\}$.

\smallskip

\item[-] The {\em $(\rho,\sigma)$-degree} of $P$ as $v_{\rho,\sigma}(P):= \max\left\{ v_{\rho,\sigma}
    (i/l,j): a_{\frac{i}{l},j} \ne 0\right\}$.

\smallskip

\item[-] The {\em $(\rho,\sigma)$-leading term} of $P$ as $\ell_{\rho,\sigma}(P):= \displaystyle
    \sum_{\{\rho \frac{i}{l} + \sigma j = v_{\rho,\sigma}(P)\}} a_{\frac{i}{l},j} x^{\frac{i}{l}} y^j$.

\smallskip

\end{itemize}
\end{definition}

\begin{remark} To abbreviate expressions we set $v_{\rho,\sigma}(0) := -\infty$ and
$\ell_{\rho,\sigma}(0) := 0$, for all $(\rho,\sigma)\in \mathfrak{V}$. Moreover, instead of $\Supp(P) = \{a\}$
we will write $\Supp(P) = a$.
\end{remark}

\begin{definition}\label{not elementos rho-sigma homogeneos} We say that $P\in L^{(l)}$ is {\em
$(\rho,\sigma)$-homogeneous} if $P = \ell_{\rho,\sigma}(P)$.
\end{definition}

\begin{definition} We assign to each direction its corresponding unit vector in $S^1$, and we define an
{\em interval} in $\mathfrak{V}$ as the preimage under this map of an arc of $S^1$ that is not the whole
circle. We consider each interval endowed with the order that increases counterclockwise.
\end{definition}

For each $P\in L^{(l)}\setminus \{0\}$, we let $\HH(P)$ denote the convex hull of the support of $P$. As it is
well known, $\HH(P)$ is a polygon, called the {\em Newton polygon of $P$}, and it is evident that each one of
its edges is the convex hull of the support of $\ell_{\rho,\sigma}(P)$, where $(\rho,\sigma)$ is orthogonal to the
given edge and points outside of $\HH(P)$.

\begin{notation}\label{Comienzo y Fin de un elemento de W} Let $(\rho,\sigma)\!\in\! \mathfrak{V}$ arbitrary.
We let $\st_{\rho,\sigma}(P)$ and $\en_{\rho,\sigma}(P)$ denote the first and the last point that we find on
$H(\ell_{\rho,\sigma}(P))$ when we run counterclockwise along the boundary of $H(P)$. Note that these points
coincide when $\ell_{\rho,\sigma}(P)$ is a monomial\endnote{In the case in which $P$ is
$(\rho,\sigma)$-homogeneous but is not a monomial this definition is a little confusing.
Remark~\ref{starting and end with cross} clarifies this situation.
}.
\end{notation}

The {\em cross product} of two vectors $A = (a_1,a_2)$ and $B = (b_1,b_2)$ in $\mathds{R}^2$ is
$A\times B:=\det\left(\begin{smallmatrix} a_1 & a_2\\ b_1 & b_2 \end{smallmatrix}\right)$.

\begin{remark}\label{starting and end with cross}
Note that if $\ell_{\rho,\sigma}(P)$ is not a monomial, then $(\rho,\sigma)\times (\en_{\rho,\sigma}(P)
-\st_{\rho,\sigma}(P))>0$\endnote{This property, combined with the fact that $\st_{\rho,\sigma}(P)$ and
$\en_{\rho,\sigma}(P)$ are the end points of $H(\ell_{\rho,\sigma}(P))$, determines  $\st_{\rho,\sigma}(P)$
and $\en_{\rho,\sigma}(P)$.
}.
\end{remark}

\begin{remark}\label{starting vs end}
If $(\rho_0,\sigma_0)<(\rho,\sigma)<(-\rho_0,-\sigma_0)$, then $v_{\rho_0,\sigma_0}(\en_{\rho,\sigma}(P))\le
v_{\rho_0,\sigma_0}(\st_{\rho,\sigma}(P))$, while if $(\rho_0,\sigma_0)>(\rho,\sigma)>(-\rho_0,-\sigma_0)$,
then $v_{\rho_0,\sigma_0}(\en_{\rho,\sigma}(P))\ge
v_{\rho_0,\sigma_0}(\st_{\rho,\sigma}(P))$,
 with equality in both cases only if $\ell_{\rho,\sigma}(P)$ is a monomial. Moreover,
 in the first case
\begin{equation*}
\st_{\rho,\sigma}(P) = \Supp(\ell_{\rho_0,\sigma_0}(\ell_{\rho,\sigma}(P)))\qquad\text{and}\qquad
\en_{\rho,\sigma}(P) = \Supp(\ell_{-\rho_0,-\sigma_0}(\ell_{\rho,\sigma}(P))).
\end{equation*}
Hence, if $(\rho,\sigma)\in \mathfrak{V}_{>0}$, then
$$
\st_{\rho,\sigma}(P) = \Supp(\ell_{1,-1}(\ell_{\rho,\sigma}(P))) \qquad\text{and}\qquad
\en_{\rho,\sigma}(P) = \Supp(\ell_{-1,1}(\ell_{\rho,\sigma}(P))),
$$
and, if $\rho+\sigma<0$, then
$$
\st_{\rho,\sigma}(P) = \Supp(\ell_{-1,1} (\ell_{\rho,\sigma}(P)))\qquad
\text{and}\qquad \en_{\rho,\sigma}(P) = \Supp(\ell_{1,-1}(\ell_{\rho,\sigma}(P))).
$$
\end{remark}

\begin{remark}\label{pr v de un producto} Let $P,Q\in L^{(l)}\setminus\{0\}$ and
$(\rho,\sigma)\in  \mathfrak{V}$. The following assertions hold:
\begin{enumerate}

\smallskip

\item $\ell_{\rho,\sigma}(PQ) = \ell_{\rho,\sigma}(P)\ell_{\rho,\sigma}(Q)$.

\smallskip

\item If $P = \sum_i P_i$, $v_{\rho,\sigma}(P_i)=v_{\rho,\sigma}(P)$ and $\sum_i \ell_{\rho,\sigma}(P_i)\ne 0$, then $\ell_{\rho,\sigma}(P) = \sum_i
    \ell_{\rho,\sigma}(P_i)$.

\smallskip

\item $v_{\rho,\sigma}(PQ) = v_{\rho,\sigma}(P) + v_{\rho,\sigma}(Q)$.

\smallskip

\item $\st_{\rho,\sigma}(PQ) = \st_{\rho,\sigma}(P)+\st_{\rho,\sigma}(Q)$.

\smallskip

\item $\en_{\rho,\sigma}(PQ)= \en_{\rho,\sigma}(P)+\en_{\rho,\sigma}(Q)$.

\smallskip

\item $-v_{-\rho,-\sigma}(P)\le v_{\rho,\sigma}(P)$.

\end{enumerate}
\end{remark}
We will use freely these facts throughout the article.

\begin{notation}\label{jacobiano} For $P,Q\in L^{(l)}$ we write $[P,Q] := \det \J(P,Q)$, where $\J(P,Q)$
is the jacobian matrix of $(P,Q)$.
\end{notation}

\begin{definition} Let $P,Q\in L^{(l)}$. We say that $(P,Q)$ is a {\em Jacobian pair} if $[P,Q]\in K^{\times}$.
\end{definition}

\begin{remark}\label{re v de un conmutador} Let $P,Q\in L^{(l)}\setminus\{0\}$ and let $(\rho,\sigma)\in
\mathfrak{V}$. We have:

\begin{enumerate}

\smallskip

\item If $P$ and $Q$ are $(\rho,\sigma)$-homogeneous, then $[P,Q]$ is also. If moreover $[P,Q] \ne 0$,
then
$$
v_{\rho,\sigma}([P,Q]) = v_{\rho,\sigma}(P) + v_{\rho,\sigma}(Q) - (\rho+\sigma).
$$

\smallskip

\item If $P=\sum_{i} P_{i}$ and $Q=\sum_{j} Q_{j}$ are the $(\rho,\sigma)$-homogeneous decompositions
of $P$ and~$Q$, then the$(\rho,\sigma)$-homogeneous decomposition $[P,Q]=\sum_k [P,Q]_k$ is given by
\begin{equation}\label{descomposicion}
[P,Q]_k=\sum_{i+j=k+\rho+\sigma} [P_i,Q_j].
\end{equation}

\smallskip

\item If $[P,Q]=0$, then $[\ell_{\rho,\sigma}(P),\ell_{\rho,\sigma}(Q)]=0$.
\end{enumerate}
\end{remark}

\begin{proposition}\label{pr v de un conmutador} Let $P,Q\in L^{(l)}\setminus \{0\}$ and $(\rho,\sigma)\in
\mathfrak{V}$. We have
\begin{equation}
v_{\rho,\sigma}([P,Q])\le v_{\rho,\sigma}(P) + v_{\rho,\sigma}(Q) - (\rho+\sigma).\label{dfpvc}
\end{equation}
Moreover
$$
v_{\rho,\sigma}([P,Q])= v_{\rho,\sigma}(P) + v_{\rho,\sigma}(Q) - (\rho+\sigma)\Longleftrightarrow
[\ell_{\rho,\sigma}(P),\ell_{\rho,\sigma}(Q)]\ne 0
$$
and in this case $[\ell_{\rho,\sigma}(P),\ell_{\rho, \sigma}(Q)] = \ell_{\rho,\sigma}([P,Q])$.
\end{proposition}

\begin{proof} It follows directly from the decomposition~\eqref{descomposicion}.\end{proof}

\begin{definition}\label{def vectors alineados} We say that two vectors $A,B\in\mathds{R}^2$ are
{\em aligned} and write $A\sim B$, if $A\times B = 0$.
\end{definition}

\begin{remark} \label{relacion de equivalencia}
Note that the restriction of $\sim$ to $\mathds{R}^2\setminus\{0\}$ is an equivalence
relation. Note also that if $A\in \mathds{R}\times \mathds{R}_{>0}$, $B\in \mathds{R}\times \mathds{R}_{\ge 0}$
and $A\sim B$, then $B=\lambda A$ for some $\lambda \ge 0$.
\end{remark}

\section{Shape of Jacobian pairs}
\setcounter{equation}{0}

The results in this section appear in several papers, for instance~\cite{A}, \cite{H} and \cite{J},
but we need to establish them in a slightly different form, including the geometric information about
the shape of the support.

\begin{proposition}\label{P y Q alineados} Let $(\rho,\sigma)\in \mathfrak{V}$ and let $P,Q\in
L^{(l)}\setminus\{0\}$ be two $(\rho,\sigma)$-homogeneous elements. Set $\tau:=v_{\rho,\sigma}(P)$ and
$\mu:=v_{\rho,\sigma}(Q)$.

\begin{enumerate}

\smallskip

\item If $\tau=\mu=0$, then $[P,Q] = 0$.

\smallskip

\item Assume that $[P,Q]\! =\! 0$ and $(\mu,\tau)\!\ne\! (0,0)$. Let $m,n\!\in\!\mathds{Z}$ with
    $\gcd(m,n)\!=\!1$ and $n\tau\! =\! m\mu$. Then
\begin{enumerate}

\smallskip

\item[a)]  There exists $\alpha\in K^{\times}$ such that $P^n = \alpha Q^m$.

\smallskip

\item[b)] There exist $R\in L^{(l)}$ and $\lambda_P,\lambda_Q\in K^\times$, such that
\begin{equation}
P = \lambda_P R^m\quad\text{and}\quad Q = \lambda_Q R^n.\label{chvh5}
\end{equation}
Moreover

\begin{itemize}

\smallskip

\item[-] if $\mu\tau<0$, then $P,Q\in K[x^{1/l},x^{-1/l}]$,

\smallskip

\item[-] if $\mu\tau\ge 0$, then we can choose $m,n\in \mathds{N}_0$,

\smallskip

\item[-] if $P,Q\in L$, then $R\in L$.

\end{itemize}

\end{enumerate}

\end{enumerate}

\end{proposition}

\begin{proof} (1)\enspace If $\rho=0$, then $P,Q\in K[x^{1/l},x^{-1/l}]$ and if $\rho\ne 0$, then
$P,Q\in K[z]$ where $z:=x^{-\sigma/\rho}y$. In both cases, $[P,Q]=0$ follows easily.

\smallskip

\noindent (2a)\enspace This is~\cite{J}*{Proposition~2.1(2)}.

\smallskip

\noindent (2b)\enspace Assume first that $\mu\tau<0$ and take $n,m\in\mathds{Z}$ coprime with $n\tau =m\mu$.
By statement~(a), there exists $\alpha\in K^{\times}$ such that $P^n = \alpha Q^m$. Since $mn<0$, necessarily
$P,Q\in K[x^{1/l},x^{-1/l}]$ and $\rho\ne 0$. Moreover, since $P$ and $Q$ are $(\rho,\sigma)$-homogeneous,
$$
P=\lambda_P x^{\frac{r}{l}}\quad\text{and}\quad Q=\lambda_Q x^{\frac{u}{l}},
$$
for some $\lambda_P,\lambda_Q\in K^\times$ and $r,u\in\mathds{Z}$ with $rn = um$. Clearly
$R:=x^{\frac{r}{lm}}=x^{\frac{u}{ln}}$ satisfies~\eqref{chvh5}. In order to finish the proof we only must note
that, since $m$ and $n$ are coprime, $R\in L^{(l)}$.

\smallskip

Assume now $\mu\tau\ge 0$ and let $m,n\in\mathds{N}_0$ be such that $n\tau = m\mu$ and $\gcd(m,n) = 1$. Set
$$
z := \begin{cases} x^{-\frac{\sigma}{\rho}} y & \text{if $\rho\ne 0$,}\\ x^{\frac{1}{l}} & \text{if
$\rho=0$,}\end{cases}
$$
and write
$$
P = x^{\frac{r}{l}}y^s f(z)\quad\text{and}\quad Q = x^{\frac{u}{l}}y^v g(z),
$$
where $f,g\in K[z]$ with $f(0)\ne 0\ne g(0)$. By statement~(2a) there exists $\alpha\in K^{\times}$ such that
$$
x^{\frac{nr}{l}}y^{ns} f^n(z) = P^n = \alpha Q^m = \alpha x^{\frac{mu}{l}}y^{mv} g^m(z),
$$
from which we obtain
$$
(nr/l,ns) = (mu/l,mv)\quad\text{and}\quad f^n = \alpha g^m.
$$
Since $\gcd(m,n) = 1$, by the second equality there exist $h\in K[z]$ and $\lambda_P,\lambda_Q\in K^{\times}$
such that
\begin{equation}
f = \lambda_P h^m(z)\quad\text{and}\quad g = \lambda_Q h^n(z)\label{chvh1}
\end{equation}
Take $c,d\in \mathds{Z}$ such that $cm+dn = 1$ and define $(a/l,b):= c (r/l,s) + d (u/l,v)$. Since
$$
m(a/l,b)=(r/l,s)\quad\text{and}\quad n(a/l,b)\!=\!(u/l,v),
$$
it follows from~\eqref{chvh1}, that $R:=x^{\frac{a}{l}}y^b h(z)$ satisfies~\eqref{chvh5}, as desired.

\smallskip

Finally, if $P,Q\!\in\! L$, then $v_{-1,0}(R)\! =\! \frac{1}{m} v_{-1,0}(P)\! \le\! 0$, which combined with the
fact that $R\!\in\! L^{(1)}$ implies that $R\!\in \! L$.
\end{proof}

\begin{lemma}\label{Lema Joseph} Let $(\rho,\sigma)\in \mathfrak{V}$ and let $P,Q\in L^{(l)} \setminus
\{0\}$ be such that $[P,Q]\in K^{\times}$. If $v_{\rho,\sigma}(P)\ne 0$, then there exists $G_0\in K[P,Q]$ such that
$$
[\ell_{\rho,\sigma}(G_0),\ell_{\rho,\sigma}(P)]\ne 0\quad\text{and}\quad
[[\ell_{\rho,\sigma}(G_0),\ell_{\rho,\sigma}(P)],\ell_{\rho,\sigma}(P)]=0.
$$
Moreover, if we define recursively $G_i:=[G_{i-1},P]$, then
$[\ell_{\rho,\sigma}(G_i),\ell_{\rho,\sigma}(P)]=0$ for $i\ge 1$.
\end{lemma}

\begin{proof} Let $t\in\mathds{N}$ and set
$$
M(t):=\linspan \{P^iQ^j:i,j=0, \dots,t\}.
$$
Since $\{P^iQ^j\}$ is linearly in\-dependent, we have
\begin{equation}
\dim M(t)=(t+1)^2.\label{chvh2}
\end{equation}
On the other hand, a direct computation shows that
$$
mt\le -v_{-\rho,-\sigma}(z)\le v_{\rho,\sigma}(z)\le Mt\quad\text{for each $z\in M(t)$,}
$$
where
\begin{align*}
&m:=\min\{0,-v_{-\rho,-\sigma}(P),-v_{-\rho,-\sigma}(Q),-v_{-\rho,-\sigma}(P)-v_{-\rho,-\sigma}(Q)\}\\
\shortintertext{and}
&M:= \max\{0,v_{\rho,\sigma}(P),v_{\rho,\sigma}(Q),v_{\rho,\sigma}(P) +v_{\rho,\sigma}(Q)\}.
\end{align*}
Consequently,
$$
J:=v_{\rho,\sigma}(M(t)\setminus\{0\})\subseteq \frac{1}{l}\mathds{Z} \cap [mt,Mt].
$$
For each $\beta\in J$ we take a $z_{\beta}\in M(t)$ with $v_{\rho,\sigma}(z_{\beta}) = \beta$. We first prove
that there exist $t\in \mathds{N}$ and $H\in M(t)$, such that
\begin{equation}
[\ell_{\rho,\sigma}(P),\ell_{\rho,\sigma}(H)]\ne 0.\label{chvh3}
\end{equation}
Assume by contradiction that
$$
[\ell_{\rho,\sigma}(P),\ell_{\rho,\sigma}(H)]=0\qquad\text{for all $H\in M(t)$ and all $t\in \mathds{N}$.}
$$
We claim that then $M(t) = \linspan\{z_{\beta}:\beta\in J\}$. In fact, suppose this equality is false and take
$z\in M(t)\setminus \linspan \{z_{\beta}:\beta\in J\}$ with $\beta:= v_{\rho,\sigma}(z)$ minimum. By assumption
$$
[\ell_{\rho,\sigma}(P),\ell_{\rho,\sigma}(z)] = 0 = [\ell_{\rho,\sigma}(P), \ell_{\rho,\sigma}(z_{\beta})].
$$
Since $v_{\rho,\sigma}(P) \ne 0$ and $v_{\rho,\sigma}(z) = v_{\rho,\sigma} (z_{\beta})$, by
Proposition~\ref{P y Q alineados}(2b) there exist $R\in L^{(l)}$, $\lambda,\lambda_{\beta}\in K^\times$ and
$n\in\mathds{Z}$, such that
$$
\ell_{\rho,\sigma}(z) = \lambda R^n\quad\text{and}\quad \ell_{\rho,\sigma}(z_{\beta}) = \lambda_{\beta} R^n.
$$
Hence $v_{\rho,\sigma}(z-\lambda\lambda_{\beta}^{-1}z_{\beta})< v_{\rho,\sigma}(z)$, which contradicts the
choice of $z$, finishing the proof of the claim. Consequently $\dim M(t)\le l(M-m)t$, which
contradicts~\eqref{chvh2} if we take $t\ge l(M-m)$. Thus we can find $H\in K[P,Q]$ such that~\eqref{chvh3} is
satisfied.

We now define recursively $(H_j)_{j\ge 0}$ by setting
$$
H_0:=H,\quad\text{and}\quad H_{j+1}:=[H_j,P].
$$
Since $H_0\in K[P,Q]$, eventually $H_n=0$. Let $k$ be the largest index for which $H_k\ne0$. By
Remark~\ref{re v de un conmutador}(3) we know that $[\ell_{\rho,\sigma}(H_k),\ell_{\rho,\sigma}(P)]=0$.
But we also have $[\ell_{\rho,\sigma}(H_0),\ell_{\rho,\sigma}(P)]\ne 0$ and hence there exists a largest
$j$ such that $[\ell_{\rho,\sigma}(H_j),\ell_{\rho,\sigma}(P)]\ne 0$. By
Proposition~\ref{pr v de un conmutador} we have
$$
[\ell_{\rho,\sigma}(H_j),\ell_{\rho,\sigma}(P)]=\ell_{\rho,\sigma}(H_{j+1}),
$$
and so $G_0:=H_j$ satisfies the required conditions.
\end{proof}

\begin{proposition}\label{extremosalineados} Let $P,Q\in L^{(l)}\setminus \{0\}$ and $(\rho,\sigma)\in
\mathfrak{V}$. If $[\ell_{\rho,\sigma}(P), \ell_{\rho,\sigma}(Q)]=0$, then
$$
\st_{\rho,\sigma}(P)\sim\st_{\rho,\sigma}(Q)\quad\text{and}\quad
\en_{\rho,\sigma}(P)\sim\en_{\rho,\sigma}(Q).
$$
\end{proposition}

\begin{proof} Consider $(\rho_0,\sigma_0)$ such that $(\rho_0,\sigma_0)<(\rho,\sigma)<(-\rho_0,-\sigma_0)$. By
Remark~\ref{re v de un conmutador}(3),
$$
0=[\ell_{(\rho_0,\sigma_0)}(\ell_{\rho,\sigma}(P)),\ell_{(\rho_0,\sigma_0)}(\ell_{\rho,\sigma}(Q))].
$$
On the other hand, by Remark~\ref{starting vs end} there exist $\mu_P,\mu_Q\in K^{\times}$ such that
$$
\ell_{\rho_0,\sigma_0}(\ell_{\rho,\sigma}(P)) = \mu_Px^{\frac{r}{l}}y^s\quad\text{and}\quad
\ell_{\rho_0,\sigma_0}(\ell_{\rho,\sigma}(Q)) = \mu_Q x^{\frac{u}{l}}y^v,
$$
where $(r/l,s)=\st_{\rho,\sigma}(P)$ and $(u/l,v)=\st_{\rho,\sigma}(Q)$. Clearly
$$
0=[\ell_{\rho_0,\sigma_0}(\ell_{\rho,\sigma}(P)),\ell_{\rho_0,\sigma_0}(\ell_{\rho,\sigma}(Q))] =
\mu_P\mu_Q\left(\frac{rv}{l}-\frac{us}{l}\right) x^{\frac{r+u}{l}-1}y^{s+v-1},
$$
from which $\st_{\rho,\sigma}(P)\sim\st_{\rho,\sigma}(Q)$ follows. Similar arguments yield
$\en_{\rho,\sigma}(P)\sim\en_{\rho,\sigma}(Q)$, finishing the proof.
\end{proof}

\begin{proposition}\label{extremosnoalineados} Let $P,Q,R\in L^{(l)}\setminus \{0\}$ be such that
$$
[\ell_{\rho,\sigma}(P),\ell_{\rho,\sigma}(Q)]=\ell_{\rho,\sigma}(R),\endnote{
By Proposition~\ref{pr v de un conmutador} we know that $\ell_{\rho,\sigma}(R) = \ell_{\rho,\sigma}([P,Q])$.}
$$
where $(\rho,\sigma)\in \mathfrak{V}$. We have:

\begin{enumerate}

\smallskip

\item $\st_{\rho,\sigma}(P)\nsim\st_{\rho,\sigma}(Q)$ if and only if $\st_{\rho,\sigma}(P)
    +\st_{\rho,\sigma}(Q) - (1,1) =\st_{\rho,\sigma}(R)$.

\smallskip

\item $\en_{\rho,\sigma}(P)\nsim\en_{\rho,\sigma}(Q)$ if and only if
    $\en_{\rho,\sigma}(P)+\en_{\rho,\sigma}(Q)-(1,1)=\en_{\rho,\sigma}(R)$.

\end{enumerate}

\end{proposition}

\begin{proof} (1)\enspace It is enough to prove it when $P$, $Q$ and $R$ are $(\rho,\sigma)$-homogeneous,
so we will assume it. Choose $(\rho_0,\sigma_0)\in \mathfrak{V}$ such that
$(\rho_0,\sigma_0)<(\rho,\sigma)<(-\rho_0,-\sigma_0)$. By Remark~\ref{starting vs end}
\begin{equation}\label{starting}
\ell_{\rho_0,\sigma_0}(P) = \mu_P x^{\frac{r}{l}}y^s,\quad
\ell_{\rho_0,\sigma_0}(Q) = \mu_Q x^{\frac{u}{l}}y^v\quad
\text{and}\quad \ell_{\rho_0,\sigma_0}(R) = \mu_R x^{\frac{a}{l}}y^b,
\end{equation}
where
$$
\mu_P,\mu_Q,\mu_R\in K^{\times},\quad \left(\frac{r}{l},s\right):=\st_{\rho,\sigma}(P),\quad
\left(\frac{u}{l},v\right):=\st_{\rho,\sigma}(Q)\quad\text{and}\quad
\left(\frac{a}{l},b\right):=\st_{\rho,\sigma}(R).
$$
Clearly
$$
[\ell_{\rho_0,\sigma_0}(P),\ell_{\rho_0,\sigma_0}(Q)] = \mu_P\mu_Q\left(\frac{rv}{l}-\frac{us}{l}\right)
x^{\frac{r+u}{l}-1}y^{s+v-1}
$$
and hence, by Proposition~\ref{pr v de un conmutador},
\begin{equation}
\st_{\rho,\sigma}(P)\nsim \st_{\rho,\sigma}(Q)\Longleftrightarrow
[\ell_{\rho_0,\sigma_0}(P),\ell_{\rho_0,\sigma_0}(Q)]\ne 0 \Longleftrightarrow \ell_{\rho_0,\sigma_0}(R) =
[\ell_{\rho_0,\sigma_0}(P),\ell_{\rho_0,\sigma_0}(Q)].\label{chvh4}
\end{equation}
Consequently if $\st_{\rho,\sigma}(P)\nsim \st_{\rho,\sigma}(Q)$, then
$$
\mu_R x^{\frac{a}{l}}y^b = \mu_P\mu_P\left(\frac{rv}{l}-\frac{us}{l}\right) x^{\frac{r+u}{l}-1}y^{s+v-1},
$$
which evidently implies that
$$
\st_{\rho,\sigma}(P) + \st_{\rho,\sigma}(Q) - (1,1) = \st_{\rho,\sigma}(R).
$$
Reciprocally if this last equation holds, then by~\eqref{starting}
$$
v_{\rho_0,\sigma_0}(P) + v_{\rho_0,\sigma_0}(Q) -(\rho_0+\sigma_0)= v_{\rho_0,\sigma_0}(R),
$$
and so, again by Proposition~\ref{pr v de un conmutador},
$$
[\ell_{\rho_0,\sigma_0}(P),\ell_{\rho_0,\sigma_0}(Q)]\ne 0,
$$
which by~\eqref{chvh4} implies that $\st_{\rho,\sigma}(P)\nsim \st_{\rho,\sigma}(Q)$.

\smallskip

\noindent (2)\enspace It is similar to the proof of statement~(1).
\end{proof}

\begin{remark}\label{F no es monomio} Let $(\rho,\sigma)\in \mathfrak{V}$ and let $P,F\in L^{(l)}$ be
$(\rho,\sigma)$-homogeneous such that $[F,P] = P$. If $F$ is a monomial, then $F = \lambda xy$ with $\lambda
\in K^{\times}$, and, either $\rho+\sigma = 0$ or $P$ is also a monomial\endnote{Write $F = \lambda
x^{\frac{a}{l}} y^b$ and $P = x^{\frac{r}{l}}y^s p(z)$ with
$$
\lambda\in K^{\times},\qquad p(0)\ne 0 \qquad\text{and}\qquad z:=\begin{cases} x^{-\sigma/\rho}y & \text{if
$\rho\ne 0$,}\\ x^{1/l} &\text{if $\rho =0$.}\end{cases}
$$
Clearly
$$
\Supp(P) = \Supp([F,P]) \subseteq \Supp(P) + \Bigl(\frac{a}{l},b\Bigr) - (1,1),
$$
and so $\Supp(F) = (1,1)$. Hence $F = \lambda xy$ with $\lambda\in K^{\times}$. Moreover
\begin{align*}
x^{\frac{r}{l}}y^s p(z)&=\bigl[\lambda xy, x^{\frac{r}{l}}y^sp(z)\bigr]\\
 &= \lambda x^{\frac{r}{l}}y^s [xy,p(z)] + \lambda
\bigl[xy,x^{\frac{r}{l}}y^s\bigr]p(z)\\
& = \lambda x^{\frac{r}{l}}y^s p'(z) [xy,z] + \lambda \Bigl(s- \frac{r}{l}\Bigr)x^{\frac{r}{l}}y^s p(z)\\
&= \lambda v_{-1,1}(z) x^{\frac{r}{l}}y^s p'(z) z + \lambda \Bigl(s- \frac{r}{l}\Bigr)x^{\frac{r}{l}}y^s p(z),
\end{align*}
which implies
$$
p(z) = \lambda\Bigl(v_{-1,1}(z)p'(z) z + \Bigl(s- \frac{r}{l}\Bigr) p(z)\Bigr).
$$
 Hence $p(z)\mid v_{-1,1}(z) p'(z) z$, and so, if $P$ is not a monomial, then $v_{-1,1}(z)=0$, from which
 $\rho+\sigma = 0$ follows.}.
\end{remark}

\begin{theorem}\label{central} Let $P\in L^{(l)}$ and let $(\rho,\sigma)\in \mathfrak{V}_{> 0}$ be such that
$v_{\rho,\sigma}(P)>0$. If $[P,Q]\in K^{\times}$ for some $Q\in L^{(l)}$, then there exists $G_0\in K[P,Q]\setminus\{0\}$
and a $(\rho,\sigma)$-ho\-mo\-ge\-neous element $F\in L^{(l)}$ such that
\begin{equation}\label{eq central}
v_{\rho,\sigma}(F)=\rho+\sigma,\quad[F,\ell_{\rho,\sigma}(P)]= \ell_{\rho,\sigma}(P)\quad \text{and}\quad [
\ell_{\rho,\sigma}(G_0), \ell_{\rho,\sigma}(P)]F= \ell_{\rho,\sigma}(G_0)\ell_{\rho,\sigma}(P).
\end{equation}
Moreover, we have
\begin{enumerate}

\smallskip

\item If $P,Q\in L$, then we can take $F\in L$.

\smallskip

\item $\st_{\rho,\sigma}(P)\sim\st_{\rho,\sigma}(F)$ or $\st_{\rho,\sigma}(F) = (1,1)$.

\smallskip

\item $\en_{\rho,\sigma}(P)\sim\en_{\rho,\sigma}(F)$ or $\en_{\rho,\sigma}(F)= (1,1)$.

\smallskip

\item $\st_{\rho,\sigma}(P)\nsim (1,1)\nsim\en_{\rho,\sigma}(P)$.

\smallskip

\item If we define recursively $G_i:=[G_{i-1},P]$, then $[\ell_{\rho,\sigma}(G_i),\ell_{\rho,\sigma}(P)]=0$
  for $i\ge 1$.
\end{enumerate}
\end{theorem}

\begin{proof}
From Lemma~\ref{Lema Joseph} we obtain $G_0$ such that the hypotheses of Lemma~2.2 of~\cite{J}
are satisfied. Hence, by this lemma,
$$
F:=\frac{\ell_{\rho,\sigma}(G_0)\ell_{\rho,\sigma}(P)}{[\ell_{\rho,\sigma}(G_0), \ell_{\rho,\sigma}(P)]}\in
L^{(l)}
$$
and if $P$ and $Q$ are in $L$, then $F\!\in\! L$. Hence statement~(1) is true. Furthermore an easy
computation shows that statement (5) is also true and equalities~\eqref{eq central} are satisfied\endnote{
By Lemma~\ref{Lema Joseph} we know that
 $[[\ell_{\rho,\sigma}(G_0),\ell_{\rho,\sigma}(P)],\ell_{\rho,\sigma}(P)]] = 0$. Hence
\begin{align*}
[\ell_{\rho,\sigma}(G_0),\ell_{\rho,\sigma}(P)][F,\ell_{\rho,\sigma}(P)] &= [[\ell_{\rho,\sigma}(G_0),
\ell_{\rho,\sigma}(P)]F,\ell_{\rho,\sigma}(P)]\\
& = [\ell_{\rho,\sigma}(G_0)\ell_{\rho,\sigma}(P),\ell_{\rho,\sigma}(P)]\\
&= [\ell_{\rho,\sigma}(G_0),\ell_{\rho,\sigma}(P)]\ell_{\rho,\sigma}(P),
\end{align*}
which implies $[F,\ell_{\rho,\sigma}(P)] = \ell_{\rho,\sigma}(P)$ because $[\ell_{\rho,\sigma}(G_0),
\ell_{\rho,\sigma}(P)] \ne 0$. Moreover, by Proposition~\ref{pr v de un conmutador},
$$
[F,\ell_{\rho,\sigma}(P)] = \ell_{\rho,\sigma}(P)\Longrightarrow v_{\rho,\sigma}(F) = \rho+\sigma.
$$
Finally, statement~(5) is an immediate consequence of Lemma~\ref{Lema Joseph}.}.
Statements~(2) and~(3) follow from Proposition~\ref{extremosnoalineados}. For statement~(4), assume that
$\st_{\rho,\sigma}(P)\sim (1,1)$. We claim that this implies that $\st_{\rho,\sigma}(F) = (1,1)$. Otherwise,
by statement~(2) we have
$$
\st_{\rho,\sigma}(F)\sim \st_{\rho,\sigma}(P)\sim (1,1),
$$
which implies $\st_{\rho,\sigma}(F)\sim (1,1)$, since $\st_{\rho,\sigma}(F)\ne (0,0)\ne \st_{\rho,\sigma}(P)$.
So there exists $\lambda \in \mathds{Q}\setminus \{1\}$ such that $\st_{\rho,\sigma}(F) = \lambda (1,1)$. But
this is impossible because $v_{\rho,\sigma}(F) = \rho + \sigma$ implies $\lambda=1$. Hence the claim is true,
and so
$$
\st_{\rho,\sigma}(P) + \st_{\rho,\sigma}(F)-(1,1) = \st_{\rho,\sigma}(P),
$$
which by Proposition~\ref{extremosnoalineados}(1) leads to the contradiction
$$
\st_{\rho,\sigma}(P)\nsim \st_{\rho,\sigma}(F) = (1,1).
$$
Similarly $\en_{\rho,\sigma}(P)\nsim (1,1)$.
\end{proof}

\begin{remark} \label{teorema central con rho mas sigma < 0}
In general, the conclusions of Theorem~\ref{central} do not hold if $\rho+\sigma < 0$. For instance, consider
the following pair in $L^{(1)}$:
$$
P=x^{-1} + x^3 y (2 + 18 x^2 y + 36 x^4 y^2) + x^9 y^3 (8 + 72 x^2 y + 216 x^4 y^2 + 216 x^6 y^3)
$$
and
$$
Q=x^2 y + x^6 y^2 (1 + 6 x^2 y + 9 x^4 y^2).
$$
Clearly $[P,Q]=-1$ and $v_{1,-2}(P)=3>0$. However, one can show that there is no $F\in L^{(1)}$ such that
$[F,\ell_{1,-2}(P)]=\ell_{1,-2}(P)$\endnote{
Assume on the contrary that there exists $F\in L^{(1)}$ such that
$$
[F,\ell_{1,-2}(P)] = \ell_{1,-2}(P).
$$
Replacing $F$ with $\ell_{1,-2}(F)$ we can assume that $F$ is $(1,-2)$-homogeneous.  If $\st_{1,-2}(F)\sim
\st_{1,-2}(P) = (9,3)$, then $\st_{1,-2}(F)=\lambda(3,1)$ for some $\lambda\in K$. But by
Proposition~\ref{pr v de un conmutador} we know that $v_{1,-2}(F) = -1$ and so $\lambda=-1$ which is
impossible, since $\st_{1,-2}(F)\in \mathds{Z}\times \mathds{N}_0$. Hence, by Theorem~\ref{central}(2),
we have $\st_{1,-2}(F)=(1,1)$. The same argument shows that $\en_{1,-2}(F)=(1,1)$.
But then $F$ is a monomial, which contradicts Remark~\ref{F no es monomio}.
}.
\end{remark}

\begin{remark}\label{polinomio asociado f^{(l)}} Let $P\in L^{(l)}\setminus\{0\}$ and $(\rho,\sigma)\in
\mathfrak{V}$ with $\rho>0$. If $\ell_{\rho,\sigma}(P)= x^{\frac{r}{l}}y^s p(x^{-\frac{\sigma}
{\rho}}y)$, where
$$
p:=\sum_{i=0}^{\gamma} a_i x^i\in K[x]\quad\text{with $a_0\ne 0$ and $a_{\gamma}\ne 0$},
$$
then, by Remark~\ref{starting vs end} with $(\rho_0,\sigma_0)= (0,-1)$,
\begin{equation}
\st_{\rho,\sigma}(P)=\Bigl(\frac{r}{l},s\Bigr)\quad\text{and}\quad\en_{\rho,\sigma}(P)=
\Bigl(\frac{r}{l}-\frac{\gamma\sigma}{\rho},s +\gamma\Bigr).\label{eq57}
\end{equation}
\end{remark}

\begin{definition}\label{forma debil} Let $P\in L^{(l)}\setminus\{0\}$. We define the set of
{\em directions associated with} $P$ as
$$
\Dir(P):=\{(\rho,\sigma)\in\mathfrak{V}:\#\Supp(\ell_{\rho,\sigma}(P))>1\}.
$$
\end{definition}

\begin{remark} \label{180 grados}
Note that if $P\in L^{(l)}\setminus\{0\}$ is a monomial, then $\Dir(P) =
\emptyset$ and that if $P\in L^{(l)}\setminus\{0\}$ is
$(\rho,\sigma)$-homogeneous, but is not a monomial, then $\Dir(P) =
\{(\rho,\sigma),(-\rho,-\sigma)\}$. Furthermore, if $P\in
L^{(l)}\setminus\{0\}$ is not homogeneous, then any two consecutive
directions of $P$ are separated by less than $180$°.
\end{remark}

\begin{proposition}\label{pavadass} Let $(\rho,\sigma)\in\mathfrak{V}^0$ and $P,F\in L^{(l)}\setminus\{0\}$.
Assume that $F$ is $(\rho,\sigma)$-ho\-mo\-ge\-neous, $v_{\rho,\sigma}(P)>0$ and
\begin{equation}\label{conmutador de F y P}
[F,\ell_{\rho,\sigma}(P)]=\ell_{\rho,\sigma}(P).
\end{equation}
Write
$$
F=x^{\frac{u}{l}}y^v f(z)\quad\text{and}\quad \ell_{\rho,\sigma}(P)=x^{\frac{r}{l}}y^s p(z) \quad\text{with
$z:=x^{-\frac{\sigma}{\rho}}y$ and $p(0)\ne 0\ne f(0)$.}
$$
Then
\begin{enumerate}

\smallskip

\item $f$ is separable and every irreducible factor of $p$ divides $f$.

\smallskip

\item If $(\rho,\sigma)\in \Dir(P)$, then $v_{0,1}(\st_{\rho,\sigma}(F))< v_{0,1}(\en_{\rho,\sigma}(F))$.

\smallskip

\item Suppose that $p,f\in K[z^k]$ for some $k\in\mathds{N}$ and let $\ov p$ and $\ov f$ denote the
    univariate polynomials defined by $p(z)=\ov p(z^k)$ and $f(z)=\ov f(z^k)$. Then $\ov f$ is separable
    and every irreducible factor of $\ov p$ divides $\ov f$.

\smallskip

\item If $P,F\in L$ and $v_{0,1}(\en_{\rho,\sigma}(F))-v_{0,1}(\st_{\rho,\sigma} (F)) = \rho$, then the
    multiplicity of each linear factor (in an algebraic closure of $K$) of $p$ is equal to
$$
\frac{1}{\rho} \deg(p)= \frac{1}{\rho}\bigl(v_{0,1}(\en_{\rho,\sigma} (P)) -v_{0,1}(\st_{\rho,\sigma}(P)
\bigr).
$$

\smallskip

\item Assume that $(\rho,\sigma)\in \Dir(P)$. If $s>0$ or $\# \factors(p)>1$, then there exist no
    $(\rho,\sigma)$-ho\-mo\-geneous element $R\in L^{(l)}$ such that
\begin{equation}\label{condicions para R}
v_{\rho,\sigma}(R)=\rho+\sigma\quad\text{and}\quad [R,\ell_{\rho,\sigma}(P)]=0.
\end{equation}
Consequently, in this case $F$ satisfying~\eqref{conmutador de F y P} is unique.
\end{enumerate}
\end{proposition}

\begin{proof} Note that, since $[-,-]$ is a derivation in both variables, we have
$$
[F,\ell_{\rho,\sigma}(P)] = \bigl[x^{\frac{u}{l}}y^v f(z),x^{\frac{r}{l}}y^s p(z)\bigr] =
x^{\frac{u+r}{l}-1}y^{v+s-1} \bigl(c f(z) p(z) + azf(z)p'(z) - b zf'(z)p(z)\bigr),\endnote{
The complete computation is
\begin{align*}
[F,\ell_{\rho,\sigma}(P)] & = \bigl[x^{\frac{u}{l}}y^v f(z),x^{\frac{r}{l}}y^s p(z)\bigr]\\
&=\bigl[x^{\frac{u}{l}}y^v ,x^{\frac{r}{l}}y^s\bigr]f(z) p(z) + x^{\frac{u}{l}}y^v
\bigl[f(z),x^{\frac{r}{l}}y^s\bigr] p(z) + x^{\frac{r}{l}}y^s\bigl[x^{\frac{u}{l}}y^v , p(z)\bigr]f(z)\\
&=\bigl[x^{\frac{u}{l}}y^v ,x^{\frac{r}{l}}y^s\bigr]f(z) p(z) + x^{\frac{u}{l}}y^v
\bigl[z,x^{\frac{r}{l}}y^s\bigr] f'(z) p(z) + x^{\frac{r}{l}}y^s\bigl[x^{\frac{u}{l}}y^v, z\bigr]f(z)p'(z)\\
&= x^{\frac{u+r}{l}-1}y^{v+s-1} \bigl(c f(z) p(z) + azf(z)p'(z) - b zf'(z)p(z)\bigr).
\end{align*}}
$$
where
$$
c:=\binom {\frac ul}{v} \times \binom{\frac rl}{s},\quad a:= \binom {\frac ul}{v}\times
\binom{-\frac{\sigma}{\rho}}{1}=\frac{1}{\rho} v_{\rho,\sigma}(F)\quad\text{and}\quad b:=
\binom{\frac rl}{s}\times \binom{-\frac{\sigma}{\rho}}{1}=\frac{1}{\rho} v_{\rho,\sigma}(P).
$$
Hence, by equality~\eqref{conmutador de F y P} there exists $h\in \mathds{N}_0$ such that
\begin{equation}
z^h p = cp f+ az p' f-bzf'p.\label{eqq1}\endnote{
This follows from the following easy lemma.
\begin{lemma}\label{igualdad} Let $(\rho,\sigma)\in \mathfrak{V}$ with $\rho\ne0$ and let $p,g\in K[z]$ with
$z:=x^{-\sigma/\rho}y$. If $p(0)\ne 0$ and there are $l\in \mathds{\mathds{N}}$, $b,d\in \mathds{N}_0$ and
$a,c\in \mathds{Z}$ such that
$$
x^{\frac{a}{l}}y^bp(z) = x^{\frac{c}{l}}y^dg(z),
$$
then there exists $h\in \mathds{N}_0$ such that $z^hp(z) = g(z)$.
\end{lemma}

\begin{proof} Write $g(z)=z^h \ov g(z)$, with $\ov g(0)\ne 0$. Comparing in the equality
$x^{\frac{a}{l}}y^b p(z)= x^{\frac{c}{l}}y^d z^h\ov g(z)$ the monomials with lowest $y$-degree,
we obtain $x^{\frac{a}{l}}y^b = x^{\frac{c}{l}} y^d z^h$. Consequently $p(z) = \ov g(z)$, as desired.
\end{proof}
}
\end{equation}
Let $g$ be a linear factor of $p$ in an algebraic closure of $K$, with multiplicity $m$. Write
$p=p_1 g^m$ and $f=f_1 g^n$, where $n\ge 0$ is the multiplicity of $g$ in $f$. Since
$$
p'=p_1 m g^{m-1}g'+p_1' g^m\quad\text{and}\quad f'=f_1 n g^{n-1}g'+f_1' g^n,
$$
equality~\eqref{eqq1} can be written
$$
z^h p_1 g^m = g^{m+n-1}\bigl(g(cp_1f_1+azf_1p_1'-bzf_1'p_1)+(am-bn)zf_1p_1 g'\bigr),
$$
which implies $n\le 1$. But $n=0$ is impossible since $a,m > 0$. So, statement~(1) follows.

\smallskip

Assume now $(\rho,\sigma)\in \Dir(P)$. Then $\deg p>0$, and so, by statement~(1) we
have $\deg f>0$. Consequently statement~(2) follows from Remark~\ref{polinomio asociado f^{(l)}}. Using now
that
$$
zp'(z)=kt\ov p'(t)\quad\text{and}\quad zf'(z)=kt\ov f'(t)\qquad\text{where $t:= z^k$,}
$$
we deduce from~\eqref{eqq1} the equality
$$
z^h \ov p(t) = c\ov p(t) \ov f(t)+ at \ov p' \ov f(t)-bt\ov f'(t)\ov p(t).
$$
The same procedure as above, but using this last equality instead of~\eqref{eqq1}, yields statement~(3).

\smallskip

Now we prove statement~(4). Write
$$
F = \sum_{i=0}^{\alpha} b_i x^{u-i\sigma} y^{v+i\rho}\quad \text{and}\quad \ell_{\rho,\sigma}(P) =
\sum_{i=0}^{\gamma} c_i x^{r-i\sigma} y^{s+i\rho}
$$
with $b_0\ne 0$, $b_{\alpha}\ne 0$, $c_0\ne 0$ and $c_{\gamma}\ne 0$. By definition
$$
f = \sum_{i=0}^{\alpha} b_iz^{i\rho}\quad\text{and}\quad p = \sum_{i=0}^{\gamma} c_iz^{i\rho}.
$$
Moreover, since by~\eqref{eq57},
$$
\alpha\rho = v_{0,1}(\en_{\rho,\sigma}(F))- v_{0,1}( \st_{\rho,\sigma}(F)),
$$
it follows from the hypothesis that $\alpha=1$. Hence
$$
f (z)= b_0 + b_1 z^{\rho} = \mu (z^{\rho}-\mu')=\ov f(z^{\rho}),
$$
where $\mu:=b_1$ and $\mu':= b_0/b_1$. Consequently, by statement~(3), there exists $\mu_P\in K^{\times}$ such
that
$$
p(z) = \mu_P(z^{\rho}-\mu')^{\gamma},
$$
from which statement~(4) follows easily. Finally we prove statement~(5). For this we first
prove~\eqref{condicion} below, and then we prove that for any $R$ satisfying~\eqref{condicions para R} there
exists $\lambda\in K^{\times}$ such that $F_\lambda:=F-\lambda R$ satisfies~\eqref{conmutador de F y P} and
$\en_{\rho,\sigma}(P)\nsim\en_{\rho,\sigma}(F_\lambda)$, which is a contradiction.

Assume that $\# \factors(p)>1$ or that $s>0$. We claim that $\en_{\rho,\sigma}(F)\ne (1,1)$.
If the first inequality holds, then, by statement~(1), we have $\deg(f) > 1$. Consequently, by
Remark~\ref{polinomio asociado f^{(l)}}, it is impossible that $\en_{\rho,\sigma}(F)= (1,1)$.
Assume that $s>0$.
By Proposition~\ref{extremosnoalineados}(1), either
$$
\st_{\rho,\sigma}(F)=(1,1)\quad\text{or}\quad \st_{\rho,\sigma}(F)\sim \st_{\rho,\sigma}(P)=(r/l,s).
$$
In the first case
$v_{0,1}(\st_{\rho,\sigma}(F)) =1$, while in the second one, since by
Remark~\ref{re v de un conmutador}(1) we know that $\st_{\rho,\sigma} (F) \ne (0,0)$ , there exists $\lambda>0$ such that $\st_{\rho,\sigma}(F) =
\lambda\st_{\rho,\sigma}(P)$. So
$v_{0,1}(\st_{\rho,\sigma}(F)) =  \lambda s>0$.
In both cases, by statement~(2),
$$
v_{0,1}(\en_{\rho,\sigma}(F))>v_{0,1}(\st_{\rho,\sigma}(F))\ge 1,
$$
which clearly implies $\en_{\rho,\sigma}(F)\ne (1,1)$, as desired. Thus, by
Proposition~\ref{extremosnoalineados}(2) we conclude that, if $\# \factors(p)>1$ or $s>0$, then
\begin{equation}\label{condicion}
[F,\ell_{\rho,\sigma}(P)] = \ell_{\rho,\sigma}(P) \Longrightarrow \en_{\rho,\sigma}(P)\sim\en_{\rho,\sigma}(F).
\end{equation}
Suppose that $R\!\in\! L^{(l)}$ is a $(\rho,\sigma)$-homogeneous element that satisfies
condition~\eqref{condicions para R}. By Proposition~\ref{extremosalineados} we know that
$\en_{\rho,\sigma}(P)\!\sim\!\en_{\rho,\sigma}(R)$ and so $\en_{\rho,\sigma}(F)\!\sim\!\en_{\rho,\sigma}(R)$.
Since by Remark~\ref{re v de un conmutador}(1)
\begin{equation}\label{e1}
v_{\rho,\sigma}(F) = \rho+\sigma = v_{\rho,\sigma}(R),
\end{equation}
this implies that
\begin{equation}\label{e2}
\en_{\rho,\sigma}(F) = \en_{\rho,\sigma}(R).
\end{equation}
Let $\bar{r}$ be an univariate polynomial such that $\bar{r}(0)\ne 0$ and $R = x^{\frac{h}{l}}y^k\bar{r}(z)$.
We have
$$
F = x^{\frac{\ov{u}}{l}}\mathfrak{f}(z)\quad\text{and}\quad R = x^{\frac{\ov{h}}{l}}\mathfrak{r}(z),
$$
where $\mathfrak{f}(z):= z^v f(z)$, $\mathfrak{r}(z):= z^k \bar{r}(z)$, $\ov{u}:=u+v\sigma l/\rho$ and
$\ov{h}:= h+k\sigma l/\rho$. By Remark~\ref{polinomio asociado f^{(l)}} and equality~\eqref{e2}
$$
\deg(\mathfrak{f})=\deg(f)+v = v_{0,1}(\en_{\rho,\sigma}(F)) = v_{0,1}(\en_{\rho,\sigma}(R)) = \deg(\bar{r})+k
 =\deg(\mathfrak{r}).
$$
Moreover $\ov{u} = \ov{h}$ since, by equality~\eqref{e1},
$$
\rho \frac{\ov{u}}{l} = v_{\rho,\sigma}(F) = v_{\rho,\sigma}(R) = \rho \frac{\ov{h}}{l}.
$$
Let $\lambda\in K^{\times}$ be such that $\deg(\mathfrak{f}-\lambda \mathfrak{r})< \deg(\mathfrak{f})$ and let
$$
F_{\lambda}:= F-\lambda R = x^{\frac{\ov{u}}{l}}\bigl(\mathfrak{f}(z)-\lambda\mathfrak{r}(z)\bigr).
$$
Again by Remark~\ref{polinomio asociado f^{(l)}}
$$
\en_{\rho,\sigma}(F_{\lambda})=\en_{\rho,\sigma}(F)-t(-\sigma,\rho)\qquad \text{where }
t:=\frac{\deg(\mathfrak{f})- \deg(\mathfrak{f}-\lambda \mathfrak{r})}{\rho}>0.
$$
Hence
$$
\en_{\rho,\sigma}(P)\times \en_{\rho,\sigma}(F_\lambda)= -t(\en_{\rho,\sigma}(P)\times(-\sigma,\rho))
= -t v_{\rho,\sigma}(P)<0,
$$
and so $\en_{\rho,\sigma}(P)\nsim\en_{\rho,\sigma}(F_\lambda)$. But, since
$$
[F_{\lambda},\ell_{\rho,\sigma}(P)]= [F-\lambda R,\ell_{\rho,\sigma}(P)]= [F,\ell_{\rho,\sigma}(P)]
- \lambda[R,\ell_{\rho,\sigma}(P)]=[F,\ell_{\rho,\sigma}(P)]=\ell_{\rho,\sigma}(P),
$$
this contradicts~\eqref{condicion}, and hence, such an $R$ cannot exist. Clearly the uniqueness of $F$ follows,
since any other $F'$ satisfying~\eqref{conmutador de F y P} yields $R:=F-F'$ which satisfies~\eqref{condicions para R}.
\end{proof}

\section{More on the order on directions}

\setcounter{equation}{0}

In this section we consider the same order on directions as other authors, e. g.~\cite{A} and~\cite{H},
but we profit from the following characterization of this order in small intervals: If $I$ is an interval in
$\mathfrak{V}$ and if there is no closed half circle contained in $I$, which means that there is no
$(\rho,\sigma)\in I$ with $(-\rho,-\sigma)\in I$, then for $(\rho,\sigma),(\rho'\sigma')\in I$ we have
\begin{equation}\label{order}
(\rho,\sigma)<(\rho',\sigma') \Longleftrightarrow (\rho,\sigma)\times (\rho',\sigma')>0.
\end{equation}
We also present in Proposition~\ref{varphi preserva el Jacobiano} the chain rule for Jacobians in a
convenient way.

\smallskip

\begin{remark}\label{a remark} Let $(\rho,\sigma)\in \mathfrak{V}$ and let $P,Q\in L^{(l)}$. If
$$
v_{\rho,\sigma}(P)>0,\quad v_{\rho,\sigma}(Q)\ge 0\quad\text{and}\quad
[\ell_{\rho,\sigma}(P),\ell_{\rho,\sigma}(Q)]=0,
$$
then by Proposition~\ref{P y Q alineados}(2) we know that there exist $\lambda_P,\lambda_Q\!\in\! K^{\times}$,
$m,n\!\in\! \mathds{N}_0$ coprime and a $(\rho,\sigma)$-ho\-mo\-ge\-neous element $R\in L^{(l)}$, with
$R\in L$ if $P,Q\in L$, such that
\begin{equation*}
\ell_{\rho,\sigma}(P)=\lambda_P R^m\qquad\text{and}\qquad \ell_{\rho,\sigma}(Q) = \lambda_Q R^n.
\end{equation*}
Note that $v_{\rho,\sigma}(P)>0$ implies $m\in \mathds{N}$. Consequently, we have

\begin{enumerate}

\smallskip

\item $\frac n m = \frac{v_{\rho,\sigma}(Q)}{v_{\rho,\sigma}(P)}$,

\smallskip

\item $\st_{\rho,\sigma}(Q) = \frac{n}{m}\st_{\rho,\sigma}(P)$,

\smallskip

\item $\en_{\rho,\sigma}(Q) = \frac{n}{m}\en_{\rho,\sigma}(P)$,

\end{enumerate}
and, if moreover $v_{\rho,\sigma}(Q)>0$, then
\begin{align*}
(\rho,\sigma)\in \Dir(P) & \Leftrightarrow \text{$\ell_{\rho,\sigma}(P)$ is not a monomial}\\
& \Leftrightarrow \text{$R$ is not a monomial}\\
& \Leftrightarrow \text{$\ell_{\rho,\sigma}(Q)$ is not a monomial}\\
& \Leftrightarrow (\rho,\sigma)\in \Dir(Q).
\end{align*}
By Proposition~\ref{pr v de un conmutador} the condition $[\ell_{\rho,\sigma}(P),\ell_{\rho,\sigma}(Q)]=0$ can
be replaced by
$$
v_{\rho,\sigma}([P,Q])< v_{\rho,\sigma}(P)+v_{\rho,\sigma}(Q)-(\rho +\sigma).
$$
We will use freely this fact.
\end{remark}

For each $(r/l,s)\in\frac{1}{l}\mathds{Z}\times\mathds{Z}\setminus\mathds{Z}(1,1)$ there exists a unique
$(\rho,\sigma)\in\mathfrak{V}_{>0}$, denoted by $\dir(r/l,s)$, such that $v_{\rho,\sigma}(r/l,s)=0$. In
fact clearly
\begin{equation}
(\rho,\sigma)=\begin{cases}\left(-ls/d,r/d\right)&\text{ if $r-ls>0$,}\\ \left(ls/d,-r/d\right)&\text{ if
$r-ls<0$,}\end{cases}\label{val}
\end{equation}
where $d:=\gcd(r,ls)$, satisfies the required condition, and the uniqueness is evident.

\begin{remark}\label{valuacion depende de extremos} Note that if $(\rho,\sigma)\in\mathfrak{V}_{>0}$,
$(r/l,s)\ne(r'/l,s')$ and $v_{\rho,\sigma}(r/l,s) = v_{\rho,\sigma}(r'/l,s')$ then
$$
(\rho,\sigma) = \dir\left(\left(\frac rl,s\right)-\left(\frac{r'}{l},s'\right) \right).
$$
In particular
$$
(\rho,\sigma) = \dir\bigl(\en_{\rho,\sigma}(P) -\st_{\rho,\sigma}(P)\bigr)\qquad \text{for all $P\in
L^{(l)}\setminus\{0\}$ and $(\rho,\sigma)\in\Dir(P)\cap \mathfrak{V}_{>0}$}.
$$
\end{remark}

\begin{remark} \label{intervalo de direcciones}
Let $(a/l,b)\in \frac 1l\mathds{Z}\times\mathds{N}$ and set
$$
(\ov\rho,\ov\sigma):=\frac 1d (bl,-a),\quad\text{where $d:=\gcd(bl,a)$.}
$$
Then, for any $(\rho,\sigma)\in\mathfrak{V}$, we have
$$
v_{\rho,\sigma}(a/l,b)>0\Longleftrightarrow (\ov\rho,\ov\sigma)\times (\rho,\sigma)>0\Longleftrightarrow
(\ov\rho,\ov\sigma)<(\rho,\sigma)<(-\ov\rho,-\ov\sigma)\endnote{
Clearly, if $u,v\in S^{1}$, then $u\times v>0$ if and only if $v\in \ ]u,-u[$, with the counterclockwise order.
}.
$$
\end{remark}

\begin{definition}\label{Sucesor y predecesor} Let $P\in L^{(l)}\setminus\{0\}$ which is not a monomial
and $(\rho,\sigma)\in\mathfrak{V}$. We define the {\em successor} $\Succ_P(\rho,\sigma)$ of $(\rho,\sigma)$
to be the first element of $\Dir(P)$ that one encounters starting from $(\rho,\sigma)$ and running
counterclockwise, and the {\em predecessor} $\Pred_P(\rho,\sigma)$, to be the first one, if we run
clockwise\endnote{Note that $\Pred_P(\rho,\sigma)\ne \Succ_P(\rho,\sigma)$ except in the case in which $P$ is
$(\rho,\sigma)$-homogeneous. In this case
$$
\Pred_P(\rho,\sigma) = \Succ_P(\rho,\sigma) = (-\rho,-\sigma).
$$
}.
\end{definition}

Note that $\mathfrak{V}_{>0}$ is the interval $](1,-1),(-1,1)[$ and the order on $\mathfrak{V}_{>0}$ is given
by~\eqref{order}.

\begin{lemma}\label{formula basica de orden} Let $(a/l,b),(c/l,d)\in \frac 1l \mathds{Z}\times \mathds{Z}$ and
$(\rho,\sigma)\in \mathfrak{V}_{>0}$. If $v_{1,-1}(a/l,b) > v_{1,-1}(c/l,d)$, then
\begin{align*}
&v_{\rho,\sigma}\left(\frac al,b\right)>v_{\rho,\sigma}\left(\frac cl,d\right) \Longleftrightarrow
\dir\left(\left(\frac al,b\right)- \left(\frac cl,d\right)\right) >(\rho,\sigma)\\
\shortintertext{and}
&v_{\rho,\sigma}\left(\frac al,b\right)<v_{\rho,\sigma}\left(\frac cl,d\right) \Longleftrightarrow
\dir\left(\left(\frac al,b\right)- \left(\frac cl,d\right)\right) < (\rho,\sigma).
\end{align*}
\end{lemma}

\begin{proof} Let
$$
(\rho',\sigma'):=\dir\bigl((a/l,b)-(c/l,d)\bigr)\quad\text{and}\quad g:=\gcd(bl-dl,a-c).
$$
Since $v_{1,-1}(a/l,b)>v_{1,-1}(c/l,d)$ implies $a-c>bl-dl$, we have
$$
(\rho',\sigma')=\left(\frac{dl-bl}{g},\frac{a-c}{g}\right).
$$
Consequently
$$
v_{\rho,\sigma}\left(\frac al,b\right)-v_{\rho,\sigma}\left(\frac cl,d\right) =
\frac{g}{l}\left(\rho\frac{a-c}{g} -\sigma\frac{dl-bl}{g}\right)= \frac{g}{l}\bigl(\rho
\sigma'-\rho'\sigma\bigr) = \frac{g}{l}\, (\rho,\sigma)\times (\rho',\sigma'),
$$
and so, the result follows immediately from~\eqref{order}.
\end{proof}

\begin{corollary}\label{formula basica de orden'} Let $(a/l,b),(c/l,d)\in \frac 1l \mathds{Z}\times
\mathds{Z}$ and $(\rho,\sigma)<(\rho',\sigma')$ in $\mathfrak{V}_{>0}$. If
$$
v_{1,-1}(a/l,b) > v_{1,-1}(c/l,d),
$$
then
\begin{align*}
&v_{\rho',\sigma'}\left(\frac al,b\right)\ge v_{\rho',\sigma'}\left(\frac cl,d\right) \Longrightarrow
v_{\rho,\sigma}\left(\frac al,b\right)> v_{\rho,\sigma}\left(\frac cl,d\right)\\
\shortintertext{and}
&v_{\rho,\sigma}\left(\frac al,b\right)\le v_{\rho,\sigma}\left(\frac cl,d\right) \Longrightarrow
v_{\rho',\sigma'}\left(\frac al,b\right) < v_{\rho',\sigma'}\left(\frac cl,d\right).
\end{align*}
\end{corollary}

\begin{proof} It follows easily from Lemma~\ref{formula basica de orden}.
\end{proof}

The next two propositions are completely clear. The first one asserts that if
you have two consecutive edges of a Newton polygon, then all that is between them is
the common vertex. The second one asserts that if the end point of an edge coincides with
the starting point of another edge, then they are consecutive.

\begin{proposition}\label{le basico} Let $P\!\in\! L^{(l)}\!\setminus\!\{0\}$ and let
$(\rho_1,\sigma_1)$ and $(\rho_2,\sigma_2)$ be consecutive elements in $\Dir(P)$. If $(\rho_1,\sigma_1)<
(\rho,\sigma) <(\rho_2,\sigma_2)$, then $\en_{\rho_1,\sigma_1}(P) = \Supp(\ell_{\rho,\sigma}(P)) =
\st_{\rho_2,\sigma_2}(P)$.
\end{proposition}

\begin{proposition}\label{le bbasico} Let $P\in L^{(l)}\setminus\{0\}$ and let
$(\rho,\sigma),(\rho',\sigma')\in \mathfrak{V}$. If $\en_{\rho,\sigma}(P) = \st_{\rho',\sigma'}(P)$,
then there is no $(\rho'',\sigma'')\in \Dir(P)$ such that $(\rho,\sigma)<(\rho'',\sigma'') <(\rho',\sigma')$.
\end{proposition}

\begin{proposition}\label{pr ell por automorfismos} For $k\in\mathds{Z}$ consider the au\-to\-mor\-phism
of $L^{(l)}$ defined by
$$
\varphi\bigl(x^{\frac{1}{l}}\bigr) := x^{\frac{1}{l}}\qquad\text{and}\qquad \varphi(y) := y+\lambda
x^{\frac{k}{l}}.
$$
Let $(\rho,\sigma)$ be the direction defined by $\rho > 0$ and
$\frac{\sigma}{\rho} =\frac{k}{l}$. We have
$$
\ell_{\rho,\sigma}(\varphi(P)) = \varphi(\ell_{\rho,\sigma}(P)),\quad \ell_{-\rho,-\sigma}(\varphi(P)) =
\varphi(\ell_{-\rho,-\sigma}(P))\quad\text{and}\quad \ell_{\rho_1,\sigma_1}(\varphi(P)) = \ell_{\rho_1,
\sigma_1}(P),
$$
for all $P\in L^{(l)}\setminus\{0\}$ and all $(\rho,\sigma)< (\rho_1,\sigma_1) < (-\rho,-\sigma)$. Moreover
$\en_{\rho,\sigma}(\varphi(P)) = \en_{\rho,\sigma}(P)$.
\end{proposition}

\begin{proof} Take $d:=\gcd(k,l)>0$, $\rho:=l/d$ and $\sigma:=k/d$. Clearly $\frac{\sigma}{\rho} = \frac kl$.
Moreover, since $\varphi$ is $(\rho,\sigma)$-homogeneous it is also clear that
$$
\ell_{\rho,\sigma}(\varphi(P)) = \varphi(\ell_{\rho,\sigma}(P))\qquad\text{and}\qquad
\ell_{-\rho,-\sigma}(\varphi(P)) = \varphi(\ell_{-\rho,-\sigma}(P))
$$
for all $P\in L^{(l)}\setminus\{0\}$. Now we prove that the last equality is also true. By the hypothesis
about $(\rho_1,\sigma_1)$ we have $\rho_1\sigma < \rho \sigma_1$. Thus
$$
\ell_{\rho_1,\sigma_1}\bigl(y+\lambda x^{\frac{\sigma}{\rho}}\bigr) = y,
$$
since $\rho>0$. Consequently
\begin{equation*}
\ell_{\rho_1,\sigma_1}\bigl(\varphi(x^{\frac{i}{l}} y^j)\bigr)= \ell_{\rho_1,\sigma_1}\bigl(x^{\frac{i}{l}}
(y+\lambda x^{\frac{\sigma}{\rho}})^j \bigr)= x^{\frac{i}{l}} y^j,
\end{equation*}
from which
$$
\ell_{\rho_1,\sigma_1}(\varphi(P)) = \ell_{\rho_1,\sigma_1}(\varphi( \ell_{\rho_1,\sigma_1}(P))) =
\ell_{\rho_1,\sigma_1}(P),
$$
follows\endnote{In fact, from
\begin{equation*}
\ell_{\rho_1,\sigma_1}\bigl(\varphi(x^{\frac{i}{l}} y^j)\bigr) = x^{\frac{i}{l}} y^j \quad
\text{for all $i\in \mathds{Z}$ and $j\in \mathds{N}$,}
\end{equation*}
it follows that
\begin{equation}
\ell_{\rho_1,\sigma_1}\bigl(\varphi(P)\bigr) = P\quad\text{for all $(\rho_1,\sigma_1)$-homogeneous $P\in
L^{(l)}\setminus \{0\}$}\label{eq27}
\end{equation}
and
\begin{equation}
v_{\rho_1,\sigma_1}(\varphi(R)) = v_{\rho_1,\sigma_1}(R)\quad\text{for all $R\in L^{(l)}\setminus
\{0\}$.}\label{eq28}
\end{equation}
Fix now $P\in L^{(l)}\setminus\{0\}$ and write
$$
P = \ell_{\rho_1,\sigma_1}(P) + R\quad\text{with $R = 0$ or $v_{\rho_1,\sigma_1}(R)<v_{\rho_1,\sigma_1}(P)$.}
$$
Applying first equality~\eqref{eq28} and then equality~\eqref{eq27}, we obtain
$$
\ell_{\rho_1,\sigma_1}(\varphi(P)) = \ell_{\rho_1,\sigma_1}(\varphi( \ell_{\rho_1,\sigma_1}(P))) =
\ell_{\rho_1,\sigma_1}(P),
$$
as desired.
}.
The last assertion follows from the second equality in~\eqref{eq57} and the fact that the monomials of
greatest degree in $y$ of $\ell_{\rho,\sigma}(\varphi(P))$ and $\ell_{\rho,\sigma}(P)$ coincide.
\end{proof}

\begin{proposition}\label{varphi preserva el Jacobiano} Let $R_0,R_1\in\{L,L^{(l)}\}$, $P,Q\in R_0$ and
$\varphi\colon R_0\to R_1$ an algebra morphism. Then
\begin{equation}\label{Chain rule for Jacobian}
[\varphi(P),\varphi(Q)]=\varphi([P,Q])[\varphi(x),\varphi(y)].
\end{equation}
\end{proposition}

\begin{proof} Recall the (formal) Jacobian chain rule (see for example~\cite{A2}*{(1.7), p. 1160}) which
generalizes the (formal) derivative chain rule and which says that given any 2-variable rational functions
$f_1(x,y),f_2(x,y),g_1(x,y),g_2(x,y)\in K(x,y)$, we have
$$
J_{(x,y)}(h_1,h_2)=J_{(f_1,f_2)}(g_1,g_2)J_{(x,y)}(f_1,f_2),
$$
where by definition $h_i(x,y) := g_i(f_1(x,y), f_2(x,y))$, and
\begin{equation}\label{aux}
J_{(f_1,f_2)}(g_1,g_2) := j (f_1(x,y), f_2(x,y))\quad\text{with}\quad j (x,y) := J_{(x,y)}(g_1,g_2).
\end{equation}
Assume first that $l=1$. Then equality~\eqref{Chain rule for Jacobian} follows applying equality~\eqref{aux}
with
$$
g_1:=P,\quad g_2:=Q,\quad f_1:=\varphi(x)\quad\text{and}\quad f_2:=\varphi(y),
$$
since $\varphi([P,Q]) = j(\varphi(x),\varphi(y))$, where $j(x,y):=[P,Q]\in L^{(1)}\subseteq K(x,y)$. Assume
now that $l$ is arbitrary. Identifying $L^{(l)}$ with $K[z,z^{-1},y]$ via $z=x^{1/l}$, we obtain
$$
[P,Q]=(P_z Q_y-P_y Q_z)\frac{1}{lz^{l-1}},\quad\text{for $P,Q\in L^{(l)}$}.
$$
Consequently equality~\eqref{Chain rule for Jacobian} is valid for $R_0,R_1\in\{L,L^{(l)}\}$.
\end{proof}

\section{Minimal pairs and $\bm{(m,n)}$-pairs}
\label{section4}
\setcounter{equation}{0}

Our next aim is to determine a lower bound for
$$
B := \begin{cases}\infty & \text{if the jacobian conjecture is true,}\\
\min\bigl(\gcd(v_{1,1}(P),v_{1,1}(Q))\bigr)&\text{if JC is false, where $(P,Q)$ runs on the counterexamples.}
\end{cases}
$$

A {\em minimal pair} is a counterexample $(P,Q)$ to JC such that $B=\gcd(v_{1,1}(P),v_{1,1}(Q))$.

An $(m,n)$-pair is a Jacobian pair $(P,Q)$ with $P,Q\in L^{(l)}$ for some $l$, that satisfies certain
conditions (see Definition~\ref{Smp}).

In this section we prove that if $B<\infty$, then there exists a minimal pair that is also an $(m,n)$-pair
for some $m,n\in\mathds{N}$. We could prove the result using only our previous results, but
we prefer to use the well known fact that a counterexample to JC can be brought into a subrectangular shape,
following an argument communicated by Leonid Makar-Limanov.

\smallskip

\begin{proposition}\label{para11} Let $(\rho,\sigma)\in \mathfrak{V}$ be such that $(1,0)\le (\rho,\sigma)\le
(0,1)$. If $(P,Q)$ is a counterexample to JC, then
$$
v_{\rho,\sigma}(P)>0,\qquad v_{\rho,\sigma}(Q)>0\qquad\text{and}\qquad v_{\rho,\sigma}(P)+v_{\rho,\sigma}(Q) -
(\rho+\sigma) >0.
$$
\end{proposition}

\begin{proof} Note that if $(1,0)\le (\rho,\sigma)\le (1,1)$, then $\rho\ge \sigma\ge 0$, while if $(1,1)\le
(\rho,\sigma)\le (0,1)$, then $\sigma\ge \rho\ge 0$. In the case $\rho\ge \sigma\ge 0$ it is enough to prove
that $v_{\rho,\sigma}(P),v_{\rho,\sigma}(Q) >\rho$. Assume for example that $v_{\rho,\sigma}(P)\le \rho$, then
$$
(i,j)\in \Supp(P) \Longrightarrow i\rho+j\sigma\le \rho \Longrightarrow i=0, \text{ or } i=1 \text{ and } j=0,
$$
which means that $P=\mu x+f(y)$ for some $\mu\in K$ and $f\in K[y]$, and obviously $(P,Q)$ can not be a
counterexample to JC. The case $\sigma\ge \rho\ge 0$ is similar.
\end{proof}

\begin{remark}\label{no se dividen} If $(P,Q)$ is a minimal pair, then neither $v_{1,1}(P)$ divides
$v_{1,1}(Q)$ nor $v_{1,1}(Q)$ divides $v_{1,1}(P)$. This fact can be proven using a classical
argument\endnote{Assume for example that $v_{1,1}(P)$ divides $v_{1,1}(Q)$. By
Proposition~\ref{P y Q alineados}(2), there exist $\lambda_P,\lambda_Q\in K^{\times}$,
$m,n\in \mathds{N}$ and a $(1,1)$-homogeneous polynomial $R\in L$, such that
$$
\ell_{1,1}(P) = \lambda_P R^m\qquad\text{and}\qquad\ell_{1,1}(Q) = \lambda_Q R^n.
$$
Since $v_{1,1}(P)\mid v_{1,1}(Q)$, we have $n = k m$ for some $k\in\mathds{N}$. Hence,
$$
\ell_{1,1}(Q) = \ell_{1,1}\biggl(\frac{\lambda_Q}{\lambda_P^k} P^k\biggr),
$$
and so $Q_1:= Q-\frac{\lambda_Q}{\lambda_P^k} P^k$ satisfies $v_{1,1}(Q_1)<v_{1,1}(Q)$. Moreover it is clear
that $[P,Q_1] = [P,Q]\in K^{\times}$. Now we can construct successively $Q_2,Q_3,\dots$, such that $[P,Q_k]\in K^{\times}$ and
$v_{1,1}(Q_k)<v_{1,1}(Q_{k-1})$, until we obtain that $v_{1,1}(P)$ does not divide $v_{1,1}(Q_k)$. Then
$B_1:=\gcd(v_{1,1}(Q_k),v_{1,1}(P))\!<\!B$, which contradicts the minimality of~$B$. Similarly $v_{1,1}(Q)$
divides $v_{1,1}(P)$ is impossible.}
given for example in the proof of~\cite{vdE}*{Theorem~10.2.23}.
\end{remark}

\begin{definition}\label{Smp} Let $m,n\in\mathds{N}$ be coprime with $n,m>1$. A pair $(P,Q)$ of elements
$P,Q\in L^{(l)}$ (respectively $P,Q\in L$) is called an {\em $(m,n)$-pair} in $L^{(l)}$
(respectively in $L$), if
$$
[P,Q]\in K^{\times},\quad \frac{v_{1,1}(P)}{v_{1,1}(Q)}=
\frac{v_{1,0}(P)}{v_{1,0}(Q)} = \frac mn\quad\text{and}\quad v_{1,-1}(\en_{1,0}(P))<0.
$$
An $(m,n)$-pair $(P,Q)$ is called {\em a standard $(m,n)$-pair} if $P,Q\in L^{(1)}$ and
$v_{1,-1}(\st_{1,0}(P))<0$.
\end{definition}

\begin{lemma}\label{lema geometrico 1}
Let $(\rho,\sigma),(\rho',\sigma')\in\mathfrak{V}$ and $A,B\in \frac 1l\mathds{Z}\times \mathds{N}_0$ such that
$$
v_{\rho,\sigma}(A)v_{\rho',\sigma'}(B)=v_{\rho,\sigma}(B)v_{\rho',\sigma'}(A)\quad\text{and}\quad (\rho,\sigma)\times(\rho',\sigma')\ne 0.
$$
Then $A\sim B$.
 \end{lemma}

\begin{proof}
Write $A=(a_1,a_2)$ and $B=(b_1,b_2)$. The Lemma follows immediately from the equality
$$
\begin{pmatrix}\rho&\sigma\\ \rho'&\sigma'\end{pmatrix}\begin{pmatrix}a_1&b_1\\ a_2&b_2\end{pmatrix}=\begin{pmatrix}v_{\rho,\sigma}(A)&v_{\rho,\sigma}(B)\\ v_{\rho',\sigma'}(A)&v_{\rho',\sigma'}(B)\end{pmatrix},
$$
taking determinants.
\end{proof}

\begin{remark}\label{lema geometrico 2}
Let $P,Q\in L^{(l)}$ and $(\rho,\sigma)\in\mathfrak{V}$. Assume that $v_{\rho,\sigma}(P)\ne 0$, $\st_{\rho,\sigma}(P)\sim \st_{\rho,\sigma}(Q)$ and $\en_{\rho,\sigma}(P)\sim \en_{\rho,\sigma}(Q)$. Then
$$
\en_{\rho,\sigma}(Q)=\lambda \en_{\rho,\sigma}(P)\quad\text{and}\quad \st_{\rho,\sigma}(Q)=\lambda \st_{\rho,\sigma}(P),\quad\text{with}\quad
\lambda:=\frac{v_{\rho,\sigma}(Q)}{v_{\rho,\sigma}(P)}.
$$
 \end{remark}
If $(P,Q)$ is an $(m,n)$-pair, then $(Q,P)$ is an $(n,m)$-pair, as is shown by the following proposition.
\begin{proposition}\label{propiedades de los pares}
Let $(P,Q)$ be an $(m,n)$-pair. Then the following properties hold:
\begin{enumerate}
  \item $v_{1,0}(P),v_{1,0}(Q)> 0$.
  \item $\en_{1,0}(Q) \sim \en_{1,0}(P)$ and $\en_{1,0}(Q) = \frac nm \en_{1,0}(P)$.
  \item $\frac{1}{m}\en_{1,0}(P)=\frac{1}{n}\en_{1,0}(Q)\in \frac{1}{l}\mathds{N}\times \mathds{N}$ and
$v_{1,-1}(\en_{1,0}(Q))<0$.
\item $v_{0,-1}(\en_{1,0}(P))<-1$ and $v_{0,-1}(\en_{1,0}(Q))<-1$.
\item Neither $P$ nor $Q$ are monomials.
\end{enumerate}
\end{proposition}

\begin{proof}
Item~(1) follows from inequality~\eqref{dfpvc}, since $v_{10}(P)<0$ implies $v_{10}(Q)<0$. Now we prove item~(2). Assume by
contradiction that $\en_{1,0}(Q) \nsim \en_{1,0}(P)$. By Propositions~\ref{pr v de un conmutador}, \ref{extremosalineados}, and~\ref{extremosnoalineados}(2) we have
\begin{equation}\label{e8}
\en_{1,0}(Q)+\en_{1,0}(P)=(1,1),
\end{equation}
which combined with the fact that $v_{1,-1}(\en_{1,0}(P))<0$ and $v_{1,0}(P)> 0$ implies that there exists $0<r<l$ with
$\en_{1,0}(P)=(r/l,1)$ and $\en_{1,0}(Q)=((l-r)/l,0)$. Set
$$
M:=\{(1,1)\}\cup ((\Dir(P)\cup \Dir(Q))\cap ](1,0),(1,1)[)= \{(\rho_0,\sigma_0)<\dots<(\rho_k,\sigma_k)=(1,1)\}.
$$
We claim that
\begin{equation}\label{suma de valuaciones mayor que 1}
v_{0,1}(\st_{\rho_j,\sigma_j}(P))+v_{0,1}(\st_{\rho_j,\sigma_j}(Q))>1,\quad\text{for $j>0$}.
\end{equation}
In fact, if $k=0$ this is trivial. Otherwise, by Proposition~\ref{le basico} and Remark~\ref{starting vs end}, we have
$$
v_{0,1}(\st_{\rho_j,\sigma_j}(P))\le v_{0,1}(\en_{\rho_j,\sigma_j}(P))=v_{0,1}(\st_{\rho_{j+1},\sigma_{j+1}}(P)),\quad\text{for $0\le j < k$,}
$$
with strict inequality if $(\rho_j,\sigma_j)\in \Dir(P)$, and the same is true for $Q$. The claim follows immediately from these facts, since
$$
(\rho_0,\sigma_0)\in\Dir(P)\cup\Dir(Q)\quad\text{and}\quad v_{0,1}(\st_{\rho_0,\sigma_0}(P))+v_{0,1}(\st_{\rho_0,\sigma_0}(Q))=1,
$$
where the equality follows from Proposition~\ref{le basico} and equality~\eqref{e8}.

Inequality~\eqref{suma de valuaciones mayor que 1} implies that
$\st_{\rho_j,\sigma_j}(P)+\st_{\rho_j,\sigma_j}(Q)\ne (1,1)$ for $j>0$, and so Proposition~\ref{extremosnoalineados}(1) and
Proposition~\ref{le basico} yield
\begin{equation}\label{e10}
\en_{\rho_j,\sigma_j}(P) = \st_{\rho_{j+1},\sigma_{j+1}}(P)\sim \st_{\rho_{j+1},\sigma_{j+1}}(Q)=\en_{\rho_j,\sigma_j}(Q)\quad\text{for $0\le j<k$.}
\end{equation}
On the other hand, $\rho_j>0$ because $(\rho_j,\sigma_j)\in\, ](1,0),(1,1)]$, and hence,
$$
v_{\rho_j,\sigma_j}(Q)\ge v_{\rho_j,\sigma_j}(\en_{1,0}(Q))= \frac{l-r}{l}>0.
$$
This allows us to use Remark~\ref{lema geometrico 2} combined with~\eqref{e10}, in order to prove inductively that
\begin{equation}\label{cocientes iguales}
\frac{v_{\rho_0,\sigma_0}(P)}{v_{\rho_0,\sigma_0}(Q)}=\frac{v_{\rho_k,\sigma_k}(P)}{v_{\rho_k,\sigma_k}(Q)}=\frac mn= \frac{v_{1,0}(P)}{v_{1,0}(Q)}.
\end{equation}
Set $A:=\en_{1,0}(P)$ and $B:=\en_{1,0}(Q)$. By Proposition~\ref{le basico}, we have $v_{\rho_0,\sigma_0}(A)=v_{\rho_0,\sigma_0}(P)$ and
$v_{\rho_0,\sigma_0}(B)=v_{\rho_0,\sigma_0}(Q)$. Consequently, by~\eqref{cocientes iguales},
$$
v_{1,0}(A)v_{\rho_0,\sigma_0}(B) - v_{1,0}(B)v_{\rho_0,\sigma_0}(A) =  v_{1,0}(P)v_{\rho_0,\sigma_0}(Q) - v_{1,0}(Q)v_{\rho_0,\sigma_0}(P) = 0,
$$
which, by Lemma~\ref{lema geometrico 1} with $(\rho,\sigma)=(1,0)$ and $(\rho',\sigma')=(\rho_0,\sigma_0)$, leads to $A\sim B$, contradicting the assumption that $\en_{1,0}(Q) \nsim \en_{1,0}(P)$ and proving item~(2).

From item~(2) we obtain
$$
\frac{1}{m}\en_{1,0}(P)=\frac{1}{n}\en_{1,0}(Q)\in \frac{1}{l}\mathds{N}\times \mathds{N}_0
\quad\text{and}\quad v_{1,-1}(\en_{1,0}(Q))<0.
$$
But $v_{0,1}(\en_{1,0}(P)),v_{0,1}(\en_{1,0}(Q))>0$, since $v_{1,-1}(\en_{1,0}(P)),v_{1,-1}(\en_{1,0}(Q))<0$, and so item~(3) holds.
Thus $v_{0,-1}(\en_{1,0}(P))<-1$ and $v_{0,-1}(\en_{1,0}(Q))<-1$\endnote{
We prove this only for $P$ since the proof for $Q$ is similar. Since $v_{0,-1}(\en_{1,0}(P))\in -m\mathds{N}_0$, it suffices
to show that $v_{0,-1}(\en_{1,0}(P))<0$. But otherwise, $\en_{1,0}(P) = (a,0)$ for some $a\in \mathds{Q}$, and so
$$
v_{1,0}(P) = v_{1,0}(\en_{1,0}(P)) = v_{1,-1}(\en_{1,0}(P))<0,
$$
which is impossible since $v_{1,0}(P)>0$.},
which is item~(4). In order to check item~(5), assume for instance that $P$ is a monomial. Then, by item~(4),
$$
v_{0,-1}(P)+v_{0,-1}(Q)=v_{0,-1}(\en_{1,0}(P))+v_{0,-1}(Q)<-1+0,
$$
which contradicts inequality~\eqref{dfpvc}.
\end{proof}

\begin{figure}[htb]
\centering
\begin{tikzpicture}
\fill[gray!20] (0,0) -- (0.25,0.75) -- (1,2.5) -- (1,0.5) -- (0.5,0) -- (0,0);
\draw  [thick] (0.25,0.75) -- (1,2.5) -- (1,0.5);
\draw[step=.5cm,gray,very thin] (0,0) grid (2,3.6);
\draw [->] (0,0) -- (2.3,0) node[anchor=north]{$x$};
\draw [->] (0,0) --  (0,4) node[anchor=east]{$y$};
\draw[dotted] (0,0) -- (2.3,2.3);
\draw[dotted] (0,1.5) -- (1,2.5);
\draw [->,thick] (1.5,2.5) -- (2,2.5) node[fill=white,anchor=west]{\tiny{(1,0)}};
\draw [->,thick] (0.5,2.5) -- (-0.5,3) node[fill=white,anchor=east]{\tiny{$\Succ_{\varphi(P)}(1,0)$}};
\end{tikzpicture}
\caption{The shape of $\varphi(P)$ according to Proposition~\ref{primera condicion estandar}.}
\end{figure}
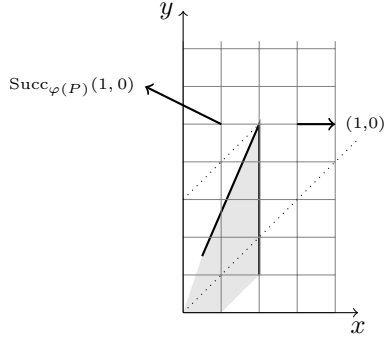

\begin{proposition}\label{primera condicion estandar} Let $(P,Q)$ be a minimal pair. Then there exist
$m,n\in \mathds{N}$ which are coprime, and $\varphi\in\Aut(L)$ such that $(\varphi(P),\varphi(Q))$ is an
$(m,n)$-pair satisfying $v_{1,1}(\varphi(P))=v_{1,1}(P)$ and $v_{1,1}(\varphi(Q))=v_{1,1}(Q)$.
Moreover,
\begin{equation}
(-1,1)<\Succ_{\varphi(P)}(1,0),\Succ_{\varphi(Q)}(1,0)<(-1,0).\label{ekua4}
\end{equation}
\end{proposition}

\begin{proof} Since $(P,Q)$ is a counterexample to JC, by \cite{vdE}*{Corollary 10.2.21} there exists an
automorphism $\varphi$ of $L$ and integers $1\le a\le b$ such that
\begin{equation}\label{subrectangular}
(a,b)\in \Supp(\varphi(P))\subseteq \{(i,j): 0\le i\le a,\ 0\le j\le b\}.
\end{equation}
We can also achieve that inequality~\eqref{ekua4} is satisfied. This is a well known fact (see for
ins\-tance~\cite{ML}*{page~8} or \cite{L}*{discussion at~1.12})\endnote{
Let
$(\ov{P},\ov{Q}):=(\varphi(P),\varphi(Q))$. It is enough to prove that there exists $\varphi'\in \Aut(L)$ such
that
\begin{equation}
(-1,1)<\Succ_{\varphi'(\ov{P})}(1,0),\Succ_{\varphi'(\ov{Q})}(1,0)<(-1,0),\label{ekuaa1}
\end{equation}
and
\begin{equation}\label{subrectangular1}
(a,b)\in \Supp(\varphi'(\ov{P}))\subseteq \{(i,j): 0\le i\le a,\ 0\le j\le b\}.
\end{equation}
Let $\psi_1\in \Aut(L)$ be the flip given by $\psi_1(x):=y$ and $\psi_1(y):=-x$. Define
$P_1 := \psi_1(\ov{P})$ and $Q_1 := \psi_1(\ov{Q})$. In order to obtain~\eqref{ekuaa1}
and~\eqref{subrectangular1} it suffices to show that there exists $\varphi_1'\in \Aut(L)$ such that
\begin{equation}
(0,-1)<\Pred_{\varphi'_1(P_1)}(0,1),\Pred_{\varphi'_1(Q_1)}(0,1)<(1,-1).\label{ekuaa2}
\end{equation}
and
\begin{equation}\label{subrectangular2}
(b,a)\in \Supp(\varphi'_1(P_1))\subseteq \{(i,j): 0\le i\le b,\ 0\le j\le a\}.
\end{equation}
In fact, these facts imply
\begin{equation*}
(-1,1)<\Succ_{\psi_1(\varphi'_1(P_1))}(1,0),\Succ_{\psi_1(\varphi'_1(Q_1))}(1,0)<(-1,0),
\end{equation*}
and
\begin{equation*}
(a,b)\in \Supp(\psi_1(\varphi'_1(P_1)))\subseteq \{(i,j): 0\le i\le a,\ 0\le j\le b\}.
\end{equation*}
which allow us to take $\varphi' := \psi_1\xcirc \varphi'_1\xcirc \psi_1$. Note that the Jacobian
de\-ter\-mi\-nant of $\psi_1$ is $1$, and so $(P_1,Q_1)$ is a counterexample to JC. Moreover
\begin{equation*}
(b,a)\in \Supp(P_1)\subseteq \{(i,j): 0\le i\le b,\ 0\le j\le a\}.
\end{equation*}
Consequently, $\en_{P_1}(1,0) = (b,a) = \st_{P_1}(0,1)$ and so, by Proposition~\ref{le bbasico},
$$
(0,-1) \le \Pred_{P_1}(0,1)\le (1,0).
$$
Now we are going to prove that there exists $\varphi_1'\in \Aut(L)$ such that~\eqref{subrectangular2} is
satisfied and
$$
(0,-1) \le \Pred_{\varphi_1'(P_1)}(0,1) <(1,0).
$$
If $(0,-1) \le \Pred_{P_1}(0,1)< (1,0)$, the  we can take $\varphi_1':=\ide$. So, we can assume that
$\Pred_{P_1}(0,1) = (1,0)$. Since $v_{1,0}(P_1) = b>0$, by Theorem~\ref{central} there exists a
$(1,0)$-homogeneous polynomial $F$ such that
$$
v_{1,0}(F) = 1\qquad\text{and}\qquad [F,\ell_{1,0}(P_1)] = \ell_{1,0}(P_1).
$$
Moreover by item~3) of that theorem, we know that
$$
\en_{1,0}(F) = (1,1)\qquad\text{or}\qquad \en_{1,0}(F) \sim \en_{1,0}(P_1) = (b,a).
$$
On the other hand since $v_{1,0}(F) = 1$ and $F$ is $(1,0)$-homogeneous, there exists a univariate polynomial
$f\ne 0$ such that $F = xf(y)$, which implies that $\Supp(F)\subseteq \{(1,k):k\in \mathds{N}_0\}$.
Consequently $\en_{1,0}(F) = (1,1)$, because $(1,k)\sim (b,a)$ is impossible since $b>a>0$. By
Remark~\ref{F no es monomio} and Proposition~\ref{pavadass}(2) we know that $\st_{(1,0)}(F) = (1,0)$. Thus
$F = \mu_Fx(y-\lambda)$ for some $\mu_F,\lambda\in K^{\times}$. Therefore, by
Proposition~\ref{pavadass}(1), there exists $\mu_{P_1}\in K^{\times}$ such that $\ell_{1,0}(P_1) = \mu_{P_1}
x^b(y-\lambda)^a$. Let $\varphi'_1\in \Aut(L)$ be the automorphism defined by $\varphi'_1(x) := x$ and
$\varphi'_1(y) := y+\lambda$. By Proposition~\ref{pr ell por automorfismos},
$$
\ell_{1,0}(\varphi'_1(P_1)) = \mu_{P_1} x^by^a\qquad\text{and}\qquad \ell_{\rho_1,\sigma_1}(\varphi'_1(P_1)) =
\ell_{\rho_1,\sigma_1}(P_1)\quad\text{for all $(1,0)<(\rho_1,\sigma_1)\le (0,1)$.}
$$
So,~\eqref{subrectangular2} is satisfied and $(0,-1)\le \Pred_{\varphi_1'(P_1)}(0,1)<(1,0)$. Combining this
with Proposition~\ref{para11} and Remark~\ref{a remark}, we obtain that
$$
\Dir(\varphi_1'(Q_1))\cap [(1,0),(0,1)] = \Dir(\varphi_1'(P_1))\cap [(1,0),(0,1)] = \emptyset,
$$
and so $(0,-1)\le \Pred_{\varphi_1'(Q_1)}(0,1) <(1,0)$. In order to finish the proof we must see that
$$
\Pred_{\varphi_1'(P_1)}(0,1),\Pred_{\varphi_1'(Q_1)}(0,1)\notin [(1,-1),(1,0)[.
$$
We will prove this only for $\varphi_1'(P_1)$ since the argument for $\varphi_1'(Q_1)$ is the same.
By~\cite{CN} it is  impossible that $(\rho,\sigma)=(1,-1)$. Assume $(\rho,\sigma)\in
\hspace{0.7pt}](1,-1),(1,0)[$, which
means that $\rho>-\sigma>0$. Since
$$
v_{\rho,\sigma}(\varphi_1'(P_1)) \ge v_{\rho,\sigma}(b,a) = \rho b+\sigma a > (\rho +\sigma) a \ge \rho+
\sigma >0,
$$
we can apply Theorem~\ref{central}. Hence, there exist a $(\rho,\sigma)$-homogeneous polynomial $F_1$ such
that $v_{\rho,\sigma}(F_1) = \rho + \sigma$ and
$$
(b,a) = \en_{\rho,\sigma}(\varphi_1'(P_1))\sim \en_{\rho,\sigma}(F_1) \qquad\text{or}\qquad
\en_{\rho,\sigma}(F_1) = (1,1).
$$
If $(b,a)\sim \en_{\rho,\sigma}(F_1)$, then there exists $\lambda>0$ such that $\en_{\rho,\sigma}(F_1) =
\lambda (b,a)$. So
$$
\rho + \sigma = v_{\rho,\sigma}(F_1) = \rho \lambda b + \lambda\sigma a > \lambda a(\rho+\sigma)
\Longrightarrow
0<\lambda a < 1,
$$
which is impossible, since $\lambda a = v_{0,1}(\en_{\rho,\sigma}(F_1))\in \mathds{Z}$. Consequently
$\en_{\rho,\sigma}(F_1) = (1,1)$. By Remarks~\ref{starting vs end} and~\ref{F no es monomio}, we have
$v_{0,1}(\st_{\rho,\sigma}(F_1)) < v_{0,1}(\en_{\rho,\sigma}(F_1)) = 1$. Therefore $\st_{\rho,\sigma}(F_1) =
(k,0)$ for some $k\in \mathds{N}$, which leads to the contradiction $\rho + \sigma =
v_{\rho,\sigma}(\st_{\rho,\sigma}(F_1)) = \rho k \ge \rho > \rho + \sigma$, and finishes the proof.
}.
We claim that there exist $m,n\in \mathds{N}$ such that $(\ov{P},\ov{Q}):=(\varphi(P),\varphi(Q))$ is an
$(m,n)$-pair. Clearly
\begin{equation}
\en_{1,0}(\ov{P})=(a,b)=\st_{1,1}(\ov{P}).\label{ekua1}
\end{equation}
Moreover $v_{1,-1}(\en_{1,0}(\ov{P})) = a - b < 0$, since by Theorem~\ref{central}(4), we have
$a<b$. Now we prove that there exist $m,n\in \mathds{N}$ coprime, such that
\begin{equation}
\frac{v_{1,1}(\ov{P})}{v_{1,1}(\ov{Q})} = \frac mn = \frac{v_{1,0}(\ov{P})}{v_{1,0}(\ov{Q})}. \label{ekua0}
\end{equation}
By Proposition~\ref{para11} the hypotheses of Remark~\ref{a remark} are satisfied for $(\ov{P},\ov{Q})$
and all $(\rho,\sigma)\in \mathfrak{V}$ such that $(1,0)\le (\rho,\sigma) \le (1,1)$. Hence there exists
$m,n\in \mathds{N}$ coprime such that $\frac{v_{11}(\ov{P})}{v_{11}(\ov{Q})} = \frac m n$,
\begin{equation}
(a,b) = \st_{1,1}(\ov{P}) = \frac m n \st_{1,1}(\ov{Q})\label{ekua2}
\end{equation}
and
$$
\Dir(\ov{Q})\cap \hspace{0.7pt} ](1,0),(1,1)[\, = \Dir(\ov{P})\cap
\hspace{0.9pt}](1,0),(1,1)[\, = \emptyset,
$$
where the last equality follows from~\eqref{ekua1} and Proposition~\ref{le bbasico}. Hence by
Proposition~\ref{le basico}
we have $\en_{1,0}(\ov{Q})=\st_{1,1}(\ov{Q})$ which, combined with~\eqref{ekua1} and~\eqref{ekua2}, gives
\begin{equation*}
\en_{1,0}(\ov{P}) = \frac m n \en_{1,0}(\ov{Q}).
\end{equation*}
This yields equality~\eqref{ekua0}.

Next we prove that
\begin{equation}
v_{1,1}(\ov{P})=v_{1,1}(P)\qquad\text{and}\qquad v_{1,1}(\ov{Q})=v_{1,1}(Q).\label{ekuaa3}
\end{equation}
For this consider the inverse $\psi:=\varphi^{-1}$. Set $M:=v_{1,1}(\psi(x))$ and $N:=v_{1,1}(\psi(y))$.
By~\cite{J} and~\cite{vdE}*{Corollary~5.1.6(a)}, we know that either $N|M$ or $M|N$. If $M=N=1$ then
clearly $\psi$ and $\varphi$ preserve $v_{1,1}$, as desired.

We assert that the case $M|N$ and $N>1$, and the case $N|M$ and $M>1$, are impossible. Assume for example $M|N$
and $N>1$ and set $R:=\ell_{1,1}(\psi(x))$. Since
$$
v_{1,1}([\psi(x),\psi(y)]) = 0 < M + N -2 =  v_{1,1}(\psi(x)) + v_{1,1}(\psi(y)) - 2,
$$
It follows from Proposition~\ref{pr v de un conmutador}, that $[\ell_{1,1}(\psi(x)),\ell_{1,1}(\psi(y))] = 0$.
Hence, by Proposition~\ref{P y Q alineados},
$$
\ell_{1,1}(\psi(y))=\lambda R^k\qquad\text{for some $\lambda\in K^{\times}$ and $k\in \mathds{N}$.}
$$
By~\eqref{subrectangular} we know that $\ell_{1,1}(\ov{P}) = \lambda_P x^ay^b$ for some $\lambda_P\in
K^{\times}$, and that
$$
i\le a,\quad j\le b\quad\text{and}\quad i+j< a+b\qquad\text{for all $(i,j)\in
\Supp(\ov{P})\setminus\{(a,b)\}$.}
$$
Hence, for all such $(i,j)$, we have
$$
v_{1,1}(\psi(x^i y^j))=i v_{1,1}(R)+ j v_{1,1}(R^k)<a v_{1,1}(R)+ b v_{1,1}(R^k)= v_{1,1}(\psi(x^ay^b)),
$$
and so
\begin{equation}
v_{1,1}(\psi(\ov{P}))=v_{1,1}(\psi(x^ay^b))= v_{1,1}(R)(a+kb).\label{ekua5}
\end{equation}
On the other hand by equality~\eqref{ekua2} we can write $a = \bar a m$ and $b = \bar b m$ with $\bar a,\bar
b\in \mathds{N}$. Hence equality~\eqref{ekua5} can be written as
\begin{equation*}
v_{1,1}(\psi(\ov{P})) = m v_{1,1}(R)(\bar a+k \bar b).
\end{equation*}
By~\eqref{ekua2} we have $\st_{1,1}(\ov{Q}) = \frac n m (a,b) = n(\bar a,\bar b)$. So, by
Proposition~\ref{para11} and Remark~\ref{a remark},
\begin{equation}\label{subrectangular'}
(n\bar a,n\bar b)\in \Supp(\ov{Q})\subseteq \{(i,j): 0\le i\le n\bar a,\ 0\le j\le n\bar b\}.
\end{equation}
A similar computation
as above, but using~\eqref{subrectangular'} instead of~\eqref{subrectangular}, shows that
$$
v_{1,1}(\psi(\ov{Q}))=n v_{1,1}(R)(\bar a+k \bar b).\endnote{
By~\eqref{subrectangular'}, there exists $\lambda_Q\in K^{\times}$ such that $\ell_{1,1}(\ov{Q}) = \lambda_Q
x^{n\bar a}y^{n\bar b}$, and
$$
i\le n\bar a,\quad j\le n\bar b\quad\text{and}\quad i+j< n\bar a+n\bar b\qquad\text{for all $(i,j)\in
\Supp(\ov{Q})\setminus\{(n\bar a,n\bar b)\}$.}
$$
Thus, for all such $(i,j)$, we have
$$
v_{1,1}(\psi(x^i y^j))=i v_{1,1}(R)+ j v_{1,1}(R^k)<n\bar a v_{1,1}(R)+ n\bar b v_{1,1}(R^k)=
v_{1,1}(\psi(x^{n\bar a}y^{n\bar b})),
$$
and so $v_{1,1}(\psi(\ov{Q}))=v_{1,1}(\psi(x^{n\bar a}y^{n\bar b}))= nv_{1,1}(R)(\bar a+k\bar b)$.
}
$$
Consequently
$$
\gcd(v_{1,1}(\psi(\ov{P}))v_{1,1}(\psi(\ov{Q})))=v_{1,1}(R)(\bar a+k\bar b)\ge \bar a + \bar b
=\gcd(v_{1,1}(\ov{P}), v_{1,1}(\ov{Q})),
$$
where the last equality follows from equality~\eqref{ekua2}.

Since $(\psi(\ov{P}),\psi(\ov{Q})))=(P,Q)$ is a minimal pair, equality must hold, and so we have
$k=1$ and $v_{1,1}(R)=1$, which contradicts $k v_{1,1}(R) = v_{1,1}(\ell_{1,1}(\psi(y))) = N >1$.

Similarly one discards the case $N|M$ and $M>1$, which finishes the proof of~~\eqref{ekuaa3}.

Hence $(\ov{P},\ov{Q})$ is minimal pair and so, by Remark~\ref{no se dividen}, we have $m,n>1$.
\end{proof}

\section{Regular corners of $\bm{(m,n)}$-pairs}

\setcounter{equation}{0}

It is known (see e.g. \cite{vdE}*{Theorem 10.2.1}) that the Newton polygons of a Jacobian pair $(P,Q)$
in $L$ are similar. The same is not true in $L^{(l)}$, but it is almost true. One of the basic geometric
reasons for this difference is the fact that, by Propositions~\ref{pr v de un conmutador},
\ref{extremosalineados} and~\ref{extremosnoalineados}, if two corners of $P$ and $Q$ are not aligned, then
they must sum to $(1,1)$. In $L$ this is only possible for $(1,0)$ and $(0,1)$, but in $L^{(l)}$ this happens for
all $(k/l,0)$ and $(1-k/l,1)$ if $k\in \mathds{Z}\setminus\{0\}$ (see Case~I.b),
equality~\eqref{conjunto de starting}).

We will analyze the edges and corners of the Newton polygons of an $(m,n)$-pair, corresponding to the
directions in
$$
I:=\,](1,-1),(1,0)]=\{(\rho,\sigma)\in \mathfrak{V}: (1,-1)<(\rho,\sigma)\le (1,0) \}.
$$
Note that for $(\rho,\sigma)\in I$ we have $\rho+\sigma>0$, $\sigma\le 0$ and $\rho>0$. In particular we
will analyze what we call regular corners (see Definition~\ref{def regular corner}). The conditions
we will find on regular corners will allow us to discard many ``small'' cases
in Sections~\ref{Lower bounds} and~\ref{More conditions on B}, and to obtain lower bounds for $B$.

From now on we assume that $K$ is algebraically closed unless otherwise stated.

\begin{lemma}\label{v positivos}
Let $(P,Q)$ be an $(m,n)$-pair in $L^{(l)}$ and let $(\rho,\sigma)\in I$. If $\en_{\rho,\sigma}(P)=\frac mn
\en_{\rho,\sigma}(Q)$, then $v_{\rho,\sigma}(P)>0$ and $v_{\rho,\sigma}(Q)>0$. Moreover, if
$v_{0,-1}(\st_{\rho,\sigma}(P))<-1$ or $v_{0,-1}(\st_{\rho,\sigma}(Q))<-1$,
 then $[\ell_{\rho,\sigma}(P),\ell_{\rho,\sigma}(Q)]= 0$.
\end{lemma}

\begin{proof}
Assume by contradiction that $v_{\rho,\sigma}(P)\le 0$. Then $v_{\rho,\sigma}(Q)= \frac nm v_{\rho,\sigma}(P)
\le 0$. But then, since $\rho+\sigma>0$, we have
$$
v_{\rho,\sigma}(P)+ v_{\rho,\sigma}(Q)-(\rho+\sigma)<0=v_{\rho,\sigma}([P,Q]),
$$
which contradicts~\eqref{dfpvc} and proves $v_{\rho,\sigma}(P)>0$. The same argument proves that
$v_{\rho,\sigma}(Q)>0$.

Now assume for instance that
$$
v_{0,-1}(\st_{\rho,\sigma}(P))<-1\qquad\text{and}\qquad [\ell_{\rho,\sigma}(P),\ell_{\rho,\sigma}(Q)]\ne 0.
$$
Since $(0,-1)<(\rho,\sigma)<(0,1)$, by Remark~\ref{starting vs end} we have
$$
v_{0,-1}(\ell_{\rho,\sigma}(P))=v_{0,-1}(\ell_{0,-1}(\ell_{\rho,\sigma}(P)))=v_{0,-1}(\st_{\rho,\sigma}(P))<-1,
$$
and so, we obtain
$$
v_{0,-1}(\ell_{\rho,\sigma}(P))+v_{0,-1}(\ell_{\rho,\sigma}(Q))-(-1+0)<0=
v_{0,-1}([\ell_{\rho,\sigma}(P),\ell_{\rho,\sigma}(Q)]),
$$
which contradicts inequality~\eqref{dfpvc}, proving $[\ell_{\rho,\sigma}(P),\ell_{\rho,\sigma}(Q)]= 0$.
Similar arguments apply to the case $v_{0,-1}(\st_{\rho,\sigma}(Q))<-1$.
\end{proof}

For $P\in L^{(l)}\setminus \{0\}$ we set
$$
A(P):=\{(\rho,\sigma)\in \Dir(P)\cap I: v_{0,-1}(\st_{\rho,\sigma}(P))<-1\text{ and }
v_{1,-1}(\st_{\rho,\sigma}(P))<0\}.
$$

\begin{proposition}\label{A(P) vs Dir(P)} Let $(\rho,\sigma)\in \Dir(P)\cap I$. If $(\rho',\sigma')
<(\rho,\sigma)\le (1,0)$ for some $(\rho',\sigma')\in A(P)$, then $(\rho,\sigma)\in A(P)$.
\end{proposition}

\begin{proof} It suffices to prove the result in the case in which $\en_{\rho',\sigma'}(P) =
\st_{\rho,\sigma}(P)$. In this case, since $(0,-1) <(\rho',\sigma')<(0,1)$ and $(1,-1)<(\rho',\sigma')
<(-1,1)$, it follows from Remark~\ref{starting vs end} that
$$
v_{0,-1}(\st_{\rho,\sigma}(P)) = v_{0,-1}(\en_{\rho',\sigma'}(P)) < v_{0,-1}(\st_{\rho',\sigma'}(P)) < -1
$$
and
$$
v_{1,-1}(\st_{\rho,\sigma}(P)) = v_{1,-1}(\en_{\rho',\sigma'}(P)) <v_{1,-1}(\st_{\rho',\sigma'}(P)) < 0,
$$
which implies that $(\rho,\sigma)\in A(P)$.
\end{proof}

\begin{proposition}\label{A(P) vs A(Q)} Let $(P,Q)$ be an $(m,n)$-pair and $(\ov{\rho},\ov{\sigma}):=
\max(A(P))$. Then
$$
](\ov{\rho},\ov{\sigma}),(1,0)]\cap \Dir(Q) =  \emptyset.
$$
\end{proposition}

\begin{proof} Assume that the statement is false and take
$$
(\rho,\sigma):=\,\max\bigl(](\ov{\rho},\ov{\sigma}),(1,0)]\cap \Dir(Q)\bigr).
$$
By Proposition~\ref{A(P) vs Dir(P)} we know that $](\ov{\rho},\ov{\sigma}),(1,0)]\cap \Dir(P) = \emptyset$.
Hence, by Proposition~\ref{le basico},
\begin{equation}\label{pepitito3}
\st_{\rho,\sigma}(P) = \en_{\rho,\sigma}(P) = \en_{1,0}(P),
\end{equation}
and so, by Proposition~\ref{propiedades de los pares}(4),
$$
v_{0,-1}(\st_{\rho,\sigma}(P)) = v_{0,-1}(\en_{1,0}(P)) < -1.
$$
On the other hand,
$$
\en_{\rho,\sigma}(P) = \en_{1,0}(P) = \frac m n \en_{1,0}(Q) = \frac m n \en_{\rho,\sigma}(Q),
$$
where the first equality follows from~\eqref{pepitito3}, the second on from the definition of $(m,n)$-pair,
and the third one, from
the fact that $](\rho,\sigma),(1,0)]\cap \Dir(Q) = \emptyset$ and Proposition~\ref{le basico}.
Hence, by Lemma~\ref{v positivos} and Remark~\ref{a remark}, we conclude that $(\rho,\sigma)\in \Dir(P)$,
which is a contradiction.
\end{proof}

\begin{proposition}\label{esquinas regulares} If $(P,Q)$ is an $(m,n)$-pair and $(\rho,\sigma)\in A(P)$, then
\begin{enumerate}

\smallskip

\item $\en_{\rho,\sigma}(P)=\frac mn \en_{\rho,\sigma}(Q)$,

\smallskip

\item $\st_{\rho,\sigma}(P)=\frac mn \st_{\rho,\sigma}(Q)$,

\smallskip

\item $(\rho,\sigma)\in \Dir(Q)$.
\end{enumerate}
Moreover $A(Q)=A(P)$ and, if we set
$$
(\rho_1,\sigma_1):=\begin{cases} \min(A(P))&\text{if $A(P)\ne\emptyset$,}\\ \min(\Succ_P(1,0),\Succ_Q(1,0))
&\text{if $A(P)=\emptyset$,}
\end{cases}
$$
then $\Pred_P(\rho_1,\sigma_1) =\Pred_Q(\rho_1,\sigma_1) \in I$.
\end{proposition}

\begin{proof}
Assume $A(P)\ne \emptyset$ and write $A(P)=\{(\rho_1,\sigma_1)<(\rho_2,\sigma_2)<\dots < (\rho_k,\sigma_k)\}$,
where we are considering the order of $I$. We will prove inductively statements~(1), (2) and~(3) for
$(\rho_j,\sigma_j)$, starting from $j=k$. Let $(\ov{\rho},\ov{\sigma}):=\max(A(P)\cup A(Q))$. We have
\begin{equation*}\label{1 cond}
\en_{\ov\rho,\ov\sigma}(P)=\en_{1,0}(P)=\frac mn \en_{1,0}(Q)=\frac mn \en_{\ov\rho,\ov\sigma}(Q),
\end{equation*}
where the first equality follows from Propositions~\ref{le basico} and~\ref{A(P) vs Dir(P)},
the second one, from Proposition~\ref{propiedades de los pares}(2) and the third one, from
Propositions~\ref{le basico},
since $](\ov{\rho},\ov{\sigma}),(1,0)]\cap\Dir(Q) =\emptyset$ by Propositions~\ref{A(P) vs Dir(P)}
and~\ref{A(P) vs A(Q)}. Hence,
by Lemma~\ref{v positivos} and Remark~\ref{a remark}, we have
$$
(\ov\rho,\ov\sigma)\in\Dir(P)\cap \Dir(Q)\qquad\text{and}\qquad \st_{\ov\rho,\ov\sigma}(P) = \frac mn
\st_{\ov\rho,\ov\sigma}(Q).
$$
On the other hand by Proposition~\ref{A(P) vs Dir(P)}, we have $(\ov\rho,\ov\sigma)=(\rho_k,\sigma_k)$, and so
statements~(1), (2) and~(3) hold for $(\rho_k,\sigma_k)$.

Let now $j\ge 1$, assume that statements~(1), (2) and~(3) hold for $(\rho_{j+1},\sigma_{j+1})$ and set
$$
(\tilde{\rho},\tilde{\sigma})=\max\{\Pred_P(\rho_{j+1},\sigma_{j+1}),\Pred_Q(\rho_{j+1},\sigma_{j+1})\}.
$$
Then
\begin{equation*}\label{cond1'}
\en_{\tilde{\rho},\tilde{\sigma}}(P)=\st_{\rho_{j+1},\sigma_{j+1}}(P)=\frac mn\st_{\rho_{j+1},\sigma_{j+1}}(Q)
=\frac mn\en_{\tilde{\rho},\tilde{\sigma}}(Q),
\end{equation*}
where the second equality holds by condition~2) for $(\rho_{j+1},\sigma_{j+1})$. Moreover by
Propositions~\ref{A(P) vs Dir(P)} and~\ref{le basico},
$$
(\tilde{\rho},\tilde{\sigma}) = (\rho_j,\sigma_j)\qquad\text{or}\qquad \st_{\tilde{\rho},\tilde{\sigma}}(P) =
\en_{\tilde{\rho},\tilde{\sigma}}(P) = \st_{\rho_{j+1},\sigma_{j+1}}(P),
$$
and so $v_{0,-1}(\st_{\tilde{\rho},\tilde{\sigma}}(P))<-1$. Hence, by Lemma~\ref{v positivos} and
Remark~\ref{a remark}, we have
$$
(\tilde{\rho},\tilde{\sigma})\in\Dir(P)\cap \Dir(Q)\qquad\text{and}\qquad \st_{\tilde{\rho},\tilde{\sigma}}(P)
= \frac mn \st_{\tilde{\rho},\tilde{\sigma}}(Q).
$$
On the other hand, by Proposition~\ref{A(P) vs Dir(P)} we have $(\tilde{\rho},\tilde{\sigma}) =
(\rho_j,\sigma_j)$, and so statements~(1), (2) and~(3) hold for $(\rho_j,\sigma_j)$.

\smallskip

Now we will prove that $A(P) = A(Q)$. By symmetry it suffices to prove that $A(P)\subseteq A(Q)$. Let
$(\rho,\sigma)\in A(P)$. By statement~(3) we already know $(\rho,\sigma)\in \Dir(Q)$. So we have to prove only
that
$$
v_{0,-1}(\st_{\rho,\sigma}(Q))<-1\quad\text{and}\quad v_{1,-1}(\st_{\rho,\sigma}(Q))<0.
$$
By statement~(2)
$$
v_{1,-1}(\st_{\rho,\sigma}(Q))=\frac nm v_{1,-1}(\st_{\rho,\sigma}(P))<0.
$$
Note now that again by statement~(2)
$$
\frac 1m v_{0,-1}(\st_{\rho,\sigma}(P))\in\mathds{Z},
$$
and so
$$
\frac 1m v_{0,-1}(\st_{\rho,\sigma}(P))\le -1,
$$
since $v_{0,-1}(\st_{\rho,\sigma}(P))<0$. Hence, once again by statement~(2),
$$
v_{0,-1}(\st_{\rho,\sigma}(Q))=\frac nm v_{0,-1}(\st_{\rho,\sigma}(P))\le -n<-1,
$$
which proves $A(P)\subseteq A(Q)$, as desired.

\smallskip

Now we prove
$$
(\rho_0,\sigma_0):=\Pred_P(\rho_1,\sigma_1)=\Pred_Q(\rho_1,\sigma_1)\in I.
$$
Set $(\hat{\rho},\hat{\sigma}):=\max\{\Pred_P(\rho_{1},\sigma_{1}),\Pred_Q(\rho_{1},\sigma_{1})\}$.
We first prove that
$$
(1,-1)< (\hat{\rho},\hat{\sigma}) < (\rho_{1},\sigma_{1}).
$$
Assume by contradiction that $(-\rho_1,-\sigma_1)\le (\hat{\rho},\hat{\sigma})\le (1,-1)$, which, by
Proposition~\ref{le basico}, implies that
\begin{equation}\label{pepitito1}
\en_{1,-1}(P)=\st_{\rho_1,\sigma_1}(P)\quad\text{and}\quad \en_{1,-1}(Q)=\st_{\rho_1,\sigma_1}(Q).
\end{equation}
If $(\rho_1,\sigma_1)\in A(P)\cap A(Q)$, then this implies
$$
v_{1,-1}(P)=v_{1,-1}(\st_{\rho_1,\sigma_1}(P))<0\quad\text{and}\quad v_{1,-1}(Q)=
v_{1,-1}(\st_{\rho_1,\sigma_1}(Q)) <0,
$$
and so, by inequality~\eqref{dfpvc}, we have $v_{1,-1}([P,Q])<0$, which contradicts that $[P,Q]\in
K^{\times}$. Hence we can suppose that $A(P)=\emptyset$, which by Proposition~\ref{le basico}, implies that
$$
\st_{\rho_1,\sigma_1}(P) = \en_{1,0}(P)\quad\text{and}\quad \st_{\rho_1,\sigma_1}(Q) = \en_{1,0}(Q).
$$
Consequently, by~\eqref{pepitito1},
$$
\en_{1,-1}(P)=\en_{1,0}(P)\quad\text{and}\quad \en_{1,-1}(Q)= \en_{1,0}(Q).
$$
By the definition of $(m,n)$-pair and Proposition~\ref{propiedades de los pares}(3), this implies  that
$$
v_{1,-1}(P)=v_{1,-1}(\en_{1,0}(P))<0\quad\text{and}\quad v_{1,-1}(Q) = v_{1,-1}(\en_{1,0}(Q)) <0,
$$
and so, again by inequality~\eqref{dfpvc}, we have $v_{1,-1}([P,Q])<0$, which contradicts that $[P,Q]\in
K^{\times}$.

\smallskip

In order to conclude the proof, we must show that $\Pred_P(\rho_1,\sigma_1)= \Pred_Q(\rho_1,\sigma_1)$. Assume
this is false and suppose for example that $\Pred_P(\rho_{1},\sigma_{1})<\Pred_Q(\rho_{1},\sigma_{1})$, which
implies
\begin{equation}\label{pepitito2}
\st_{\hat{\rho},\hat{\sigma}}(P)=\en_{\hat{\rho},\hat{\sigma}}(P)=\st_{\rho_1,\sigma_1}(P).
\end{equation}
If $A(P)\ne\emptyset$, then by Lemma~\ref{v positivos}, the conditions of Remark~\ref{a remark} are satisfied
for $(\hat{\rho},\hat{\sigma})$. Consequently, by this remark, $(\hat{\rho},\hat{\sigma})\in \Dir(P)$,
contradicting~\eqref{pepitito2}. Assume now $A(P)=\emptyset$, which implies that $(\hat{\rho},\hat{\sigma})\le
(1,0)<(\rho_1,\sigma_1)$. Hence, by~\eqref{pepitito2}, Proposition~\ref{le basico} and Proposition~\ref{propiedades
de los pares}(2),
$$
\st_{\hat{\rho},\hat{\sigma}}(P) = \en_{\hat{\rho},\hat{\sigma}}(P) = \en_{1,0}(P) = \frac nm \en_{1,0}(Q) =
\frac nm \en_{\hat{\rho},\hat{\sigma}}(Q)
$$
and
$$
v_{0,-1}(\st_{\hat{\rho},\hat{\sigma}}(P)) = v_{0,-1}(\en_{1,0}(P))< -1.
$$
Hence again by Lemma~\ref{v positivos}, the conditions of Remark~\ref{a remark} are satisfied
for $(\hat{\rho},\hat{\sigma})$, and therefore $(\hat{\rho},\hat{\sigma})\in \Dir(P)$, which
contradicts~\eqref{pepitito2}. The case $\Pred_Q(\rho_1,\sigma_1)<\Pred_P(\rho_1,\sigma_1)$, can be
discarded using a similar argument.
\end{proof}


\begin{definition}\label{def regular corner}
A {\em regular corner} of an $(m,n)$-pair $(P,Q)$ in $ L^{(l)}$, is a pair $(A,(\rho,\sigma))$, where
$A =(a/l,b)\in \frac 1l \mathds{Z}\times \mathds{N}_0$ and $(\rho,\sigma) \in I$ such that
\begin{enumerate}

\smallskip

\item $b\ge 1$ and $b>a/l$,

\smallskip

\item $(\rho,\sigma)\in \Dir(P)$,

\smallskip

\item $\left(\frac al,b\right)=\frac 1m \en_{\rho,\sigma}(P)$.
\end{enumerate}
A regular corner $(A,(\rho,\sigma))$ is said to be {\it at} the point $A$.
\end{definition}

\begin{proposition}\label{cases of corners}
If $(A,(\rho,\sigma))$ is a regular corner of an $(m,n)$-pair $(P,Q)$, then at least one of the following
three facts is true:
\begin{enumerate}

\smallskip

\item[\emph{(a)}] $(\rho,\sigma)\in A(P)$,

\smallskip

\item[\emph{(b)}] $\en_{\rho,\sigma}(P)=\en_{1,0}(P)$,

\smallskip

\item[\emph{(c)}] $(\rho,\sigma)=\Pred_P(\rho_1,\sigma_1)$, where $(\rho_1,\sigma_1):=\min( A(P))$.
\end{enumerate}
Moreover, there exists exactly one regular corner $(A,(\rho,\sigma))$ such that $(\rho,\sigma)\notin A(P)$.
\end{proposition}

\begin{proof}
Assume that $(\rho,\sigma)\notin A(P)$ and define $(\rho_1,\sigma_1):=\Succ_P(\rho,\sigma)$. If
$(\rho,\sigma)<(\rho_1,\sigma_1)\le (1,0)$, then $(\rho_1,\sigma_1)\in A(P)$\endnote{
By Proposition~\ref{le
basico} we know that $\en_{\rho,\sigma}(P) = \st_{\rho_1,\sigma_1}(P)$. Hence
$$
v_{0,-1}(\st_{\rho_1,\sigma_1}(P)) = v_{0,-1}(\en_{\rho,\sigma}(P))<-1\qquad\text{and}\qquad
v_{1,-1}(\st_{\rho_1,\sigma_1}(P)) = v_{1,-1}(\en_{\rho,\sigma}(P)) <0,
$$
where for the first inequality we have used that $\frac 1m v_{0,1}(\en_{\rho,\sigma}(P))\ge 1$. Moreover it is
clear that $(\rho_1,\sigma_1)\in \Dir(P)$. Hence, in order to finish the proof we only must show that
$(\rho_1,\sigma_1)\in I$, which
follows immediately since $(\rho,\sigma)\in I$ and, by hypothesis,
$(\rho,\sigma)<(\rho_1,\sigma_1)\le (1,0)$.},
which implies that $(\rho_1,\sigma_1)=\min (A(P))$ by Proposition~\ref{A(P) vs Dir(P)},
 and so item~(c) holds. Otherwise $(\rho,\sigma)\le (1,0)<
(\rho_1,\sigma_1)$ and, by
Proposition~\ref{le basico}, we conclude that $\en_{\rho,\sigma}(P) = \en_{1,0}(P)$.
\end{proof}

\begin{corollary}\label{some properties of corners}
If $(A,(\rho,\sigma))$ is a regular corner of an $(m,n)$-pair $(P,Q)$, then
\begin{enumerate}

\smallskip

\item $v_{\rho,\sigma}(P)>0$ and $v_{\rho,\sigma}(Q)>0$,

\smallskip

\item $\en_{\rho,\sigma}(P)=\frac mn \en_{\rho,\sigma}(Q)$,

\smallskip

\item $(\rho,\sigma)\in \Dir(Q)$.
\end{enumerate}
\end{corollary}

\begin{proof}
If $(\rho,\sigma)\in A(P)$, or $(\rho,\sigma)=\Pred_P(\rho_1,\sigma_1)$, where $(\rho_1,\sigma_1):=
\min( A(P))$,
then Lemma~\ref{v positivos} and Proposition~\ref{esquinas regulares} yield the result\endnote{
If
$(\rho,\sigma)\in A(P)$, then by items~(1) and~(3)
of Proposition~\ref{esquinas regulares}, we have
$$
\en_{\rho,\sigma}(P) = \frac mn \en_{\rho,\sigma}(Q)\qquad\text{and}\qquad (\rho,\sigma)\in\Dir(Q).
$$
Applying now Lemma~\ref{v positivos}, we obtain
$$
v_{\rho,\sigma}(P)>0 \qquad\text{and}\qquad  v_{\rho,\sigma}(Q)>0.
$$
Assume now that $(\rho,\sigma)=\Pred_P(\rho_1,\sigma_1)$. Then by Proposition~\ref{esquinas regulares},
$$
(\rho,\sigma)=\Pred_Q(\rho_1,\sigma_1)\in \Dir(Q).
$$
Moreover, by Propositions~\ref{le basico} and~\ref{esquinas regulares}(2),
$$
\en_{\rho,\sigma}(P) = \st_{\rho_1,\sigma_1}(P) = \frac mn \st_{\rho_1,\sigma_1}(Q) = \frac mn
\en_{\rho,\sigma}(Q).
$$
Finally, again by Lemma~\ref{v positivos},
$$
v_{\rho,\sigma}(P)>0 \qquad\text{and}\qquad  v_{\rho,\sigma}(Q)>0.
$$
}.
Hence, by Proposition~\ref{cases of corners}, it suffices to prove the assertions when
$\en_{\rho,\sigma}(P)=\en_{1,0}(P)$ and $(\rho,\sigma)\notin A(P)$. We claim that $A(P)=\emptyset$. In fact,
if there exists $(\rho',\sigma')\in A(P)$ with $(\rho',\sigma')<(\rho,\sigma)$, then $(\rho,\sigma)\in A(P)$
by Proposition~\ref{A(P) vs Dir(P)}. On the other hand, if there exists $(\rho',\sigma')\in A(P)$ with
$(\rho',\sigma')>(\rho,\sigma)$, then
$$
(\rho',\sigma')\ge \Succ_{P}(\rho,\sigma)=\Succ_P(1,0)>(1,0),\endnote{
Since $\en_{\rho,\sigma}(P)=\st_{\Succ_P(1,0)}(P)$, by
Proposition~\ref{le bbasico} we know that $(\rho,\sigma)$ and $\Succ_P(1,0)$ are consecutive elements in
$\Dir(P)$, and hence $\Succ_{P}(\rho,\sigma)=\Succ_P(1,0)$.}
$$
which contradicts $(\rho',\sigma')\in I$. Now we set
$$
(\rho_1',\sigma_1'):=\min(\Succ_P(1,0),\Succ_Q(1,0)).
$$
Then $(\rho,\sigma)= \Pred_P(\rho_1',\sigma_1')$\endnote{
Since $(1,0)<(\rho_1',\sigma_1') \le \Succ_P(1,0)$, by
Proposition~\ref{le basico}
$$
\en_{\rho,\sigma}(P)= \en_{1,0}(P) = \st_{\rho_1',\sigma_1'}(P).
$$
Hence, since $(\rho,\sigma)\in \Dir(P)$, from Proposition~\ref{le bbasico} it follows that $(\rho,\sigma)=
\Pred_P(\rho'_1,\sigma'_1)$.}
and from Proposition~\ref{esquinas regulares} we obtain
$$
(\rho,\sigma)=\Pred_Q(\rho_1',\sigma_1')\in \Dir(Q).
$$
Since $(\rho,\sigma)\le (1,0)< (\rho_1',\sigma_1')$, by Proposition~\ref{le basico} this implies
$\en_{\rho,\sigma}(Q)=\en_{1,0}(Q)$. Consequently, by Proposition~\ref{propiedades de los pares}(2),
$$
\en_{\rho,\sigma}(P)=\en_{1,0}(P)=\frac mn \en_{1,0}(Q)=\frac mn \en_{\rho,\sigma}(Q),
$$
and Lemma~\ref{v positivos} concludes the proof.
\end{proof}

\begin{remark}\label{a>0} If $((a/l,b),(\rho,\sigma))$ is a regular corner of an $(m,n)$-pair $(P,Q)$ in
$L^{(l)}$, then
$a>0$\endnote{
Since $(1,-1)<(\rho,\sigma)\le(1,0)$, we have $\rho>0$ and $\sigma \le 0$. Combining this with the
facts that $b>0$ and, by Corollary 5.7(1),
$$
\rho\frac{a}{l} + \sigma b = v_{\rho,\sigma}(P) > 0,
$$
we obtain $a>0$.}.
\end{remark}


Let $(A,(\rho,\sigma))$ be a regular corner of an $(m,n)$-pair $(P,Q)$ in $L^{(l)}$. Write
\begin{equation}\label{definicion de p}
\ell_{\rho,\sigma}(P)=x^{k/l}\mathfrak{p}(z)\quad\text{where $z:=x^{-\sigma/\rho}y$ and $\mathfrak{p}(z)\in
K[z]$.}
\end{equation}

Since $(\rho,\sigma)\in\Dir(P)$ the polynomial $\mathfrak{p}(z)$ is not a constant. Moreover by
Corollary~\ref{some properties of corners}(1) and  Theorem~\ref{central}(4) we know that
$v_{1,-1}(\st_{\rho,\sigma}(P))\ne 0$. Hence, one of the following five
mutually excluding conditions is true:
\begin{enumerate}

\smallskip

\item[\textrm{I.a)}] $[\ell_{\rho,\sigma}(P),\ell_{\rho,\sigma}(Q)]\ne 0$ and  $\st_{\rho,\sigma}(P)\sim
\st_{\rho,\sigma}(Q)$.

\smallskip

\item[\textrm{I.b)}] $[\ell_{\rho,\sigma}(P),\ell_{\rho,\sigma}(Q)]\ne 0$ and $\st_{\rho,\sigma}(P)\nsim
\st_{\rho,\sigma}(Q)$.

\smallskip

\item [\textrm{II.a)}] $[\ell_{\rho,\sigma}(P),\ell_{\rho,\sigma}(Q)]= 0$, $\# \factors(\mathfrak{p}(z))>1$
    and
$v_{1,-1}(\st_{\rho,\sigma}(P))<0$.

\smallskip

\item [\textrm{II.b)}] $[\ell_{\rho,\sigma}(P),\ell_{\rho,\sigma}(Q)]= 0$, $\# \factors(\mathfrak{p}(z))>1$
 and
  $v_{1,-1}(\st_{\rho,\sigma}(P))>0$.

\smallskip

\item [\textrm{III)}] $[\ell_{\rho,\sigma}(P),\ell_{\rho,\sigma}(Q)]= 0$ and
    $\mathfrak{p}(z)=\mu(z-\lambda)^r$
for some $\mu,\lambda\in K^{\times}$ and $r\in \mathds{N}$.
\end{enumerate}

\begin{remark}\label{mayores estan en A(P)} Let $(P,Q)$ be an $(m,n)$-pair in $L^{(l)}$ and let
$(\rho,\sigma)\in \Dir(P)\cap I$. If $(A,(\rho',\sigma'))$ is a regular corner and $(\rho',\sigma')<(\rho,\sigma)\le (1,0)$,
then $(\rho,\sigma)\in A(P)$.
In fact, by  Proposition~\ref{A(P) vs Dir(P)}, it suffices to consider
the case in which $\en_{\rho',\sigma'}(P) =\st_{\rho,\sigma}(P)$, and in that case it
follows easily from the definition of $A(P)$ and Definition~\ref{def regular corner}.
\end{remark}

\begin{remark}\label{estamos en Case IIa} Let $(P,Q)$ be an $(m,n)$-pair in $L^{(l)}$ and let
$(\rho,\sigma)\in\mathfrak{V}$.
If $(\rho,\sigma)\in A(P)$ then $\bigl(\frac 1m \en_{\rho,\sigma}(P),(\rho,\sigma)\bigr)$ is a regular corner
and we are in the Case II.a)\endnote{
By item~(1) of Proposition~\ref{esquinas regulares},
$$
\frac 1m \en_{\rho,\sigma}(P)\in \frac{1}{l} \mathds{Z}\times \mathds{N}_0.
$$
On the other hand clearly $(\rho,\sigma)\in I$ and items~(2) and~(3) of Definition~\ref{def regular corner}
are true. Moreover, since
$(\rho,\sigma)\in A(P)$ and $\frac{\rho+\sigma}{\rho}>0$, by equalities~\eqref{eq57} we have
$$
v_{0,-1}(\en_{\rho,\sigma}(P)) < v_{0,-1}(\st_{\rho,\sigma}(P))<-1\quad\text{and}\qquad
v_{1,-1}(\en_{\rho,\sigma}(P)) < v_{1,-1}(\st_{\rho,\sigma}(P))<0,
$$
which proves item~(1). Now Lemma~\ref{v positivos} proves that
$[\ell_{\rho,\sigma}(P),\ell_{\rho,\sigma}(Q)]=0$. Moreover, by the first equality
in~\eqref{eq57}, we know that $z$ is a factor of $\mathfrak{p}$, because $s>0$. But $\mathfrak{p}$ is not a
power of $z$, since
$\ell_{\rho,\sigma}(P)$ is not a monomial, and consequently, $\#\factors(\mathfrak{p}(z))>1$.}.
\end{remark}

\begin{remark}\label{Case IIa consecuencia} In the Case II.a), if we set
$(\rho',\sigma'):=\Pred_P(\rho,\sigma)$, then
$\bigl(\frac 1m \st_{\rho,\sigma}(P),(\rho',\sigma')\bigr)$
is a regular corner of $(P,Q)$\endnote{
Item~(2) of Definition~\ref{def regular corner} is trivial and
item~(3) follows from
Proposition~\ref{le basico}. Now we prove item~(1). If we set $\frac 1 m \st_{\rho,\sigma}(P) = (a'/l,b')$,
then $a'/l < b'$, since $\frac 1 m
v_{1,-1}(\st_{\rho,\sigma}(P))<0$, and $b' > 0$, since otherwise we have $v_{\rho,\sigma}(\st_{\rho,\sigma}(P))
 = \rho a'/l< 0$, which contradicts
Corollary~\ref{some properties of corners}(1). It remains to check that $(\rho',\sigma')\in I$. But this
follows immediately from
Proposition~\ref{esquinas regulares}, since $(\rho,\sigma)\in A(P)$, because $v_{0,-1}(\st_{\rho,\sigma}(P)) =
- m b'<-1$.
}.
\end{remark}

\begin{remark}\label{first component}
If $(P,Q)$ is an $(m,n)$-pair, then $\frac 1 m \en_{1,0}(P)$ is the first component of a regular corner of $(P,Q)$\endnote{
Assume first $(1,0)\in\Dir(P)$. In this case, using
Proposition~\ref{propiedades de los pares}, it is easy to check that
$\bigl(\frac 1 m \en_{1,0}(P),(1,0)\bigr)$ is a regular corner. Suppose
now $(1,0)\notin\Dir(P)$ and set $(\rho,\sigma):=\Pred_P(1,0)$. Since, by
Proposition~\ref{le basico}
$$
\frac 1 m \en_{\rho,\sigma}(P) = \frac 1 m \st_{1,0}(P) = \frac 1 m
\en_{1,0}(P)
$$
it is enough to show that $\bigl(\frac 1 m
\en_{\rho,\sigma}(P),(\rho,\sigma)\bigr)$ is a regular corner of $(P,Q)$.
If $A(P)=\emptyset$ then this follows from Proposition~\ref{propiedades de los
pares} and the additional part of Propositions~\ref{esquinas regulares}.
Otherwise
$$
(\rho,\sigma)=\max(\Dir(P)\cap I)=\max(A(P)),
$$
where the last equality follows from Proposition~\ref{A(P) vs Dir(P)}.
Remark~\ref{estamos en Case IIa} concludes the proof.
}.
\end{remark}

\begin{proposition}[\textrm{Cases I.a) and I.b)}]\label{extremosfinales}
Let $(P,Q)$ be an $(m,n)$-pair in $L^{(l)}$, and $((a/l,b),(\rho,\sigma))$ a regular corner of
$(P,Q)$.
Assume $[\ell_{\rho,\sigma}(P),\ell_{\rho,\sigma}(Q)]\ne 0$. Then $l-a/b > 1$ and the following assertions
hold:
\begin{enumerate}

\smallskip

\item[a)]If $\st_{\rho,\sigma}(P)\sim\st_{\rho,\sigma}(Q)$, then
$$
\frac{1}{m}\st_{\rho,\sigma}(P)\in \frac{1}{l} \mathds{Z}\times \mathds{N}_0\quad \text{and}\quad
\st_{\rho,\sigma}(P)\sim (1,0).
$$

\smallskip

\item[b)] If $\st_{\rho,\sigma}(P)\nsim\st_{\rho,\sigma}(Q)$, then there exists $k\in\mathds{N}$, with
    $k<l-\frac{a}{b}$, such that
\begin{equation}\label{conjunto de starting}
\{\st_{\rho,\sigma}(P),\st_{\rho,\sigma}(Q)\} = \left\{\left(\frac{k}{l},0\right),
\left(1-\frac{k}{l},1\right)\right\}.
\end{equation}
\end{enumerate}
\end{proposition}

\begin{proof}
\noindent a)\enspace Since $\st_{\rho,\sigma}(P)\sim\st_{\rho,\sigma}(Q)$, it follows from
Corollary~\ref{some properties of corners}, that
$$
\frac{1}{m} \st_{\rho,\sigma}(P)=\frac{1}{n} \st_{\rho,\sigma}(Q),
$$
and so
\begin{equation}\label{che}
A':=\frac 1m \st_{\rho,\sigma}(P)\in \frac 1{l} \mathds{Z}\times \mathds{N}_0,
\end{equation}
because $m$ and $n$ are coprime. Hence $A' = (a'/l,b')$ with $a'\in \mathds{Z}$ and $b'\in \mathds{N}_0$.
Now we prove that $\st_{\rho,\sigma}(P) \sim (1,0)$ or, equivalently, that $b' = 0$. Assume by contradiction
that $b'>0$. By Remark~\ref{polinomio asociado f^{(l)}} we can write
$$
\ell_{\rho,\sigma}(P)=x^{\frac {ma'}{l}}y^{mb'} f (z)\quad\text{and}\quad \ell_{\rho, \sigma}(Q) =
x^{\frac{na'}{l}} y^{nb'} g(z),
$$
where $z:=x^{-\frac{\sigma}{\rho}}y$ and $f(z),g(z)\in K[z]$. Since $nb',mb' \ge 1$, the term $y$
divides both $\ell_{\rho,\sigma}(P)$ and $\ell_{\rho,\sigma} (Q)$. Consequently $y$ is a factor of
$[\ell_{\rho,\sigma}(P),\ell_{\rho,\sigma}(Q)]$.
Since by Proposition~\ref{pr v de un conmutador}, we know that $[\ell_{\rho,\sigma} (P),\ell_{\rho,\sigma}(Q)]
=\ell_{\rho,\sigma}([P,Q])\in
K^{\times}$, this is a contradiction which proves that $b' = 0$.

We next prove $l-a/b > 1$ in this case. Since, by Corollary~\ref{some properties of corners}(1),
$$
a' = \frac{l}{\rho} v_{\rho,\sigma}\left(\frac{a'}{l},0\right) = \frac{l}{\rho m} v_{\rho,\sigma}
(\en_{\rho,\sigma}(P)) >0,
$$
it suffices to show that $l-a/b > a'$. Assume that this is false. Then $1-\frac{a'}{l}\le \frac{a}{bl}$, and so
$$
v_{\rho,\sigma}\left(1-\frac{a'}l,1\right) \le \frac{1}{b} v_{\rho,\sigma}\left(\frac{a}l,b\right) =
\frac{1}{bm} v_{\rho,\sigma}(\en_{\rho,\sigma}(P)),
$$
since $\rho>0$. Moreover,
$$
v_{\rho,\sigma}\left(\frac{a'}l,0\right) = \frac{1}{m}v_{\rho,\sigma}(\st_{\rho,\sigma}(P)),
$$
and so, by Proposition~\ref{pr v de un conmutador},
$$
v_{\rho,\sigma}(P) + v_{\rho,\sigma}(Q) = \rho+\sigma = v_{\rho,\sigma}\left(1-\frac{a'}l,1\right) +
v_{\rho,\sigma}\left(\frac{a'}l,0\right) \le \left(\frac{1}{bm} + \frac{1}{m}\right)v_{\rho,\sigma}(P) \le
v_{\rho,\sigma}(P).
$$
But this is impossible since $v_{\rho,\sigma}(Q)>0$ by Corollary~\ref{some properties of corners}(1).

\smallskip

\noindent b)\enspace By Proposition~\ref{pr v de un conmutador},
$$
[\ell_{\rho,\sigma}(P),\ell_{\rho,\sigma}(Q)] = \ell_{\rho,\sigma}([P,Q])\in K^{\times},
$$
and consequently, by Proposition~\ref{extremosnoalineados}(1),
$$
\st_{\rho,\sigma}(P)+\st_{\rho,\sigma}(Q)=(1,1).
$$
Therefore equality~\eqref{conjunto de starting} is true for some $k\in \mathds{Z}$. Applying $v_{\rho,\sigma}$
we obtain
\begin{equation*}
\left\{\rho\frac k{l},\rho\left(1-\frac k{l}\right) + \sigma\right\} =
\{v_{\rho,\sigma}(P),v_{\rho,\sigma}(Q)\},
\end{equation*}
which by Corollary~\ref{some properties of corners}(1), implies $k>0$. Assume that $\st_{\rho,\sigma}(Q) =
\bigl(1-\frac{k}{l},1\bigr)$.
By Corollary~\ref{some properties of corners}(2),
$$
n\left(\rho\frac{a}{l}+ \sigma b\right) = v_{\rho,\sigma}(\en_{\rho,\sigma}(Q))= v_{\rho,\sigma}
\left(1-\frac{k}{l},1\right)=\rho-\rho \frac{k}{l}+\sigma,
$$
and so
\begin{equation}
k = l-n a+\frac{\sigma l}{\rho}(1-bn).\label{chvh7}
\end{equation}
On the other hand, since $v_{\rho,\sigma}(P)>0$ and $\frac{l}{\rho b m}>0$, we have
$$
\frac{l\sigma}{\rho} + \frac{a}{b} = \frac{l}{\rho b}\left(\rho\frac{a}{l} + \sigma
b\right)= \frac{l}{\rho b m}v_{\rho,\sigma}(\en_{\rho,\sigma}(P)) >0.
$$
Multiplying this inequality by $bn-1>0$, we obtain
$$
\frac{\sigma l}{\rho}(1-bn)<\frac ab(bn-1).
$$
Combining this with equality~\eqref{chvh7} we conclude that
$$
k = l-n a+\frac{\sigma l}{\rho}(1-bn) < l-na+\frac ab(bn-1) = l-\frac{a}{b},
$$
as desired. In the case $\st_{\rho,\sigma}(P) = \bigl(1-\frac{k}{l},1\bigr)$ the proof of $k<l-\frac{a}{b}$ is
similar. Since $k\ge 1$ we also obtain $l-a/b > 1$ in the case~b).
\end{proof}

\begin{proposition}[\textrm{Case II}]\label{case II}
Let $(P,Q)$ be an $(m,n)$-pair in $L^{(l)}$, let $((a/l,b),(\rho,\sigma))$ be a regular corner of
$(P,Q)$ and let $F$ be as
in Theorem~\ref{central}. Assume that $[\ell_{\rho,\sigma}(P),\ell_{\rho,\sigma}(Q)] = 0$ and write
$\ell_{\rho,\sigma}(P)=x^{k/l}\mathfrak{p}(z)$,
where $z:=x^{-\sigma/\rho}y$ and $\mathfrak{p}(z)\in K[z]$. If $\# \factors(\mathfrak{p}(z))>1$, then
\begin{enumerate}

\smallskip

\item $\en_{\rho,\sigma}(F)\sim (a/l,b)$.

\smallskip

\item $\rho/\gcd(\rho,l)\le b$.

\end{enumerate}
Set $d:=\gcd(a,b)$, $\ov a:=a/d$, $\ov b:=b/d$ and write $\en_{\rho,\sigma}(F)=\mu (\ov a/l,\ov b)$. We have:
\begin{enumerate}[resume]

\smallskip

\item $(\rho,\sigma)=\dir\left(\en_{\rho,\sigma}(F)-(1,1)\right)=\dir(\mu \ov a -l,\mu \ov b l-l)$.

\smallskip

\item $\mu\in \mathds{N}$, $\mu\le l(bl-a)+1/\ov b$, $d\nmid \mu$ and $d>1$.

\end{enumerate}

\end{proposition}

\begin{proof}
Write $F=x^{1+\sigma/\rho}\mathfrak{f}(z)$, where $\mathfrak{f}(z)\in K[z]$. Note that
$\mathfrak{p}(z)=z^sp(z)$ and $\mathfrak{f}(z)=z^v f(z)$,
where $p$, $f$, $s$ and $v$ are the same as in Proposition~\ref{pavadass}. Moreover $s>0$ implies $v>0$ by
Remark~\ref{polinomio asociado f^{(l)}} and Theorem~\ref{central}(2), and so, by
Proposition~\ref{pavadass}(1), each irreducible
factor of $\mathfrak p$ divides $\mathfrak f$. Since $\# \factors(\mathfrak{p}(z))>1$, we have
$\deg(\mathfrak f)\ge 2$. Hence
$\en_{\rho,\sigma}(F)\ne (1,1)$ by Remark~\ref{polinomio asociado f^{(l)}}. Consequently, by
Theorem~\ref{central}(3), we have
$\en_{\rho,\sigma}(F)\sim \en_{\rho,\sigma}(P)$ which yields statement~(1).

Now we prove statement~(2). Let $A_1':=\frac 1m \st_{\rho,\sigma}(P)$. By Remark~\ref{a remark}(2) we have
$A_1'\in \frac 1l \mathds{Z}\times \mathds{N}_0$. Write $A_1' = (a'/l,b')$. Since $b'<b$ by
Remark~\ref{polinomio asociado f^{(l)}},
and $v_{\rho,\sigma}(a/l,b)=v_{\rho,\sigma}(a'/l,b')$, there exists $h\in\mathds{N}$, such that
$$
\left(-\frac{\sigma h l/\rho}{l},h\right)=h\left(-\frac{\sigma}{\rho},1\right) =
\left(\frac a l,b\right)-\left(\frac{a'}{l},b'\right)\in \frac 1l \mathds{Z}\times \mathds{N}_0,
$$
Hence $\rho$ divides $\sigma h l$. Set
$$
\ov \rho:=\frac{\rho}{\gcd(\rho,l)}\quad\text{and}\quad \ov l:= \frac{l}{\gcd(\rho,l)}.
$$
Clearly $\ov \rho$ divides $h \sigma \ov l$, and so $\ov \rho\mid h=b-b'$, which implies $\ov \rho\le b$, as
desired.

Statement~(3) follows from Remark~\ref{valuacion depende de extremos} and the fact that
$\en_{\rho,\sigma}(F)\ne (1,1)$.

It remains to  prove statement~(4). First note that $\mu\in \mathds{N}$, since $\ov{b}\in \mathds{N}$,
$\mu \ov{b}\in \mathds{N}$,
$\mu \ov{a}\in \mathds{Z}$ and $\gcd(\ov{a},\ov{b}) = 1$. On the other hand, by Remark~\ref{a remark} we know
that there exists
$\lambda_P,\lambda_Q\in K^{\times}$ and a $(\rho,\sigma)$-homogeneous element $R\in L^{(l)}$, such that
$$
\ell_{\rho,\sigma}(P)=\lambda_P R^m\qquad\text{and}\qquad \ell_{\rho,\sigma}(Q) = \lambda_Q R^n,
$$
which implies
$$
\en_{\rho,\sigma}(R) = (a/l,b) = d(\ov{a}/l,\ov{b}) = \frac{d}{\mu} \en_{\rho,\sigma}(F).
$$
Next we prove that $d\nmid \mu$. In fact, if we assume that $d|\mu$, then we have
$$
v_{\rho,\sigma}(R^{\mu/d}) = v_{\rho,\sigma}(F)=\rho+\sigma\qquad\text{and}\qquad
[R^{\mu/d},\ell_{\rho,\sigma}(P)]=0,
$$
where the last equality follows from the fact that $[-,-]$ is a Poisson bracket and
$[R^n,\ell_{\rho,\sigma}(P)]=0$. But this
contradicts Proposition~\ref{pavadass}(5) and proves $d\nmid \mu$. From this it follows immediately that $d>1$.
Finally we prove that $\mu\le l(bl-a)+1/\ov b$. Since
$$
(\mu \ov a -l) - (\mu \ov b l-l) = \mu (\ov a - \ov b l) = \frac{\mu}{d}(a-b l)< 0,
$$
from equalities~\eqref{val} and statement~(3) it follows that
$$
\rho=\frac{(\mu\ov b-1)l}{d_1},\quad\text{where $d_1:=\gcd(\mu \ov a -l,\mu \ov b l-l)$.}
$$
Now note that $d_1$ divides $\ov b l(\mu \ov a -l)-\ov a(\mu \ov b l-l)=l(\ov a-\ov b l)$, and therefore
$$
d_1\le l(\ov b l-\ov a),
$$
since $\ov b l-\ov a>0$. Hence, by statement~(2),
$$
b\ge \frac{\rho}{\gcd(\rho,l)}\ge \frac {\rho}{l}=\frac{(\mu\ov b-1)}{d_1}\ge
\frac{(\mu\ov b-1)}{l(\ov b l-\ov a)},
$$
which implies $\mu\ov b -1 \le bl(\ov b l -\ov a) = \ov b l(bl-a)$, as desired.
\end{proof}

\begin{figure}[htb]
\centering
\begin{tikzpicture}
%
\draw[step=.25cm,gray,very thin] (0,0) grid (3.1,3.6);
\draw [thick]  (0.625,0) -- (0.833,0.833) node[fill=white,right=2pt]{\tiny{$\lambda$(1,1)}}
--(1.25,2.5) node[fill=white,right=2pt]{\tiny{$m(a/l,b)$}} -- (1.3,3);
\draw [->] (0,0) -- (3.3,0) node[anchor=north]{$x$};
\draw [->] (0,0) --  (0,3.8) node[anchor=east]{$y$};
\draw[dotted] (0,0) -- (2.3,2.3);
\filldraw [gray]  (0.833,0.833)    circle (1.5pt)
                  (1.25,2.5)    circle (1.5pt);
\draw [->,thick] (1.5,2) -- (2.5,1.75) node[fill=white,anchor=west]{\tiny{$(\rho,\sigma)$}};
\end{tikzpicture}
\caption{Remark~\ref{lambda en la diagonal} with $\lambda=
\frac ml\left(\frac{a\rho+bl\sigma}{\rho+\sigma}\right)$.}
\end{figure}
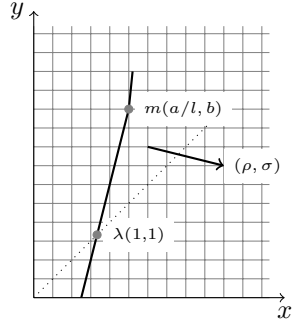

\begin{remark}\label{lambda en la diagonal}
Let $((a/l,b),(\rho,\sigma))$ be a regular corner of an $(m,n)$-pair $(P,Q)$ and let $L$ be the straight line
that includes
$\Supp(\ell_{\rho,\sigma}(P))$. The intersection of $L$ with the diagonal $x=y$ is the point
$$
\lambda(1,1),\quad\text{where}\quad \lambda=\frac ml\left(\frac{a\rho+bl\sigma}{\rho+\sigma}\right).
$$
In fact
$$
\lambda(\rho+\sigma)
=v_{\rho,\sigma}\bigl(\lambda(1,1)\bigr)=v_{\rho,\sigma}(\ell_{\rho,\sigma}(P))=v_{\rho,\sigma}(m(a/l,b))=\frac
ml(a\rho+ bl\sigma),
$$
from which the assertion follows.
\end{remark}

The following proposition about multiplicities can be traced back to~\cite{J}*{Corollary 2.6(2)}.
The algebraic parallel is not so clear, but the geometric meaning, which will be proved in
Proposition~\ref{encima de la diagonal}, is that one can cut the support of
$\ell_{\rho,\sigma}(P)$ above the diagonal.

\begin{proposition}[\textrm{Case II.b)}]\label{case IIb}
Let $(P,Q)$ and $((a/l,b),(\rho,\sigma))$ be as in Proposition~\ref{case II}. Assume that
$[\ell_{\rho,\sigma}(P),\ell_{\rho,\sigma}(Q)]= 0$ and write $\ell_{\rho,\sigma}(P)=x^{k/l}\mathfrak p(z)$
where $z:=x^{-\sigma/\rho}y$ and
$\mathfrak p(z)\in K[z]$. If $\# \factors(\mathfrak{p}(z))\!>\!1$ and  $v_{1,-1}(\st_{\rho,\sigma}(P))\!>\!0$,
then there exists $\lambda\in K^{\times}$ such that $z-\lambda$ has mul\-tiplicity
$$
m_\lambda\ge \frac ml\left(\frac{a\rho+bl\sigma}{\rho+\sigma}\right)
$$
in $\mathfrak{p}(z)$.
\end{proposition}

\begin{proof}
Let $F$ be as in Theorem~\ref{central} and write $F=x^{1+\sigma/\rho}\mathfrak{f}(z)$. In the proof of
Proposition~\ref{case II} it was shown that
each irreducible factor of $\mathfrak p$ divides $\mathfrak f$. Hence, there is a linear factor of
$\mathfrak{p}$ with multiplicity greater than
or equal to $\deg(\mathfrak{p})/\deg(\mathfrak{f})$. Since $\en_{\rho,\sigma}(P)=(ma/l,mb)$, it follows
from Remark~\ref{polinomio asociado f^{(l)}}, that $\deg(\mathfrak{p}) = mb$. Similarly, if we write
$\en_{\rho,\sigma}(F)=(M_0,M)$, then $M= \deg(\mathfrak{f})$, an so $\deg(\mathfrak{p})/\deg(\mathfrak{f}) =
mb/M$. Consequently, in order to finish the proof it suffices to check that
\begin{equation}\label{cociente}
    \frac{mb}{M}=\frac ml\left(\frac{a\rho+bl\sigma}{\rho+\sigma}\right).
\end{equation}
Since
$$
\rho+\sigma=v_{\rho,\sigma}(F)=\rho M_0+\sigma M,
$$
we have
\begin{equation*}
    M_0=\frac{1}{\rho}(\rho+\sigma-\sigma M).
\end{equation*}
Hence, by Proposition~\ref{case II}(1),
\begin{equation*}
    \frac{a}{bl}=\frac{M_0}{M}=\frac{\rho+\sigma-\sigma M}{\rho M},
\end{equation*}
which implies
$$
M=\frac{bl(\rho+\sigma)}{a\rho +bl\sigma}.
$$
Therefore equality~\eqref{cociente} is true.
\end{proof}

\begin{proposition}[\textrm{Case III}]\label{case III}
Let $(P,Q)$ be an $(m,n)$-pair in $L^{(l)}$ and let $((a/l,b),(\rho,\sigma))$ be a regular corner
of $(P,Q)$.
Assume that $[\ell_{\rho,\sigma}(P),\ell_{\rho,\sigma}(Q)]= 0$ and write $\ell_{\rho,\sigma}(P)=x^{k/l}
\mathfrak p(z)$ where $z:=x^{-\sigma/\rho}y$
and $\mathfrak p(z)\in K[z]$. If there exist $\mu,\lambda\!\in\! K^{\times}$ and $r\!\in\! \mathds{N}$,
such that
$\mathfrak p(z)\!=\!\mu(z-\lambda)^r$, then $\rho\!\mid\! l$. Moreover, the  au\-to\-mor\-phism $\varphi$ of
$L^{(l)}$, defined by
$\varphi(x^{1/l}):=x^{1/l}$ and $\varphi(y):= y+\lambda x^{\sigma/\rho}$, satisfies
\begin{enumerate}

\smallskip

\item $\en_{\rho,\sigma}(\varphi(P))= \en_{\rho,\sigma}(P)$ and for all
    $(\rho,\sigma)<(\rho'',\sigma'') < (-\rho,-\sigma)$ the equalities
$$
\ell_{\rho'',\sigma''}(\varphi(P))=\ell_{\rho'',\sigma''}(P)\quad\text{and}\quad
\ell_{\rho'',\sigma''}(\varphi(Q))=\ell_{\rho'',\sigma''}(Q),
$$
hold.

\smallskip

\item $(\varphi(P),\varphi(Q))$ is an $(m,n)$-pair in $L^{(l)}$.

\smallskip

\item $((a/l,b),(\rho',\sigma'))$ is a regular corner of $(\varphi(P),\varphi(Q))$, where
    $(\rho',\sigma'):=\Pred_{\varphi(P)}(\rho,\sigma)$.

\smallskip

\item $(a/l,b) = \frac{1}{m}\st_{\rho,\sigma}(\varphi(P))$.

\end{enumerate}
\end{proposition}

\begin{proof}
Clearly the conditions imply that
$$
\ell_{\rho,\sigma}(P)=x^{k/l}\mu\left(\lambda^r-\binom r1\lambda^{r-1}z+\cdots\right).
$$
Hence $(k/l-\sigma/\rho,1)\in \Supp(\ell_{\rho,\sigma}(P))\subseteq\frac 1l \mathds{Z}\times \mathds{N}_0$.
So $\sigma/\rho\in \frac 1l\mathds{Z}$,
which evidently implies $\rho|l$, because $\gcd(\rho,\sigma)=1$. From
Proposition~\ref{pr ell por automorfismos}
 we obtain statement~(1). Statement~(2)
follows easily from Proposition~\ref{varphi preserva el Jacobiano}, statement~(1) and the fact that by
Proposition~\ref{pr ell por automorfismos} we
know that $\en_{1,0}(\varphi(P))\!=\! \en_{1,0}(P)$, even in the case where $(\rho,\sigma)\! =\! (1,0)$.
Finally, statements~(3) and~(4) follow from
Pro\-po\-si\-tions~\ref{le basico}, \ref{pr ell por automorfismos} and~\ref{esquinas regulares}
\endnote{
Clearly items~(1) and~(2) of
Definition~\ref{def regular corner} are fulfilled. So, we have to prove only that
$$
\frac 1m\en_{\rho',\sigma'}(\varphi(P)) = (a/l,b) \qquad\text{and}\qquad (\rho',\sigma')\in I.
$$
By Proposition~\ref{pr ell por automorfismos},
$$
\ell_{\rho,\sigma}(\varphi(P)) = \varphi\bigl(\ell_{\rho,\sigma}(P)\bigr) = \mu x^{\frac
kl}\Bigl(x^{-\frac{\sigma}{\rho}} \bigl(y+\lambda
x^{\frac{\sigma}{\rho}}\bigr) - \lambda\Bigr)^r = \mu x^{\frac kl-\frac{\sigma r}{\rho}}y^r,
$$
and so, by Propositions~\ref{le basico} and~\ref{pr ell por automorfismos},
$$
\frac 1m\en_{\rho',\sigma'}(\varphi(P)) = \frac 1m\st_{\rho,\sigma}(\varphi(P)) = \frac
1m\en_{\rho,\sigma}(\varphi(P)) = \frac 1m\en_{\rho,\sigma}(P)
= (a/l,b),
$$
as desired. It remains to verify that $(\rho',\sigma')\in I$. Since
$(\rho',\sigma')=\Pred_{\varphi(P)}(\rho,\sigma)
<(\rho,\sigma)\le (1,0)$
we have only to prove that $(1,-1)< \Pred_{\varphi(P)}(\rho,\sigma)$. Let
$(\rho_2,\sigma_2):=\Succ_{\varphi(P)}(\rho,\sigma)$. By Proposition~\ref{esquinas regulares}, in order to
obtain this inequality it is enough to show that $(\rho_1,\sigma_1) \le (\rho_2,\sigma_2)$,
where
$$
(\rho_1,\sigma_1):=\begin{cases} \min\bigl(A(\varphi(P))\bigr) & \text{if $A(\varphi(P))\ne\emptyset$,}\\
\min\bigl(\Succ_{\varphi(P)}(1,0),\Succ_{\varphi(Q)}(1,0)\bigr) & \text{if
$A(\varphi(P))=\emptyset$.}\end{cases}
$$
When $(\rho_2,\sigma_2)\notin I$ this is evident. Assume that $(\rho_2,\sigma_2)\in I$ and set $J:=
\{(\rho'',\sigma''): (\rho,\sigma)< (\rho'',\sigma'')\le  (-\rho,-\sigma)\}$. By
Proposition~\ref{pr ell por automorfismos},
$$
(\rho_2,\sigma_2)=\Succ_{\varphi(P)}(\rho,\sigma) = \Succ_{P}(\rho,\sigma)\qquad\text{and}\qquad
A(P)\cap J = A(\varphi(P))\cap J,
$$
while by Proposition~\ref{le basico}
$$
\st_{\rho_2,\sigma_2}(P) = \en_{\rho,\sigma}(P) = m(a/l,b).
$$
Consequently $(\rho_2,\sigma_2)\in A(P)\cap J \subseteq A(\varphi(P))$, and so $(\rho_1,\sigma_1) \le
(\rho_2,\sigma_2)$ follows.
}.
\end{proof}

By Proposition~\ref{case IIb} the hypotheses of the next proposition are always fulfilled in Case~II.b).
Sometimes they are fulfilled in Case~II.a).

\begin{proposition}[\textrm{Case II}]\label{encima de la diagonal}

Let $(P,Q)$ and $((a/l,b),(\rho,\sigma))$ be as in Proposition~\ref{case II} and let $l'\!:=\lcm(\rho,l)$.
Assume that $[\ell_{\rho,\sigma}(P),\ell_{\rho,\sigma}(Q)]\!=\! 0$ and write $\ell_{\rho,\sigma}(P)\!=x^{k/l}
\mathfrak p(z)$ where $z\!:=x^{-\sigma/\rho}y$ and $\mathfrak p(z)\in K[z]$.
Assume also that $\# \factors(\mathfrak{p}(z))>1$
and that there exists $\lambda\in K^{\times}$ such that the multiplicity
$m_\lambda$ of $z-\lambda$ in $\mathfrak{p}(z)$, satisfies
\begin{equation}\label{mlambda}
m_\lambda\ge\frac ml\left(\frac{a\rho+bl\sigma}{\rho+\sigma}\right).
\end{equation}
Define $\varphi\in\Aut(L^{(l')})$ by $\varphi(x^{1/l'}):=x^{1/l'}$ and $\varphi(y):=y+\lambda x^{\sigma/\rho}$,
 and set
$$
A^{(1)}:=\frac 1m\st_{\rho,\sigma}(\varphi(P))\qquad\text{and}\qquad
(\rho',\sigma'):=\Pred_{\varphi(P)}(\rho,\sigma).
$$
Then

\begin{enumerate}

\smallskip

\item We have $\en_{\rho,\sigma}(\varphi(P))= \en_{\rho,\sigma}(P)$ and for all $(\rho,\sigma)
    <(\rho'',\sigma'') < (-\rho,-\sigma)$ the equalities
$$
\ell_{\rho'',\sigma''}(\varphi(P))=\ell_{\rho'',\sigma''}(P)\quad\text{and}\quad\ell_{\rho'',\sigma''}(\varphi(Q))=\ell_{\rho'',\sigma''}(Q)
$$
hold.

\smallskip

\item $(\varphi(P),\varphi(Q))$ is an $(m,n)$-pair in $L^{(l')}$.

\smallskip

\item $(\rho,\sigma)\in \Dir(\varphi(P))$, $\st_{\rho,\sigma}(\varphi(P))=\Bigl(\frac kl,0\Bigr)+
    m_\lambda\Bigl(-\frac{\sigma}{\rho},1\Bigr)$ and $m\mid m_\lambda$.

\smallskip

\item $(A^{(1)},(\rho',\sigma'))$ and $((a/l,b),(\rho,\sigma))$ are regular corners of $(\varphi(P),\varphi(Q))$.
The second one is of type~IIa).

\end{enumerate}
\end{proposition}

\begin{proof} Statements~(1) and~(2) follows as in the proof of Proposition~\ref{case III}. Now we are going
to prove item~(3). For this we write $\mathfrak{p}(z)=(z-\lambda)^{m_\lambda}\ov p(z)$ with
$\ov p(\lambda)\ne 0$. Since
$$
\varphi(z)=\varphi(x^{-\sigma/\rho})\varphi(y)=x^{-\sigma/\rho}(y+\lambda x^{\sigma/\rho}) =z+\lambda,
$$
by Proposition~\ref{pr ell por automorfismos}, we have
$$
\ell_{\rho,\sigma}(\varphi(P))=\varphi(\ell_{\rho,\sigma}(P))=\varphi(x^{k/l}\mathfrak{p}(z))=
x^{k/l}\varphi((z-\lambda)^{m_\lambda}\ov p(z))= x^{k/l}z^{m_\lambda}\ov p(z+\lambda),
$$
which implies that $(\rho,\sigma)\in \Dir(\varphi(P))$, because
$$
\#\factors(z^{m_\lambda}\ov p(z+\lambda)) = \#\factors(\mathfrak{p}(z))>1.
$$
Moreover, since $\ov p(\lambda)\ne 0$, from the first equality in~\eqref{eq57} it follows that
$$
\st_{\rho,\sigma}(\varphi(P))=(k/l,0)+m_\lambda(-\sigma/\rho,1),
$$
and so statement~(3) holds. By statement~(2) and Remarks~\ref{estamos en Case IIa}
and~\ref{Case IIa consecuencia}, in order to prove statement~(4) it suffices to verify that $(\rho,\sigma)\in
A(\varphi(P))$. Since $(\rho,\sigma)\in \Dir(\varphi(P))\cap I$ we only must check that
\begin{equation}
v_{1,-1}(\st_{\rho,\sigma}(\varphi(P)))<0\qquad\text{and}\qquad
v_{0,-1}(\st_{\rho,\sigma}(\varphi(P)))<-1.\label{e3}
\end{equation}
Since
$$
\frac{k\rho}{l}=v_{\rho,\sigma}\Bigl(\frac kl,0\Bigr)=
v_{\rho,\sigma}(P)=v_{\rho,\sigma}\Bigl(\frac{ma}{l},mb\Bigr)=\frac ml(a\rho+bl\sigma),
$$
by inequality~\eqref{mlambda}, we have
$$
v_{1,-1}(\st_{\rho,\sigma}(\varphi(P)))=\frac kl-m_\lambda\left(\frac{\sigma}{\rho}+1\right)\le
\frac{m}{\rho l}(a\rho+bl\sigma)-\frac
ml\left(\frac{a\rho+bl\sigma}{\rho+\sigma}\right)\left(\frac{\rho+\sigma}{\rho}\right)=0.
$$
But $v_{1,-1}(\st_{\rho,\sigma}(\varphi(P)))=0$ is impossible by Theorem~\ref{central}(4), and hence the first
inequality in~\eqref{e3} holds. We next deal with the second one. By
Proposition~\ref{varphi preserva el Jacobiano},
$$
[\ell_{\rho,\sigma}(\varphi(P)),\ell_{\rho,\sigma}(\varphi(Q))] = 0,
$$
while by Proposition~\ref{pr ell por automorfismos} and Corollary~\ref{some properties of corners}(1),
$$
v_{\rho,\sigma}(\varphi(P)) = v_{\rho,\sigma}(P) >0 \qquad\text{and}\qquad v_{\rho,\sigma}(\varphi(Q)) =
v_{\rho,\sigma}(Q) >0.
$$
Hence, by Remark~\ref{a remark}(2), we have $\frac{1}{m}\st_{\rho,\sigma}(\varphi(P))\in
\frac{1}{l'}\mathds{Z}\times \mathds{N}_0$, and so $m\mid m_\lambda$ and
$$
v_{0,-1}(\st_{\rho,\sigma})(\varphi(P)) \le -m < -1,
$$
since $v_{0,1}(\st_{\rho,\sigma}(\varphi(P))) = m_{\lambda}\ge 1$, by statement~(3).
\end{proof}

\begin{proposition}[First criterion for regular corners]\label{criterion}
If $(a/l,b)$ is the first entry of a regular corner of an $(m,n)$-pair in $L^{(l)}$, then it is
the first entry of a regular corner of an (possibly different) $(m,n)$-pair in $L^{(l)}$ of type~I or type~II.
Moreover, in the first case $l-a/b>1$, while in the second one $\gcd(a,b)>1$.
If $l=1$ then necessarily case~II holds.
\end{proposition}

\begin{proof}
Assume that we are in case~III. By Proposition~\ref{case III}(3), there exists $\varphi\in \Aut(L^{(l)})$ such
that
$((a/l,b),(\rho_1,\sigma_1))$ is a regular corner of $(\varphi(P),\varphi(Q))$, where
$(\rho_1,\sigma_1):=\Pred_{\varphi(P)}(\rho,\sigma)$. If Case III holds for this corner, then we  can find
$(\rho_2,\sigma_2)<(\rho_1,\sigma_1)$ such that $((a/l,b),(\rho_2,\sigma_2))$ is a regular corner.
As long as Case III occurs, we can find $(\rho_{k+1},\sigma_{k+1})<(\rho_k,\sigma_k)$ such that
$((a/l,b),(\rho_{k+1},\sigma_{k+1}))$ is a regular corner.
But there are only finitely many $\rho_k$'s with $\rho_k|l$. Moreover,
$0<-\sigma_k<\rho_k$, since $(1,-1)<(\rho_k,\sigma_k)<(1,0)$, and so there are only finitely many
$(\rho_k,\sigma_k)$ possible, which
proves that eventually cases I or II must occur. In case~I Proposition~\ref{extremosfinales}
gives $l-a/b>1$ and in case~II, by Proposition~\ref{case II}(4), we have $\gcd(a,b)>1$.
The last statement is clear, since $1-a/b<1$, because $a,b>0$.
\end{proof}

\begin{proposition}\label{todos son Smp} For each $(m,n)$-pair $(P,Q)$ in $L^{(1)}$, there exists an
automorphism $\varphi$ of $L^{(1)}$ such that $(\varphi(P),\varphi(Q))$ is a standard $(m,n)$-pair with
$$
v_{1,1}(\varphi(P))=v_{1,1}(P),\quad v_{1,1}(\varphi(Q))=v_{1,1}(Q)\quad\text{and} \quad
\en_{1,0}(\varphi(P))=\en_{1,0}(P).
$$
Moreover, if $(-1,1)<\Succ_{P}(1,0),\Succ_{Q}(1,0)<(-1,0)$, then
$$
(-1,1)<\Succ_{\varphi(P)}(1,0),\Succ_{\varphi(Q)}(1,0)<(-1,0).
$$
Furthermore, if $P,Q\in L$, then we can take $\varphi\in \Aut(L)$.
\end{proposition}

\begin{proof} If $v_{1,-1}(\st_{1,0}(P))<0$, then we can take $\varphi:=\ide$.
Otherwise $\bigl(\frac 1m \en_{1,0}(P),(1,0)\bigr)$ is a regular corner\endnote{
By Proposition~\ref{propiedades de los pares}(3) we know that $\frac 1m
\en_{1,0}(P)\in \mathds{N}\times \mathds{N}$.
Moreover condition~(3) of Definition~\ref{def regular corner} is
trivially fulfilled and condition~(2) holds
because $v_{1,-1}(\en_{1,0}(P))<0$. It remains to prove that if $(a,b) := \frac 1m \en_{1,0}(P)$, then $b>a$
and $b\ge 1$. But this follows from the
fact that $a-b = v_{1,-1}(\en_{1,0}(P))<0$ and $a\in \mathds{N}$.}.
Write $(a,b):= \frac 1m \en_{1,0}(P)$. Case I is impossible because $a,b>0$ and
Proposition~\ref{extremosfinales}
gives $1-a/b>1$, and the Case II.a) is discarded, because $v_{1,-1}(\st_{1,0}(P))\ge 0$. By
Propositions~\ref{case IIb}
and~\ref{encima de la diagonal} in Case~II.b) and by Proposition~\ref{case III} in Case~III, we can find a
$\varphi\in\Aut(L^{(1)})$ such that

\begin{itemize}

\smallskip

\item[-] $\bigl(\frac 1m\st_{1,0}(\varphi(P)),(\rho',\sigma')\bigr)$ is a regular corner, for some
    $(\rho',\sigma')$,

\smallskip

\item[-] $\ell_{1,1}(\varphi(P))=\ell_{1,1}(P)$, $\ell_{1,1}(\varphi(Q))=\ell_{1,1}(Q)$ and
    $\en_{1,0}(\varphi(P))=\en_{1,0}(P)$,

\smallskip

\item[-] If $\Succ_{P}(1,0),\Succ_{Q}(1,0)<(-1,0)$, then
$$
\quad\qquad\Succ_{\varphi(P)}(1,0) = \Succ_{P}(1,0)\qquad\text{and}\qquad \Succ_{\varphi(Q)}(1,0) =
\Succ_{Q}(1,0).
$$

\smallskip

\item[-] $\varphi(x)=x$ and $\varphi(y)=y+\lambda$, for some $\lambda\in K^{\times}$.

\smallskip

\end{itemize}
The assertions in the statement follow immediately from these facts.
\end{proof}

\begin{corollary}\label{B finito}
If $B<\infty$ (i. e., if the
Jacobian conjecture is false), then there exists a Jacobian pair $(P,Q)$ and $m,n\in \mathds{N}$ coprime with $m,n>1$, such that
\begin{enumerate}

\smallskip

  \item $(P,Q)$ is a stan\-dard $(m,n)$-pair  in $L$,

\smallskip

  \item $(P,Q)$ is a minimal pair (i. e., $\gcd\bigl(v_{1,1}(P),v_{1,1}(Q)\bigr)=B$),

\smallskip

  \item $\st_{1,1}(P)=\en_{1,0}(P)$,

\smallskip

  \item $(-1,1)<\Succ_{P}(1,0),\Succ_{Q}(1,0)<(-1,0).$

\smallskip

\end{enumerate}
\end{corollary}

\begin{proof}
By Propositions~\ref{le basico}, \ref{primera condicion estandar} and~\ref{todos son Smp}.
\end{proof}

\begin{proposition}\label{esquina regular unica} Each $(m,n)$-pair $(P,Q)$ in $L^{(l)}$ has
a unique regular corner $((a/l,b),(\rho,\sigma))$ with $(\rho,\sigma)\notin A(P)$.
\end{proposition}

\begin{proof}
If $A(P)\ne \emptyset$, then the existence follows from Remarks~\ref{estamos en Case IIa}
and~\ref{Case IIa consecuencia}, since
$$
\Pred_P(\min(A(P)))\notin A(P).
$$
Otherwise, by Proposition~\ref{esquinas regulares}, we know that $(\rho,\sigma):=\Pred_P(\rho_1,\sigma_1)
\in I$,
where
$$
(\rho_1,\sigma_1):= \min\bigl(\Succ_P(1,0),\Succ_Q(1,0)\bigr).
$$
Clearly item~(2) of Definition~\ref{def regular corner} is fulfilled for $\bigl(\frac 1m \en_{\rho,\sigma}(P),
(\rho,\sigma)\bigr)$. Moreover, by Proposition~\ref{le basico} and Proposition~\ref{propiedades de los pares}(3),
$$
\frac1m \en_{\rho,\sigma}(P) = \frac1m \en_{1,0}(P)\in \frac{1}{l}\mathds{Z}\times \mathds{N},
$$
and so item~(3) is also satisfied. In order to prove item~(1) we write $(a/l,b): = \frac 1m
\en_{\rho,\sigma}(P)$. By Definition~\ref{Smp},
$$
a/l-b=v_{1,-1}\left(\frac1m \en_{\rho,\sigma}(P)\right)=v_{1,-1}\left(\frac1m \en_{1,0}(P)\right)<0,
$$
while by Proposition~\ref{propiedades de los pares}(4), we have $b=-v_{0,-1}(\en_{1,0}(P))>1$. This ends the proof
of the
existence. The uniqueness follows from Proposition~\ref{A(P) vs Dir(P)} and the fact that, by
Proposition~\ref{le basico}
and Definition~\ref{def regular corner},
if $(A,(\rho,\sigma))$ is a regular corner of $(P,Q)$, then $\Succ_P(\rho,\sigma)\in I$ implies that
$\Succ_P(\rho,\sigma)\in A(P)$.
\end{proof}

\section{Lower bounds}\label{Lower bounds}

\setcounter{equation}{0}

By Corollary~\ref{B finito}, if $B<\infty$ (i. e., if the
Jacobian conjecture is false), then there exists a stan\-dard $(m,n)$-pair $(P,Q)$ in $L$, which is also
 a minimal pair (i. e., $\gcd(v_{1,1}(P),v_{1,1}(Q))=B$). In this
section we will first prove that $B\ge 16$. The argument is nearly the same as in~\cite{G-G-V1},
but we will also need lower bounds for $(m,n)$-pairs in $L^{(1)}$, and not only in $L$.
The reason is the following: One technical result, Proposition~\ref{impossibles}, says something about
$(m,n)$-pairs in $L$ with $\frac 1m v_{2,-1}(P)\le 4$.
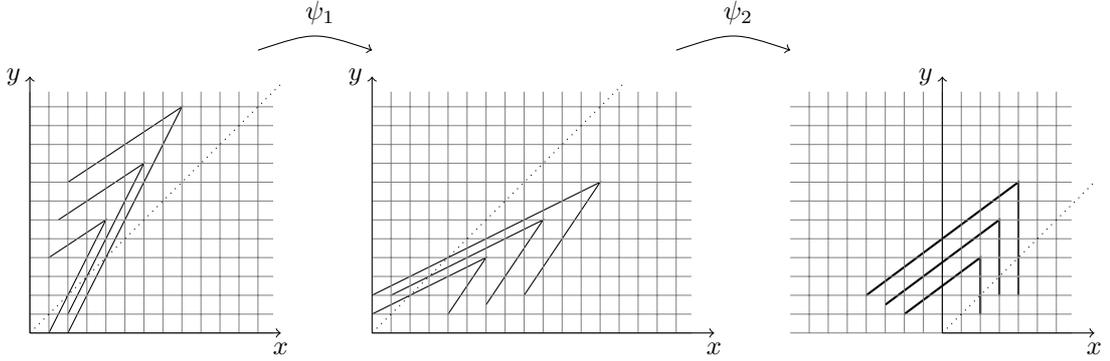
\begin{figure}[htb]
\centering
\begin{tikzpicture}
%
%
%
\draw [thick]  (12.5,0.25) -- (12.5,1) -- (11.5,0.25);
\draw [thick]  (12.75,0.5) -- (12.75,1.5) -- (11.25,0.375);
\draw [thick]  (13,0.5) -- (13,2) -- (11,0.5);
\draw[step=.25cm,gray,very thin] (10,0) grid (13.7,3.2);
\draw [->] (12,0) -- (14,0) node[anchor=north]{$x$};
\draw [->] (12,0) --  (12,3.4) node[anchor=east]{$y$};
\draw[dotted] (12,0) -- (14,2);
\draw   (0.25,0) -- (1,1.5) -- (0.25,1);
\draw   (0.5,0.25) -- (1.5,2.25) -- (0.375,1.5);
\draw   (0.5,0) -- (2,3) -- (0.5,2);
\draw[step=.25cm,gray,very thin] (0,0) grid (3.2,3.2);
\draw [->] (0,0) -- (3.3,0) node[anchor=north]{$x$};
\draw [->] (0,0) --  (0,3.4) node[anchor=east]{$y$};
\draw[dotted] (0,0) -- (3.3,3.3);
\draw   (4.5,0.25) -- (6,1) -- (5.5,0.25);
\draw   (4.75,0.5) -- (6.75,1.5) -- (6,0.375);
\draw   (4.5,0.5) -- (7.5,2) -- (6.5,0.5);
%
\draw[step=.25cm,gray,very thin] (4.5,0) grid (8.7,3.2);
\draw [->] (4.5,0) -- (9,0) node[anchor=north]{$x$};
\draw [->] (4.5,0) --  (4.5,3.4) node[anchor=east]{$y$};
\draw[dotted] (4.5,0) -- (7.8,3.3);
\draw[->] (8.5,3.75) .. controls (9.25,4) .. (10,3.75);
\draw (9,4.25) node[right,text width=2cm]{$\psi_2$};
\draw[->] (3,3.75) .. controls (3.75,4) .. (4.5,3.75);
\draw (3.5,4.25) node[right,text width=2cm]{$\psi_1$};
\end{tikzpicture}
\caption{Applying $\psi_1$ and $\psi_2$ to elements $P$ with $v_{2,-1}(P)\le 4$.}
\end{figure}
\noindent Via the flip $\psi_1$ this is the same as saying something about Jacobian pairs in $L$ with
$\frac 1m v_{-1,2}(P)\le 4$. Applying the automorphism $\psi_2$ defined by $\psi_2(x):=x$ and
$\psi_2(y):=  x^2y$,
this amounts to
proving facts about $(m,n)$-pairs in $L^{(1)}$ with $\frac 1m v_{1,0}(P)\le 4$, which we will do in the
sequel.

\begin{proposition}\label{primitivo} Let $(P,Q)$ be a standard $(m,n)$-pair. There exists exactly
one regular corner $((a,b),(\rho,\sigma))$ of $(P,Q)$ of type~II.b).
Moreover,
\begin{enumerate}

\smallskip

\item $\sigma<0$,

\smallskip

\item $v_{\rho,\sigma}(P)>0$ and $v_{\rho,\sigma}(Q)>0$,

\smallskip

\item $\frac{v_{\rho,\sigma}(P)}{v_{\rho,\sigma}(Q)}= \frac mn$,

\smallskip

\item $[\ell_{\rho,\sigma}(P),\ell_{\rho,\sigma}(Q)]=0$,

\smallskip

\item $(\rho,\sigma)\in \Dir(P)\cap\Dir(Q)$,

\smallskip

\item There exists $\mu\in \mathds{Q}$ greater than $0$ such that
$$
\en_{\rho,\sigma}(F) = \frac{\mu}{m}\en_{\rho,\sigma}(P),
$$
where $F\in L^{(1)}$ is the $(\rho,\sigma)$-homogeneous element obtained
in Theorem~\ref{central},

\smallskip

\item $v_{1,-1} \bigl(\en_{\rho,\sigma}(P) \bigr)<0$ and
    $v_{1,-1} \bigl(\en_{\rho,\sigma}(Q)\bigr)<0$,

\smallskip

\item $\frac{v_{\rho',\sigma'}(P)}{v_{\rho',\sigma'}(Q)}=\frac mn$ for all
    $(\rho,\sigma)<(\rho',\sigma')<(1,0)$,

\smallskip

\item $v_{1,1}(\en_{\rho,\sigma}(P))\le v_{1,1}(\en_{1,0}(P))$,

\end{enumerate}

\end{proposition}

\begin{proof} The uniqueness follows immediately from the definition of $A(P)$ and
Proposition~\ref{esquina regular unica}. The same proposition yields a regular corner
$((a,b),(\rho,\sigma))$  such that
\begin{equation}\label{not in A(P)}
(\rho,\sigma)\not\in A(P).
\end{equation}
Statements~(2), (3) and~(5) follow now from Corollary~\ref{some properties of corners}.
By Remark~\ref{a>0} we have $1-a/b<1$. Hence, by Proposition~\ref{extremosfinales}
we are in case II or in case III, and so statement~(4) holds. Fur\-thermore, by
Remark~\ref{a remark} we have
$$
\frac 1m \st_{\rho,\sigma}(P)=\frac 1n \st_{\rho,\sigma}(Q),
$$
which implies
$\frac 1m \st_{\rho,\sigma}(P)\in\mathds{Z}\times \mathds{N}_0$. We will prove that
\begin{equation}\label{condicion caso IIb}
v_{1,-1}(\st_{\rho,\sigma}(P))>0.
\end{equation}
 Assume by contradiction
that $v_{1,-1}(\st_{\rho,\sigma}(P))\le0$, which implies
$v_{1,-1}(\st_{\rho,\sigma}(P))<0$, by Theorem~\ref{central}(4).
If $v_{0,-1}(\frac 1m \st_{\rho,\sigma}(P))\le -1$ then $(\rho,\sigma)\in A(P)$, which
contradicts~\eqref{not in A(P)}.
Hence $v_{0,-1}(\frac 1m \st_{\rho,\sigma}(P))=0$, and so $\st_{\rho,\sigma}(P)=(k,0)$,
for some $k<0$. But then
$$
0<v_{\rho,\sigma}(P)=\rho k<0,
$$
a contradiction which proves~\eqref{condicion caso IIb}. This implies
statement~(1) since, by the definition of standard
$(m,n)$-pair, $v_{1,-1}(\st_{1,0}(P))<0$.

Assume that $((a,b),(\rho,\sigma))$ is of type~III.
Then, by Proposition~\ref{case III}, we know that $\rho|l=1$, and so
$(\rho,\sigma)=(1,0)$, which contradicts statement~(1). Hence, by
inequality~\eqref{condicion caso IIb}, we are in case~II.b).

Statement~(6) follows from items~(1) and~(4) of Proposition~\ref{case II}. Statement~(7) for $P$
follows from Definition~\ref{def regular corner}, and then it follows for $Q$, by
Corollary~\ref{some properties of corners}(2).

By Proposition~\ref{le basico} and Definition~\ref{def regular corner}, if $\Succ_P(\rho,\sigma)\in I$,
then
$\Succ_P(\rho,\sigma)\in A(P)$. Consequently, by Proposition~\ref{A(P) vs Dir(P)},
$$
\Dir(P)\cap \,\, ](\rho,\sigma),(1,0)]\subseteq A(P).
$$
Statement~(8) now
follows easily from Proposition~\ref{le basico}, Remark~\ref{a remark} and the fact that, by
Proposition~\ref{esquinas regulares},
statement~(3) holds for all
$(\rho_j,\sigma_j)\in A(P)$. Finally, by Proposition~\ref{le basico} and
Remark~\ref{starting vs end},
$$
v_{1,1}(\en_{\rho',\sigma'}(P)) = v_{1,1}(\st_{\rho'',\sigma''}(P)) < v_{1,1}(\en_{\rho'',\sigma''}(P))
$$
for consecutive directions $(\rho',\sigma')<(\rho'',\sigma'')$ in $\Dir(P)\cap I$,
from which statement~(9) follows.
\end{proof}

\begin{definition}\label{starting triple}
The  {\it starting triple} of a  standard $(m,n)$-pair $(P,Q)$ is
$(A_0,A_0',(\rho,\sigma))$, where $(A_0,(\rho,\sigma))$
is the unique regular corner of $(P,Q)$ with $(\rho,\sigma)\notin A(P)$,
and $A_0'=\frac 1m \st_{\rho,\sigma}(P)$.  The point $A_0$ is called the
{\it primitive corner} of $(P,Q)$.
\end{definition}

\begin{remark}\label{starting triple es IIb}
By  Propositions~\ref{esquina regular unica} and~\ref{primitivo} and Remark~\ref{estamos en Case IIa},
in the previous definition $(A_0,(\rho,\sigma))$ is the unique regular corner of type~II.b). Consequently
$v_{1,-1}(\st_{\rho,\sigma}(P))>0$.
\end{remark}

Let $(P,Q)$ be a standard $(m,n)$-pair and $(A_0,A_0',(\rho,\sigma))$ its
starting triple. Let $\lambda$ and $m_\lambda$ be as in
Proposition~\ref{case IIb}, let $\varphi\in\Aut(L^{(\rho)})$
and $A^{(1)}$ be as in Proposition~\ref{encima de la diagonal} and let $F$
be as in Proposition~\ref{case II}.
Note that $F\in L$. In fact, for $(i,j)\in \Supp(F)$, we have
$$
\rho i+\sigma j = v_{\rho,\sigma}(i,j) = v_{\rho,\sigma}(F) = \rho+\sigma > 0,
$$
which implies that $i\ge 0$, since $\rho>0$, $\sigma <0$ and $j\ge 0$. Write
$$
(f_1,f_2):=\en_{\rho,\sigma}(F),\quad (u,v):=A_0,\quad (r',s'):=A_0' \quad\text{and}\quad
\gamma:=\frac{m_\lambda}{m}.
$$
\begin{figure}[htb]
\centering
\begin{tikzpicture}[scale=0.25]
\begin{scope}[xshift=0cm, yshift=0cm]
\draw [dash pattern=on 2pt off 1.5pt, ultra thin] (9,12) -- (9,0) (9,12) -- (0,12) (7.5,6) -- (7.5,0)
(7.5,6) -- (0,6) (1,1.5) -- (1,0) (1,1.5) -- (0,1.5);
\draw (8,-0.7) node[ above=0pt, right=0pt]{$\scriptstyle u$};
\draw (-1.7,1.5) node[
 above=0pt, right=0pt]{$\scriptstyle f_2$};
\draw (0.35,-0.6) node[
 above=0pt, right=0pt]{$\scriptstyle f_1$};
\draw (5.8,-0.5) node[
 above=0pt, right=0pt]{$\scriptstyle  r'$};
\draw (-2.5,12) node[
 above=0pt, right=0pt]{$\scriptstyle v$};
\draw (-2.5,6) node[
 above=0pt, right=0pt]{$\scriptstyle  s'$};
\draw (8.5,12) node[
 above=0pt, right=0pt]{$\scriptstyle  A_0$};
\draw (14,9) node[
above=0pt, right=0pt]{$\scriptstyle  A^{(1)}$};
\draw (7.2,6) node[
above=0pt, right=0pt]{$\scriptstyle A'_0$};
\draw [->] (-0.5,0)--(24,0) node[anchor=north]{$x$};
\draw [->] (0,-0.5)--(0,24) node[anchor=east]{$y$};
\draw[dotted,  thin] (-0.5,-0.5) --(23.6,23.6);
\draw[dotted,  thin] (-0.5,-0.5) --(9,12);
\draw [-] (0.833,0.833) -- (1,1.5);
\draw [-,thick] (3,18) -- (6,18) -- (9,15) -- (9,12) -- (8.5,10) -- (5.5,4);
\draw [densely dashdotted] (8.5,10) -- (7.5,6);
\draw[->] (14,9) ..controls (12,9) and (10.6,9.5) .. (8.6,10);
\end{scope}
\end{tikzpicture}
\caption{Illustration of Proposition~\ref{final}.}
\end{figure}
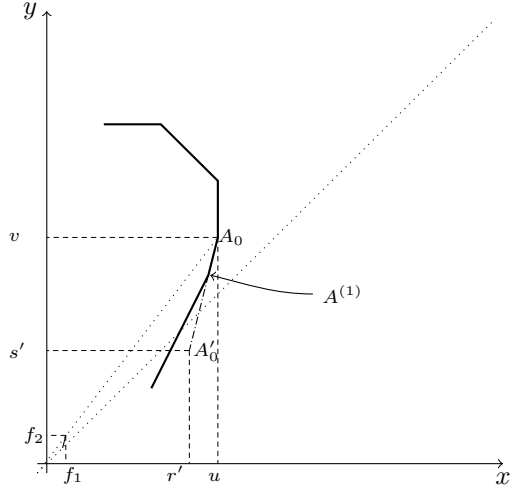

\begin{proposition}\label{final} It is true that $A_0,A_0'\in\mathds{N}_0\times \mathds{N}_0$ and
$v_{\rho,\sigma}(A_0) = v_{\rho,\sigma}(A_0')$. Moreover,

\begin{enumerate}

\smallskip

\item $s'<r'<u<v$,

\smallskip

\item $2\le f_1<u$,

\smallskip

\item $\gcd(u,v)>1$,

\smallskip

\item $\en_{\rho,\sigma}(F) = \mu A_0$ for some $0<\mu<1$,

\smallskip

\item $uf_2 = vf_1$ and $\rho\le u$,

\smallskip

\item $(\rho,\sigma)=\dir(f_1-1,f_2-1)=\left(\frac{f_2-1}{d},\frac{1-f_1}{d}\right)$, where
    $d:=\gcd(f_1-1,f_2-1)$,

\smallskip

\item $A^{(1)}=A_0'+(\gamma-s')\left(-\frac{\sigma}{\rho},1 \right)$,

\smallskip

\item If $A^{(1)}=(a'/\rho,b')$, then $\rho-a'/b'>1$ or $\gcd(a',b')>1$,

\smallskip

\item $\gamma\le (v-s')/\rho$. Moreover, if $d=\gcd(f_1-1,f_2-1)=1$, then $\gamma=(v-s')/\rho$.

\smallskip

\end{enumerate}
\end{proposition}

\begin{proof} By statements~(2), (3) and~(4) of Proposition~\ref{primitivo} and statement~(2b) of
Proposition~\ref{P y Q alineados}, there exist $\lambda_P,\lambda_Q\in K^{\times}$ and a
$(\rho,\sigma)$-homogeneous element $R\in L^{(1)}$ such that
\begin{equation}\label{R existe}
\ell_{\rho,\sigma}(P) = \lambda_P R^m\quad\text{and}\quad\ell_{\rho,\sigma}(Q) = \lambda_P R^n.
\end{equation}
This implies that
$$
A_0 =\en_{\rho,\sigma}(R)\qquad\text{and}\qquad A'_0 =\st_{\rho,\sigma}(R).
$$
Hence, $v_{\rho,\sigma}(A_0) = v_{\rho,\sigma}(A_0')$. Moreover, the same
argument given above for $F$ shows that $R\in L$, and so $A_0,A'_0\in
\mathds{N}_0\times \mathds{N}_0$

Statement~(1) follows from the fact that, by inequality~\eqref{condicion caso IIb} and
Proposition~\ref{primitivo}(7)
$$
v_{1,-1}(\en_{\rho,\sigma}(P))<0\qquad\text{and}\qquad v_{1,-1}(\st_{\rho,\sigma} (P))>0,
$$
and, by Remark~\ref{starting vs end},
\begin{equation}\label{dat2}
v_{1,0}(\st_{\rho,\sigma}(P)) < v_{1,0}(\en_{\rho,\sigma}(P)),
\end{equation}
since $(1,-1)<(\rho,\sigma)\!<\!(1,0)$. Proposition~\ref{primitivo}(6) gives statement~(4) except the
inequality
$\mu<1$. But this is true because $\mu\ge1$ implies
$$
v_{\rho,\sigma}(A_0')=v_{\rho,\sigma}(A_0)=\frac{1}{\mu}v_{\rho,\sigma}(F)\le v_{\rho,\sigma}(F)=\rho+\sigma,
$$
which is is impossible since, by statement~(1) and Proposition~\ref{primitivo}(1),
$$
v_{\rho,\sigma}(A_0')=r'\rho+s'\sigma=(r'-s')\rho+s'(\rho+\sigma)\ge(r'-s')\rho\ge \rho>\rho+\sigma.
$$
We claim that $v_{1,0}(\st_{\rho,\sigma}(F))\!\ge\! 1$. In fact, otherwise
$\st_{\rho,\sigma}(F)\!=\!(0,h)$ for some $h\!\in\!\mathds{N}_0$, which
implies $v_{\rho,\sigma}(F)\!=\!\sigma h\!\le\! 0$. But this is impossible
since $v_{\rho,\sigma}(F) = \rho+\sigma>0$. Hence, by Remark~\ref{starting vs end},
$$
f_1=v_{1,0}(\en_{\rho,\sigma}(F)) > v_{1,0}(\st_{\rho,\sigma}(F))\ge 1,
$$
which combined with $f_1\!=\!\mu u$ and $0<\mu<1$ proves statement~(2). Moreover,
if $\gcd(u,v)\!=\!1$, then there is no
$\mu\!\in\,\, ]0,1[$ such that $\mu(u,v)\in\mathds{N}_0\times\mathds{N}_0$, and so statement~(3) is true. Next
we prove statement~(5). From statement~(4) it follows that
$uf_2 = vf_1$. Equivalently $u(f_1,f_2) = f_1(u,v)$, and so
$$
v_{\rho,\sigma}(f_1A_0')= f_1 v_{\rho,\sigma}(A'_0)= f_1 v_{\rho,\sigma}(A_0)= u v_{\rho,\sigma}(F) = u
v_{\rho,\sigma}(1,1)=v_{\rho,\sigma}(u,u).
$$
Hence there exists $t\in\mathds{Z}$ such that $f_1A_0'=(u,u)-t(-\sigma,\rho)$. Thus
$$
u-t\rho = v_{0,1}(f_1A_0') = f_1v_{0,1}(A_0') \ge 0,
$$
and so $u\ge t\rho$. Therefore, in order to finish the proof of statement~(5) we only must note that $t\le 0$
is impossible, because it implies $f_1 v_{1,-1}(A_0')\le 0$, contradicting Remark~\ref{starting triple es IIb}.

The first equality in statement~(6) follows from the fact that
$v_{\rho,\sigma}(f_1,f_2)=v_{\rho,\sigma}(1,1)$ and Remark~\ref{valuacion depende de extremos}. So,
by~\eqref{val}
$$
(\rho,\sigma) = \pm \left(\frac{f_2-1}{d},\frac{1-f_1}{d}\right),
$$
where $d:=\gcd(f_1-1,f_2-1)$. Since $\rho+\sigma>0$ and, by statements~(1) and~(4), we have $f_2-f_1>0$,
necessarily
$$
(\rho,\sigma) = \left(\frac{f_2-1}{d},\frac{1-f_1}{d}\right),
$$
which ends the proof of statement~(6). Next we prove statement~(7). The point $A^{(1)}$ is completely
determined
by $v_{0,1}(A^{(1)})$ and $v_{\rho,\sigma}(A^{(1)})$. Let $\varphi$ be as in
Proposition~\ref{encima de la diagonal}.
Since
$$
v_{0,1}\left(A_0'+(\gamma-s')\left(-\frac{\sigma}{\rho},1 \right)\right)=\gamma=v_{0,1}(A^{(1)})
$$
by Proposition~\ref{encima de la diagonal}(3), and
$$
v_{\rho,\sigma}\left(A_0'+(\gamma-s')\left(-\frac{\sigma}{\rho},1 \right)\right)
=v_{\rho,\sigma}(A_0')=v_{\rho,\sigma}(A^{(1)})
$$
because $\varphi$ is $(\rho,\sigma)$-homogeneous,
statement~(7) is true. Statement~(8) follows directly from
Propositions~\ref{encima de la diagonal}(4) and~\ref{criterion}. It remains
to prove statement~(9). Write $\ell_{\rho,\sigma}(P)=x^r y^s p(z)$, where $z=x^{-\sigma/\rho}y$ and $p(0)
\ne 0$.
Since $R\in L$ and $\ell_{\rho,\sigma}(P) = \lambda_P R^m$, we have $\ell_{\rho,\sigma}(P)\in L$,
 which implies $p(z)\in K[x,y]$. Hence
$$
\ell_{\rho,\sigma}(P)=x^r y^s \ov p(z^\rho),\quad\text{where $\ov p(z^{\rho})\in K[z^{\rho}]$}.
$$
By Proposition~\ref{case IIb} we know that $\lambda\ne 0$, so $m_\lambda$ is the multiplicity of a root of
$p(z)$. Since the multiplicities of the roots of $p(z)$ are the same as the multiplicities of the roots of
$\ov p(z)$, we have $m_\lambda\le \deg(\ov p)$. Combining this with Remark~\ref{polinomio asociado f^{(l)}},
we obtain
$$
m_\lambda\le \deg(\ov p)= \frac{v_{0,1}(\en_{\rho,\sigma}(P))-v_{0,1} (\st_{\rho,\sigma}(P))} {\rho}=
m\left(\frac{v-s'}{\rho}\right),
$$
which proves the first part of statement~(9). We claim that $\st_{\rho,\sigma}(F)\!=\!(1,1)$. In fact,
other\-wise
$$
(\alpha,\beta):=\st_{\rho,\sigma} (F)=(1-\sigma i,1+\rho i),\qquad\text{with $i>0$.}
$$
Note that $\alpha<\beta$, since $\rho>-\sigma$. But this is impossible because $r'>s'$, and, by
Theorem~\ref{central}(2),
 we have
$$
\st_{\rho,\sigma}(F) \sim \st_{\rho,\sigma}(P)= m(r',s').
$$
Consequently, if $d=1$, then by statement~(6)
$$
v_{0,1}(\en_{\rho,\sigma}(F))- v_{0,1}(\st_{\rho,\sigma}(F)) = f_2-1 = \rho,
$$
and so, by Proposition~\ref{pavadass}(4),
$$
m_\lambda = \frac{1}{\rho}\left(v_{0,1}(\en_{\rho,\sigma}(P))- v_{0,1}(\st_{\rho,\sigma} (P))\right)=
\frac{m(v-s')}{\rho},
$$
as desired.
\end{proof}

\begin{proposition}\label{primera cota para primitivos} If $A_0$ is as before Proposition~\ref{final}, then
$v_{1,1}(A_0)\ge 16$.
\end{proposition}

\begin{proof} By Proposition~\ref{final} it suffices to prove that there is no pair $A_0=(u,v)$ with $u+v
\le 15$,
for which there exist $(f_1,f_2)$, $A_0'=(r',s')$, $\gamma$ and $A^{(1)}$, such that all the conditions of that
proposition are satisfied.
\begin{table}[htb]
\begin{center}
\setlength{\tabcolsep}{8pt}
\ra{1.2}
\begin{tabular}{ccccccc}
\toprule
$A_0$ &$(f_1,f_2)$ & $(\rho,\sigma)$& $A_0'$ & d & $\gamma$ & $A^{(1)}$\\
\midrule
(3,6) & (2,4) & (3,-1) & (1,0) & 1& 2& $\left(\frac 53,2\right) $\\
              (3,9) & (2,6) & (5,-1) & $\times$ & &&\\
              (3,12) & (2,8) & (7,-1) & $\times$ &&&\\
              (4,6) & (2,3) & (2,-1) & (1,0) & 1& 3& $\left(\frac 52,3\right) $\\
              (4,8) & (2,4) & (3,-1) & $\times$ &&&\\
              (4,8) & (3,6) & (5,-2) & $\times$ &&&\\
              (4,10) & (2,5) & (4,-1) & $\times$ &&&\\
              (5,10) & (2,4) & (3,-1) & (2,1) &1&3&  $\left(\frac 83,3\right) $\\
              (5,10) & (3,6) & (5,-2) & (1,0) &1&2&  $\left(\frac{9}5,2\right) $\\
              (5,10) & (4,8) & (7,-3) & $\times$ &&&\\
              (6,8) & (3,4) & (3,-2) & $\times$ &&&\\
              (6,9) & (2,3) & (2,-1) & (2,1) &1&4& $\left(\frac 72,4\right) $\\
              (6,9) & (4,6) & (5,-3) & $\times$ &&&\\
\bottomrule
\end{tabular}
\end{center}
\caption{}
\label{Tabla1}
\end{table}
\noindent In Table~\ref{Tabla1} we first list all possible pairs $(u,v)$ with $v>u>2$, $\gcd(u,v)>1$ and
$u+v\le 15$. We also list all the possible $(f_1,f_2)=\mu (u,v)$ with $f_1\ge 2$ and $0<\mu<1$. Then we
compute the co\-rres\-ponding $(\rho,\sigma)$ using Proposition~\ref{final}(6) and we verify if there is
an~$A_0':=(r',s')$ with $s'<r'<u$ and $v_{\rho,\sigma}(u,v) = v_{\rho,\sigma}(r',s')$. This happens in
five cases. In all these cases $d:=\gcd(f_1-1,f_2-1)=1$. Then, by Proposition~\ref{final}(9), we
have $\gamma=(v-s')/\rho$. Using these values, we compute $A^{(1)}$ in each of the five cases using
statement~(7) of the same
proposition. Finally we verify that in none of this cases condition~(8) of Proposition~\ref{final}
is satisfied, concluding the proof.
\end{proof}

\begin{corollary} We have $B\ge 16$.
\end{corollary}

\begin{proof} Suppose $B<\infty$ and take $(P,Q)$ and $(m,n)$ as in Corollary~\ref{B finito}. Assume that
$(\rho,\sigma)$ and $A_0$ are as above Proposition~\ref{final}.
By Proposition~\ref{primera cota para primitivos},
$$
B =\gcd\bigl(v_{1,1}(P),v_{1,1}(Q)\bigr)= \frac{1}{m} v_{1,1}(P)\ge \frac{1}{m}v_{1,1}(\en_{\rho,\sigma}(P))
= v_{1,1}(A_0) \ge 16,
$$
as desired.
\end{proof}

\begin{proposition}\label{u(u-1)} Let $(P,Q)$ be a standard $(m,n)$-pair and let $A_0=(u,v)$ be as before
Proposi\-tion~\ref{final}. Then $v\le u(u-1)$ and $u\ge 4$.
\end{proposition}

\begin{proof} Let $F$, $(f_1,f_2)=\en_{\rho,\sigma}(F)$ and $d=\gcd(f_1-1,f_2-1)$ be as before
Proposition~\ref{final}. By statements~(5) and~(6) of Proposition~\ref{final},
$$
\frac{f_1v-u}{du} = \frac{f_2-1}{d} = \rho \le u.
$$
Hence
$$
v \le \frac{du^2+u}{f_1} = u\frac{du+1}{f_1} \le u\frac{(f_1-1)u+1}{f_1} = u\frac{f_1u-(u-1)}{f_1} =
u\left(u-\frac{u-1}{f_1}\right)\le u(u-1),
$$
where the last inequality follows from Proposition~\ref{final}(2). Again by Proposition~\ref{final}(2),
we know that $u\ge 3$, so we must only check that the case $u=3$ is impossible.
But if $A_0=(3,v)$, then by the first statement necessarily
$v\le 6$, which contradicts Proposition~\ref{primera cota para primitivos}.
\end{proof}

\begin{remark}
The inequality $u\ge 4$ is related to~\cite{H}*{Proposition 2.22}. It shows that for a standard $(m,n)$-pair $(P,Q)$, the greatest common divisor of
$\deg_x(P)=v_{1,0}(P)$ and $\deg_x(Q)=v_{1,0}(Q)$ is greater than or equal to $4$. Using similar techniques as in the proof of
Proposition~\ref{primera condicion estandar}, one can prove that this inequality holds for any counterexample.
\end{remark}

Let $\psi_1\in \Aut(L)$ be the map defined by $\psi_1(x):=y$ and
$\psi_1(y):=-x$. Since $[\psi_1(x),\psi_1(y)] = 1$,
by Proposition~\ref{varphi preserva el Jacobiano}, this map preserves
Jacobian pairs. Moreover, the action induced by $\psi_1$
on the Newton polygon of a polynomial $P$ is the orthogonal reflection at
the main diagonal, and so, it maps edges of the convex
hull of $\Supp(P)$ into edges of the convex hull of $\Supp(\psi_1(P))$,
interchanging $\st$ and $\en$.

Similarly the automorphism
$\psi_2$ of $L^{(1)}$, defined by $\psi_2(x):=-x^{-1}$ and $\psi_2(y):=x^2
y$ preserves Jacobian pairs and it induces on the Newton polygon of
each $P\in L^{(1)}$ a
reflection at the main diagonal, parallel to the $X$-axis . Hence it also maps edges of the convex hull of
$\Supp(P)$ into edges of the convex hull of $\Supp(\psi_2(P))$,
interchanging $\st$ and $\en$.

Moreover, an elementary computation shows that
if we define
\begin{equation}\label{cambio de direccion}
\ov \psi_1(\rho,\sigma):=(\sigma,\rho)\quad\text{and}\quad \ov \psi_2(\rho,\sigma):= (-\rho,2\rho+\sigma),
\end{equation}
and set $(\rho_k,\sigma_k):=\ov\psi_k(\rho,\sigma)$ for  $k=1,2$, then
\begin{equation}\label{polinomios por cambio de direccion}
v_{\rho_k,\sigma_k}(\psi_k(P))=v_{\rho,\sigma}(P)\quad\text{and}\quad
\ell_{\rho_k,\sigma_k}(\psi_k(P))=\psi_k(\ell_{\rho,\sigma}(P)),
\end{equation}
for all  $(\rho,\sigma)\in\mathfrak{V}$ and $P\in L^{(1)}$ (when $k=1$ we assume $P\in L$).

\begin{proposition}\label{impossibles} Let $(P,Q)$ be a standard $(m,n)$-pair in $L$ and let $(\rho,\sigma)$
and $A_0$ be as before Proposition~\ref{final}. If $(\rho,\sigma)=(2,-1)$, then it is impossible
that $v_{\rho,\sigma}(A_0)\le 3$ or that $A_0=(8,12)$.
\end{proposition}

\begin{proof} Let $\varphi\colon L\to L^{(1)}$ be the morphism defined by $\varphi:=\psi_2\circ \psi_1$.
Write $\en_{2,-1}(P)=(a,b)$, so that $A_0 = \frac 1m (a,b)$. We claim that $\en_{1,0}(\varphi(P))=(2a-b,a)$.
In order to prove the claim, note first that
$$
\ov\psi_1(-1,1)=(1,-1),\quad\ov\psi_1(2,-1)=(-1,2),\quad\ov\psi_2(-1,2)=(1,0)\quad \text{and} \quad
\ov\psi_2(1,-1)=(-1,1).
$$
Since, by Remark~\ref{starting vs end},
$$
\Supp\bigl(\ell_{-1,1}(\ell_{2,-1}(P))\bigr) = \en_{2,-1}(P) = (a,b),
$$
and, by the second equality in~\eqref{polinomios por cambio de direccion},
$$
\ell_{-1,1}(\ell_{1,0}(\varphi(P))) = \ell_{-1,1}(\psi_2(\ell_{-1,2}(\psi_1(P)))) =
\psi_2(\ell_{1,-1}(\psi_1(\ell_{2,-1}P)))= \varphi(\ell_{-1,1}(\ell_{2,-1}(P))),
$$
we have, again by Remark~\ref{starting vs end},
$$
\en_{1,0}(\varphi(P))=\Supp(\varphi(x^a y^b))= (2a-b,a),
$$
which proves the claim. Moreover, $(\varphi(P),\varphi(Q))$ is an $(m,n)$-pair, because
by Proposition~\ref{varphi preserva el Jacobiano}, we have $[\varphi(P),\varphi(Q)]=1$;
it is true that
$$
v_{1,-1}(\en_{1,0}(\varphi(P)))=a-b=v_{1,-1}(\en_{2,-1}(P))<0;
$$
and, by the first equality in~\eqref{polinomios por cambio de direccion}, statements~(3) and~(8) of
Proposition~\ref{primitivo}, and the fact that $\ov\psi_2(\ov\psi_1(2,-1))=(1,0)$ and
$\ov\psi_2(\ov\psi_1(3,-1))=(1,1)$, we have
$$
\frac{v_{1,0}(\varphi(P))}{v_{1,0}(\varphi(Q))}=\frac{v_{2,-1}(P)}{v_{2,-1}(Q)}=\frac mn\quad\text{and}\quad
\frac{v_{1,1}(\varphi(P))} {v_{1,1}(\varphi(Q))} = \frac{v_{3,-1}(P)}{v_{3,-1}(Q)}=\frac mn.
$$
Applying Proposition~\ref{todos son Smp} we obtain a standard $(m,n)$-pair $(\wt{P},\wt{Q})$ with
$$
\frac 1m \en_{1,0}(\wt{P}) = \frac 1m \en_{1,0}(\varphi(P)) = \frac 1m (2a-b,a).
$$
Hence,
$$
\frac 1m v_{1,0}(\wt{P}) = \frac 1m v_{1,0}(\en_{1,0}(\wt{P})) = \frac 1m(2a-b) = \frac 1m
v_{2,-1}(\en_{2,-1}(P))=v_{\rho,\sigma}(A_0).
$$
Let $\wt{A}_0 = (u,v)$ be the primitive corner of $(\wt{P},\wt{Q})$. Since
$m(u,v)\in \Supp(\wt{P})$, we have
$$
u\le \frac 1m v_{1,0}(\wt{P}) = v_{\rho,\sigma}(A_0).
$$
So, if $v_{\rho,\sigma}(A_0)\le 3$, then $u\le 3$, which contradicts Proposition~\ref{u(u-1)}. If $\frac 1 m
(a,b) = A_0=(8,12)$, then, by Proposition~\ref{primitivo}(9), we have
$$
v_{1,1}(\wt{A}_0)\le \frac 1m v_{1,1}(\en_{1,0}(\wt{P})) = v_{1,1}(4,8)=12,
$$
which is impossible by Proposition~\ref{primera cota para primitivos}.
\end{proof}

%

\section{More conditions on $\bm{B}$}\label{More conditions on B}

\setcounter{equation}{0}
In this section we prove that $B=16$ or $B>20$, and that $B\ne 2p$ for all prime $p$. The first result can be
inferred from~\cite{H}*{Theorem~2.24}, whose proof is hidden behind a computer search. We give a complete
proof, without the use of a computer. Abhyankar allegedly developed a proof of the second result according
to~\cite{H}*{Page~50}, but we could not find any published article of Abhyankar with such proof. Heitmann says
that it is possible to adapt the proof of~\cite{H}*{Proposition~2.21} to prove $B\ne 2p$, however we were not
able to do this. On the other hand this is also claimed to be proven in~\cite{Z}*{Theorem~4.12}. But the proof
relies on~\cite{Z}*{Lemma~4.10}, which has a gap, since it claims without proof that $I_2 \subseteq \frac 1m
\Gamma(f_2)$, an assertion which cannot be proven to be true. The main technical results in this section are
Propositions~\ref{proporcionalidad de direcciones mayores} and~\ref{proporcionalidad de direcciones menores},
together with its Corollaries~\ref{fracciones de F} and~\ref{fracciones de F1}. They are closely related
to~\cite{M}*{Propositions~6.3 and~6.4} and seem to be a generalization of them. These results are
interesting on their own, but they also allow to establish a very strong criterion for the possible
regular corners (Theorem~\ref{divisibilidad}) which leads to the proof of
$B\ne 2p$. Another application of these results is the generalization of the reduction of degree
technique of Moh, which we apply to the case $B=16$ in the last section. In the present section we also
describe in Corollary~\ref{forma final en L} the shape of a
possible counterexample to the Jacobian conjecture in the case $B=16$.

\setcounter{equation}{0}

\begin{proposition}\label{proporcionalidad de direcciones mayores} Let $m,n\in \mathds{N}$ be coprime with
$m,n>1$ and let $P,Q\in L^{(l)}$ with
$$
[P,Q]\in K^{\times}\quad\text{and}\quad \frac{v_{1,1}(P)}{v_{1,1}(Q)}=\frac{v_{1,0}(P)}{v_{1,0}(Q)} =\frac mn,
$$
Take $T_0\in K[P,Q]$ and set $T_j:=[T_{j-1},P]$ for $j\ge 1$. Assume that $(\rho_0,\sigma_0)\in
\mathfrak{V}_{\ge0}$ satisfies
\begin{enumerate}

\smallskip

  \item $(\rho_0,\sigma_0)\in\Dir(P)$ and $v_{\rho_0,\sigma_0}(P)>0$,

\smallskip

  \item $\en_{\rho_0,\sigma_0}(T_j)\sim \en_{\rho_0,\sigma_0}(P)$ for all $j$ with $T_j\ne 0$,

\smallskip

  \item $\frac{1}{m}\en_{\rho_0,\sigma_0}(P) = \frac{1}{n}\en_{\rho_0,\sigma_0}(Q) \in \frac 1l
\mathds{Z}\times\mathds{N}$,

\smallskip

  \item $b>a/l$, where $(a/l,b):=\frac{1}{m}\en_{\rho_0,\sigma_0}(P)$.

\smallskip

\end{enumerate}
Let $I_0:=[(\rho_0,\sigma_0),(0,-1)[$ and
$$
(\tilde\rho,\tilde\sigma):=\max\{(\rho,\sigma)\in \Dir(P)\cap I_0: v_{\rho',\sigma'}(P)>0 \text{ for all }
(\rho_0,\sigma_0)\le(\rho',\sigma')\le(\rho,\sigma) \}
$$
Then for all $(\rho,\sigma)\in \mathfrak{V}$ with $(\rho_0,\sigma_0)<(\rho,\sigma)\le(\tilde\rho,\tilde\sigma)$
and all $j\ge 0$ we have
\begin{equation}\label{cocientes}
[\ell_{\rho,\sigma}(T_j),\ell_{\rho,\sigma}(P)]=0\quad\text{and}\quad
\frac{v_{\rho,\sigma}(T_j)}{v_{\rho,\sigma}(P)}= \frac{v_{\rho_0,\sigma_0}(T_j)}{v_{\rho_0,\sigma_0}(P)}.
\end{equation}
\end{proposition}

\noindent {\bf Idea of the proof:} We must prove that there is a partial homothety between $P$ and $T_j$
for $(\rho,\sigma)>(\rho_0,\sigma_0)$. The basic idea is that otherwise $\en (T_{j+n})\nsim \en(P)$ for some
direction and all $n>0$, and then $T_{j+n}\ne 0$ for all $n>0$, which is impossible.

\begin{proof}  Let
$$
(\rho_1,\sigma_1)<\dots<(\rho_k,\sigma_k)=(\tilde\rho,\tilde\sigma)
$$
be the directions in $\Dir(P)$ between  $(\rho_0,\sigma_0)$ and $(\tilde\rho,\tilde\sigma)$.
We will use freely that $v_{\rho',\sigma'}(P)>0$ for all
$(\rho_0,\sigma_0)\le(\rho',\sigma')\le(\tilde\rho,\tilde\sigma)$.
By Remark~\ref{relacion de equivalencia} and conditions~(2), (3) and~(4), we have
\begin{equation}\label{condini}
v_{1,-1}(\en_{\rho_0,\sigma_0}(P))<0\qquad\text{and}\qquad \en_{\rho_0,\sigma_0}(T_j) = \mu_j
\en_{\rho_0,\sigma_0}(P) \quad\text{with $\mu_j\ge 0$,}
\end{equation}
for all $j$ with $T_j\ne 0$.
We claim that if  there exists $0\le i<k$ such that
\begin{equation}\label{condind}
v_{1,-1}(\en_{\rho_i,\sigma_i}(P))<0\qquad\text{and}\qquad \en_{\rho_i,\sigma_i}(T_j) = \mu_j
\en_{\rho_i,\sigma_i}(P) \quad\text{with $\mu_j\ge 0$,}
\end{equation}
for all $j$ with $T_j\ne 0$, then
\begin{enumerate}

\smallskip

\item[(a)] If $T_j\ne 0$, then
$\en_{\rho_{i},\sigma_{i}}(T_j)=\st_{\rho_{i+1},\sigma_{i+1}}(T_j)$.

\smallskip

\item[(b)] $[\ell_{\rho_{i+1},\sigma_{i+1}}(T_j),\ell_{\rho_{i+1},\sigma_{i+1}}(P)]=0$, for all $j$.

\smallskip

\end{enumerate}
In order to check this, we write
$$
\en_{\rho_i,\sigma_i}(P) = r_i (a_i/l,b_i)\quad\text{with $r_i\ge 0$ and $\gcd(a_i,b_i)=1$.}
$$
We define the auxiliary direction
$$
(\ov\rho,\ov\sigma):=\frac 1d(lb_i,-a_i),\quad\text{where $d:=\gcd(lb_i,a_i)$.}
$$
By~\eqref{val} and the inequality in~\eqref{condind}, we have $(\ov\rho,\ov\sigma)=\dir(a_i/l,b_i)$.
Furthermore $r_i\in\mathds{N}$, because $\gcd(a_i,b_i)=1$ and
$\en_{\rho_i,\sigma_i}(P)\ne(0,0)$.
Note that
\begin{equation}\label{alineado con cruz}
    (c,d)\sim (a_i/l,b_i)\quad\text{if and only if}\quad v_{\ov\rho,\ov\sigma}(c,d)=0.
\end{equation}
Since $v_{\rho_i,\sigma_i}(a_i/l,b_i),v_{\rho_{i+1},\sigma_{i+1}}(a_i/l,b_i)>0$, by Remarks~\ref{180 grados} and~\ref{intervalo de direcciones}
 we know that
\begin{equation}\label{donde esta rho sigma}
    (\ov\rho,\ov\sigma)<    (\rho_i,\sigma_i)<    (\rho_{i+1},\sigma_{i+1})<    (-\ov\rho,-\ov\sigma).
\end{equation}
\smallskip
Next we prove condition~(a). For this it suffices to prove that if $T_j\ne
0$, then
$$
\Dir(T_j)\cap\, ](\rho_{i},\sigma_{i}),(\rho_{i+1},\sigma_{i+1})[\, =\emptyset.
$$
In order to check this fact, assume by contradiction that it is false and set $(\hat\rho,\hat\sigma):=\Succ_{T_j}(\rho_i,\sigma_i)$. Since
$(\hat\rho,\hat\sigma)\in ](\rho_{i},\sigma_{i}),(\rho_{i+1},\sigma_{i+1})[$, by~\eqref{donde esta rho sigma} we have
\begin{equation}\label{dirab nueva}
(\ov\rho,\ov\sigma)<(\hat\rho,\hat\sigma)<(-\ov\rho,-\ov\sigma).
\end{equation}
By Remark~\ref{starting vs end} and~\eqref{alineado con cruz}, we have
\begin{equation}\label{v de T menor que cero}
   v_{\ov\rho,\ov\sigma}(\en_{\hat\rho,\hat\sigma}(T_{j}))<
    v_{\ov\rho,\ov\sigma}(\st_{\hat\rho,\hat\sigma}(T_{j}))=0,
\end{equation}
since $(a_i/l,b_i)\sim \en_{\rho_i,\sigma_i}(T_j)=\st_{\hat\rho,\hat\sigma}(T_{j})$, by~\eqref{condind}.
We assert that
\begin{equation}\label{ch4}
T_{j+k}\ne0\quad\text{and}\quad\en_{\hat\rho,\hat\sigma}(T_{j+k})=\en_{\hat\rho,\hat\sigma}(T_{j})+k\en_{\hat\rho,\hat\sigma}(P)-k(1,1),
\end{equation}
for all $k\in \mathds{N}_0$. We will prove this by induction on $k$.
For $k=0$ this is trivial. Assume that~\eqref{ch4} is true for some $k$.
Then,
\begin{align*}
v_{\ov\rho,\ov\sigma}(\en_{\hat\rho,\hat\sigma}(T_{j+k}))&=v_{\ov\rho,\ov\sigma}(\en_{\hat\rho,\hat\sigma}(T_{j}))
+kv_{\ov\rho,\ov\sigma}(\en_{\hat\rho,\hat\sigma}(P))-kv_{\ov\rho,\ov\sigma}(1,1)\\
&= v_{\ov\rho,\ov\sigma}(\en_{\hat\rho,\hat\sigma}(T_{j}))-k(\ov\rho+\ov\sigma)\\
&<0,
\end{align*}
since
$v_{\ov\rho,\ov\sigma}(\en_{\hat\rho,\hat\sigma}(P))=v_{\ov\rho,\ov\sigma}(\en_{\rho_i,\sigma_i}(P))=0$
by Proposition~\ref{le basico} and~\eqref{alineado con cruz},
 $v_{\ov\rho,\ov\sigma}(\en_{\hat\rho,\hat\sigma}(T_{j}))<0$
by~\eqref{v de T menor que cero},
and $\ov\rho+\ov\sigma>0$. But then, again by~\eqref{alineado con cruz},
$$
\en_{\hat\rho,\hat\sigma}(T_{j+k})\nsim\en_{\hat\rho,\hat\sigma}(P)=r_i(a_i/l,b_i).
$$
Hence, by Propositions~\ref{pr v de un conmutador} and~\ref{extremosalineados}
$$
\ell_{\hat\rho,\hat\sigma}(T_{j+k+1})=[\ell_{\hat\rho,\hat\sigma}(T_{j+k}),\ell_{\hat\rho,\hat\sigma}(P)]\ne 0.
$$
Consequently, by Proposition~\ref{extremosnoalineados}(2) and~\eqref{ch4} for $k$,
\begin{align*}
\en_{\hat\rho,\hat\sigma}(T_{j+k+1})&=
\en_{\hat\rho,\hat\sigma}(T_{j+k})+\en_{\hat\rho,\hat\sigma}(P)-(1,1)\\
&= \en_{\hat\rho,\hat\sigma}(T_{j})+(k+1)\en_{\hat\rho,\hat\sigma}(P)-(k+1)(1,1),
\end{align*}
which ends the proof of the assertion. But
$T_{j+k}\ne0$ for all~$k$ is impossible, since from $[P,Q] \in K^{\times}$ and $T_0\in K[P,Q]$ it follows
easily that $T_n = 0$ for $n$ large enough. Therefore statement~(a) is true.

Now we are going to prove statement~(b). Assume by contradiction that
$$
[\ell_{\rho_{i+1},\sigma_{i+1}}(T_j),\ell_{\rho_{i+1},\sigma_{i+1}}(P)]\ne 0,
$$
which by Proposition~\ref{pr v de un conmutador} implies
\begin{equation}\label{rata2 nueva}
[\ell_{\rho_{i+1},\sigma_{i+1}}(T_j),\ell_{\rho_{i+1},\sigma_{i+1}}(P)] =
\ell_{\rho_{i+1},\sigma_{i+1}}([T_j,P]) = \ell_{\rho_{i+1},\sigma_{i+1}}(T_{j+1}).
\end{equation}
By~\eqref{donde esta rho sigma} we have $(\ov\rho,\ov\sigma)< (\rho_{i+1},\sigma_{i+1})<
(-\ov\rho,-\ov\sigma)$ and so, by Remark~\ref{starting vs end}
\begin{align*}
\st_{\rho_{i+1},\sigma_{i+1}}(P) &=
 \Supp(\ell_{\ov\rho,\ov\sigma}(\ell_{\rho_{i+1},\sigma_{i+1}}(P))), \\
\st_{\rho_{i+1},\sigma_{i+1}}(T_j) &=
 \Supp(\ell_{\ov\rho,\ov\sigma}(\ell_{\rho_{i+1},\sigma_{i+1}}(T_j))), \\
\st_{\rho_{i+1},\sigma_{i+1}}(T_{j+1}) &=
 \Supp(\ell_{\ov\rho,\ov\sigma}(\ell_{\rho_{i+1},\sigma_{i+1}}(T_{j+1}))).
\end{align*}
But then, by Proposition~\ref{pr v de un conmutador} and equivalence~\eqref{alineado con cruz},
\begin{align*}
 v_{\ov \rho,\ov\sigma} (\st_{\rho_{i+1},\sigma_{i+1}}(T_{j+1}))&=
  v_{\ov \rho,\ov\sigma} (\ell_{\rho_{i+1},\sigma_{i+1}}(T_{j+1}))\\
  &\le v_{\ov \rho,\ov\sigma} (\ell_{\rho_{i+1},\sigma_{i+1}}(T_{j}))
  +v_{\ov \rho,\ov\sigma} (\ell_{\rho_{i+1},\sigma_{i+1}}(P))
  -(\ov \rho+\ov\sigma)\\
  &=v_{\ov \rho,\ov\sigma} (\st_{\rho_{i+1},\sigma_{i+1}}(T_{j}))
  +v_{\ov \rho,\ov\sigma} (\st_{\rho_{i+1},\sigma_{i+1}}(P))
  -(\ov \rho+\ov\sigma)\\
  &=-(\ov \rho+\ov\sigma)<0,
\end{align*}
since by item~(a), Proposition~\ref{le basico} and~\eqref{condind},
$$
\st_{\rho_{i+1},\sigma_{i+1}}(T_{j}) = \en_{\rho_i,\sigma_i}(T_{j})\sim
(a_i/l,b_i)\sim \en_{\rho_i,\sigma_i}(P) =
\st_{\rho_{i+1},\sigma_{i+1}}(P).
$$
Hence,
by item~(a), Proposition~\ref{le basico} and~\eqref{alineado con cruz},
$$
\en_{\rho_{i},\sigma_{i}}(T_{j+1})=\st_{\rho_{i+1},\sigma_{i+1}}(T_{j+1})\nsim
(a_i/l,b_i)\sim \en_{\rho_{i},\sigma_{i}}(P),
$$
which contradicts~\eqref{condind}, thus proving~(b) and finishing the proof of the claim.

In order to prove~\eqref{cocientes}, we must check that
\begin{equation}\label{ch0}
[\ell_{\rho,\sigma}(T_j),\ell_{\rho,\sigma}(P)]=0\qquad\text{and}\qquad
\frac{v_{\rho,\sigma}(T_j)}{v_{\rho,\sigma}(P)}= \frac{v_{\rho_i,\sigma_i}(T_j)}{v_{\rho_i,\sigma_i}(P)}
\end{equation}
hold for all $(\rho,\sigma)$ with $(\rho_{i+1},\sigma_{i+1})\ge (\rho,\sigma)>(\rho_i,\sigma_i)$ and all $i$.
We proceed by induction, using the claim and~\eqref{condini}. More precisely, we are going to prove
 for any $i$, that~\eqref{condind} implies that~\eqref{ch0}
hold for all $(\rho,\sigma)$ with $(\rho_{i+1},\sigma_{i+1})\ge (\rho,\sigma)>(\rho_i,\sigma_i)$, and that
condition~\eqref{condind} is true for $i+1$.

 In fact, if
$(\rho_{i+1},\sigma_{i+1})> (\rho,\sigma)>(\rho_i,\sigma_i)$, then by Proposition~\ref{le basico},
\begin{equation}\label{ch1}
\en_{\rho_i,\sigma_i}(P) =\Supp(\ell_{\rho,\sigma}(P)) = \st_{\rho_{i+1},\sigma_{i+1}}(P),
\end{equation}
while, again by Proposition~\ref{le basico} and statement~(a), for the same $(\rho,\sigma)$
\begin{equation}\label{ch2}
\en_{\rho_i,\sigma_i}(T_j) =\Supp(\ell_{\rho,\sigma}(T_j)) = \st_{\rho_{i+1},\sigma_{i+1}}(T_j).
\end{equation}
Consequently, since $\en_{\rho_i,\sigma_i}(T_j)\sim \en_{\rho_i,\sigma_i}(P)$,
\begin{align*}
&[\ell_{\rho,\sigma}(T_j),\ell_{\rho,\sigma}(P)]=0 && \text{for all $(\rho,\sigma)$ with
$(\rho_{i+1},\sigma_{i+1})> (\rho,\sigma)>(\rho_i,\sigma_i)$}
\shortintertext{and}
&\frac{v_{\rho,\sigma}(T_j)}{v_{\rho,\sigma}(P)}= \frac{v_{\rho_i,\sigma_i}(T_j)}{v_{\rho_i,\sigma_i}(P)} &&
\text{for all $(\rho,\sigma)$ with $(\rho_{i+1},\sigma_{i+1})\ge (\rho,\sigma)>(\rho_i,\sigma_i)$.}
\end{align*}
Hence the equalities in~\eqref{ch0} hold for all required $(\rho,\sigma)$'s. Next we prove that
condition~\eqref{condind} is true for $i+1$.
We first prove that
\begin{equation}\label{condind2}
v_{1,-1}(\en_{\rho_{i+1},\sigma_{i+1}}(P))<0.
\end{equation}
If $\rho_{i+1}+\sigma_{i+1}\ge 0$, then by Proposition~\ref{le basico}, Remark~\ref{starting vs end} and the
inequality in~\eqref{condind},
$$
v_{1,-1}(\en_{\rho_{i+1},\sigma_{i+1}}(P)) \le v_{1,-1}(\st_{\rho_{i+1},\sigma_{i+1}}(P)) =
v_{1,-1}(\en_{\rho_i,\sigma_i}(P))< 0,
$$
as desired. Assume that $\rho_{i+1}+\sigma_{i+1}< 0$ and set $A:=\en_{\rho_{i+1},\sigma_{i+1}}(P)$.
First we are going to prove that
$v_{1,-1}(A)\ne 0$. Otherwise $A=k(1,1)$ for some $k\in\mathds{N}_0$, which is impossible, since then
$$
v_{\rho_{i+1},\sigma_{i+1}}(P)=v_{\rho_{i+1},\sigma_{i+1}}(A)=k(\rho_{i+1}+\sigma_{i+1})\le 0,
$$
contradicting the definition of $(\tilde\rho,\tilde\sigma)$.
Assume that
\begin{equation}\label{va de A mayor que cero}
v_{1,-1}(A)>0=v_{1,-1}(0,0).
\end{equation}
Since $\rho_{i+1}+\sigma_{i+1}<0$,  $(\rho_{i+1},\sigma_{i+1})\in I_0$ and $A\in\Supp(P)$, we have
$$
(-1,1)<(\rho_{i+1},\sigma_{i+1})<(0,-1)\quad\text{and}\quad v_{0,1}(A)\ge 0 = v_{0,1}(0,0).
$$
Thus, by Corollary~\ref{formula basica de orden'} (which we can apply because
$({-\rho_{i+1},-\sigma_{i+1}})<(0,1)$ in $\mathfrak{V}_{>0}$), we have
$$
v_{\rho_{i+1},\sigma_{i+1}}(P)=v_{\rho_{i+1},\sigma_{i+1}}(A)=
-v_{-\rho_{i+1},-\sigma_{i+1}}(A)< -v_{-\rho_{i+1},-\sigma_{i+1}}(0,0)=0,
$$
which contradicts again the definition of $(\tilde\rho,\tilde\sigma)$ and ends the proof of~\eqref{condind2}.

It remains to check that the second assertion in~\eqref{condind} holds for $i+1$.
By equalities~\eqref{ch1} and~\eqref{ch2},
$$
\st_{\rho_{i+1},\sigma_{i+1}}(T_j) = \en_{\rho_i,\sigma_i}(T_j) = \mu_j \en_{\rho_i,\sigma_i}(P) = \mu_j
\st_{\rho_{i+1},\sigma_{i+1}}(P),
$$
which implies $v_{\rho_{i+1},\sigma_{i+1}}(T_j) = \mu_jv_{\rho_{i+1},\sigma_{i+1}}(P) \ge 0$.
Therefore, by~(b), we can
apply Remark~\ref{a remark} in order to obtain that
$$
\en_{\rho_{i+1},\sigma_{i+1}}(T_j) = \mu_j \en_{\rho_{i+1},\sigma_{i+1}}(P),
$$
as desired. This proves~\eqref{cocientes} and concludes the proof.
\end{proof}

\begin{corollary}\label{fracciones de F} Let $m,n\in \mathds{N}$ be coprime with $m,n>1$ and let
$P,Q\in L^{(l)}$ with
$$
[P,Q]\in K^{\times}\quad\text{and}\quad \frac{v_{1,1}(P)}{v_{1,1}(Q)}=\frac{v_{1,0}(P)}{v_{1,0}(Q)} =\frac mn.
$$
Assume that $(\rho_0,\sigma_0)\in
\mathfrak{V}_{\ge0}$ satisfies
\begin{enumerate}

\smallskip

  \item $(\rho_0,\sigma_0)\in\Dir(P)$ and $v_{\rho_0,\sigma_0}(P)>0$,

\smallskip

  \item $\frac{1}{m}\en_{\rho_0,\sigma_0}(P) = \frac{1}{n}\en_{\rho_0,\sigma_0}(Q) \in \frac 1l
\mathds{Z}\times\mathds{N}$,

\smallskip

  \item $b>a/l$, where $(a/l,b):=\frac{1}{m}\en_{\rho_0,\sigma_0}(P)$.

\smallskip

\end{enumerate}
Let $(\tilde\rho,\tilde\sigma)$ be as in
Proposition~\ref{proporcionalidad de direcciones mayores} and let $F\in L^{(l)}$ be the
$(\rho_0,\sigma_0)$-ho\-mo\-geneous element obtained in Theorem~\ref{central}.
If there exist $p,q\in \mathds{N}$ coprime such that $\en_{\rho_0,\sigma_0}(F) =
\frac pq (a/l,b)$,
then for all $(\rho,\sigma)\in \mathfrak{V}$ with $(\rho_0,\sigma_0)<(\rho,\sigma)\le(\tilde\rho,\tilde\sigma)$
 there exists  a
    $(\rho,\sigma)$-ho\-mo\-geneous element $R\in L^{(l)}$ such that $\ell_{\rho,\sigma}(P)= R^{qm}$.
\end{corollary}

\begin{proof} Let $G_0$ and $G_1$ be as in Theorem~\ref{central}. Since $G_0,P\ne 0$, by the last equality
in~\eqref{eq central} we have
$$
[\ell_{\rho_0,\sigma_0}(G_0),\ell_{\rho_0,\sigma_0}(P)]\ne 0,
$$
which, by Proposition~\ref{pr v de un conmutador}, implies that
$$
\ell_{\rho_0,\sigma_0}(G_1)=[\ell_{\rho_0,\sigma_0}(G_0),\ell_{\rho_0,\sigma_0}(P)]\ne 0.
$$
By Proposition~\ref{extremosalineados} and Theorem~\ref{central}(5) there exists $g_1\in \mathds{Q}$ such that
\begin{equation}\label{ck2}
\en_{\rho_0,\sigma_0}(G_1) = g_1\en_{\rho_0,\sigma_0}(P).
\end{equation}
Moreover,
\begin{equation}\label{ck3}
\en_{\rho_0,\sigma_0}(F) = \frac pq (a/l,b) = \frac p{qm}\en_{\rho_0,\sigma_0}(P).
\end{equation}
On the other hand, using again the last equality in~\eqref{eq central}, we obtain
$$
\ell_{\rho_0,\sigma_0}(G_1)F = \ell_{\rho_0,\sigma_0}(G_0) \ell_{\rho_0,\sigma_0}(P),
$$
and hence
$$
\en_{\rho_0,\sigma_0}(G_1)+\en_{\rho_0,\sigma_0}(F) = \en_{\rho_0,\sigma_0}(G_0) + \en_{\rho_0,\sigma_0}(P).
$$
Consequently, by~\eqref{ck2} and~\eqref{ck3},
$$
\en_{\rho_0,\sigma_0}(G_0) = \en_{\rho_0,\sigma_0}(G_1)+\en_{\rho_0,\sigma_0}(F)- \en_{\rho_0,\sigma_0}(P) =
\Bigl(g_1+\frac p{qm}-1\Bigr)\en_{\rho_0,\sigma_0}(P).
$$
Set $g_0:=g_1+\frac p{qm}-1$ and take $r\in \mathds{Z}$ and $s\in\mathds{N}$ coprime, such that $g_0=r/s$. Note
that by~\eqref{ck2}, \eqref{ck3} and the fact that $g_1=\frac rs+1-\frac p{qm}$, we have
\begin{equation}\label{valuaciones de G}
\frac{1}{v_{\rho_0,\sigma_0}(P)}\bigl(v_{\rho_0,\sigma_0}(G_0), v_{\rho_0,\sigma_0}(G_1),
v_{\rho_0,\sigma_0}(P),v_{\rho_0,\sigma_0}(Q) \bigr)= \left(\frac rs,\frac rs+1-\frac {p}{qm},1,
\frac nm\right).
\end{equation}
Let $(\rho,\sigma)>(\rho_0,\sigma_0)$. Applying
Proposition~\ref{proporcionalidad de direcciones mayores} with $T_0:=G_0$, with $T_0:=G_1$ and with
$T_0:=Q$, we obtain that
$$
[\ell_{\rho,\sigma}(G_0),\ell_{\rho,\sigma}(P)]=0,\quad
[\ell_{\rho,\sigma}(G_1),\ell_{\rho,\sigma}(P)]=0\quad\text{and}\quad
[\ell_{\rho,\sigma}(Q),\ell_{\rho,\sigma}(P)]=0.
$$
Hence, by Proposition~\ref{P y Q alineados}(2b), there exist $\gamma_0,\gamma_1,\gamma_2,\gamma_3\in
K^{\times}$, a $(\rho,\sigma)$-homogeneous element $R_0\in L$ and $u_0,u_1,u_2,u_3 \in \mathds{N}$, such that
$$
\ell_{\rho,\sigma}(G_0) = \gamma_0 R_0^{u_0},\quad \ell_{\rho,\sigma}(G_1) = \gamma_1 R_0^{u_1},\quad
\ell_{\rho,\sigma}(P) = \gamma_2 R_0^{u_2}\quad \text{and} \quad \ell_{\rho,\sigma}(Q) =
\gamma_3 R_0^{u_3},\endnote{
Write $\ell_{\rho,\sigma}(P)=P_1^{m_1}\cdot\dots\cdot P_r^{m_r}$ where the $P_j$'s
are irreducible. Let $d:=\gcd(m_1,\dots,m_r)$. Then we can take
$R_0:=P_1^{m_1/d}\cdot\dots\cdot P_r^{m_r/d}$.}
$$
and clearly we can assume that $\gcd(u_0,u_1,u_2,u_3)\!=\!1$\endnote{
Replace $R_0$ by $R_0^{\gcd(u_0,u_1,u_2,u_3)}$.}.
But then, by
Proposition~\ref{proporcionalidad de direcciones mayores} and e\-qua\-li\-ty~\eqref{valuaciones de G},
\begin{align*}
v_{\rho,\sigma}(R_0)(u_0,u_1,u_2,u_3) & = \bigl(v_{\rho,\sigma}(G_0), v_{\rho,\sigma}(G_1),
v_{\rho,\sigma}(P),v_{\rho,\sigma}(Q) \bigr)\\
& = \frac{v_{\rho,\sigma}(P)}{v_{\rho_0,\sigma_0}(P)}\bigl(v_{\rho_0,\sigma_0}(G_0), v_{\rho_0,\sigma_0}(G_1),
v_{\rho_0,\sigma_0}(P),v_{\rho_0,\sigma_0}(Q) \bigr)\\
& = \frac{v_{\rho,\sigma}(P)}{sqm}(rqm,rqm+sqm-ps,sqm,sqn),
\end{align*}
and so we have $u_2 = \frac{sqm}{d}$, where $d:=\gcd(rqm,rqm+sqm-ps,sqm,sqn)$. Since
$$
d=\gcd(rqm,ps,sqm,sqn)=\gcd(qm,s),
$$
we obtain that $d|s$. Consequently we can write $\ell_{\rho,\sigma}(P) = \gamma_2 R_0^{sqm/d}= \gamma_2
(R_0^{s/d})^{qm}$. We conclude the proof setting  $R:=\gamma R_0^{s/d}$, where we choose $\gamma\in K^{\times}$ such that
$\gamma^{qm}=\gamma_2$. \end{proof}

\begin{proposition}\label{proporcionalidad de direcciones menores} Let $m,n\in \mathds{N}$ be coprime with
$m,n>1$ and let $P,Q\in L^{(l)}$ with
$$
[P,Q]\in K^{\times}\quad\text{and}\quad \frac{v_{1,1}(P)}{v_{1,1}(Q)}=\frac{v_{0,1}(P)}{v_{0,1}(Q)} =\frac mn.
$$
Take $T_0\in K[P,Q]$ and set $T_j:=[T_{j-1},P]$ for $j\ge 1$. Assume that
$(\rho_0,\sigma_0)\in\mathfrak{V}_{\ge 0}$ satisfies
\begin{enumerate}

\smallskip

  \item $(\rho_0,\sigma_0)\in\Dir(P)$ and $v_{\rho_0,\sigma_0}(P)>0$,

\smallskip

  \item $\st_{\rho_0,\sigma_0}(T_j)\sim \st_{\rho_0,\sigma_0}(P)$ for all $j$ with $T_j\ne 0$,

\smallskip

  \item $\frac{1}{m}\st_{\rho_0,\sigma_0}(P) = \frac{1}{n}\st_{\rho_0,\sigma_0}(Q) \in \frac 1l
\mathds{Z}\times\mathds{N}$,

\smallskip

  \item $b<a/l$, where $(a/l,b):=\frac{1}{m}\st_{\rho_0,\sigma_0}(P)$.

\smallskip

\end{enumerate}
Let $I_1:=[(0,-1),(\rho_0,\sigma_0)]$ and
$$
(\tilde\rho,\tilde\sigma):=
\min\{(\rho,\sigma)\in \Dir(P)\cap I_1 : v_{\rho',\sigma'}(P)>0
\text{ for all } (\rho_0,\sigma_0)\ge(\rho',\sigma')\ge(\rho,\sigma)\}.
$$
Then for all $(\rho,\sigma)\in\mathfrak{V}$ with $(\tilde\rho,\tilde\sigma)
\le (\rho,\sigma)<(\rho_0,\sigma_0)$ and all $j\ge 0$ we have
$$
[\ell_{\rho,\sigma}(T_j),\ell_{\rho,\sigma}(P)]=0\quad\text{and}\quad
\frac{v_{\rho,\sigma}(T_j)}{v_{\rho,\sigma}(P)}= \frac{v_{\rho_0,\sigma_0}(T_j)}{v_{\rho_0,\sigma_0}(P)}.
$$
\end{proposition}

\begin{proof} Mimic the proof of Proposition~\ref{proporcionalidad de direcciones mayores}.
\end{proof}

\begin{corollary}\label{fracciones de F1} Let $m,n\in \mathds{N}$ be coprime with $m,n>1$ and let
$P,Q\in L^{(l)}$ with
$$
[P,Q]\in K^{\times}\quad\text{and}\quad \frac{v_{1,1}(P)}{v_{1,1}(Q)}=\frac{v_{0,1}(P)}{v_{0,1}(Q)} =\frac mn.
$$
Assume that $(\rho_0,\sigma_0)\in
\mathfrak{V}_{\ge0}$ satisfies
\begin{enumerate}

\smallskip

  \item $(\rho_0,\sigma_0)\in\Dir(P)$ and $v_{\rho_0,\sigma_0}(P)>0$,

\smallskip

  \item $\frac{1}{m}\st_{\rho_0,\sigma_0}(P) = \frac{1}{n}\st_{\rho_0,\sigma_0}(Q) \in \frac 1l
\mathds{Z}\times\mathds{N}$,

\smallskip

\item $b<a/l$, where $(a/l,b):=\frac{1}{m}\st_{\rho_0,\sigma_0}(P)$.

\smallskip

\end{enumerate}
Let $(\tilde\rho,\tilde\sigma)$ be as in
Proposition~\ref{proporcionalidad de direcciones menores} and let $F\in L^{(l)}$ be the
$(\rho_0,\sigma_0)$-ho\-mo\-geneous element obtained in Theorem~\ref{central}.
If there exist $p,q\in \mathds{N}$ coprime, such that $\st_{\rho_0,\sigma_0}(F) = \frac pq
(a/l,b)$, then for all
$(\rho,\sigma)\in \mathfrak{V}$ with $(\tilde\rho,\tilde\sigma)\le(\rho,\sigma)<(\rho_0,\sigma_0)$ there exists
a $(\rho,\sigma)$-ho\-mo\-geneous element $R\in L^{(l)}$ such that $\ell_{\rho,\sigma}(P)= R^{qm}$.
\end{corollary}

\begin{proof} Mimic the proof of Corollary~\ref{fracciones de F}.
\end{proof}

\begin{remark}\label{direcciones intermedias positivas}
Let $P\in L^{(l)}\setminus\{0\}$ and let $(\rho',\sigma')$ and
$(\rho'',\sigma'')$ be consecutive elements in $\Dir(P)$.
It follows from  Remarks~\ref{180 grados} and~\ref{intervalo de direcciones} that if
$v_{\rho',\sigma'}(P),v_{\rho'',\sigma''}(P)>0$, then
$v_{\rho,\sigma}(P)>0$ for all
$(\rho',\sigma')<(\rho,\sigma)<(\rho'',\sigma'')$.
\end{remark}

The following theorem is related to~\cite{H}*{Proposition 1.10} and also to~\cite{Z}*{Remark 5.12}.
In this theorem and in Proposition~\ref{factores}, we consider the order in $I=\,\, ](1,-1),(1,0)]$.

\begin{theorem}\label{divisibilidad}
Let $(A_0,(\rho_0,\sigma_0)),(A_1,(\rho_1,\sigma_1)),\dots,(A_k,(\rho_k,\sigma_k))$ be the regular corners
of an $(m,n)$-pair $(P,Q)$ in $L^{(l)}$, where $(\rho_i,\sigma_i)<(\rho_{i+1},\sigma_{i+1})$ for all $i<k$.
The following facts hold:
\begin{enumerate}

\smallskip

\item $A(P)=\{(\rho_1,\sigma_1),\dots,(\rho_k,\sigma_k)\}$.  In particular, if $(P,Q)$ is a standard $(m,n)$-pair, then
$(A_0,A'_0,(\rho_0,\sigma_0))$ is the starting triple of $(P,Q)$, where $A'_0:=\frac{1}{m}
\st_{\rho_0,\sigma_0}(P)$.

\smallskip

\item For all $j\ge 1$ there exists $d_j\in\mathds{N}$ maximum such that
 $\ell_{\rho_j,\sigma_j}(P)=R_j^{md_j}$ for some $(\rho_j,\sigma_j)$-
ho\-mo\-geneous $R_j\in L^{(l)}$. If $A_0$ is of type~II, then this holds also for $j=0$.

\smallskip

\item
For all $j>0$ the element $F_j$ constructed via Theorem~\ref{central} satisfies
$$
\en_{\rho_j,\sigma_j}(F_j)=\frac{p_j}{q_j}\frac 1m \en_{\rho_j,\sigma_j}(P),
$$
where $p_j$ and $q_j$ are coprime. If $A_0$ is of type~II, then this holds also for $j=0$.

\smallskip

    \item $q_i\nmid d_i$ for all $i>0$.

\smallskip

    \item $q_j \mid d_i$ for all $i>j>0$.

\smallskip

    \item $q_i\nmid q_j$ for all $i>j>0$.
\end{enumerate}
\smallskip

\noindent Set $D_j:=\gcd(a_j,b_j,a_{j-1},b_{j-1})$, where $A_j=(a_j/l,b_j)$ and
    $A_{j-1}=(a_{j-1}/l,b_{j-1})$. Then
\smallskip
\begin{enumerate}[resume]
    \item $d_j\mid D_j$ and $\Omega(D_j)\ge \Omega(d_j)\ge j-1$ for all $j>0$,
    where for $n\in\mathds{N}$ we let $\Omega(n)$ denote the number of prime factors of $n$,
     counted with multiplicity.

\smallskip

    \item If $A_0$ is of type~II, then $q_0\nmid d_0$ and for all $i>0$, we have
    $$
   q_0\mid d_i,\quad q_i\nmid q_0,\quad \text{and}\quad \Omega(d_i)\ge i.
    $$
  \end{enumerate}
\end{theorem}

\begin{proof}
By Remark~\ref{estamos en Case IIa} and Propositions~\ref{A(P) vs Dir(P)} and~\ref{esquina regular unica} statement~(1) is true.
By Corollary~\ref{some properties of corners}(1) we know that $v_{\rho_j,\sigma_j}(P)>0$ for all $j$. If $A_0$ is of type~II,
then $[\ell_{\rho_0,\sigma_0}(P),\ell_{\rho_0,\sigma_0}(Q)]=0$.
In the general case, when $j\ge 1$, by
Remark~\ref{estamos en Case IIa}, we are in Case II.a), and so
$[\ell_{\rho_j,\sigma_j}(P),\ell_{\rho_j,\sigma_j}(Q)]=0$. Hence, by
Proposition~\ref{P y Q alineados}(2b), statement~(2) holds.
Statement~(3) follows from Remark~\ref{estamos en Case IIa} and Proposition~\ref{case II}(1).

In order to prove statement~(4), assume by contradiction that $q_j\mid d_j$. Then
$\widetilde R:=R_j^{p_j d_j/q_j}$ satisfies
\begin{equation}\label{existencia de R}
[\widetilde R,\ell_{\rho_j,\sigma_j}(P)]=0\quad\text{and}\quad
v_{\rho_j,\sigma_j}(\widetilde R)=v_{\rho_j,\sigma_j}(F_j)=\rho_j+\sigma_j,
\end{equation}
where the second equality follows from the fact that
$$
\en_{\rho_j,\sigma_j}(F_j)=\frac{p_j}{q_j}\frac 1m \en_{\rho_j,\sigma_j}(P)=\en_{\rho_j,\sigma_j}(\widetilde R).
$$
But the existence of $\widetilde R$ satisfying~\eqref{existencia de R} contradicts Proposition~\ref{pavadass}(5)
(The condition $s>0$ or $\# \factors(p)>1$ required in
Proposition~\ref{pavadass}(5) is satisfied if and only if $\# \factors(\mathfrak{p}(z))>1$,
which holds because we are in case II).

By Corollary~\ref{some properties of corners}(1) we have $v_{\rho_j,\sigma_j}(P)>0$ for all $j\ge 0$, and
hence,
by Remark~\ref{direcciones intermedias positivas}, we have
$v_{\rho,\sigma}(P)>0$ if $(\rho,\sigma)$ lies between $(\rho_0,\sigma_0)$ and
$(\rho_{k},\sigma_{k})$. Let $(\tilde\rho,\tilde\sigma)$ be as in
Proposition~\ref{proporcionalidad de direcciones mayores}. By its very
definition $(\tilde\rho,\tilde\sigma)\ge
(\rho_i,\sigma_i)>(\rho_j,\sigma_j)$.
Thus the hypotheses of Corollary~\ref{fracciones de F} are satisfied with
$(\rho_0,\sigma_0)=(\rho_j,\sigma_j)$ and $(\rho,\sigma)=(\rho_i,\sigma_i)$, and hence we have
$$
R_i^{md_i}=\ell_{\rho_i\sigma_i}(P)=R^{mq_j}\quad\text{for some $R\in L^{(l)}$,}
$$
which gives statement~(5) by the maximality of $d_i$.

Statement~(6) follows from~(4) and~(5). In order to prove statement~(7), note that $d_j|D_j$ since
$$
A_j=d_j\en_{\rho_j\sigma_j}(R_j)\quad\text{and}\quad A_{j-1}=d_j\st_{\rho_j\sigma_j}(R_j),
$$
and a straightforward computation using~(4), (5) and~(6) proves the last assertion of~(7). The proof of statement~(8) follows along the lines of the
proofs of~(4), (5), (6) and~(7).
\end{proof}

\begin{remark}\label{multiplicidad de la potencia} Let $f,\ov f\in K[x]$ be polynomials. If $f(x)=\ov f(x^n)$, then $\lambda$ is a root of $f$ if
and only if $\lambda^n$ is a root of $\ov f$. Moreover, if $\lambda\ne 0$, then the multiplicity $m_\lambda$ of $\lambda$ in $f$ is the same as the
multiplicity $\ov m_{\lambda^n}$ of $\lambda^n$ in $\ov f$ and consequently, if $f(0)\ne 0$ and $f$ is a $d$th power, then $\ov f$ is also a $d$th
power.
\end{remark}

\begin{proposition}\label{factores}
With the notations of Theorem~\ref{divisibilidad}, we have:
\begin{enumerate}

\smallskip

  \item If $\Omega(d_{j_0})= j_0-1$ for some $j_0>0$, then $l-\frac{a_{j_0}}{b_{j_0}}>1$.

\smallskip

  \item Assume that $(A_0,(\rho_0,\sigma_0))$ is of type~II.b) and that
$\ell_{\rho_0,\sigma_0} (P)$ is at
most an $m$th power in $L^{(l')}$, where $l'\!:=\!\lcm(\rho_0,l)$. Write
$A^{(1)}\!=\!(a/l',b)$, where $A^{(1)}$ is constructed via
Pro\-positions~\ref{case IIb} and~\ref{encima de la diagonal}. Then
$$
l'-\frac ab>1.
$$

\smallskip

  \item Let $(P,Q)$ be a standard $(m,n)$-pair. Assume that $A_0=(1,0)+r(1,\rho_0)$ for
   some $r\ge 1$, $\sigma_0=-1$ and $A_0'=(1,0)$.
    If $\gcd(r,\rho_0)=1$ or $\rho_0$ is a prime number, then
  $\ell_{\rho_0,\sigma_0} (P)$ is at most an $m$th power in $L^{(\rho_0)}$. Moreover,
  in both cases $\gamma(\rho_0-2)>\rho_0$, where $\gamma=\frac{m_\lambda}{m}$
      is as before Proposition~\ref{final}.

\smallskip

  \item If $(P,Q)$ is a standard $(m,n)$-pair, then $\gcd(a_0,b_0)\ne 2$ and
$q_0\ne 2$.
\end{enumerate}
\end{proposition}

\begin{proof}
\noindent(1)\enspace Note that $(A_0,(\rho_0,\sigma_0))$ cannot be of type~II by Theorem~\ref{divisibilidad}(8).
If $(A_0,(\rho_0,\sigma_0))$ is of type~I,
then Proposition~\ref{extremosfinales} yields $l-\frac{a_0}{b_0}>1$. If $(A_0,(\rho_0,\sigma_0))$
is of type~III, then
we take $\varphi$ as in Proposition~\ref{case III}. By statement~(3) of that proposition, we know that
$((a_0/l,b_0),(\rho',\sigma'))$ is a regular corner of $(\varphi(P),\varphi(Q))$, where
$(\rho',\sigma'):=\Pred_{\varphi(P)}(\rho_0,\sigma_0)$, while by statement~(1), we have
\begin{equation}\label{phi no cambia A}
\ell_{\rho_j,\sigma_j}(P)=\ell_{\rho_j,\sigma_j}(\varphi(P))\quad\text{and}\quad
\ell_{\rho_j,\sigma_j}(Q)=\ell_{\rho_j,\sigma_j}(\varphi(Q)),\quad\text{for $j>0$}.
\end{equation}
Hence $(A_0,(\rho',\sigma')),(A_1,(\rho_1,\sigma_1)),\dots,(A_k,(\rho_k,\sigma_k))$ are regular corners of
$(\varphi(P),\varphi(Q))$. We claim that there are no other. In fact, given a new regular corner
$(A'',(\rho'',\sigma''))$, it is
impossible that $(\rho'',\sigma'')>(\rho',\sigma')$, since otherwise $(1,0)\ge(\rho'',\sigma'')>(\rho_0,\sigma_0)$,
by Proposition~\ref{case III}\endnote{
It is impossible that $(\rho'',\sigma'')=(\rho_0,\sigma_0)$ since by statements~(1) and~(4) of Proposition~\ref{case III},
$$
\en_{(\rho_0,\sigma_0)}(\varphi(P)) = \en_{(\rho_0,\sigma_0)}(P) = m(a/l,b) = \st_{(\rho_0,\sigma_0)}(\varphi(P)).
$$
}
and
$(\rho'',\sigma'')\in A(\varphi(P))$, by
Theorem~\ref{divisibilidad}(1), which contradicts
the fact that, by~\eqref{phi no cambia A}, Remark~\ref{mayores estan en A(P)} and Proposition~\ref{le bbasico}, we have
$$
A(P)\cap\ ](\rho_0,\sigma_0),(1,0)]=A(\varphi(P))\cap \ ](\rho_0,\sigma_0),(1,0)].
$$
On the other hand, if $(\rho'',\sigma'')<(\rho',\sigma')$, then there exists $r\ge 1$, such that
the regular corners of $(\varphi(P),\varphi(Q))$ form a set
$$
\{(\tilde A_0,(\tilde \rho_0,\tilde\sigma_0)),\dots,(\tilde A_{k+r},(\tilde\rho_{k+r},\tilde\sigma_{k+r}))\},
$$
where the first $r$ corners are new,
$$
(\tilde
A_{r},(\tilde\rho_{r},\tilde\sigma_{r}))=(A_0,(\rho',\sigma'))\quad\text{and}\quad
(\tilde A_{j+r},(\tilde\rho_{j+r},\tilde\sigma_{j+r}))=
(A_j,(\rho_j,\sigma_j))\quad\!\text{for $j=1,\dots,k$.}
$$
For each $0<j\le k+r$, let $\tilde d_j$ be as $d_j$, but for the
$(m,n)$-pair $(\varphi(P),\varphi(Q))$. Then, by~\eqref{phi no cambia A}, we have
$d_{j}=\tilde d_{j+r}$ for all $j$, and the fact that $\tilde d_{j_0+r}=j_0-1$
contradicts Theorem~\ref{divisibilidad}(7). This proves the claim.

Once again, $(A_0,(\rho',\sigma'))$ cannot be of type~II by Theorem~\ref{divisibilidad}(8), and
if $(A_0,(\rho',\sigma'))$ is of type~I,
then Proposition~\ref{extremosfinales} yields $l-\frac{a_0}{b_0}>1$. If it is of type~III,
then we repeat the same proceeding as above. The same argument as in the proof of
Proposition~\ref{criterion} shows that case~III can occur only finitely many times,
finishing the proof of~(1).

\smallskip

\noindent (2)\enspace Using Proposition~\ref{encima de la diagonal} it is easy to check that
$(A_0,(\rho_0,\sigma_0))$ and $(A^{(1)},(\rho',\sigma'))$ are regular corners of the $(m,n)$-pair
$(\varphi(P),\varphi(Q))$,
where $(\rho',\sigma'):=\Pred_{\varphi(P)}(\rho_0,\sigma_0)$. But then there exists $r\ge 1$ such that the regular corners of
$(\varphi(P),\varphi(Q))$ are
$$
(\tilde A_0,(\tilde \rho_0,\tilde \sigma_0)),\dots,(\tilde
A_{k+r},(\tilde\rho_{k+r},\tilde\sigma_{k+r})),
$$
where for $j\ge 0$ we have $\tilde A_{j+r}= A_j$. Let
$\tilde d_r$ be maximum such that
 $\ell_{\tilde \rho_r,\tilde \sigma_r}(\varphi(P))=R^{m\tilde d_r}$ for some $R\in L^{(l')}$.
Since $\ell_{\rho_0,\sigma_0} (\varphi(P))$ is at
 most an $m$th power in $L^{(l')}$ and $(\rho_0,\sigma_0)=(\tilde \rho_r,\tilde \sigma_r)$, we obtain
$\tilde d_r=1$. Hence,
by Theorem~\ref{divisibilidad}(7), we have $r=1$, and so
$(\tilde A_0,(\tilde \rho_0,\tilde \sigma_0))=(A^{(1)},(\rho',\sigma'))$. By statement~(1) we have
$l'-a/b>1$, as desired.

\smallskip

\noindent (3)\enspace By Corollary~\ref{some properties of corners}(1) and
 Remark~\ref{a remark}, we have $\ell_{\rho_0,\sigma_0} (P)=\mu R^m$ for some $\mu\in K^{\times}$ and $R\in L^{(1)}$.
Furthermore, since
$$
\st_{\rho_0,\sigma_0}(R)=A'_0=(1,0),\quad \en_{\rho_0,\sigma_0}(R)=A_0= (1,0)+ r(1,\rho_0),\quad\sigma_0=-1\quad\text{and}\quad R\in L^{(1)},
$$
it follows from Remark~\ref{polinomio asociado f^{(l)}}, that there exist $f,\ov f \in K[x]$ with $f(0)\ne 0$ and $\deg( f)=r\rho_0$, such that
$$
R  =x f(z)= x \ov{f}(z^{\rho_0}),
$$
where $z:=x^{1/\rho_0}y$.
We must prove that if $R=R_1^d$, for some $R_1\in L^{(\rho_0)}$, then $d=1$. But, under this assumption, $\st_{\rho_0,\sigma_0}(R_1)=(1/d,0)$,
hence $d\mid\rho_0$. Using
Remark~\ref{multiplicidad de la potencia} we obtain
$d\mid r$, and so, if $\gcd(r,\rho_0)=1$, then $d=1$, as desired.
On the other hand, since $\st_{\rho_0,\sigma_0}(R)=(1,0)$,
if $R=\mu_1 R_1^{\rho_0}$, then
$$
v_{\rho_0,\sigma_0}(R_1^{\rho_0+\sigma_0}) =\frac{\rho_0+\sigma_0}{\rho_0}v_{\rho_0,\sigma_0}(R)
=\rho_0+\sigma_0\qquad\text{and}\qquad [R_1^{\rho_0+\sigma_0},\ell_{\rho_0,\sigma_0}(P)]=0,
$$
which contradicts Proposition~\ref{pavadass}(5). Consequently $d\ne \rho_0$, and  so, for $\rho_0$ prime we have $d=1$.
Hence, if $\gcd(r,\rho_0)=1$ or $\rho_0$ is prime, then statement~(2) applies\endnote{
By Theorem~\ref{central}(1) and Remark~\ref{starting triple es IIb} we know that
$(A_0,(\rho_0,\sigma_0))$ is of type~IIb).
}
 and yields
$l'-a/b>1$, where $l'=\rho_0$. By Propositions~\ref{case IIb} and~\ref{encima de la diagonal}(3)
$$
\Bigl(\frac{a}{l'},b\Bigr)=A^{(1)}=(1,0)+\gamma(1/\rho_0,1)=\Bigl(\frac{\rho_0+\gamma}{\rho_0},\gamma\Bigl),
$$
from which we deduce
$\rho_0-\frac{\rho_0+\gamma}{\gamma}=l'-\frac ab >1$ and then $\gamma(\rho_0-2)>\rho_0$, as desired.

\smallskip

\noindent (4)\enspace
Note that $(\rho,\sigma)=(\rho_0,\sigma_0)$, $A_0'=(r',s')$ and $F=F_0$ are as before Proposition~\ref{final}. By
Proposition~\ref{primitivo}(1) we know that $(1,-1)<(\rho,\sigma)<(1,0)$, and by Proposition~\ref{final}(1)
we have
\begin{equation}\label{dat1}
 v_{1,-1}(A_0') = r'-s'>0 = v_{1,-1}(2,2),
\end{equation}
which implies $A'_0\ne (2,2)$.
Assume by contradiction that $q_0=2$ or $\gcd(a_0,b_0)=2$\endnote{
$\gcd(a_0,b_0) = 2$ implies $q_0=2$. In fact, since $F\in L$, by
Theorem~\ref{divisibilidad}(2) we have $q_0\mid \gcd(a_0,b_0) =2$. But
$q_0 = 1$ is impossible by Proposition~\ref{final}(4).
}. Let $\mu$ be as in
Proposition~\ref{final}(4). Since $0<\mu<1$ and $\mu=p_0/q_0$, this assumption implies $\mu=1/2$.
We have
$$
v_{\rho,\sigma}(A'_0) = v_{\rho,\sigma}(A_0) = 2v_{\rho,\sigma}(F)= 2v_{\rho,\sigma}(1,1)=v_{\rho,\sigma}(2,2).
$$
Then, by Remark~\ref{valuacion depende de extremos}
$$
\dir (A_0'-(2,2))=(\rho,\sigma)<(1,0).
$$
From this, the second inequality in~\eqref{dat1} and Lemma~\ref{formula basica de orden}, it follows that
$$
0\le s'<r' = v_{1,0}(A_0')<v_{1,0}(2,2) = 2.
$$
Hence necessarily $A_0'=(1,0)$. Therefore $(\rho,\sigma)=(2,-1)$, and we are in the situation of statement~(3)
with $\rho_0=\rho=2$. The contradiction $0=\gamma(\rho_0-2)>\rho_0$ concludes the proof.
\end{proof}

In the proof the next corollary, nearly all facts were more or less known,
except statement~(4) of Proposition~\ref{factores}, which is the missing piece of the puzzle. This statement relies on the impossibility
of $A_0'=(1,0)$ and $(\rho,\sigma)=(2,-1)$, which we prove via Proposition~\ref{factores}(3)\endnote{
Specializing the proof of Proposition~\ref{factores}(3), the direct proof of this fact runs as follows:

\noindent Assume $A_0'=(1,0)$ and $(\rho,\sigma)=(2,-1)$. Write $\ell_{2,-1}(P)=R^{dm}$ with $R\in L^{(2)}$ and $d$ maximal.

\noindent If $d>1$, then $d=2$,
since $A_0'=(1,0)$. But then $v_{2,-1}(R)=v_{2,-1}(F)$ and $[R,\ell_{2,-1}(P)]=0$, which contradicts Proposition~\ref{pavadass}(5).
Hence $d=1$.

\noindent Take $\varphi$ as in Proposition~\ref{case IIb} and take $\left(\frac a2,b\right):=\frac 1m \st_{2,-1}(\varphi(P))$. By Proposition~\ref{divisibilidad}(8)
it cannot be a corner of type II, hence, by Proposition~\ref{criterion}, we should have $2-\frac ab>1$. But this is not possible, since
the equality $\left(\frac a2,b\right)= (1,0)+k\left(\frac 12,1\right)=\left(\frac{k+2}{2},k\right)$ implies $2-\frac ab=2-\frac{k+2}{k}=1-\frac 2k<1$.
}, but which can also be proven
using Proposition~\ref{impossibles}, since in that case $v_{2,-1}(A_0)=v_{2,-1}(A_0')=2$, which is impossible by Proposition~\ref{impossibles}.

\begin{corollary} Let $(P,Q)$ be a standard $(m,n)$-pair in $L$. Write $(a,b):=\frac 1m \en_{1,0}(P)$.
Then $(a,b)\in\mathds{N}\times \mathds{N}$ and $\gcd(a,b)>2$. Furthermore $B\ne p$ and $B\ne 2p$ for any
prime $p$, where $B$ is as at the beginning of Section~\ref{section4}.
\end{corollary}

\begin{proof}
By Remark~\ref{first component} we know that $(a,b)$ is the first component of a regular corner of $(P,Q)$. Hence,
by Remark~\ref{a>0} we have $(a,b)\in\mathds{N}\times \mathds{N}$ and by Proposition~\ref{criterion}
we know $\gcd(a,b)>1$. Next we discard $\gcd(a,b)=2$. Let $k$ be the number of regular corners in $A(P)$.
If $k=0$, then $\gcd(a,b)=2$ contradicts Proposition~\ref{factores}(4).
Assume $k>0$. Then $(a,b)=(a_k,b_k)$ and the very definitions of $q_k$ and $d_k$ show that
$q_k|\gcd(a,b)$ and $d_k|\gcd(a,b)$. Moreover, by Theorem~\ref{divisibilidad}(8) we have
$q_0| d_k$, $q_0\nmid d_0$ and $q_k\nmid q_0$. Hence $\gcd(a,b)$ is a composite number
and so $\gcd(a,b)>2$.

Now assume that $(P,Q)$ is as in Corollary~\ref{B finito}. In particular,
$$
B=\gcd(v_{1,1}(P),v_{1,1}(Q))=\frac{1}{m}v_{1,1}(P)\quad\text{and}\quad
(a,b) = \frac 1m \en_{1,0}(P) = \frac{1}{m} \st_{1,1}(P),
$$
and so $a + b = B$, which implies $\gcd(a,b)\mid B$. Now, if $B=p$ or $B=2p$ for some prime $p$, then
$\gcd(a,b)\in\{1,2,p,2p\}$. Since $\gcd(a,b)>2$ and $\gcd(a,b)=2p$ is impossible, we have to discard only the
case $\gcd(a,b)=p$. But in that case $a=b=p$, which contradicts $a<b$ and finishes the proof.
\end{proof}

In the following proposition we give a condition under which the Newton polygon of $P$
has no vertical edge at the right hand side.

\begin{corollary}\label{no hay lados verticales} Let $(P,Q)$ be a standard $(m,n)$-pair and let
$(\rho,\sigma)$, $A_0$ and $F$ be as in the discussion above Proposition~\ref{final}. By
Proposition~\ref{final}(4) there exist $p,q\in \mathds{N}$ coprime, such that
$\en_{\rho,\sigma}(F)=\frac pq A_0$. Assume that
$A_0=\frac 1m \st_{1,0}(P)$.
If $A_0=(q,b)$, then $\st_{1,0}(P)=\en_{1,0}(P)$.
\end{corollary}

\begin{proof} Along the proof we use the notations of Theorem~\ref{divisibilidad}.
Assume that $\st_{1,0}(P)\ne\en_{1,0}(P)$.
Note that $(\rho_0,\sigma_0) = (\rho,\sigma)$, $q_0 = q$, $k=1$ and
$(A_1,(\rho_1,\sigma_1))= \left(\frac 1m \en_{1,0}(P),(1,0)\right)$. By
Theorem~\ref{divisibilidad}(3),
$$
\en_{1,0}(F_1)=\frac {p_1}{q_1}\frac 1m\en_{1,0}(P),
$$
and so $1 = v_{1,0}(F_1) = p_1v_{1,0}(A_0)/q_1 = p_1q/q_1$. Consequently,
$q_1 = q = q_0$, which contradicts Theorem~\ref{divisibilidad}(6) and
concludes the proof.
\end{proof}

In the following theorem we do not assume that $K$ is algebraically closed.

\begin{theorem}\label{cota para primitivos} Let $(P,Q)$ be a standard $(m,n)$-pair and let $(\rho,\sigma)$,
$A_0$, $F$ and $\gamma$ be as before Proposition~\ref{final}. We have either $v_{1,1}(A_0)>20$, or
$A_0=(4,12)$, $(\rho,\sigma)=(4,-1)$, $\gamma=3$ and $\en_{\rho,\sigma}(F)=(3,9)$.
\end{theorem}

\begin{proof} Without loss of generality we can assume that $K$ is algebraically closed. Let $A'_0$ and
$A^{(1)}$
be as in the above  Proposition~\ref{final}. We will analyze all possible pairs $A_0 = (u,v)$ with $u+v\le 20$,
satisfying all the conditions of Propositions~\ref{final}, \ref{primera cota para primitivos}, \ref{u(u-1)}
and Proposition~\ref{factores}(4).
\begin{table}[htb]
\begin{center}
\ra{1.2}
\begin{tabular}[t]{cccc}
\toprule
$A_0$ &$(f_1,f_2)$ & $(\rho,\sigma)$& $A'_0$\\
\midrule
              (4,12) & (2,6) & (5,-1) & $\times$ \\
              (4,12) & (3,9) & (4,-1) & (1,0) \\
              (5,15) & (2,6) & (5,-1) & (2,0) \\
              (5,15) & (3,9) & (4,-1) & $\times$ \\
              (5,15) & (4,12) & (11,-3) & $\times$\\
              (6,12) & (2,4) & (3,-1) & (2,0) \\

\bottomrule
\end{tabular}
\qquad
\begin{tabular}[t]{cccc}
\toprule
$A_0$ &$(f_1,f_2)$  & $(\rho,\sigma)$& $A'_0$ \\
\midrule
              (6,12) & (3,6) & (5,-2) & $\times$\\
              (6,12) & (4,8) & (7,-3) & $\times$ \\
              (6,12) & (5,10) & (9,-4) & $\times$ \\
              (8,12) & (2,3) & (2,-1) & (2,0) or (3,2) \\
              (8,12) & (4,6) & (5,-3) & $\times$ \\
              (8,12) & (6,9) & (8,-5) & $\times$ \\
\bottomrule
              \end{tabular}
\end{center}
\caption{}
\label{tabla2}
\end{table}
In Table~\ref{tabla2} we list all possible pairs $(u,v)$ with $3<u<v\le u(u-1)$, $\gcd(u,v)>2$ and
$16\le u+v\le 20$. We also list all the possible $(f_1,f_2)=\mu (u,v)$ with $f_1\ge 2$ and $0<\mu<1$. Then we
compute the corresponding $(\rho,\sigma)$ using Proposition~\ref{final}(6) and we verify if there can be an
$A'_0=(r',s')$ with $s'<r'<u$ and $v_{\rho,\sigma}(u,v)=v_{\rho,\sigma}(r',s')$. This happens in four cases.
By Proposition~\ref{impossibles} the case with $(\rho,\sigma)=(2,-1)$ is impossible. By
Proposition~\ref{final}(9), we know that in the remaining three cases
$\gamma\!\le\!\frac{v-s'}{\rho}$. Moreover, in two of these
cases $d\!:=\! \gcd(f_1 - 1, f_2 - 1)\!=\!1$, and so $\gamma=\frac{v-s'}{\rho}$.
By statements~(2) and~(4) of Proposition~\ref{encima de la diagonal} we know that $A^{(1)}\in L^{(l')}$ is
the first entry of a regular corner of an $(m,n)$ pair in $L^{(l')}$,
where $l'=\lcm(l,\rho)$. In the cases we are considering $l=1$ and so $l'=\rho$. We write
$A^{(1)}=(a/l',b)$.
In the case $A_0=(4,12)$, $(\rho,\sigma)=(4,-1)$ and $A_0'=(1,0)$ we are in the framework of
Proposition~\ref{factores}(3) with $r=3$ and $(\rho_0,\sigma_0)=(\rho,\sigma)$. Since
$\gcd(\rho_0,r)=\gcd(4,3)=1$, we have
$\gamma(\rho_0-2)>\rho_0$, which implies $\gamma\ge 3$. Hence, in this case $\gamma=3$.
In Table~\ref{Tabla3} we list the three cases
with the $\gamma$'s, the corresponding $A^{(1)}$, $d$ and $l'-a/b$.
\begin{table}[htb]
\begin{center}
\setlength{\tabcolsep}{8pt}
\ra{1.2}
\begin{tabular}[t]{cccccccc}
\toprule
$A_0$ &$(f_1,f_2)$ & $(\rho,\sigma)$& $A'_0$ & d & $\gamma$ & $A^{(1)}$ &$l'-\frac{a}{b}$ \\
\midrule
              (4,12) & (3,9) & (4,-1) & (1,0) &2&3& $\left (\frac 74,3\right)$&$1+\frac 23$ \\
              (5,15) & (2,6) & (5,-1) & (2,0) &1&3&$\left(\frac{13}{5},3\right)$& $1-\frac 13$ \\
              (6,12) & (2,4) & (3,-1) & (2,0) &1&4&$\left(\frac {10}3,4\right) $ &$1-\frac 12$ \\
\bottomrule
\end{tabular}
\end{center}
\caption{}
\label{Tabla3}
\end{table}
In the first case, where
$A_0=(4,12)$, $(\rho,\sigma)=(4,-1)$, $\gamma=3$ and $\en_{\rho,\sigma}(F)=(3,9)$, by Proposition~\ref{criterion}, we may have a
regular corner of type~I at $A^{(1)}$. In order to finish
the proof, we have to discard the remaining cases.

By Proposition~\ref{criterion} the second case is impossible, since $\gcd(13,3)=1$ and $1-1/3<1$.

Again by Proposition~\ref{criterion}, in the third case there could be a regular corner
of type~II with
\begin{equation}
 A^{(1)} =\left(\frac{a}{l'},b\right)= \left(\frac{10}3,4\right).\label{ve}
\end{equation}
Note that, with the notation of Proposition~\ref{case II}, we have
$$
d:=\gcd(a,b)=2,\quad \ov a:=\frac{a}d=5,\quad \ov b:=\frac{b}{d}=2\quad\text{and}
\quad\left\lfloor l'(bl'-a)+
\frac 1{\ov b} \right\rfloor=6.
$$
According to Proposition~\ref{case II}(4) the only possible $\mu$'s are $1$, $3$ or  $5$.
Let
$$
(\rho_1,\sigma_1):= \dir(\mu \ov{a}-\rho,\mu\ov{b}\rho-\rho) = \dir(\mu
5-3,\mu 3) = \begin{cases} (3,-2) & \text{if $\mu = 1$,}\\ (5,-4) &
\text{if $\mu = 3$,}\\ (27,-22) & \text{if $\mu = 5$.}\end{cases}
$$
For $\mu=3$
we obtain $\frac{\rho_1}{\gcd(\rho_1,l')}=5>4=b$ and for $\mu=5$
we obtain $\frac{\rho_1}{\gcd(\rho_1,l')}=9>4=b$, which contradict Proposition~\ref{case II}(2).
Hence we are left with
$\mu=1$ and $(\rho_1,\sigma_1)=(3,-2)$.
Since $\lcm(l',\rho_1)=\lcm(3,3)=3=l'$, applying Propositions~\ref{case
IIb} and~\ref{encima de la diagonal} we can assume that
$\left(\left(\frac{10}3,4\right),(3,-2)\right)$ is a regular corner of $(P,Q)$ of
type~II.a) in $L^{(l')}$.
 But then,
by Remark~\ref{Case IIa consecuencia} we must have
a regular corner at $\frac 1m\st_{3,-2}(P)$.
By Definition~\ref{def regular corner}(1) the only possible choice for $\frac 1m\st_{3,-2}(P)$ is
$(8/3,3)$, which is discarded by Proposition~\ref{criterion}, concluding the proof.
\end{proof}

\begin{corollary}\label{cota para B} Let $L=K[x,y]$ be the polynomial algebra over a non necessarily
algebraically closed characteristic zero field $K$. We have $B=16$ or $B>20$.
\end{corollary}

\begin{proof} Assume that $B\le 20$. By Corollary~\ref{B finito}
there exist $m,n\in \mathds{N}$ coprime with $m,n>1$ and a standard $(m,n)$-pair $(P,Q)$ which is also a minimal pair, such that
$\gcd(v_{1,1}(P),v_{1,1}(Q))\le 20$.
Let $(\rho,\sigma)$, $A_0=(a,b)$ and $F$ be as above of Proposition~\ref{final}. Since
\begin{equation}\label{B menor igual a 20}
v_{1,1}(A_0)=\frac{1}{m}v_{1,1}(\en_{\rho,\sigma}(P))\le \frac{1}{m}v_{1,1}(P) = B\le 20,
\end{equation}
by Theorem~\ref{cota para primitivos} we have
$$
(\rho,\sigma)=(4,-1),\quad A_0 = (4,12)\quad\text{and}\quad
\en_{4,-1}(F) = (3,9) = \frac 34 (4,12) = \frac 34 A_0,
$$
and consequently $q_0=4$, according to Theorem~\ref{divisibilidad}(3). If $A_0$ is the only regular corner, then
$B=v_{1,1}(A_0)=16$. Otherwise assume by contradiction that there exists $(A_1,(\rho_1,\sigma_1))$ as in Theorem~\ref{divisibilidad} and write
$A_1=(a_1,b_1)$. Then
\begin{itemize}

\smallskip

  \item[-] $4|\gcd(a_1,b_1)$, by Theorem~\ref{divisibilidad}(8),

\smallskip

  \item[-] $a_1\ge a=4$, since $(-1,0)<(\rho_1,\sigma_1)\le (1,0)$ (see Remark~\ref{starting vs end})\endnote{
If $(-1,0)<(\rho_1,\sigma_1)< (1,0)$, then by Remark~\ref{starting vs end} and
Proposition~\ref{le basico}, we have
$$
a = \frac{1}{m}v_{1,0}\bigl(\en_{\rho_0,\sigma_0}(P)\bigr) = \frac{1}{m}
v_{1,0}\bigl(\st_{\rho_1,\sigma_1}(P)\bigr) < \frac{1}{m}
v_{1,0}\bigl(\en_{\rho_1,\sigma_1}(P)\bigr) = a_1.
$$
as desired. If $(\rho_1,\sigma_1)= (1,0)$, then trivially $a=a_1$.
},

\smallskip

  \item[-] $a_1+b_1\le 20$ (similar to \eqref{B menor igual a 20}),

\smallskip

  \item[-] $b_1\ge 4(a_1-1)$ since $v_{4,-1}(m(a_1,b_1))\le v_{4,-1}(P)=v_{4,-1}(m(4,12))$.

\smallskip

\end{itemize}
This leaves us with
the only possibility that $(a_1,b_1)=(4,16)$, hence $(\rho_1,\sigma_1) =(1,0)$.
But this contradicts Corollary~\ref{no hay lados verticales} and concludes the proof.
\end{proof}

\begin{corollary}\label{forma final en L} If $B=16$, then there exist coprime integers $m,n>1$, a standard $(m,n)$-pair $(P,Q)$
that is a minimal pair, and
$\lambda_0,\lambda_1,\lambda_P,\lambda_Q \in K^{\times}$ such that
\begin{equation}\label{dir de P}
\Dir(P)=\Dir(Q)=\{(4,-1),(-2,1),(-1,0),(0,-1)\},
\end{equation}
the $(4,-1)$-homogeneous element $R_0:=x(xy^4-\lambda_0)^3$ satisfies
\begin{equation}
\ell_{4,-1}(P)=\lambda_P R_0^m\quad\text{and}\quad \ell_{4,-1}(Q)=\lambda_Q R_0^n,\label{chh7}
\end{equation}
and the $(-2,1)$-homogeneous element $R_1:=y(xy^2-\lambda_1)$ satisfies
\begin{equation}\label{chh7'}
\ell_{-2,1}(P)=\lambda_P R_1^{4m}\quad\text{and}\quad \ell_{-2,1}(Q)=\lambda_Q R_1^{4n}.
\end{equation}
Moreover, the $(4,-1)$-homogeneous element $F$ of Theorem~\ref{central} satisfies $\en_{\rho,\sigma}(F)=(3,9)$.
\end{corollary}

\begin{proof} By Corollary~\ref{B finito},
there exist $m,n\in \mathds{N}$ coprime with $m,n>1$, a standard $(m,n)$-pair $(P,Q)$ that is a minimal pair, such that
$\gcd(v_{1,1}(P),v_{1,1}(Q))=16$, and
\begin{equation}\label{sucesor de 10}
(-1,1)<\Succ_{P}(1,0),\Succ_{Q}(1,0)<(-1,0).
\end{equation}
Let $(\rho,\sigma)$, $A_0$, $\lambda$,
$m_\lambda$ and $\gamma$ be as above of Proposition~\ref{final}. By
Theorem~\ref{cota para primitivos} we know that
\begin{equation}\label{F con mu}
(\rho,\sigma)=(4,-1),\quad A_0=(4,12),\quad \gamma=3\quad\text{and}\quad
\en_{\rho,\sigma}(F)=(3,9)=\frac 34 A_0.
\end{equation}
By statements~(2), (3) and~(4) of
Proposition~\ref{primitivo} and Remark~\ref{a remark}, there exist
$\lambda_P,\lambda_Q\in K^{\times}$ and a $(4,-1)$-homogeneous polynomial $R$ such that
\begin{equation}
\ell_{4,-1}(P)=\lambda_P R^m\quad\text{and}\quad \ell_{4,-1}(Q)=\lambda_Q R^n.\label{chh1}
\end{equation}
Moreover, since $\en_{4,-1}(R)=A_0=(4,12)$ and $R$ is $(4,-1)$-homogeneous
$$
\Supp(R) \subseteq \{(4,12),(3,8),(2,4),(1,0)\}.
$$
Thus, $R = x \mathfrak{r}(z)= x \ov{\mathfrak{r}}(z^4)$ for some $r,\mathfrak{r}\in K[x]$, where $z := x^{1/4}y$ and
$\deg(\ov{\mathfrak{r}}) = 3$.
Now, since $m_\lambda = 3m$, the
multiplicity of the factor $z-\lambda$ in $\mathfrak{r}(z)$ is $3$, and then, by
Remark~\ref{multiplicidad de la potencia} the
multiplicity of the factor $z^4-\lambda^4$ in $\ov{\mathfrak{r}}(z^4)$ is $3$. Hence there exists $\mu\in K^{\times}$, such that
$$
R=\mu x(z^4-\lambda^4)^3=\mu x(xy^4-\lambda^4)^3.
$$
Then we obtain~\eqref{chh7} with $\lambda_0:=\lambda^4$, $\lambda_P$ replaced by $\lambda_P \mu^m$ and $\lambda_Q$ replaced by $\lambda_Q \mu^n$.

$$
R=\mu x(z^4-\lambda^4)^3=\mu x(xy^4-\lambda^4)^3 = \mu
x(xy^4-\lambda_0)^3\qquad\text{where $\lambda_0:=\lambda^4$.}
$$
Combining this with~\eqref{chh1}, and replacing $\lambda_P\mu^m$ by
$\lambda_P$ and $\lambda_Q \mu^n$ by $\lambda_Q$, we obtain~\eqref{chh7}.

Note that there are no directions in $\Dir(P)\cup\Dir(Q)$ in $](4,-1),(1,0)]\subset ](4,-1),(1,1)[$, since
$$
v_{1,1}(\en_{4,-1}(P))=v_{1,1}(m A_0)=16m = v_{1,1}(P)\quad\text{and}\quad v_{1,1}(\en_{4,-1}(Q))= v_{1,1}(Q)\endnote{
In
fact, this implies $mA_0\in \ell_{1,1}(P)$. Hence $mA_0\in
\Supp(\ell_{4,-1}(\ell_{1,1}(P)))$, because $mA_0\in \ell_{4,-1}(P)$.
Thus, by Remark~\ref{starting vs end} we have $\en_{4,-1}(P) = mA_0 =
\st_{1,1}(P)$. Consequently, by Proposition~\ref{le bbasico}, there is no
direction in $\Dir(P)$ between $(4,-1)$ and $(1,1)$. Similarly there is no
direction in $\Dir(Q)$ between $(4,-1)$ and $(1,1)$.
}.
$$
Hence, by~\eqref{sucesor de 10},
$$
(-1,1)<\Succ_{P}(4,-1),\Succ_{Q}(4,-1)<(-1,0).
$$
Let $(\rho',\sigma'):=\min\{\Succ_{P}(4,-1),\Succ_{Q}(4,-1)\}$. We claim that
$(-3,1)\le (\rho',\sigma')<(-1,0)$ is impossible. Otherwise
$$
v_{-3,1}(P)=v_{-3,1}(\st_{-3,1}(P))=v_{-3,1}(\en_{4,-1}(P))=v_{-3,1}(mA_0)=0,
$$
from which we obtain $\max\{0,\deg_y(P(0,y))\}= 0$, contradicting~\cite{vdE}*{Theorem 10.2.6}.
Therefore $(-1,1)<(\rho',\sigma')<(-3,1)$, and it is also clear
 that $v_{\rho,\sigma}(P)>0$ for
$(4,-1)<(\rho,\sigma)\le(\rho',\sigma')$.\endnote{
Note that
$v_{\rho,\sigma}(\en_{4,-1}(P))=v_{\rho,\sigma}(m(4,12))>0$ if and only if
$(3,-1)<(\rho,\sigma)<(-3,1)$.}
Then, by Corollary~\ref{fracciones de F}\endnote{
We
apply Corollary~\ref{fracciones de F} with $(\rho_0,\sigma_0) = (4,-1)$.
Conditions~(1) and~(2) of this corollary
follow from statements~(1) and~(2) of Corollary~\ref{some properties of
corners}.
}
and~\eqref{F con mu}, there exist a $(\rho',\sigma')$-homogeneous polynomial
$\ov{R}$ such that $\ell_{\rho',\sigma'}(P) = \ov{R}^{4m}$. Since
$\st_{\rho',\sigma'}(P)=m(4,12)$, we know that
$\st_{\rho',\sigma'}(\ov{R})=(1,3)$. By Remark~\ref{starting vs end},
$$
v_{1,0}(\en_{\rho',\sigma'}(\ov{R}))<v_{1,0}(\st_{\rho',\sigma'}(\ov{R}))=1.
$$
So $\en_{\rho',\sigma'}(\ov{R})=(0,k)$ for some $k\in\mathds{N}_0$. But
$k=0$ is impossible since it implies $(\rho',\sigma')=(-3,1)$. Moreover
$$
2=v_{-1,1}(\st_{\rho',\sigma'}(\ov{R}))>v_{-1,1}(\en_{\rho',\sigma'}(\ov{R}))=k,
$$
and therefore $k=1$, which implies $(\rho',\sigma')=(-2,1)$. Consequently
$\ov{R}=\mu y(xy^2-\lambda_1)$ for some $\mu,\lambda_1\in K^{\times}$.
Hence, $\ell_{-2,1}(P) = \mu^m R_1^m$. But since $\en_{4,-1}(P) =
\st_{-2,1}(P)$, necessarily $\mu^m = \lambda_P$. This shows that the first
equality in~\eqref{chh7'} is true. The same argument shows that the second
one is also true. It remains to achieve~\eqref{dir de P}.
But $\st_{4,-1}(P)=m\st_{4,-1}(R_0)=(m,0)$ and
$\en_{-2,1}(P)=4m\en_{-2,1}(R_1)=(0,4m)$. Adding eventually a constant, we
can assume that $P(0,0)\ne 0$ and then an elementary geometric argument
shows that~\eqref{dir de P} holds for $P$. Similarly it also holds for
$Q$.
\end{proof}

\begin{remark} As long as we are not able to discard the possibility $B\!=\!16$, there can be expected no
real progress in proving or disproving the JC just by describing the admissible $A_0$'s. However we submit
without proof a complete list of small values. Let
$$
B_0:=\frac 1m \st_{1,0}(P)\qquad\text{and}\qquad B_1:=\frac 1m \en_{1,0}(P).
$$
If $B\le 50$, then necessarily

\begin{enumerate}
\item[a)] $A_0$ belongs to the following set:
\begin{align*}
\qquad\mathcal{X}:=\{&(4,12),(5,20),(6,15),(6,30),(7,21),(7,35),(7,42),(8,24),(8,28),(9,21),\\
&(9,24),(9,36),(10,25),(10,30),(10,40),(11,33),(12,28),(12,30), (12,33),\\
&(12,36),(14,35),(15,35),(18,30)\}.
\end{align*}

\item[b)] $B_0\in \mathcal{X}$ or $B_0=(8,40)$ and $A_0=(4,12)$.

\item[c)] $B_1\in\mathcal{X}$ or $B_1\in\{(8,32),(8,40),(6,18),(6,24),(6,36), (6,42),(9,27)\}$.
    Furthermore,

\begin{itemize}

\item[-] if $B_1=(8,32)$, then $B_0=(8,28)$,

\item[-] if $B_1=(8,40)$ then $B_0=B_1$ or $B_0=(8,28)$,

\item[-] if $B_1=(6,18+6k)$, then $B_0=(6,15)$,

\item[-] if $B_1=(9,27)$, then $B_0=(9,21)$ or $B_0=(9,24)$.

\end{itemize}
\end{enumerate}
The cases listed in a) and b) coincide with the list for $B_0$ given in~\cite{H}*{Theorem~2.24(1)}, where
$B_0=(E_1,D_1)$ is written as $(D_1,E_1)$ and $B_1=(E,D)$ is written as
$(D,E)$. However, our list in c), in addition to the cases considered in~\cite{H}, contains the pairs
$\{(6,18),(6,24),(6,36),(6,42),(9,27)\}$. His result follows from a computer search and some
computations on a parameter $\alpha$, which should be the same as our $q$. Our list was found using the same
techniques as
in Theorem~\ref{cota para primitivos} and Corollary~\ref{cota para B}.
\end{remark}

\section{The case $B=16$}

\setcounter{equation}{0}

In this section we will analyze the case $B=16$. Applying the flip automorphism $\psi_1$ of $L$, given by
$\psi_1(x):=y$ and $\psi_1(y):=-x$, to the standard $(m,n)$-pair $(P,Q)$ obtained in Corollary~\ref{forma final en L}
(which is also a minimal pair), we obtain a minimal pair $(P_0,Q_0)$ in $L$, with
\begin{equation}\label{et1}
\Dir(P_0)=\Dir(Q_0)=\{(-1,4),(1,-2),(0,-1),(-1,0)\}.
\end{equation}
Moreover, by the same corollary there exist $\lambda_0,\lambda_1,\lambda_P,\lambda_Q\in K^{\times}$ such that
$R_0:=y(x^4y-\lambda_0)^3$ satisfies
\begin{equation}\label{et2}
\ell_{-1,4}(P_0)=\lambda_P R_0^m\quad\text{and}\quad \ell_{-1,4}(Q_0)=\lambda_Q R_0^n,
\end{equation}
and $R_1:=x(x^2y-\lambda_1)$ satisfies
\begin{equation}\label{et3}
\ell_{1,-2}(P_0)=\lambda_P R_1^{4m}\quad\text{and}\quad \ell_{1,-2}(Q_0)=\lambda_Q R_1^{4n},
\end{equation}
where $n,m\in\mathds{N}$ are coprime such that
$$
\frac{v_{1,1}(P_0)} {v_{1,1}(Q_0)}=\frac{v_{0,1}(P_0)}{v_{0,1}(Q_0)}=\frac mn.
$$
In this section we start with $(P_0,Q_0)$ and we use the following two types of operations:
\begin{itemize}

\smallskip

\item[-] add constants or multiply by non-zero constants,

\smallskip

\item[-] apply automorphisms of $L$ or $L^{(1)}$,

\smallskip

\end{itemize}
in order to obtain successively $(P_1,Q_1),\dots,(P_6,Q_6)\in L^{(1)}$, such that $(P_6,Q_6)$ satisfies
$$
P_6,Q_6 \in L,\quad \ell_{1,0}(P_6)=x^3 y,\quad \ell_{1,0}(Q_6)=x^2 y\quad\text{and}\quad [P_6,Q_6]=x^4 y+\mu_3 x^3+\mu_2 x^2+\mu_1 x+\mu_0
$$
for some $\mu_0$, $\mu_1$, $\mu_2$, $\mu_3$ in $K$, with $\mu_0\ne 0$.

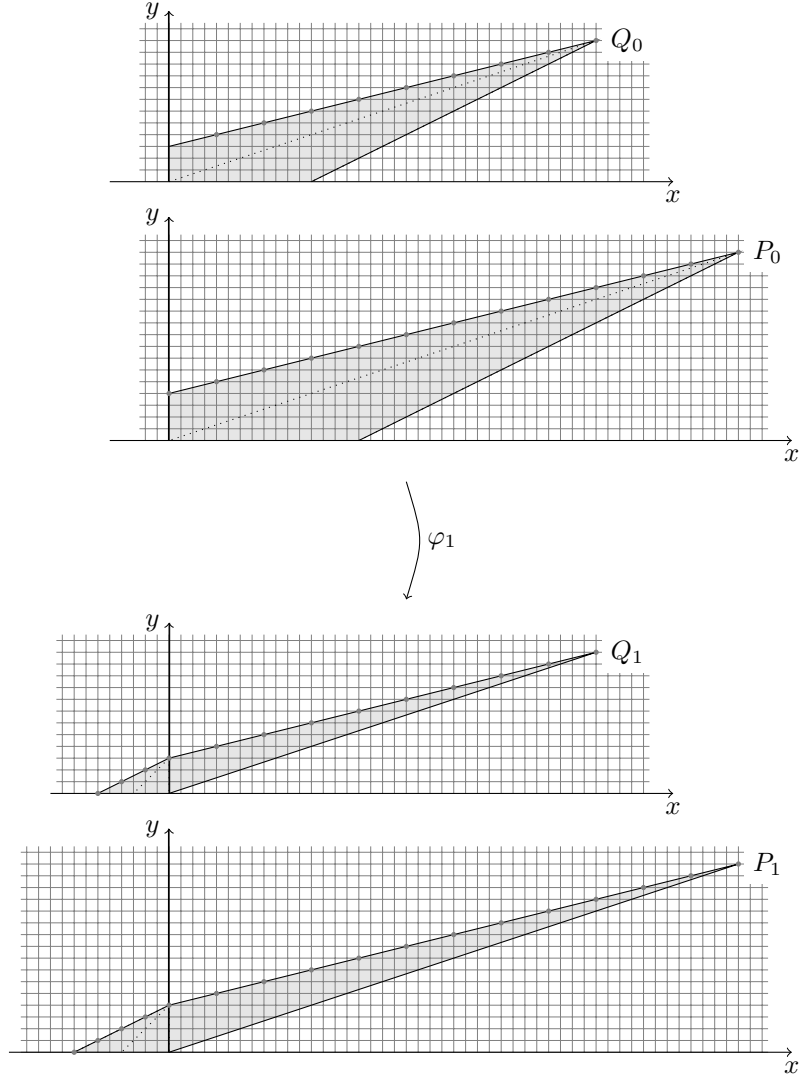
\begin{figure}[htb]
\centering
\begin{tikzpicture}[scale=0.78]
\fill[gray!20] (0,10.01) -- (3.2,10.01) -- (9.6,13.2)--(0,10.8);
\fill[gray!20] (0,-0.4) -- (9.6,2.8)--(0,0.4)--(-1.6,-0.4);
\fill[gray!20] (0,14.4) -- (2.4,14.4) -- (7.2,16.8)--(0,15);
\fill[gray!20] (0,4) -- (7.2,6.4)--(0,4.6)--(-1.2,4);
\draw[step=.2cm,gray,very thin] (-0.5,10) grid (10.1,13.5);
\draw [->] (-1,10) -- (10.5,10) node[anchor=north]{$x$};
\draw [->] (0,10) --  (0,13.8) node[anchor=east]{$y$};
\draw [-] (-0.01,10) --  (-0.01,13.8) node[anchor=east]{};
\draw (3.2,10) --  (9.6,13.2) node[fill=white,right=2pt]{$P_0$} -- (0,10.8);
\draw[dotted] (9.6,13.2) -- (0,10);
\draw[dotted] (9.6,13.2) -- (3.2,10);
\draw[step=.2cm,gray,very thin] (-2.5,-0.4) grid (10.1,3.1);
\draw [->] (-2.7,-0.4) -- (10.5,-0.4) node[anchor=north]{$x$};
\draw [->] (0,-0.4) --  (0,3.4) node[anchor=east]{$y$};
\draw [-] (-0.01,-0.4) --  (-0.01,3.4) node[anchor=east]{};
\draw (0,-0.4) --  (9.6,2.8) node[fill=white,right=2pt]{$P_1$} -- (0,0.4)--(-1.6,-0.4);
\draw[dotted] (0,0.4) -- (-1.6,-0.4);
\draw[dotted] (0,0.4) -- (-0.8,-0.4);
\draw[dotted] (0,0.4) -- (0,-0.4);
\draw[step=.2cm,gray,very thin] (-0.5,14.4) grid (8.1,17.1);
\draw [->] (-1,14.4) -- (8.5,14.4) node[anchor=north]{$x$};
\draw [->] (0,14.4) --  (0,17.3) node[anchor=east]{$y$};
\draw [-] (-0.01,14.4) --  (-0.01,17.3) node[anchor=east]{};
\draw (2.4,14.4) --  (7.2,16.8) node[fill=white,right=2pt]{$Q_0$} -- (0,15);
\draw[dotted] (7.2,16.8) -- (0,14.4);
\draw[dotted] (7.2,16.8) -- (2.4,14.4);
\draw[step=.2cm,gray,very thin] (-1.9,4) grid (8.1,6.7);
\draw [->] (-2,4) -- (8.5,4) node[anchor=north]{$x$};
\draw [->] (0,4) --  (0,6.9) node[anchor=east]{$y$};
\draw [-] (0.01,4) --  (0.01,6.9) node[anchor=east]{};
\draw (0,4) --  (7.2,6.4) node[fill=white,right=2pt]{$Q_1$} -- (0,4.6)--(-1.2,4);
\draw[dotted] (0,4.6) -- (-0.6,4);
\draw[dotted] (0,4.6) -- (0,4);
\filldraw [gray]  (0,10.8)    circle (1pt)
                  (0.8,11)      circle (1pt)
                  (1.6,11.2)      circle (1pt)
                  (2.4,11.4)      circle (1pt)
                  (3.2,11.6)    circle (1pt)
                  (4,11.8)  circle (1pt)
                  (4.8,12)  circle (1pt)
                  (5.6,12.2)      circle (1pt)
                  (6.4,12.4)    circle (1pt)
                  (7.2,12.6)      circle (1pt)
                  (8,12.8)      circle (1pt)
                  (8.8,13)    circle (1pt)
                  (9.6,13.2)  circle (1pt)
                  (0,0.4)    circle (1pt)
                  (0.8,0.6)      circle (1pt)
                  (1.6,0.8)      circle (1pt)
                  (2.4,1)      circle (1pt)
                  (3.2,1.2)    circle (1pt)
                  (4,1.4)  circle (1pt)
                  (4.8,1.6)  circle (1pt)
                  (5.6,1.8)      circle (1pt)
                  (6.4,2)    circle (1pt)
                  (7.2,2.2)      circle (1pt)
                  (8,2.4)      circle (1pt)
                  (8.8,2.6)    circle (1pt)
                  (9.6,2.8)  circle (1pt)
                  (-0.4,0.2)    circle (1pt)
                  (-0.8,0)      circle (1pt)
                  (-1.2,-0.2)      circle (1pt)
                  (-1.6,-0.4)      circle (1pt);
\filldraw [gray]  (7.2,6.4)     circle (1pt)
                  (6.4,6.2)      circle (1pt)
                  (5.6,6)      circle (1pt)
                  (4.8,5.8)      circle (1pt)
                  (4,5.6)    circle (1pt)
                  (3.2,5.4)  circle (1pt)
                  (2.4,5.2)  circle (1pt)
                  (1.6,5)      circle (1pt)
                  (0.8,4.8)    circle (1pt)
                  (0,4.6)      circle (1pt)
                  (-0.4,4.4)      circle (1pt)
                  (-0.8,4.2)    circle (1pt)
                  (-1.2,4)  circle (1pt)
                  (7.2,16.8)     circle (1pt)
                  (6.4,16.6)      circle (1pt)
                  (5.6,16.4)      circle (1pt)
                  (4.8,16.2)      circle (1pt)
                  (4,16)    circle (1pt)
                  (3.2,15.8)  circle (1pt)
                  (2.4,15.6)  circle (1pt)
                  (1.6,15.4)      circle (1pt)
                  (0.8,15.2)    circle (1pt)
                  (0,15);
\draw (4.2,8.3) node[right,text width=3cm]{$\varphi_1$};
\draw[->] (4,9.3) .. controls (4.3,8.3) .. (4,7.3);
\end{tikzpicture}
\caption{Proposition~\ref{phi1} for $m=4$ and $n=3$.}
\end{figure}

\smallskip

Let $\varphi_1\in \Aut(L^{(1)})$ defined by $\varphi_1(x):= x$ and $\varphi_1(y):= y+\lambda_1 x^{-2}$.
Adding eventually constants to $P_0$ and $Q_0$,
we can assume that

\begin{equation}\label{et4}
(0,0)\in \Supp(P_0)\cap\Supp(Q_0)\cap \Supp(\varphi_1(P_0))\cap\Supp(\varphi_1(Q_0)).
\end{equation}
Conditions~\eqref{et1}, \eqref{et2} and~\eqref{et3} remain valid under this
change, since they imply that $(0,0)$ belongs to the convex hull of
$\Supp(P_0)$, but doesn't belong to any of the straight lines containing
$\Supp(\ell_{-1,4}(P_0))$ and $\Supp(\ell_{1,-2}(P_0))$, and similarly for
$Q_0$, $\varphi_1(P_0)$ and $\varphi_1(Q_0)$.

\begin{proposition}\label{phi1} The elements $P_1:=\varphi_1(P_0)$ and $Q_1:=\varphi_1(Q_0)$ of $L^{(1)}$ satisfy:

\begin{enumerate}

\smallskip

\item $\Dir(P_1)=\Dir(Q_1)=\{(-1,4),(-1,2),(0,-1),(1,-3)\}$,

\smallskip

\item $\ell_{-1,4}(P_1)=\ell_{-1,4}(P_0)$, $\ell_{-1,4}(Q_1)=\ell_{-1,4}(Q_0)$, $v_{1,-3}(P_1)=0$ and $v_{1,-3}(Q_1)=0$,

\smallskip

\item $v_{1,1}(P_1) = v_{1,1}(P_0)$, $v_{1,1}(Q_1) = v_{1,1}(Q_0)$  and $[P_1,Q_1]\in K^{\times}$.

\end{enumerate}

\end{proposition}

\begin{proof} By Proposition~\ref{varphi preserva el Jacobiano} we know that $[P_1,Q_1]=[\varphi_1(P_0),\varphi_1(Q_0)]\in K^{\times}$ and by hypothesis and Proposition~\ref{pr ell por automorfismos}, we have
\begin{equation}
\ell_{0,1}(P_1)=\ell_{0,1}(P_0),\qquad \ell_{1,1}(P_1)=\ell_{1,1}(P_0),\qquad\ell_{-1,4}(P_1)=\ell_{-1,4}(P_0) = \lambda_P R_0^m\label{ecua1}
\end{equation}
and
\begin{equation}
\begin{aligned}
\{(-1,4)\} &= \Dir(P_0)\cap\{(\rho,\sigma)\in \mathfrak{V} : (1,-2)<(\rho,\sigma)<(-1,2) \} \\
& = \Dir(P_1)\cap\{(\rho,\sigma)\in \mathfrak{V}:(1,-2)<(\rho,\sigma)<(-1,2) \}.
\end{aligned}\label{ecua2}
\end{equation}
A direct computation using the third equality in~\eqref{ecua1} shows that
\begin{equation}
\st_{-1,4}(P_1) = (12m,4m) \qquad\text{and}\qquad\en_{-1,4}(P_1) = (0,m).\label{ecua3}
\end{equation}
Furthermore, by Proposition~\ref{le basico},
$$
\Supp(\ell_{-1,2}(P_0))=\st_{-1,0}(P_0)=\en_{-1,4}(P_0)=m\en_{-1,4}(R_0)=m(0,1),
$$
and so there exists $\mu\in k^{\times}$ such that $\ell_{-1,2}(P_0)=\mu y^m$. Hence, again by Proposition~\ref{pr ell por automorfismos},
\begin{equation}
\ell_{-1,2}(P_1)=\varphi_1(\ell_{-1,2}(P_0))=\mu (y+\lambda_1 x^{-2})^m,\label{ecua4}
\end{equation}
which combined with the equality~\eqref{ecua2}, implies that $(-1,4)$ and $(-1,2)$ are consecutive directions of $\Dir(P_1)$. Moreover, since
$\en_{-1,2}(P_1) = (-2m,0)$ by equality~\eqref{ecua4}, and $(0,0)\in \Supp(P_1)$,
the next direction in $\Dir(P_1)$ is $(0,-1)$ and
\begin{equation}\label{e4}
(0,0)\in \Supp(\ell_{0,-1}(P_1)).
\end{equation}
The same argument shows that all the results already proved for $P_1$ are also valid for $Q_1$ but with $Q_0$ instead of $P_0$ and $n$ instead of $m$. Hence the first two equalities in items~(2) and~(3) are true.

We claim that in order to finish the computation of $\Dir(P_1)$ it suffices to verify that
\begin{equation}\label{predecesor de -14}
\Pred_{P_1}(-1,4)\le(1,-3)<(-1,4).
\end{equation}
In fact, by Proposition~\ref{le basico} and the first equality in~\eqref{ecua3} this implies that
$$
\en_{1,-3}(P_1) = \st_{-1,4}(P_1) = (12m,4m),
$$
and consequently
$$
v_{1,-3}(P_1) = v_{1,-3}(12m,4m) = 0
$$
(that is the third equality in item~(2)). Since $v_{1,-3}(0,0) = 0$, we also have $(0,0)\in\!\Supp(\ell_{1,-3}(P_1))$, and
so $(1,-3)\!\in\! \Dir(P_1)$. Moreover by~\eqref{e4} it is also clear that
$$
\en_{0,-1}(P_1) = \st_{1,-3}(P_1) = (0,0),
$$
which, by Proposition~\ref{le bbasico}, finishes the proof of the claim.

Now we prove~\eqref{predecesor de -14}. Since, by Proposition~\ref{pr ell por automorfismos}
$$
\Supp(\ell_{1,-2}(P_1))=\Supp(\varphi_1(\ell_{1,-2}(P_0)))=\Supp(\varphi_1(\lambda_P R_1^{4m}))= (12m,4m)=\st_{-1,4}(P_1),
$$
again by Proposition~\ref{le bbasico}, we are reduced to prove that
\begin{equation}\label{rho sigma impossible}
\Pred_{P_1}(-1,4)\notin \ ](1,-3),(1,-2)[.
\end{equation}
Assume by contradiction that $(\rho',\sigma'):= \Pred_{P_1}(-1,4)\in\, ](1,-3),(1,-2)[$. Then
\begin{equation}\label{e5}
v_{\rho,\sigma}(P_1)>0\qquad\text{for all $(\rho',\sigma')\le(\rho,\sigma)\le (-1,4)$.}\endnote{
Note that $v_{\rho,\sigma}(\st_{-1,4}(P))=v_{\rho,\sigma}(m(12,4))>0$ if and only if $(1,-3)<(\rho,\sigma)<(-1,3)$.}
\end{equation}
By~\eqref{ecua1}, \eqref{ecua3}, their analogous for $Q_1$, and~\eqref{e5}, the hypothesis of Corollary~\ref{fracciones de F1}
are satisfied with $(\rho_0,\sigma_0) = (-1,4)$, $l=1$ and $(a,b) = (12,4)$. Moreover the $(-1,4)$-homogeneous element $F_1$
of Theorem~\ref{central} satisfies
$$
\st_{-1,4}(F_1)=(9,3)=\frac 34 \frac 1m \st_{-1,4}(P_1).\endnote{
We claim that $\st_{-1,4}(F_1)\sim \st_{-1,4}(P_1)=m(12,4)$. If the claim is true,
then $\st_{-1,4}(F_1)=\alpha (3,1)$, and so $\alpha=v_{-1,4}(F_1)=-1+4=3$, which yields $\st_{-1,4}(F_1)=(9,3)$, as desired.

If we assume by contradiction that $\st_{-1,4}(F_1)\nsim \st_{-1,4}(P_1)$, then
$\st_{-1,4}(F_1)=(1,1)$ by Theorem~\ref{central}(2), and so, by Remark~\ref{F no es monomio}, we have
$\en_{-1,4}(F_1)\ne(1,1)$, which by Theorem~\ref{central}(3) implies
$$
\en_{-1,4}(F_1)\sim \en_{-1,4}(P_1)=(0,m).
$$
But then $v_{-1,4}(F_1)=3$ leads to $\en_{-1,4}(F_1)=(0,3/4)$, which is impossible.
}
$$
Thus, by Corollary~\ref{fracciones de F1} there exist a $(\rho',\sigma')$-homogeneous element $R\in L^{(1)}$ and $\zeta\in K^{\times}$, such that
$\ell_{\rho',\sigma'}(P_1)=\zeta R^{4m}$. Since by Proposition~\ref{le basico} and the first equality in~\eqref{ecua3}
$$
\en_{\rho',\sigma'}(P_1)=m(12,4),
$$
we obtain that $\en_{\rho',\sigma'}(R)=(3,1)$. Consequently, by Remark~\ref{starting vs end},
$$
v_{0,1}(\st_{\rho',\sigma'}(R))<v_{0,1}(\en_{\rho',\sigma'}(R))=1,
$$
and so $\st_{\rho',\sigma'}(R)=(k,0)$ for some $k\in\mathds{Z}$. But then, again by Remark~\ref{starting vs end},
$$
k=v_{1,-2}(\st_{\rho',\sigma'}(R))<v_{1,-2}(\en_{\rho',\sigma'}(R))=1,
$$
and
$$
k=v_{1,-3}(\st_{\rho',\sigma'}(R))>v_{1,-3}(\en_{\rho',\sigma'}(R))=0,
$$
which yields the desired contradiction, proving~\eqref{rho sigma impossible} and concluding the computation of $\Dir(P_1)$. In order to finish the proof it suffices to apply the same argument to $Q_1$.
\end{proof}

\begin{proposition}\label{esquina de 83} Let $(\rho_1,\sigma_1)\in\mathfrak{V}_{>0}$ with $\rho_1<0$ and let $P,Q\in
L^{(1)}$ such that $[P,Q]\in K^{\times}$. Assume
$$
\frac 1m\en_{\rho_1,\sigma_1}(P)=\frac 1n \en_{\rho_1,\sigma_1}(Q) = (a,b),
$$
with $a,b,m,n\in \mathds{N}$, $a>b$, $m\ne n$ and $\gcd(m,n) = 1$. Set
$$
(\rho_2,\sigma_2):=\min\{\Succ_P(\rho_1,\sigma_1),\Succ_Q(\rho_1,\sigma_1)\}\quad\text{and}\quad (a',b'):=\frac 1m
\en_{\rho_2,\sigma_2}(P).
$$
(By the same argument as in the proof of Proposition~\ref{propiedades de
los pares}(5), neither $P$ nor $Q$ are monomials).
The following facts hold:

\begin{enumerate}

\item If $\en_{\rho_2,\sigma_2}(P)\sim \en_{\rho_2,\sigma_2}(Q)$, then $(\rho_2,\sigma_2) \in \Dir(P) \cap
    \Dir(Q)$. Furthermore
\begin{equation}\label{caso alineado}
\quad (a',b')\in \mathds{Z}\times\mathds{N}_0,\quad v_{\rho_1,\sigma_1}(a',b') < v_{\rho_1,\sigma_1}(a,b)
\quad\text{and}\quad ab'-ba'>0.
\end{equation}

\item If $\en_{\rho_2,\sigma_2}(P)\nsim \en_{\rho_2,\sigma_2}(Q)$, then there exists $k\in\mathds{N}$, such
    that
$$
\quad (k+1)b<a\qquad\text{and}\qquad\{\en_{\rho_2,\sigma_2}(P),\en_{\rho_2,\sigma_2}(Q)\}= \{(-k,0),(k+1,1)\}.
$$

\item If $(\rho_1,\sigma_1)=(-1,4)$, $(a,b)=(8,3)$ and $\en_{\rho_2,\sigma_2}(P)\sim \en_{\rho_2,\sigma_2}(Q)$,
then $(\rho_2,\sigma_2)=(-1,3)$ and there exist $\mu,\mu_P,\mu_Q \in K^{\times}$ such that
$$
\quad \ell_{-1,3}(P)=\mu_P x^{-m}(x^3y-\mu)^{3m}\qquad\text{and}\qquad \ell_{-1,3}(Q) =
\mu_Q x^{-n}(x^3y-\mu)^{3n}.
$$
\end{enumerate}
\end{proposition}

\begin{proof} By Proposition~\ref{le basico}, we know that
\begin{equation}\label{eq:proporcionalidad}
\st_{\rho_2,\sigma_2}(P)=\en_{\rho_1,\sigma_1}(P)=\frac mn \en_{\rho_1,\sigma_1}(Q)=\frac mn \st_{\rho_2,\sigma_2}(Q),
\end{equation}
which implies
\begin{equation}\label{eq:cociente}
v_{\rho_2,\sigma_2}(P) = \frac mn v_{\rho_2,\sigma_2}(Q)= mv_{\rho_2,\sigma_2}(a,b).
\end{equation}
Moreover, $(\rho_1,\sigma_1)<(\rho_2,\sigma_2)<(-1,1)$ since otherwise, by
Proposition~\ref{le basico},
$$
v_{-1,1}(P)=m(b-a)<0\quad\text{and}\quad v_{-1,1}(Q)=n(b-a)<0,
$$
which contradicts Proposition~\ref{pr v de un conmutador}. Consequently
$(\rho_2,\sigma_2)\in\mathfrak{V}_{>0}$, and so, by this same proposition
and equality~\eqref{eq:cociente},
$$
\frac{m+n}{m} v_{\rho_2,\sigma_2}(P)= v_{\rho_2,\sigma_2} (P)+v_{\rho_2,\sigma_2}(Q)\ge v_{\rho_2,\sigma_2} ([P,Q]) +
\rho_2+\sigma_2>0.
$$
Hence, again by equality~\eqref{eq:cociente},
\begin{equation}\label{desigualdades para rho' y sigma'}
v_{\rho_2,\sigma_2}(P)>0,\quad v_{\rho_2,\sigma_2}(Q)>0\quad\text{and}\quad \rho_2a+\sigma_2b=
v_{\rho_2,\sigma_2}(a,b) >0.
\end{equation}
Note also that, since $(\rho_1,\sigma_1)\in\mathfrak{V}_{>0}$,
$$
\rho<0 \Leftrightarrow (-1,1)>(\rho_1,\sigma_1)>(0,1) \Rightarrow (-1,1)>(\rho_2,\sigma_2)>(0,1) \Rightarrow \rho_2<0.
$$
Set
$$
(\rho_3,\sigma_3):=\dir(a,b)=\left(-\frac{b}{\gcd(a,b)},\frac{a}{\gcd(a,b)}\right)
$$
(Note that $(\rho_3,\sigma_3)\in\mathfrak{V}_{>0}$). By the last inequality in~\eqref{desigualdades para rho' y sigma'},
$$
\rho_2\sigma_3-\sigma_2\rho_3=\frac{1}{\gcd(a,b)}(\rho_2a+\sigma_2b)>0,
$$
which means that $(\rho_3,\sigma_3)>(\rho_2,\sigma_2)$ in $\mathfrak{V}_{>0}$.
But then from $\st_{\rho_2,\sigma_2}(P) =
\en_{\rho_1,\sigma_1}(P)=m(a,b)$ and Remark~\ref{starting vs end} it follows that
\begin{equation}\label{v del endpoint positivo}
v_{\rho_3,\sigma_3}(\en_{\rho_2,\sigma_2}(P))\ge v_{\rho_3,\sigma_3}(\st_{\rho_2,\sigma_2}(P))=m
v_{\rho_3,\sigma_3}(a,b)=0\quad\text{and}\quad v_{\rho_3,\sigma_3}(\en_{\rho_2,\sigma_2}(Q))\ge 0.
\end{equation}
Note also that the first inequality is strict if $(\rho_2,\sigma_2)\in \Dir(P)$ and the second
inequality is strict if $(\rho_2,\sigma_2)\in \Dir(Q)$.

\smallskip

\noindent Proof of statement~(1).\enspace By~\eqref{desigualdades para rho' y sigma'} we have
$v_{\rho_2,\sigma_2}(P)\ne 0$. So, by equality~\eqref{eq:cociente},
\begin{equation}\label{eqq13}
\en_{\rho_2,\sigma_2}(P)=\lambda \en_{\rho_2,\sigma_2}(Q) \Longrightarrow \lambda=m/n.
\end{equation}
Consequently, by~\eqref{eq:proporcionalidad},
$$
\en_{\rho_2,\sigma_2}(P)=\st_{\rho_2,\sigma_2}(P) \quad\text{if and only if}
\quad\en_{\rho_2,\sigma_2}(Q)=\st_{\rho_2,\sigma_2}(Q).
$$
By the definition of $(\rho_2,\sigma_2)$ this implies that
\begin{equation}\label{esta en dir}
(\rho_2,\sigma_2)\in\Dir(P)\cap\Dir(Q).
\end{equation}
Furthermore, since $m$ and $n$ are coprime, it
follows from~\eqref{eqq13}, that
$$
(a',b')=\frac 1m \en_{\rho_2,\sigma_2}(P) = \frac 1n \en_{\rho_2,\sigma_2}(Q)\in \mathds{Z}\times\mathds{N}_0.
$$
Using now Remark~\ref{starting vs end}, we obtain
$$
v_{\rho_1,\sigma_1}(a',b') = \frac 1m v_{\rho_1,\sigma_1}(\en_{\rho_2,\sigma_2}(P))< \frac 1m
v_{\rho_1,\sigma_1}(\st_{\rho_2,\sigma_2}(P))= \frac 1m v_{\rho_1,\sigma_1}(\en_{\rho_1,\sigma_1}(P)) =
v_{\rho_1,\sigma_1}(a,b).
$$
Finally, by the definition of $(\rho_3,\sigma_3)$, by the first inequality in~\eqref{v del endpoint positivo} and~\eqref{esta en dir},
$$
ab'-ba'= \gcd(a,b) v_{\rho_3,\sigma_3}(a',b')= \frac{\gcd(a,b)} m v_{\rho_3,\sigma_3}(\en_{\rho_2,\sigma_2}(P))>0.
$$

\smallskip

\noindent Proof of statement~(2).\enspace
By Propositions~\ref{pr v de un conmutador}, \ref{extremosalineados}, and~\ref{extremosnoalineados}(2) we have
$$
\en_{\rho_2,\sigma_2}(Q)+\en_{\rho_2,\sigma_2}(P)=(1,1).
$$
Thus
$$
\{\en_{\rho_2,\sigma_2}(P), \en_{\rho_2,\sigma_2}(Q)\}=\{(-k,0),(k+1,1)\}\quad\text{for some $k\in \mathds{Z}$.}
$$
Consequently,
\begin{equation}\label{rho prima en dir p}
(\rho_2,\sigma_2)\in \Dir(P)\cap\Dir(Q),
\end{equation}
 since $\st_{\rho_2,\sigma_2}(P)=m(a,b)$, $\st_{\rho_2,\sigma_2}(Q)=n(a,b)$, and
$nb,mb>1$.
Note that by~\eqref{desigualdades para rho' y sigma'}
$$
0<v_{\rho_2,\sigma_2}(-k,0)=-k\rho_2,
$$
which implies that $k\in\mathds{N}$, since $\rho_2<0$. Finally, by the definition of $(\rho_3,\sigma_3)$
and~\eqref{v del endpoint positivo} we have
$$
\frac{-b(k+1)+a}{\gcd(a,b)}=v_{\rho_3,\sigma_3}(k+1,1)>0,
$$
where the inequality is strict by~\eqref{rho prima en dir p}, and so $(k+1)b<a$.

\smallskip

\noindent Proof of statement~(3).\enspace We first prove that $(\rho_2,\sigma_2)=(-1,3)$. By statement~(1)
we know that
$$
(\rho_2,\sigma_2) =\Succ_P(\rho_1,\sigma_1)\qquad\text{and}\qquad (a',b'):=\frac 1m \en_{\rho_2,\sigma_2}(P)\in
\mathds{Z}\times\mathds{N}_0.
$$
Hence, since $(\rho_2,\sigma_2)>(0,1)$ in $\mathfrak{V}_{>0}$, from Remark~\ref{starting vs end} and
Proposition~\ref{le basico} it follows that
$$
b'=\frac 1m v_{0,1}(\en_{\rho_2,\sigma_2}(P))<\frac 1m v_{0,1}(\st_{\rho_2,\sigma_2}(P))=\frac 1m
v_{0,1}(\en_{\rho_1,\sigma_1}(P))=v_{01}(a,b)=3.
$$
Furthermore, by the inequalities in~\eqref{caso alineado}, and the fact that $(\rho_1,\sigma_1)=(-1,4)$
and $(a,b)=(8,3)$, we have
\begin{equation}\label{ocho b menos 3 a}
8b'-3a'>0\qquad\text{and}\qquad 4>4b'-a'.
\end{equation}
Now we consider several cases. In all of them we will use that, by Remark~\ref{valuacion depende de extremos},
$$
(\rho_2,\sigma_2)=\dir\bigl(\st_{\rho_2,\sigma_2}(P) - \en_{\rho_2,\sigma_2}(P)\bigr)=\dir\bigl((a,b) -
(a',b')\bigr)
= \dir\bigl((8,3) - (a',b')\bigr).
$$
If $b'=2$, then the inequalities in~\eqref{ocho b menos 3 a} imply $4<a'<16/3$, and so
$$
a'=5\qquad\text{and}\qquad (\rho_2,\sigma_2) = \dir\bigl((8,3) - (5,2)\bigr)=(-1,3),
$$
as desired. If $b'=1$, then the inequalities in~\eqref{ocho b menos 3 a} imply $0<a'<8/3$. Since by
Theorem~\ref{central}(4), the equality $(a',b')=(1,1)$ is impossible, necessarily $a'=2$ and
$(\rho_2,\sigma_2) = (-1,3)$, again as desired. Assume finally that $b'=0$. Then the inequalities
in~\eqref{ocho b menos 3 a} give $-4<a'<0$. If $a'=-1$ then $(\rho_2,\sigma_2)=(-1,3)$ as we want.
So, it will be sufficient to discard the cases $a'=-2$ and $a'=-3$. Consider the first one, in
which $(\rho_2,\sigma_2)=(-3,10)$. By Theorem~\ref{central} there exists a $(\rho_2,\sigma_2)$-homogeneous
element $F\in L^{(1)}$ satisfying the equalities in~\eqref{eq central}, with $(\rho_1, \sigma_1)$ replaced by
$(\rho_2,\sigma_2)$. Then there exists $k\in \mathds{Z}$ such that
$$
\st_{\rho_2,\sigma_2}(F)=(1,1)+k(10,3),
$$
and hence $v_{-3,8}(\st_{\rho_2,\sigma_2}(F))=5-6k\ne 0$. Consequently
$$
\st_{\rho_2,\sigma_2}(F)\sim \st_{\rho_2,\sigma_2}(P)=m(8,3)
$$
is impossible, since it implies $v_{-3,8}(\st_{\rho_2,\sigma_2}(F))=0$. So, by statement~(2) of
Theorem~\ref{central}, we have $\st_{\rho_2,\sigma_2}(F)=(1,1)$. Since by Remark~\ref{F no es monomio}
we know that $F$ is not a monomial, $\en_{\rho_2,\sigma_2}(F)\ne (1,1)$, which by Theorem~\ref{central}(3) and the fact that $b'=0$ and $a'<0$,
implies
$$
\en_{\rho_2,\sigma_2}(F)\sim \en_{\rho_2,\sigma_2}(P)\sim(-1,0).
$$
But by the first equality in~\eqref{eq central}, this leads to $\en_{\rho_2,\sigma_2}(F)=(-7/3,0)$, which is
impossible since $F\in L^{(1)}$. A similar computation proves that if $a'=-3$, then $\en_{\rho_2,\sigma_2}(F)=(-8/3,0)$, which is
also impossible. Thus, $(\rho_2,\sigma_2)=(-1,3)$ as we want.

\smallskip
\newpage
It remains to check that there exist $\mu,\mu_P,\mu_Q \in K^{\times}$ such that
$$
\quad \ell_{-1,3}(P)=\mu_P x^{-m}(x^3y-\mu)^{3m}\qquad\text{and}\qquad \ell_{-1,3}(Q) =
\mu_Q x^{-n}(x^3y-\mu)^{3n}.
$$
By equalities~\eqref{eq:cociente} we have
$$
v_{-1,3}(P)=mv_{-1,3}(8,3)=m>0\qquad\text{and}\qquad v_{-1,3}(Q)=nv_{-1,3}(8,3)=n>0.
$$
Since $m,n\in \mathds{N}$ and $m\ne n$ this implies that,
$$
v_{-1,3}(P)+v_{-1,3}(Q)-(-1+3)=m+n-2>0,
$$
and so, by Proposition~\ref{pr v de un conmutador} and Remark~\ref{a remark}, there exist
$\lambda_P, \lambda_Q\!\in\! K^{\times}$ and a $(-1,3)$-homogeneous e\-le\-ment $R\in L^{(1)}$ such that
\begin{equation}\label{de las ultimas}
\ell_{-1,3}(P)=\lambda_P R^m\qquad\text{and}\qquad \ell_{-1,3}(Q)=\lambda_Q R^n.
\end{equation}
Hence, by equalities \eqref{eq:proporcionalidad},
$$
\st_{-1,3}(R)=\frac 1m \st_{-1,3}(P)=(8,3).
$$
Consequently there exists a polynomial $\mathfrak{r}$ of degree $3$ (which can not be a monomial), such that
$$
R = x^{-1} \mathfrak{r}(z) \quad\text{where $z:=x^3 y$.}
$$
We are going to check that $\mathfrak{r}$ has no linear factor $z-\mu$ with multiplicity~1,
which implies that $\mathfrak{r}$ is
a cube of a linear polynomial, and so $\ell_{-1,3}(P)$ and $\ell_{-1,3}(Q)$ have the desired form. Assume
by contradiction that such a factor exists and consider $\varphi\in \Aut(L^{(1)})$ defined by
$$
\varphi(x) := x\qquad\text{and}\qquad \varphi(y) := y + \mu x^{-3}.
$$
Write $R=x^{-1}(z-\mu)\mathfrak{h}(z)$ with $\mathfrak{h}(\mu)\ne 0$. Then
$\varphi(R)=x^{-1}z\mathfrak{h}(z+\mu)$ and so $\en_{-1,3}(\varphi(R))=(2,1)$. Hence
$$
\en_{-1,3}(\varphi(P))= m(2,1)\quad\text{and}\quad \en_{-1,3}(\varphi(Q))= n(2,1).
$$
Set
$$
(\rho_4,\sigma_4):=\min\{\Succ_{\varphi(P)}(-1,3),\Succ_{\varphi(Q)}(-1,3)\}\quad\text{and}\quad
(a'',b''):=\frac 1m \en_{\rho_4,\sigma_4}(\varphi(P)).
$$
Since by Proposition~\ref{varphi preserva el Jacobiano} we have $[\varphi(P),\varphi(Q)] \in K^{\times}$,
we can apply statements~(1) and~(2) with
$$
(\varphi(P),\varphi(Q))\text{ instead of } (P,Q),\quad (\rho_1,\sigma_1):=(-1,3)\quad\text{and}\quad (a,b):=(2,1).
$$
So, if $\en_{\rho_4,\sigma_4} (\varphi(P)) \sim \en_{\rho_4,\sigma_4} (\varphi(Q))$, then by statement~(1)
we have
$$
(a'',b'')\in \mathds{Z}\times\mathds{N}_0,\quad 3b''-a''=v_{-1,3}(a'',b'') < v_{-1,3}(2,1) = 1
\quad\text{and}\quad 2b''-a''>0,
$$
which is impossible, while if $\en_{\rho_4,\sigma_4} (\varphi(P)) \nsim \en_{\rho_4,\sigma_4} (\varphi(Q))$,
then there exists $k\in\mathds{N}$, such that
$$
\quad k+1<2\qquad\text{and}\qquad\{\en_{\rho_4,\sigma_4}(\varphi(P)),\en_{\rho_4,\sigma_4}(\varphi(Q))\}=
\{(-k,0),(k+1,1)\},
$$
which is also impossible since there is no $k\in\mathds{N}$ with $k+1<2$. Hence we have a contradiction that
finishes the proof.
\end{proof}


\begin{proposition}\label{phi2} Let $P,Q\in L^{(1)}$ be such that $\frac{v_{1,1}(P)}{v_{1,1}(Q)}=\frac mn\ne
1$ with $m,n\in \mathds{N}$ coprime. Assume that there exist $\lambda, \lambda_P, \lambda_Q\in K^{\times}$ such that $R:=y(x^4 y-\lambda)^3$
satisfies
$$
\ell_{-1,4}(P)=\lambda_P R^m\qquad\text{and}\qquad\ell_{-1,4}(Q)=\lambda_Q R^n.
$$
Assume also that
$$
[P,Q]\in K^{\times},\quad v_{1,-3}(P)=v_{1,-3}(Q)=0\quad\text{and}\quad\Succ_{P}(1,-3)=
\Succ_{Q}(1,-3)=(-1,4).
$$

Then there exists $\mu\in K$ such that the images $\widetilde P:=\varphi(P)$ and $\widetilde Q:=\varphi(Q)$
of $P$
and $Q$ under the automorphism $\varphi$ of $L^{(1)}$ given by $\varphi(x):=x$ and $\varphi(y):=y+\lambda
x^{-4}+\mu x^{-3}$, satisfy
\begin{alignat*}{4}
&\ell_{-1,4}(\widetilde P)=\lambda_P x^{8m} y^{3m}(x^4 y+\lambda)^m,&&\quad
 \frac 1m \en_{-1,4}(\widetilde P)=  (8,3),&&\quad v_{1,1}(\widetilde P)=v_{1,1}(P),\quad
 v_{1,-3}(\widetilde P) =0,\\
&\ell_{-1,4}(\widetilde Q)=\lambda_Q x^{8n} y^{3n}(x^4 y+\lambda)^n,&&\quad
\frac 1n \en_{-1,4}(\widetilde Q)=  (8,3),&&\quad v_{1,1}(\widetilde Q)=v_{1,1}(Q),\quad
v_{1,-3}(\widetilde Q) =0,\\
&\Succ_{\widetilde P}(-1,4) = \Succ_{\widetilde Q}(-1,4),&&\quad [\widetilde P,\widetilde Q]\in K^{\times}
\hspace{0.55cm}\text{and}&&\quad \en_{\rho,\sigma}(\widetilde P)\nsim
\en_{\rho,\sigma}(\widetilde Q),
\end{alignat*}
where $(\rho,\sigma):=\Succ_{\widetilde P}(-1,4)$. Moreover
$$
\{\en_{\rho,\sigma}(\widetilde P),\en_{\rho,\sigma}(\widetilde
Q)\}=\{(-1,0),(2,1)\},\quad\st_{\rho,\sigma}(\widetilde P)
= m(8,3),\quad\st_{\rho,\sigma}(\widetilde Q) = n(8,3)
$$
and there exists $j\in \mathds{N}$ such that $(\rho,\sigma) = (-3j-1,8j+3)$ and

\begin{itemize}

\item[-] if $\en_{\rho,\sigma}(\widetilde P)=(-1,0)$, then $m =3j+1$ and $n =2j+1$,

\item[-] if $\en_{\rho,\sigma}(\widetilde Q)=(-1,0)$, then $m=2j+1$ and $n=3j+1$.

\end{itemize}

\end{proposition}

\noindent {\bf Idea of the proof:} Using the automorphism $\varphi_a$ of $L^{(1)}$ defined by
$$
\varphi_a(x):=x\qquad\text{and}\qquad \varphi_a(y):=y+\lambda x^{-4},
$$
we transform $\ell_{-1,4}(P)=\lambda_P (y(x^4 y-\lambda)^3)^m$ into $\ell_{-1,4}(\varphi_a(P))=
\lambda_P x^{8m} y^{3m}(x^4 y+\lambda)^m$. By Proposition~\ref{esquina de 83} we have
$(\rho,\sigma)=(-1,3)$ or $[\ell_{\rho,\sigma}\varphi_a(P),\ell_{\rho,\sigma}(\varphi_a(Q))]\ne 0$. In the
first case we apply the automorphism $\varphi_b$ of $L^{(1)}$ defined by
$$
\varphi_b(x):=x\qquad\text{and}\qquad \varphi_b(y):=y+\mu x^{-3},
$$
for some $\mu\in K^{\times}$ and we arrive at the second case. Using Proposition~\ref{esquina de 83} again, we
check the required conditions.

\begin{proof} We define $\varphi_a\in\Aut(L^{(1)})$ by $\varphi_a(x):=x$ and $\varphi_a(y):=y+\lambda x^{-4}$.
Set
$$
P_1:=\varphi_a(P)\qquad\text{and}\qquad Q_1:=\varphi_a(Q).
$$
By Proposition~\ref{varphi preserva el Jacobiano} we have $[P_1,Q_1] \in K^{\times}$. Moreover, since
$\varphi_a$ is
$(-1,4)$-homogeneous,
\begin{equation}\label{eqq25}
\ell_{-1,4}(P_1)=\lambda_P x^{8m} y^{3m}(x^4 y+\lambda)^m\quad\text{and}\quad \ell_{-1,4}(Q_1)=
\lambda_Q x^{8n} y^{3n}(x^4 y+\lambda)^n,
\end{equation}
and so
\begin{equation}\label{eqq14}
\frac 1m \en_{-1,4}(P_1)=\frac 1n \en_{-1,4}(Q_1)=(8,3).
\end{equation}
Furthermore, since $(1,-4)\!<\!(1,-3)<(-1,4)$, it follows from Proposition~\ref{pr ell por automorfismos} and
the hypothesis, that
\begin{equation}\label{eqq27}
v_{1,-3}(P_1) = v_{1,-3}(P) =0 \qquad\text{and}\qquad v_{1,-3}(Q_1) = v_{1,-3}(Q) = 0.
\end{equation}
Let
$$
(\rho',\sigma'):=\min\{\Succ_{P_1}(-1,4),\Succ_{Q_1}(-1,4)\}.
$$
Applying Proposition~\ref{esquina de 83}(3) to $(P_1,Q_1)$ we obtain that
\begin{align}
& \en_{\rho',\sigma'}(P_1)\nsim \en_{\rho',\sigma'}(Q_1)\label{eqq15}
\shortintertext{or}
& \en_{\rho',\sigma'}(P_1)\sim \en_{\rho',\sigma'}(Q_1)\quad\text{and} \quad (\rho',\sigma')=(-1,3).
\label{eqq16}
\end{align}
Furthermore, in the second case there exist $\mu,\mu_P,\mu_Q\in K^{\times}$ such that
\begin{equation}\label{ell de P_1 y de Q_1}
\ell_{-1,3}(P_1)=\mu_P x^{-m}(x^3y-\mu)^{3m}\quad\text{and}\quad \ell_{-1,3}(Q_1)=\mu_Q x^{-n}(x^3y-\mu)^{3n}.
\end{equation}
Assume first that condition~\eqref{eqq16} is satisfied and define $\varphi_b\in\Aut(L^{(1)})$ by
$$
\varphi_b(x):=x\quad\text{and}\quad \varphi_b(y):=y+\mu x^{-3}.
$$
By Proposition~\ref{varphi preserva el Jacobiano}, we have $[\varphi_b(P_1),\varphi_b(Q_1)] \in K^{\times}$
 and, by
Proposition~\ref{pr ell por automorfismos}, we have
$$
\ell_{\rho_1,\sigma_1}(\varphi_b(P_1))=\ell_{\rho_1,\sigma_1}(P_1)\quad\text{and}\quad
\ell_{\rho_1,\sigma_1}(\varphi_b(Q_1))= \ell_{\rho_1,\sigma_1}(Q_1)
$$
for all $(-1,4)\le (\rho_1,\sigma_1)<(-1,3)$, which implies that
\begin{equation}\label{eqq26}
\ell_{-1,4}(\varphi_b(P_1))=\ell_{-1,4}(P_1),\qquad \ell_{-1,4}(\varphi_b(Q_1))=\ell_{-1,4}(Q_1)
\end{equation}
and
$$
(\rho'',\sigma''):=\min \{\Succ_{\varphi_b(P_1)}(-1,4),\Succ_{\varphi_b(Q_1)}(-1,4)\} \ge (-1,3).
$$
Combining equalities~\eqref{eqq25} and~\eqref{eqq26} we obtain
\begin{equation*}
\ell_{-1,4}(\varphi_b(P_1))=\lambda_P x^{8m} y^{3m}(x^4 y+\lambda)^m\quad\text{and}\quad
\ell_{-1,4}(\varphi_b(Q_1))=\lambda_Q x^{8n} y^{3n}(x^4
y+\lambda)^n,
\end{equation*}
which implies
\begin{equation}\label{eqq17}
\frac 1m \en_{-1,4}(\varphi_b(P_1))=\frac 1n \en_{-1,4}(\varphi_b(Q_1))=(8,3).
\end{equation}
On the other hand, since $\varphi_b$ is $(-1,3)$-homogeneous it follows from equalities~\eqref{ell de P_1 y de Q_1}, that
\begin{equation*}\label{ell de tilde P y de tilde Q}
\ell_{-1,3}(\varphi_b(P_1)) = \mu_P x^{8m}y^{3m}\quad\text{and}\quad \ell_{-1,3}(\varphi_b(Q_1))=
\mu_Q x^{8n}y^{3n},
\end{equation*}
and so $(\rho'',\sigma'') > (-1,3)$. Consequently, applying Proposition~\ref{esquina de 83}(3) to
$(\varphi_b(P_1),\varphi_b(Q_1))$, we obtain that
$$
\en_{\rho'',\sigma''}(\varphi_b(P_1))\nsim \en_{\rho'',\sigma''}(\varphi_b(Q_1)).
$$
Moreover, by Proposition~\ref{pr ell por automorfismos} and equalities~\eqref{eqq27}, we have
\begin{equation*}
v_{1,-3}(\varphi_b(P_1)) = v_{1,-3}(P_1) =0 \qquad\text{and}\qquad v_{1,-3}(\varphi_b(Q_1)) = v_{1,-3}(Q_1)
= 0.
\end{equation*}
Now in the case~\eqref{eqq15} we set
$$
\varphi:=\varphi_a,\quad \widetilde P:=P_1,\quad \widetilde Q:=Q_1\quad\text{and}\quad
(\rho,\sigma):=(\rho',\sigma'),
$$
while in the case~\eqref{eqq16} we set
$$
\varphi:=\varphi_b\circ \varphi_a,\quad \widetilde P:=\varphi_b(P_1),\quad\widetilde
Q:=\varphi_b(Q_1)\quad\text{and}\quad (\rho,\sigma):=(\rho'',\sigma'').
$$
The conditions required in Proposition~\ref{esquina de 83} are satisfied for $\widetilde P$, $\widetilde Q$ and $(\rho_1,\sigma_1)=(-1,4)$.
Moreover $\en_{\rho,\sigma}(\widetilde P)\nsim \en_{\rho,\sigma}(\widetilde Q)$.
Since $k\!=\!1$ is the unique positive integer satisfying $(k\!+\!1)3\!<\!8$,
Proposition~\ref{esquina de 83}(2) yields
\begin{equation}\label{eqq19}
\{\en_{\rho,\sigma}(\widetilde P), \en_{\rho,\sigma}(\widetilde Q)\}=\{(-1,0),(2,1)\}.
\end{equation}
On the other hand, since $(\rho,\sigma)=\min\{\Succ_{\widetilde P}(-1,4),\Succ_{\widetilde Q}(-1,4)\}$, by
Proposition~\ref{le basico} and equalities~\eqref{eqq14} and~\eqref{eqq17}, we have
\begin{equation}\label{eqq18}
\st_{\rho,\sigma}(\widetilde P)= \en_{-1,4}(\widetilde P)= m(8,3)\quad\text{and} \quad
\st_{\rho,\sigma}(\widetilde Q)=
\en_{-1,4}(\widetilde Q)= n(8,3).
\end{equation}
Combining~\eqref{eqq19} and~\eqref{eqq18}, we obtain that $(\rho, \sigma)\in \Dir(P) \cap \Dir(Q)$, and then
$$
(\rho,\sigma)=\Succ_{\widetilde P}(-1,4)=\Succ_{\widetilde Q}(-1,4),
$$
as desired. Since $(1,-4)<(1,1)<(-1,4)$ and $(1,-3)<(1,1)<(-1,3)$, from
Proposition~\ref{pr ell por automorfismos} it follows that
$$
v_{1,1}(\widetilde P) = v_{1,1}(P)\quad\text{and}\quad v_{1,1}(\widetilde Q) = v_{1,1}(Q).
$$
Now, by Remark~\ref{valuacion depende de extremos} and equalities~\eqref{eqq18}, we have
$$
\dir(n(8,3)-\en_{\rho,\sigma}(\widetilde Q))=(\rho,\sigma)= \dir(m(8,3)-\en_{\rho,\sigma}(\widetilde P)).
$$
If $\en_{\rho,\sigma}(\widetilde P) = (-1,0)$, then
$$
m(8,3)-(-1,0)\sim n(8,3)-(2,1)\Rightarrow (8m+1)(3n-1)=3m(8n-2) \Rightarrow 3(n-1)=2(m-1),
$$
and so $m\!=\!3j\!+\!1$ and $n\!=\!2j\!+\!1$ for some $j\!\in\! \mathds{Z}$. But necessarily $j\!\in\!
\mathds{N}$ since $m,n\!\in\! \mathds{N}$ and $m\!\ne\! n$. Moreover
$$
(\rho,\sigma) = \dir((3j+1)(8,3)-(-1,0)) = \dir(24j+9,9j+3) = (-3j-1,8j+3).
$$
Similarly, if $\en_{\rho,\sigma}(\widetilde Q) = (-1,0)$, then $m\!=\!2j\!+\!1$ and $n\!=\!3j\!+\!1$ for some
$j\!\in\! \mathds{N}$, and
$$
(\rho,\sigma) = \dir((3j+1)(8,3)-(-1,0)) = (-3j-1,8j+3),
$$
as desired.
\end{proof}

\begin{remark}\label{lista}
Let $(P_0,Q_0)$ be an $(m,n)$-pair satisfying conditions~\eqref{et1}, \eqref{et2}, \eqref{et3} and~\eqref{et4}.
Applying first Proposition~\ref{phi1} to $(P_0,Q_0)$ we obtain $(P_1,Q_1)$, and then applying Proposition~\ref{phi2} to
$(P_1,Q_1)$ we obtain $j\in \mathds{N}$, $\lambda, \lambda_p, \lambda_q\in K^{\times}$, $(\rho_1,\sigma_1)>(-1,4)$ and
a pair $(P_2,Q_2)$ in $L^{(1)}$ such that

\begin{enumerate}

\smallskip

\item $\frac{v_{1,1}(P_2)}{v_{1,1}(Q_2)}=\frac mn$,

\smallskip

\item $v_{1,-3}(P_2) = 0$ and $v_{1,-3}(Q_2) = 0$,

\smallskip

\item $[P_2,Q_2]\in K^{\times}$,

\smallskip

\item $(\rho_1,\sigma_1)=\Succ_{P_2}(-1,4)= \Succ_{Q_2}(-1,4)$,

\smallskip

\item $\{\en_{\rho_1,\sigma_1}(P_2),\en_{\rho_1,\sigma_1}(Q_2)\}=
\{(-1,0),(2,1)\}$,

\smallskip

\item $\st_{\rho_1,\sigma_1}(P_2)=m(8,3)$ and
$\st_{\rho_1,\sigma_1}(Q_2)=n(8,3)$,

\smallskip

\item $\ell_{-1,4}(P_2)=\lambda_p R^m$ and $\ell_{-1,4}(Q_2)=\lambda_q
R^n$ where $R:=x^8y^3(x^4y+\lambda)$,

\smallskip

\item $(\rho_1,\sigma_1)=(-3j-1,8j+3)$,

\smallskip

\item If $\en_{\rho_1,\sigma_1}(P_2) = (-1,0)$, then $m=3j+1$ and
$n=2j+1$, while if $\en_{\rho_1,\sigma_1}(Q_2) = (-1,0)$, then $m=2j+1$
and $n=3j+1$.

\smallskip

\end{enumerate}
Interchanging $P_2$ with $Q_2$ if necessary, we can assume that
$\en_{\rho_1,\sigma_1}(P_2) = (-1,0)$, $m=3j+1$ and $n=2j+1$.
Adding constants to $P_2$ and $Q_2$ we also can assume that $(0,0)$ is in
their support. From items~(2), (5), (6) and~(7) it follows that $\{(-1,0),m(8,3),m(12,4),(0,0)\}$ are
vertices of $\HH(P_2)$ and that $\{ (2,1),n(8,3),n(12,4),(0,0)\}$ are vertices of $\HH(Q_2)$, where,
as in Section~1, $\HH(P)$ denotes the Newton polygon of $P$.
\end{remark}

\begin{remark}
Let $P_2$, $Q_2$ be as above. Then
\begin{itemize}

\smallskip

\item[-] $\HH(P_2)=\CH\{ (-1,0),m(8,3),m(12,4),(0,0)\}$,

\smallskip

\item[-] $\HH(Q_2)=\CH\{ (2,1),n(8,3),n(12,4),(0,0)\}$,

\smallskip

\end{itemize}
where for each subset $X$ of $\mathds{R}^2$, $\CH(X)$ denotes the convex hull of $X$.
For $P_2$ this is clear, but for $Q_2$ we have to prove that there is no other vertex between $(2,1)$ and $(0,0)$, or equivalently, that the
only direction between $(\rho_1,\sigma_1)$ and $(1,-3)$ is $(-1,2)$. But this follows easily
from the fact that the straight line containing $\ell_{\rho_1,\sigma_1}(Q_2)$ intersects
the $X$-axis between $(0,0)$ and $(-1,0)$, since
$$
v_{\rho_1,\sigma_1}(-1,0)=-\rho_1=3j+1>2j+1=2\rho_1+\sigma_1=v_{\rho_1,\sigma_1}(2,1)=v_{\rho_1,\sigma_1}(Q_2).
$$
\end{remark}

\begin{figure}[htb]
\centering
\begin{tikzpicture}
\fill[gray!20] (13,3.5) -- (13.125,3.625) -- (13.375,4.625)--(13,5);
\fill[gray!20] (13,0) -- (13.125,0) -- (13.5,1.5) -- (13,2);
\fill[gray!20] (5,3.5) -- (9.5,5) -- (8,4.625)--(5.25,3.625);
\fill[gray!20] (5,0) -- (11,2) -- (9,1.5)--(4.875,0);
\draw[step=.125cm,gray,very thin] (13,3.5) grid (14.7,5.2);
\draw [->] (13,3.5) -- (15,3.5) node[anchor=north]{$x$};
\draw [->] (13,3.5) --  (13,5.5) node[anchor=east]{$y$};
\draw[step=.125cm,gray,very thin] (13,0) grid (14.7,2.2);
\draw [->] (13,0) -- (15,0) node[anchor=north]{$x$};
\draw [->] (13,0) --  (13,2.5) node[anchor=east]{$y$};
\draw[step=.125cm,gray,very thin] (5,3.5) grid (10.7,5.2);
\draw [->] (4.3,3.5) -- (11,3.5) node[anchor=north]{$x$};
\draw [->] (5,3.5) --  (5,5.5) node[anchor=east]{$y$};
\draw[step=.125cm,gray,very thin] (5,0) grid (11.2,2.2);
\draw [->] (3.8,0) -- (11.5,0) node[anchor=north]{$x$};
\draw [->] (5,0) --  (5,2.5) node[anchor=east]{$y$};
\draw (13.125,0) --  (13.5,1.5) node[fill=white,right=2pt]{$P_3$} -- (13,2);
\draw (13,3.5) -- (13.125,3.625) -- (13.375,4.625) node[fill=white,right=2pt]{$Q_3$} -- (13,5);
\draw (5,0) --  (11,2) node[fill=white,right=2pt]{$P_2$} -- (9,1.5)--(4.875,0);
\draw (5,3.5) --  (9.5,5) node[fill=white,right=2pt]{$Q_2$} -- (8,4.625)--(5.25,3.625)--(5,3.5);
\filldraw [gray]  (13.125,0)    circle (1pt)
                  (13.25,0.5)      circle (1pt)
                  (13.375,1)      circle (1pt)
                  (13.5,1.5)      circle (1pt)
                  (13,2)    circle (1pt)
                  (13,3.5)  circle (1pt)
                  (13.125,3.625)  circle (1pt)
                  (13.25,4.125)      circle (1pt)
                  (13.375,4.625)    circle (1pt)
                  (13,5)      circle (1pt)
                  (5,0)      circle (1pt)
                  (11,2)    circle (1pt)
                  (9,1.5)  circle (1pt)
                  (7.625,1)  circle (1pt)
                  (6.25,0.5)  circle (1pt)
                  (4.875,0)    circle (1pt)
                  (5,3.5)    circle (1pt)
                  (9.5,5)    circle (1pt)
                  (8,4.625)      circle (1pt)
                  (6.625,4.125)  circle (1pt)
                  (5.25,3.625)    circle (1pt);
\draw[->] (11,2.9) .. controls (11.75,3.15) .. (12.5,2.9);
\draw (11.5,3.4) node[right,text width=2cm]{$\psi_3$};
\end{tikzpicture}
\caption{Proposition~\ref{transformacion afin} for $m=4$ and $n=3$.}
\end{figure}
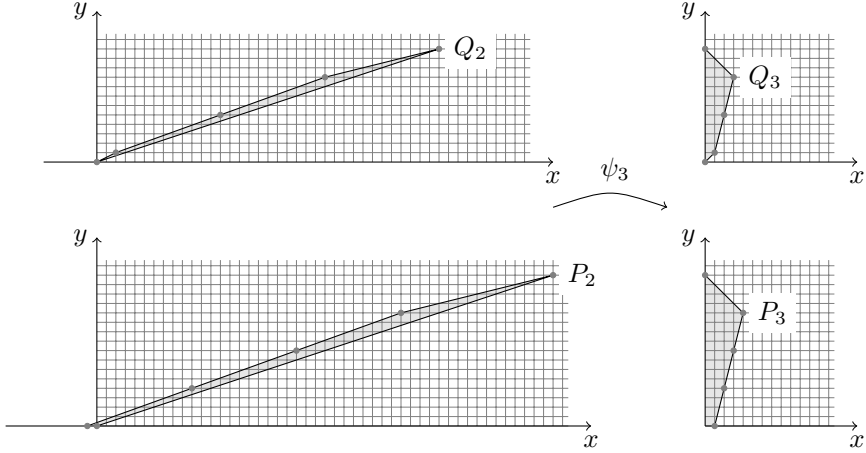

\begin{remark}\label{capsula convexa} Let $\widetilde{\psi}_3$ be the $\mathds{R}$-linear automorphism of
$\mathds{R}^2$ given by $\widetilde{\psi}_3(i,j):=(-i+3j,j)$. If $\psi_3$ is the automorphism of $L^{(1)}$ defined by
$\psi_3(x):= x^{-1}$ and $\psi_3(y):= x^3 y$, then
$$
\psi_3\left(\sum_{i,j}a_{i,j}x^i y^j\right)=\sum_{i,j}a_{i,j}x^{-i+3j} y^j,
$$
and so
$$
\Supp(\psi_3(P))=\widetilde{\psi}_3(\Supp(P))\qquad\text{for all $P\in
L^{(1)}$.}
$$
Moreover, since
$\widetilde{\psi}_3$ is $\mathds{R}$-linear, it preserves convex hulls, which means
$$
\CH(\widetilde{\psi}_3(X))=\widetilde{\psi}_3(\CH(X)),
$$
for all sets $X\subseteq \mathds{R}^2$. In particular, for $A_1,\dots, A_k\in \mathds{R}^2$ and $P\in L^{(1)}$,
\begin{equation}\label{implicacion de capsula convexa}
\CH(\Supp(P))=\CH\{ A_1,\dots, A_k \} \Longrightarrow \CH(\Supp(\psi_3(P)))=
\CH\{\widetilde{\psi}_3(A_1),\dots, \widetilde{\psi}_3(A_k)\}.
\end{equation}
An elementary computation also shows that for $\ov{\psi}_3\colon \mathfrak{V}\to\mathfrak{V}$ given by
$\ov{\psi}_3(\rho,\sigma)\coloneqq(-\rho,3\rho+\sigma)$, we have
\begin{equation}\label{ve de psi}
   \ell_{\ov{\psi}_3(\rho,\sigma)}(\psi_3(P)) = \psi_3(\ell_{\rho,\sigma}(P))\quad\text{and}\quad
   v_{\ov{\psi}_3(\rho,\sigma)}(\psi_3(P))=v_{\rho,\sigma}(P),
\end{equation}
for all $P\in L^{(1)}$ and $(\rho,\sigma)\in \mathfrak{V}$. Similar properties hold for $\widetilde\psi_1$ defined by
$\widetilde\psi_1(i,j)\coloneqq(j,i)$.
\end{remark}

\begin{proposition}\label{transformacion afin} Let $P_2$, $Q_2$, $(\rho_1,\sigma_1)$ and $j$ be as in Remark~\ref{lista}, and $\psi_3$, $\ov{\psi}_3$ as above.
Set
$$
P_3:=\psi_3(P_2),\quad Q_3:=\psi_3(Q_2)\quad\text{and}\quad(\rho_2,\sigma_2):=
\ov{\psi}_3(\rho_1,\sigma_1) = (-\rho_1,3\rho_1+\sigma_1)=(3j+1,-j).
$$
The following facts hold:
\begin{enumerate}

\smallskip

\item $P_3,Q_3\in L$,

\smallskip

\item $[P_3,Q_3]=\zeta x$, for some $\zeta\in K^{\times}$,

\smallskip

\item $\en_{\rho_2,\sigma_2}(P_3) = m(1,3)$ and $\en_{\rho_2,\sigma_2}(Q_3) = n(1,3)$,

\smallskip

\item $\st_{\rho_2,\sigma_2}(P_3)=(1,0)$ and $\st_{\rho_2,\sigma_2}(Q_3)=(1,1)$,

\smallskip

\item $\{(\rho_2,\sigma_2),(1,1)\} \subseteq \Dir(P_3)\cap \Dir(Q_3)$ and
    $(\rho_2,\sigma_2)=\Pred_{P_3}(1,1)=\Pred_{Q_3}(1,1)$,

\smallskip

\item $\ell_{1,1}(P_3)=\lambda_p R_3^m$ and $\ell_{1,1}(Q_3)=\lambda_q R_3^n$, where $R_3=y^3(y+\lambda
    x)$, with $\lambda,\lambda_p,\lambda_q\in K^{\times}$ the same elements as in Remark~\ref{lista}.

\end{enumerate}

\end{proposition}

\begin{proof} By Proposition~\ref{varphi preserva el Jacobiano} we have
$$
[P_3,Q_3]=\psi_3([P_2,Q_2])[\psi_3(x),\psi_3(y)]=\zeta x\quad\text{for some $\zeta\in K^{\times}$,}
$$
which proves statement~(2).
Next we are going to check that $P_3,Q_3\in L$, or, equivalently, that
\begin{equation}\label{eqq28}
v_{-1,0}(P_3)\le 0\qquad\text{and}\qquad v_{-1,0}(Q_3)\le 0.
\end{equation}
Since $\ov\psi_3(1,-3)=(-1,0)$, these inequalities follow from~\eqref{ve de psi} and the fact that
$v_{1,-3}(P_2) = 0$ and
$v_{1,-3}(Q_2) = 0$.
Furthermore, since $\widetilde\psi_3$ is a reflection, consecutive elements of $\Dir(P)$ are
mapped into consecutive elements of $\Dir(\psi_3(P))$, inverting the order. From
$$
(\rho_1,\sigma_1)=\Succ_{P_2}(-1,4)= \Succ_{Q_2}(-1,4)\quad\text{and}\quad
(-1,4)\in \Dir(P_2)\cap \Dir(Q_2)
$$
 it follows that
$$
(\rho_2,\sigma_2),(1,1)\in \Dir(P_3)\cap \Dir(Q_3)\quad\text{and}\quad
(\rho_2,\sigma_2)=\Pred_{P_3}(1,1)=\Pred_{Q_3}(1,1),
$$
since $(\rho_2,\sigma_2)=\ov\psi_3(\rho_1,\sigma_1)$ and $(1,1)=\ov\psi_3(-1,4)$. This proves statement~(5).

\smallskip

Let $\widetilde{\psi}_3$ be as in Remark~\ref{capsula convexa}. Since
$$
\Supp(\ell_{\rho_2,\sigma_2}(\psi_3(P))) = \Supp(\psi_3(\ell_{\rho_1,\sigma_1}(P))) =
\widetilde{\psi}_3(\Supp(\ell_{\rho_1,\sigma_1}(P))) \qquad \text{for all $P\in L^{(1)}$},
$$
and $\widetilde\psi_3$ is a reflection, we have
$$
\st_{\rho_2,\sigma_2}(\psi_3(P))=\widetilde{\psi}_3(\en_{\rho_1,\sigma_1}(P))\quad\text{and}\quad
\en_{\rho_2,\sigma_2}(\psi_3(P))=\widetilde{\psi}_3(\st_{\rho_1,\sigma_1}(P))
\qquad\text{for all $P\in L^{(1)}$.}
$$
From this the third and fourth statements follow immediately. By~\eqref{ve de psi} we also have
$$
\ell_{1,1}\bigl(\psi_3(P_2)\bigr) = \psi_3\bigl(\ell_{1,-4}(P_2)\bigr) = \psi_3(\lambda_p R^m)= \lambda_p
\psi_3(R^m) = \lambda_p R_3^m,
$$
and similarly $\ell_{1,1}\bigl(\psi_3(Q_2)\bigr)=\lambda_q R_3^n$.
\end{proof}

\begin{remark} We assert that
\begin{equation}\label{predecesor entre extremos}
(-1,1)\le \Pred_{P_3}(\rho_2,\sigma_2)\le (1,-1)\quad\text{and}\quad
(-1,1)\le \Pred_{Q_3}(\rho_2,\sigma_2)\le (1,-1).
\end{equation}
Assume on the contrary that
$(\ov\rho,\ov\sigma):=\Pred_{P_3}(\rho_2,\sigma_2)\in\Dir(P_3)\cap\mathfrak{V}_{>0}$. Then, since
$$
(\ov\rho,\ov\sigma)<(\rho_2,\sigma_2) = (3j+1,-j) < (0,1)
$$
using Remark~\ref{starting vs end}, Proposition~\ref{le basico} and Proposition~\ref{transformacion afin}(4),
we obtain that
$$
v_{0,1}(\st_{\ov\rho,\ov\sigma}(P_3))<v_{0,1}(\en_{\ov\rho,\ov\sigma}(P_3))=
v_{0,1}(\st_{\rho_2,\sigma_2}(P_3))=v_{0,1}(1,0)=0,
$$
which is impossible. On the other hand, if $(\ov\rho,\ov\sigma):=
\Pred_{Q_3}(\rho_2,\sigma_2)\in\Dir(Q_3)\cap\mathfrak{V}_{>0}$,
then again by Remark~\ref{starting vs end}, Proposition~\ref{le basico} and
Proposition~\ref{transformacion afin}(4), we have
$$
v_{0,1}(\st_{\ov\rho,\ov\sigma}(Q_3))<v_{0,1}(\en_{\ov\rho,\ov\sigma}(Q_3))=v_{0,1}(\st_{\rho_2,\sigma_2}(Q_3))
=v_{0,1}(1,1)=1,
$$
and similarly $v_{1,0}(\st_{\ov\rho,\ov\sigma}(Q_3))<1$, which implies $\st_{\ov\rho,\ov\sigma}(Q_3)=(-k,0)$
for
some $k\in \mathds{N}_0$. So
$$
v_{1,-1}(\st_{\ov\rho,\ov\sigma}(Q_3))=-k \le 0 =
v_{1,-1}(1,1)=v_{1,-1}(\st_{\rho_2,\sigma_2}(Q_3))=v_{1,-1}(\en_{\ov\rho,\ov\sigma}(Q_3)),
$$
which is impossible by Remark~\ref{starting vs end}. This ends the proof of~\eqref{predecesor entre extremos}.
Consequently
\begin{equation}\label{lo que faltaba}
\Pred_{P_3}(\rho_2,\sigma_2),\Pred_{P_3}(\rho_2,\sigma_2)< (2,-1)<(\rho_2,\sigma_2).
\end{equation}
\end{remark}

\begin{remark}\label{constante}
Combining~\eqref{predecesor entre extremos} with Proposition~\ref{le basico}
and Proposition~\ref{transformacion afin}(4), we obtain
$$
\en_{1,-1}(P_3)=(1,0)\qquad\text{and}\qquad \en_{1,-1}(Q_3)=(1,1),
$$
since $(1,-1)<(\rho_2,\sigma_2)$ in $\mathfrak{V}_{\ge0}$.
Clearly this implies
$$
\ell_{1,-1}(P_3)=\mu_P x\qquad\text{and}\qquad \ell_{1,-1}(Q_3)= \mu_Q xy + \xi
$$
with $\mu_P,\mu_Q\in K^{\times}$ and $\xi \in K$. Since by Proposition~\ref{pr v de un conmutador} and
Proposition~\ref{transformacion afin}(2)
$$
\mu_P\mu_Q x = [\mu_P x,\mu_Q xy + \xi] = [\ell_{1,-1}(P_3),\ell_{1,-1}(Q_3)]=\zeta x,
$$
dividing $P_3$ by $\mu_P$ and replacing $Q_3$ by $-\frac{\mu_P}{\zeta}(Q_3 -\xi)$, we obtain a new
 pair $(P_4,Q_4)$ satisfying statements~(1)--(6) of Proposition~\ref{transformacion afin} with $\zeta=-1$, and
such that
$$
\ell_{1,-1}(P_4)=x\qquad\text{and}\qquad \ell_{1,-1}(Q_4)=-xy.
$$
\end{remark}

\begin{proposition}\label{P4} Let $\psi_1\in\Aut(L)$ be the map given by $\psi_1(x)\coloneqq y$ and $\psi_1(y)\coloneqq -x$ and
$\ov{\psi}_1$ be the action on directions given by $\ov\psi_1(\rho,\sigma)\coloneqq(\sigma,\rho)$. Let $P_4$,
$Q_4$ be as in Remark~\ref{constante} and let $\lambda$ be as in Pro\-po\-si\-tion~\ref{transformacion afin}. Set
$(\rho_3,\sigma_3)\!:= \!\ov{\psi}_1(\rho_2,\sigma_2)\! =\! (\sigma_2,\rho_2) \!=\!
(-j,3j+1)$, where $j\in \mathds{N}$ is as in Remark~\ref{lista}. There exist $\mu_1,\mu_2,\mu_3\in K$ such that the images
$P_5:=\varphi(\psi_1(P_4))$ and $Q_5:=\varphi(\psi_1(Q_4))$ of $\psi_1(P_4)$ and $\psi_1(Q_4)$ under the automorphism $\varphi$ of
$L^{(1)}$ given by
$$
\varphi(x):=x\quad\text{and}\quad \varphi(y):=y+\mu_0x+\mu_1+\mu_2 x^{-1}+\mu_3 x^{-2},
$$
where $\mu_0:=1/\lambda$, satisfy:

\begin{enumerate}

\smallskip

\item $(-\rho_3,-\sigma_3)\le\Pred_{P_5}(\rho_3,\sigma_3),\Pred_{Q_5}(\rho_3,\sigma_3)\le(1,-3)$,

\smallskip

\item $\en_{\rho_3,\sigma_3}(P_5)\!=\!(0,1)$, $\en_{\rho_3,\sigma_3}(Q_5)\!=\!(1,1)$,
    $\st_{\rho_3,\sigma_3}(P_5)\!=\!m(3,1)$ and $\st_{\rho_3,\sigma_3}(Q_5)\!=\!n(3,1)$.

\smallskip

\item $\en_{1,-3}(P_5))=m(3,1)$ and $\en_{1,-3}(Q_5))=n(3,1)$,

\smallskip

\item $[P_5,Q_5]=-(y+\mu_0x+\mu_1+\mu_2 x^{-1}+\mu_3 x^{-2})$,

\smallskip

\item $\ell_{-1,2}(P_5)=y+\mu_3 x^{-2}$ and $\ell_{-1,2}(Q_5)=xy+\mu_3 x^{-1}$.

\end{enumerate}

\end{proposition}

\noindent {\bf Idea of the proof:} We map successively $y\mapsto y+\mu_0x$, then
$y\mapsto y+\mu_1$, then $y\mapsto y+\mu_2 x^{-1}$ and finally $y\mapsto y+\mu_3 x^{-2}$,
transforming $\Pred(\rho_3,\sigma_3)$ from $(1,1)$ to $(1,0)$, then to
$(1,-1)$, $(1,-2)$ and finally to $(1,-3)$ or lower. Then $P_5$ and $Q_5$ are as in figure~\ref{figura 3}
and we verify the other conditions.

\begin{proof}
We assert the following: If $P,Q\in L^{(1)}$ satisfy
\begin{enumerate}

\smallskip

  \item[a)] $\ell_{\rho_3,\sigma_3}(P)=\ell_{\rho_3,\sigma_3}(\psi_1(P_4))$ and
        $\ell_{\rho_3,\sigma_3}(Q)=\ell_{\rho_3,\sigma_3}(\psi_1(Q_4))$,

\smallskip

  \item[b)] $(1,-3)<(\rho,\sigma)\le(1,1)$ for $(\rho,\sigma):=\max\{\Pred_P(\rho_3,\sigma_3),
      \Pred_Q(\rho_3,\sigma_3)\}$,

\smallskip

  \item[c)] $v_{\rho,\sigma}([P,Q])\le v_{\rho,\sigma}(x+y)$,

\smallskip

\end{enumerate}
then there exists $R\in L^{(1)}$ and $\lambda_p,\lambda_q\in K^{\times}$, such that
\begin{enumerate}

\smallskip

  \item[(6)] $\ell_{\rho,\sigma}(P)=\lambda_p R^m$ and
  $\ell_{\rho,\sigma}(Q)=\lambda_q R^n$ (consequently $R$ is not a monomial),

\smallskip

  \item[(7)] $\en_{\rho,\sigma}(R)=(3,1)$ and $\st_{\rho,\sigma}(R)=(k,0)$ for some $k=1,2,3$ or $4$,

\smallskip

  \item[(8)] $(\rho,\sigma)=(1,k-3)$,

\smallskip

  \item[(9)] $R=x^3(y-\mu_{4-k}x^{k-3})$ for some $\mu_{4-k}\in K^{\times}$.

\smallskip

\end{enumerate}
Note that $\en_{\rho_2,\sigma_2}(P_4)=m(1,3)$ implies $\st_{\sigma_2,\rho_2}(\psi_1(P_4))=m(3,1)$.
Hence, by a), b) and Proposition~\ref{le basico}, we have
\begin{equation}\label{end point de P}
\en_{\rho,\sigma}(P)=\st_{\rho_3,\sigma_3}(P)=\st_{\sigma_2,\rho_2}(\psi_1(P_4))=m(3,1).
\end{equation}
On the other hand, from b) it follows that $3\rho+\sigma>0$ and $\rho\ge \sigma$, and therefore
$$
v_{\rho,\sigma}(P)=m v_{\rho,\sigma}(3,1)=m(3\rho+\sigma)>0,\quad
v_{\rho,\sigma}(Q)=n(3\rho+\sigma)>0\quad\text{and}\quad \rho>0.
$$
Then
$$
v_{\rho,\sigma}(P)+v_{\rho,\sigma}(Q)-(\rho+\sigma)=(m+n-1)(3\rho+\sigma)+2\rho>\rho=v_{\rho,\sigma}(x+y)
\ge v_{\rho,\sigma}([P,Q]),
$$
and so, by Proposition~\ref{pr v de un conmutador}, we deduce that $[\ell_{\rho,\sigma}(P),\ell_{\rho,\sigma}(Q)]=0$.
Thus statement~(6) follows from Remark~\ref{a remark}. Equality~\eqref{end point de P} implies $\en_{\rho,\sigma}(R)=(3,1)$.
So, since $(0,-1)<(\rho,\sigma)<(0,1)$, by Remark~\ref{starting vs end}, we have
$$
v_{0,-1}(\st_{\rho,\sigma}(R))>v_{0,-1}(\en_{\rho,\sigma}(R))=-1,
$$
which yields $\st_{\rho,\sigma}(R)=(k,0)$ for some $k\in \mathds{Z}$. Then $(\rho,\sigma)=\pm(1,k-3)$, and using again $\rho>0$,
we obtain that $(\rho,\sigma)=(1,k-3)$. Now, by statement~b) we have $(1,-3)<(\rho,\sigma)\le(1,1)$, which gives $k=1,2,3$ or $4$,
completing statement~(7), from which statement~(9) follows directly.

Now we will construct inductively the automorphism $\varphi$. By Proposition~\ref{transformacion afin}(6), we know that there exist
$\lambda,\lambda_p,\lambda_q\in K^{\times}$ such that
$$
\ell_{1,1}(P_4)=\lambda_p R_3^m\qquad\text{and}\qquad \ell_{1,1}(Q_4)=\lambda_q R_3^n,
$$
where $R_3:=y^3(y+\lambda x)$. Hence
\begin{align}
&\ell_{1,1}(\psi_1(P_4))= \psi_1(\ell_{1,1}(P_4))= \lambda_p \psi_1( R_3)^m = \lambda_{\widetilde P} x^{3m}(y-\mu_0 x)^m\label{eqq32}
\shortintertext{and}
&\ell_{1,1}(\psi_1(Q_4))= \psi_1(\ell_{1,1}(Q_4))= \lambda_q \psi_1( R_3)^n = \lambda_{\widetilde Q}x^{3n}(y-\mu_0 x)^n\label{eqq32'},
\end{align}
where $\lambda_{\widetilde P}:=-\lambda_p\lambda$, $\lambda_{\widetilde Q}:=-\lambda_q\lambda$ and
$\mu_0:=\lambda^{-1}$ (by Proposition~\ref{transformacion afin}(1) we can apply $\psi_1$ to $P_4$ and $Q_4$).
Now set
$$
\widetilde P_1:=\varphi_0(\psi_1(P_4)),\quad \widetilde Q_1:=\varphi_0(\psi_1(Q_4))\quad\text{and}\quad
(\ov\rho,\ov\sigma):=\max\bigl(\Pred_{\widetilde P_1}(\rho_3,\sigma_3),\Pred_{\widetilde Q_1}(\rho_3,\sigma_3)\bigr),
$$
where $\varphi_0$ is the automorphism of $L$ defined by
$$
\varphi_0(x):=x\qquad\text{and}\qquad \varphi_0(y):=y+\mu_0 x.
$$
By Proposition~\ref{transformacion afin}(5) we know that $\Pred_{P_4}(1,1)=\Pred_{Q_4}(1,1)=(\rho_2,\sigma_2)$. So, since $\psi_1$
induces on the plane the reflection $\wt{\psi}_1$, defined by $\wt{\psi}_1(i,j):=(j,i)$, we have
$$
\Succ_{\psi_1(P_4)}(1,1)=\Succ_{\psi_1(Q_4)}(1,1)=\ov\psi_1(\rho_2,\sigma_2)=(\rho_3,\sigma_3).
$$
On the other hand, by Proposition~\ref{pr ell por automorfismos},
\begin{equation}\label{eqq33}
\ell_{\rho',\sigma'}(\widetilde P_1)=\ell_{\rho',\sigma'}(\psi_1(P_4))\qquad\text{and}\qquad
\ell_{\rho',\sigma'}(\widetilde Q_1)= \ell_{\rho',\sigma'}(\psi_1(Q_4))
\end{equation}
for all $(1,1)<(\rho',\sigma')<(-1,-1)$. So, since $(\rho_3,\sigma_3)\in\, ](1,1),(-1,-1)[$, we have
\begin{equation*}
(\rho_3,\sigma_3)=\Succ_{\widetilde P_1}(1,1)=\Succ_{\widetilde Q_1}(1,1).
\end{equation*}
Consequently
\begin{equation}\label{pep1}
(-\rho_3,-\sigma_3)\le (\ov\rho,\ov\sigma)<(1,1),
\end{equation}
where the second inequality is strict since $(1,1)\notin \Dir(\widetilde P_1)\cup \Dir(\widetilde Q_1)$, because
by~\eqref{eqq32}, \eqref{eqq32'} and Proposition~\ref{pr ell por automorfismos},
\begin{align*}
\ell_{1,1}(\widetilde P_1) = \ell_{1,1}(\varphi_0(\psi_1(P_4))) = \varphi_0(\ell_{1,1}(\psi_1(P_4))) =
\widetilde \lambda_p x^{3m}\varphi_0(y-\mu_0 x)^m = \lambda_{\widetilde P} x^{3m}y^m
\shortintertext{and}
\ell_{1,1}(\widetilde Q_1) = \ell_{1,1}(\varphi_0(\psi_1(Q_4))) = \varphi_0(\ell_{1,1}(\psi_1(Q_4))) =
\widetilde \lambda_q x^{3n}\varphi_0(y-\mu_0 x)^n = \lambda_{\widetilde Q} x^{3n}y^n.
\end{align*}
Moreover by Proposition~\ref{varphi preserva el Jacobiano} we also have $[\widetilde P_1,\widetilde Q_1]=-(y+\mu_0 x)$; while from
items~(3) and~(4) of Proposition~\ref{transformacion afin}, the fact that $\widetilde\psi_1$ interchanges $\st$ and $\en$, and
equalities~\eqref{eqq33}, we obtain
\begin{equation}\label{pep2}
\en_{\rho_3,\sigma_3}(\widetilde P_1)\!=\!(0,1),\quad \en_{\rho_3,\sigma_3}(\widetilde Q_1)\!=\!(1,1), \quad\st_{\rho_3,\sigma_3}(\widetilde P_1)\!=\!m(3,1),\quad \st_{\rho_3,\sigma_3}(\widetilde Q_1)\!=\!n(3,1).
\end{equation}
If
\begin{equation}\label{desigualdad}
(-\rho_3,-\sigma_3)<(\ov\rho,\ov\sigma)\le(1,-3),
\end{equation}
then we set $\varphi:=\varphi_0$ and $\mu_1=\mu_2=\mu_3:=0$. Consequently, in this case $P_5 = \widetilde P_1$ and $Q_5 = \widetilde Q_1$.
So the previous inequalities are item~(1), while items~(2) and~(4) were proven above. Item~(3) follows from items~(1) and~(2) since
$(-\rho_3,-\sigma_3)<(1,-3)<(\rho_3,\sigma_3)$. Finally item~(5) follows applying Proposition~\ref{pr ell por automorfismos} and the fact that
$\ell_{-1,2}(\psi_1(P_4))=y$ and $\ell_{-1,2}(\psi_1(Q_4))=xy$, which is true by inequalities~\eqref{lo que faltaba} and the definition of
$\widetilde\psi_1$. So, if~\eqref{desigualdad} holds, the proof is finished.

\smallskip

Else $\widetilde P_1$, $\widetilde Q_1$  satisfy the conditions a), b) and c). Hence, by~(6), (7), (8) and~\eqref{pep1}, we have
$$
\Pred_{\widetilde P_1}(\rho_3,\sigma_3) = \Pred_{\widetilde Q_1}(\rho_3,\sigma_3) = (\ov\rho,\ov\sigma)\in \{(1,-2),(1,-1),(1,0)\}.
$$
Since $(\rho_3,\sigma_3)\in \Dir(\widetilde P_1)\cap \Dir(\widetilde Q_1)$, this implies that
\begin{equation}\label{pep4}
\Succ_{\widetilde P_1}(\ov\rho,\ov\sigma) = \Succ_{\widetilde Q_1}(\ov\rho,\ov\sigma) = (\rho_3,\sigma_3).
\end{equation}
Let $i\in \{1,2,3\}$ be such that $\ov\sigma = - i$ and let $\varphi_1$ be the algebra automorphism of $L^{(1)}$, defined by
$\varphi_1(x):=x$ and $\varphi_1(y):= y+\mu_{i+1}x^{-i}$, where $\mu_{i+1}\in K^{\times}$ is as in item~(9). Set
$$
\widetilde P_2:=\varphi_1(\widetilde P_1),\quad \widetilde Q_2:=\varphi_1(\widetilde Q_1)\quad \text{and}\quad
(\ov\rho_1,\ov\sigma_1):=\max\bigl(\Pred_{\widetilde P_2}(\rho_3,\sigma_3),\Pred_{\widetilde Q_2}(\rho_3,\sigma_3)\bigr),
$$
Arguing as before we obtain that $\widetilde P_2$, $\widetilde Q_2$ satisfy either the conditions of the proposition, or
the conditions a), b) and c) with
$$
\Pred_{\widetilde P_2}(\rho_3,\sigma_3) = \Pred_{\widetilde Q_2}(\rho_3,\sigma_3) = (\ov\rho_1,\ov\sigma_1)\in \{(1,-2),(1,-1),(1,0)\}\cap \{(1,-j):j>i\}.
$$
In this case, by a similar argument, we obtain $\widetilde P_3$, $\widetilde Q_3$ satisfying either the conditions of the proposition, or the
conditions a), b) and c) with $\Pred_{\widetilde P_3}(\rho_3,\sigma_3)=(1,-2)$. Finally, if we are in this last situation, then we can use the same
argument again in order to obtain $\widetilde P_4$, $\widetilde Q_4$ satisfying the conditions of the proposition.\endnote{
By Proposition~\ref{pr ell por automorfismos},
\begin{equation}\label{eqq33'}
\ell_{\rho',\sigma'}(\widetilde P_2)=\ell_{\rho',\sigma'}(\widetilde P_1)\qquad\text{and}\qquad
\ell_{\rho',\sigma'}(\widetilde Q_2)= \ell_{\rho',\sigma'}(\widetilde Q_1)
\end{equation}
for all $(\ov\rho,\ov\sigma)<(\rho',\sigma')<(-\ov\rho,-\ov\sigma)$. So, since $(\rho_3,\sigma_3)
\in\, ](\ov\rho,\ov\sigma),(-\ov\rho,-\ov\sigma)[$, by equality~\eqref{pep4} we have
$$
(\rho_3,\sigma_3)=\Succ_{\widetilde P_2}(\ov\rho,\ov\sigma)=\Succ_{\widetilde Q_2}(\ov\rho,\ov\sigma).
$$
Hence
\begin{equation}\label{pep6}
(-\rho_3,-\sigma_3)\le (\ov\rho_1,\ov\sigma_1)<(\ov\rho,\ov\sigma),
\end{equation}
where the second inequality is strict since $(\ov\rho,\ov\sigma)\notin \Dir(\widetilde P_2)\cup \Dir(\widetilde P_2)$, because
by~(6), (9) and Proposition~\ref{pr ell por automorfismos}, there exists $\lambda_{\widetilde P_1},\lambda_{\widetilde Q_1}\in K^{\times}$ such that
\begin{align*}
\ell_{\ov\rho,\ov\sigma}(\widetilde P_2) = \varphi_1(\ell_{\ov\rho,\ov\sigma}(\widetilde P_1) = \lambda_{\widetilde P_1}
\varphi_1\bigl(x^{3m}(y-\mu_{i+1}x^{-i})^m\bigr) = \lambda_{\widetilde P_1} x^{3m}y^m
\shortintertext{and}
\ell_{\ov\rho,\ov\sigma}(\widetilde Q_2) = \varphi_1(\ell_{\ov\rho,\ov\sigma}(\widetilde Q_1) = \lambda_{\widetilde Q_1}
\varphi_1\bigl(x^{3n}(y-\mu_{i+1}x^{-i})^n\bigr) = \lambda_{\widetilde Q_1} x^{3n}y^n.
\end{align*}
Moreover by Proposition~\ref{varphi preserva el Jacobiano} we also have $[\widetilde P_2,\widetilde Q_2]=-(y+\mu_0 x+\mu_{i+1}x^{-i})$; while by
equalities~\eqref{pep2} and equalities~\eqref{eqq33'} with $(\rho',\sigma') = (\rho_3,\sigma_3)$, we obtain
\begin{equation*}
\en_{\rho_3,\sigma_3}(\widetilde P_2)\!=\!(0,1),\quad \en_{\rho_3,\sigma_3}(\widetilde Q_2)\!=\!(1,1), \quad\st_{\rho_3,\sigma_3}(\widetilde P_2)
\!=\!m(3,1),\quad \st_{\rho_3,\sigma_3}(\widetilde Q_2)\!=\!n(3,1).
\end{equation*}
If
\begin{equation}\label{desigualdad1}
(-\rho_3,-\sigma_3)<(\ov\rho_1,\ov\sigma_1)\le(1,-3),
\end{equation}
then we set $\varphi:=\varphi_1\circ \varphi_0$ and $\mu_j:=0$ for $j\notin \{0,i+1\}$. Consequently, in this case $P_5 = \widetilde P_2$ and
$Q_5 = \widetilde Q_2$. So the previous inequalities are item~(1), while items~(2) and~(4) were proven above. Item~(3) follows from items~(1)
and~(2) since $(-\rho_3,-\sigma_3)<(1,-3)<(\rho_3,\sigma_3)$. Finally item~(5) follows applying Proposition~\ref{pr ell por automorfismos} and the
fact that $\ell_{-1,2}(\widetilde P_1)=y$ and $\ell_{-1,2}(\widetilde Q_1)=xy$. So, if~\eqref{desigualdad1} holds, the proof is finished.

\smallskip

Else $\widetilde P_2$, $\widetilde Q_2$  satisfy the conditions a), b) and c). Hence, by~(6), (7), (8) and~\eqref{pep6}, we have
$$
\Pred_{\widetilde P_2}(\rho_3,\sigma_3) = \Pred_{\widetilde Q_2}(\rho_3,\sigma_3) = (\ov\rho_1,\ov\sigma_1)\in \{(1,-2),(1,-1),(1,0)\}\cap
\{(1,-j):j>i\}.
$$
In this case, arguing as before we obtain $\widetilde P_3$, $\widetilde Q_3$ satisfying either the conditions of the Proposition, or the
conditions a), b) and c) with $\Pred_{\widetilde P_3}(\rho_3,\sigma_3)=(1,-2)$. Finally, if we are in this last situation, then we can use the same
argument again in order to obtain $\widetilde P_4$, $\widetilde Q_4$ satisfying the conditions of the Proposition.}
\end{proof}

\begin{figure}[htb]
\centering
\begin{tikzpicture}[scale=0.78]
\fill[gray!20] (13,3.5) -- (13.5,3.5) -- (14,4)--(13,5);
\fill[gray!20] (13,0) -- (14,0) -- (14.5,0.5) -- (13,2);
\fill[gray!20] (5,3.5) -- (9.5,5) -- (5.5,4)--(4.5,3.5);
\fill[gray!20] (5,0) -- (11,2) -- (5,0.5)--(4,0);
\draw[step=.5cm,gray,very thin] (13,3.5) grid (14.7,5.2);
\draw [->] (13,3.5) -- (15,3.5) node[anchor=north]{$x$};
\draw [->] (13,3.5) --  (13,5.5) node[anchor=east]{$y$};
\draw[step=.5cm,gray,very thin] (13,0) grid (14.7,2.2);
\draw [->] (13,0) -- (15,0) node[anchor=north]{$x$};
\draw [->] (13,0) --  (13,2.5) node[anchor=east]{$y$};
\draw[step=.5cm,gray,very thin] (5,3.5) grid (10.7,5.2);
\draw [->] (4.3,3.5) -- (11,3.5) node[anchor=north]{$x$};
\draw [->] (5,3.5) --  (5,5.5) node[anchor=east]{$y$};
\draw[step=.5cm,gray,very thin] (5,0) grid (11.2,2.2);
\draw [->] (3.8,0) -- (11.5,0) node[anchor=north]{$x$};
\draw [->] (5,0) --  (5,2.5) node[anchor=east]{$y$};
\draw [->] (5,3.5) --  (5,5.5) node[anchor=east]{$y$};
\draw [->] (5,0) --  (5,2.5) node[anchor=east]{$y$};
\draw (14,0) --  (14.5,0.5) node[fill=white,right=2pt]{$P_6$} -- (13,2);
\draw (13.5,3.5) --  (14,4) node[fill=white,right=2pt]{$Q_6$} -- (13,5);
\draw (5,0) --  (11,2) node[fill=white,right=2pt]{$P_5$} -- (5,0.5)--(4,0);
\draw (5,3.5) --  (9.5,5) node[fill=white,right=2pt]{$Q_5$} -- (5.5,4)--(4.5,3.5);
\filldraw [gray]  (13.5,0)    circle (2pt)
                  (14,0)      circle (2pt)
                  (13,0.5)    circle (2pt)
                  (13.5,0.5)  circle (2pt)
                  (14.5,0.5)  circle (2pt)
                  (13,1)    circle (2pt)
                  (13.5,1)      circle (2pt)
                  (14,1)      circle (2pt)
                  (13,1.5)    circle (2pt)
                  (13.5,1.5)  circle (2pt)
                  (13,2)    circle (2pt)
                  (13.5,3.5)    circle (2pt)
                  (13,4)    circle (2pt)
                  (14,4)      circle (2pt)
                  (13,4.5)    circle (2pt)
                  (13.5,4.5)    circle (2pt)
                  (13,5)      circle (2pt)
                  (5,0)      circle (2pt)
                  (13,0)      circle (2pt)
                  (5,3.5)      circle (2pt)
                  (13,3.5)      circle (2pt);
\filldraw [gray]  (5.5,0.5)    circle (2pt)
                  (4.5,0)      circle (2pt)
                  (4,0)    circle (2pt)
                  (5,0.5)  circle (2pt)
                  (6.5,0.5)  circle (2pt)
                  (7,1)    circle (2pt)
                  (8,1)      circle (2pt)
                  (7.5,1)      circle (2pt)
                  (9,1.5)    circle (2pt)
                  (9.5,1.5)  circle (2pt)
                  (11,2)    circle (2pt)
                  (4.5,3.5)    circle (2pt)
                  (5.5,4)    circle (2pt)
                  (6.5,4)      circle (2pt)
                  (7.5,4.5)    circle (2pt)
                  (8,4.5)    circle (2pt)
                  (9.5,5)      circle (2pt);
\draw[->] (11.2,3) .. controls (11.95,3.25) .. (12.7,3);
\draw (11.7,3.5) node[right,text width=2cm]{$\psi_3$};
\end{tikzpicture}
\caption{Illustration of Proposition~\ref{pr 88} for $j=1$.}
\label{figura 3}
\end{figure}
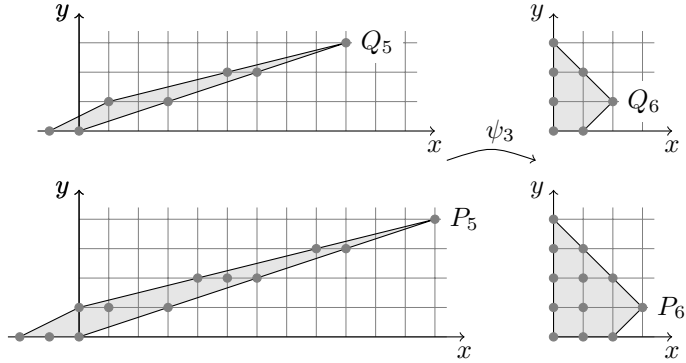

\begin{proposition}\label{pr 88} 
Let $P_5$, $Q_5$, $j$,
$m$, $n$, $\mu_0$, $\mu_1$, $\mu_2$ and $\mu_3$ be as in the previous
proposition and let $\psi_3$ be as in Remark~\ref{capsula convexa}. Define
$P_6:=\psi_3(P_5)$ and $Q_6:=\psi_3(Q_5)$. The following facts hold:
\begin{enumerate}

\smallskip

\item $\Dir(P_6)\cap\mathfrak{V}_{>0} = \Dir(Q_6)\cap\mathfrak{V}_{>0} = \{(j,1)\}$,

\smallskip

\item $\st_{j,1}(P_6)=(3,1)$, $\st_{j,1}(Q_6)=(2,1)$, $\en_{j,1}(P_6)=(0,m)$ and $\en_{j,1}(Q_6)=(0,n)$,

\smallskip

\item $P_6,Q_6\in L$,

\smallskip

\item $\ell_{1,-1}(P_6)=x^3 y+\mu_3 x^{2}$ and $\ell_{1,-1}(Q_6)=x^2y+\mu_3 x$,

\smallskip

\item $[P_6,Q_6]=x^4 y+\mu_0 +\mu_1 x+\mu_2 x^2+\mu_3 x^3$.

\end{enumerate}
\end{proposition}

\begin{proof} Recall that
$(\rho_3,\sigma_3)=(-j,3j+1)$ with $j\in \mathds{N}$ and that
$\ov{\psi}_3(\rho_3,\sigma_3)= (j,1)$.
Note that, by Remark~\ref{capsula convexa},
$$
\Supp(\ell_{j,1}(\psi_3(P))) = \Supp(\psi_3(\ell_{\rho_3,\sigma_3}(P))) =
\widetilde{\psi}_3(\Supp(\ell_{\rho_3,\sigma_3}(P))) \qquad\text{for all $P\in L^{(1)}$.}
$$
Hence
$$
\st_{j,1}(\psi_3(P))=\widetilde{\psi}_3(\en_{\rho_3,\sigma_3}(P))\quad\text{and}\quad
\en_{j,1}(\psi_3(P)) =\widetilde{\psi}_3(\st_{\rho_3,\sigma_3}(P))
\qquad\text{for all $P\in L^{(1)}$.}
$$
By Proposition~\ref{P4}(2), applying this equality with $P=P_5$ and $P=Q_5$ we get statement~(2). Using
Proposition~\ref{P4}(3) and arguing as in the proof of Proposition~\ref{transformacion afin}(1) we obtain
statement~(3).
Statement~(4) follows from Proposition~\ref{P4}(5) and~\eqref{ve de psi}.
Next we prove statement~(1). Note that by statements~(2) and~(4)
$$
\en_{1,-1}(P_6)\!=\!\st_{j,1}(P_6)\!\ne\!\en_{j,1}(P_6) = (0,m)\quad\text{and}\quad
\en_{1,-1}(Q_6)\!=\!\st_{j,1}(Q_6)\!\ne\!\en_{j,1}(Q_6)\!=\!(0,n),
$$
and so, by Proposition~\ref{le bbasico} we have
$\Dir(P_6)\cap\mathfrak{V}_{>0} = \Dir(Q_6)\cap\mathfrak{V}_{>0} = \{(j,1)\}$, as desired. Finally,
statement~(5) follows from Propositions~\ref{varphi preserva el Jacobiano} and~\ref{P4}(4).
\end{proof}

Collecting the main results in this section we have arrived at the following theorem:

\begin{theorem} Assume that $B=16$ and let $m,n$ and $(P,Q)$ be as in Corollary~\ref{forma final en L}. Interchanging $P$ with $Q$ if necessary we can assume that $m>n>1$. Then there exists $j\in\mathds{N}$ such that $m=3j+1$ and $n=2j+1$. Furthermore performing on $(P,Q)$ a series of operations as described at the beginning of this section we obtain $P_6, Q_6\in L$ and $\mu_0,\mu_1,\mu_2,\mu_3\in K$
with $\mu_0\ne 0$ such that
\begin{align*}
&[P_6,Q_6]=x^4 y\!+\mu_0\! +\!\mu_1 x\!+\!\mu_2 x^2\!+\!\mu_3 x^3,\\
& \ell_{1,-1}(P_6)=x^3 y\!+\!\mu_3 x^{2},\quad\ell_{1,-1}(Q_6)=x^2y\!+\!\mu_3
x,\\
& \Dir(P_6)\cap\mathfrak{V}_{>0} = \Dir(Q_6)\cap\mathfrak{V}_{>0} = \{(j,1)\},\\
& \st_{j,1}(P_6)=(3,1),\quad \st_{j,1}(Q_6)=(2,1),\quad \en_{j,1}(P_6)=(0,m)\quad\text{and}\quad \en_{j,1}(Q_6)=(0,n).
\end{align*}
\end{theorem}

\begin{remark} In this last section we have managed to reduce the degrees of $P,Q$ from $16m,16n$ to $m,n$.
The first step, reduction of degrees to $4m,4n$, was done in Proposition~\ref{transformacion afin} and it
is the algebraic
process that correspond to the reduction
done by Moh in the case $m=3,n=4$ using~\cite{M}*{Propositions~6.3 and~6.4}. It is based on
Corollary~\ref{fracciones de F1} and the fact that $b=q$, where $(a/l,b):=\frac{1}{m}\st_{\rho_0,\sigma_0}(P)$
and $\en_{\rho_0,\sigma_0}(F) = \frac pq (a/l,b)$. In fact, Corollary~\ref{fracciones de F1} allows us to
write $\ell_{1,-2}(P)$ as an $mq$-th power of a linear polynomial $R$, and $b=q$ allows us to ``cut"
the support until $v_{1,-3}(P_1)=0$ and $v_{1,-3}(Q_1)=0$. This correspond geometrically to squeeze the
support in order to obtain the following form:
\begin{figure}[htb]
\centering
\begin{tikzpicture}
\draw [->] (-2,0) -- (4,0) node[anchor=north]{$x$};
\draw [->] (0,0) --  (0,3) node[anchor=east]{$y$};
\draw (0,0) --  (4,2) node[fill=white,right=2pt]{$(12m,4m)$} -- (0,2);
\draw [dotted] (0.5,0.25)--(-2,0.25);
\draw [dotted] (1,0.5)--(-2,0.5);
\draw [dotted] (1.5,0.75)--(-2,0.75);
\draw [dotted] (2,1)--(-2,1);
\draw [dotted] (2.5,1.25)--(-2,1.25);
\draw [dotted] (3,1.5)--(-2,1.5);
\draw [dotted] (3.5,1.75)--(-2,1.75);
\draw (1,1.2) node[fill=white]{$P_1$};
\end{tikzpicture}
\caption{$v_{1,-3}(P_1)=0$.}
\end{figure}
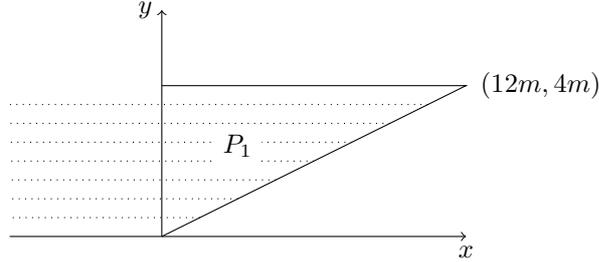

This figure is the same as Figure~11 in~\cite{A3}*{page~2779}. But in cases different from $16m,16n$ this cannot
always be obtained. For instance, consider the example of Remark~\ref{teorema central con rho mas sigma < 0}:
$$
P=x^{-1} + x^3 y (2 + 18 x^2 y + 36 x^4 y^2) + x^9 y^3 (8 + 72 x^2 y + 216 x^4 y^2 + 216 x^6 y^3)
$$
and
$$
Q=x^2 y + x^6 y^2 (1 + 6 x^2 y + 9 x^4 y^2).
$$
Here $[P,Q]=-1$, $m=3$, $n=2$, $b=4$ and $q=2$; so Corollary~\ref{fracciones de F1} applies with $(\rho_0,\sigma_0)=(-2,7)$. In order to
obtain the desired squeezing we need $\ell_{-1,2}(P)=\lambda_p x^2(x^2 y-\lambda)^4$ for some
$\lambda,\lambda_p\in K^{\times}$ but this cannot be achieved.
\end{remark}
\paragraph{Acknowledgment}
We wish to thank the referee for useful comments and suggestions that helped us to reduce the size
of the paper and improve the presentation.
We also thank Leonid Makar-Limanov  for
the argument in the proof of Proposition~\ref{primera condicion estandar}.

\setcounter{theorem}{0}
\setcounter{equation}{0}
\renewcommand\thesection{A}

\theendnotes

\end{document}